\documentclass[twoside,11pt,reqno]{amsart}
\usepackage{amsmath,amssymb,amscd,mathrsfs,epic}
\usepackage{latexsym,amsthm,amscd}

\makeatletter

\newif\iffomin@
\fomin@false

\@addtoreset{equation}{section}

\newtheorem{Proposition}{Proposition}[section]
\newtheorem{Lemma}[Proposition]{Lemma}
\newtheorem{Theorem}[Proposition]{Theorem}
\newtheorem{Corollary}[Proposition]{Corollary}

\newtheorem{Remark}[Proposition]{Remark}

\def\phantomsubsection#1{\vspace{2mm}\noindent{\bf\em #1.}}

\newcommand{\fin}{{f\:\!\!in}}

\def\longline{\text{--}}

\def\lonestar{{{^{*\!}}}}
\def\red{\operatorname{red}}
\def\upperred{\red}
\def\lowerred{\red}
\def\caps{{\operatorname{caps}}}
\def\cups{{\operatorname{cups}}}
\def\op{{\operatorname{op}}}
\def\wt{{\operatorname{wt}}}
\def\flip{\operatorname{flip}}

\newcommand{\sF}{\mathscr{F}\!}

\newcommand{\cO}{\mathcal{O}}

\newbox\squ  
\setbox\squ=\hbox{\vrule width.6pt
   \vbox{\hrule height.6pt width.4em\kern1ex
         \hrule height.6pt}%
   \vrule width.6pt}

\def\Seq{\mathscr{R}}

\def\deg{\operatorname{deg}}
\def\defect{\operatorname{def}}
\def\sh{{\operatorname{sh}}}
\def\rk{{\operatorname{rk}}}
\def\op{\operatorname{op}}

\def\down{{\scriptstyle\vee}}
\def\up{{\scriptstyle\wedge}}
\def\cross{{\scriptstyle\times}}

\def\Proj#1{\operatorname{Proj}(#1)}

\def\mod#1{#1\!\operatorname{-mod}}
\def\Mod#1{#1\!\operatorname{-mod}_{l\!f}}

\def\rad{\operatorname{rad}}

\def\C{{\mathbb C}}
\def\E{{\mathbf e}}
\def\F{{\mathbf f}}
\def\G{{\mathbf g}}

\def\Z{{\mathbb Z}}

\def\0{{\bar 0}}
\def\1{{\bar 1}}

\def\HHH{{\operatorname{H}}}
\def\hom{{\operatorname{Hom}}}

\def\End{{\operatorname{End}}}

\def\res{{\operatorname{res}}}
\def\ind{{\operatorname{ind}}}
\def\ires{{i\text{-}\!\operatorname{res}}}
\def\iind{{i\text{-}\!\operatorname{ind}}}

\def\adj{{\operatorname{adj}}}
\def\residue{{\operatorname{cont}}}
\def\im{{\operatorname{im}}}
\def\soc{{\operatorname{soc}\:}}

\def\eps{{\varepsilon}}
\def\phi{{\varphi}}

\def\la{{\lambda}}
\def\La{{\Lambda}}
\def\Ga{{\Gamma}}
\def\ga{{\gamma}}

\def\De{{\Delta}}
\def\al{{\alpha}}
\def\be{{\beta}}

\def\bi{\text{\boldmath$i$}}
\def\bj{\text{\boldmath$j$}}
\def\bmathj{\text{\boldmath$\jmath$}}
\def\bk{\text{\boldmath$k$}}

\def\tP{{\mathtt P}}
\def\tQ{{\mathtt Q}}

\def\tT{{\mathtt T}}
\def\tU{{\mathtt U}}
\def\tV{{\mathtt V}}
\def\tX{{\mathtt X}}
\def\tY{{\mathtt Y}}
\def\tS{{\mathtt S}}

\newdimen\hoogte    \hoogte=14pt    
\newdimen\breedte   \breedte=14pt   
\newdimen\dikte     \dikte=0.5pt    

\newenvironment{young}{\begingroup
       \def\vr{\vrule height0.8\hoogte width\dikte depth 0.2\hoogte}
       \def\fbox##1{\vbox{\offinterlineskip
                    \hrule height\dikte
                    \hbox to \breedte{\vr\hfill##1\hfill\vr}
                    \hrule height\dikte}}
       \vbox\bgroup \offinterlineskip \tabskip=-\dikte \lineskip=-\dikte
            \halign\bgroup &\fbox{##\unskip}\unskip  \crcr }
       {\egroup\egroup\endgroup}
\def\diagram#1{\relax\ifmmode\vcenter{\,\begin{young}#1\end{young}\,}\else%
              $\vcenter{\,\begin{young}#1\end{young}\,}$\fi}
\begin{document}

\title[Walled Brauer algebras]{\boldmath 
Gradings on walled Brauer algebras and Khovanov's arc algebra}
\author{Jonathan Brundan and Catharina Stroppel}

\address{Department of Mathematics, University of Oregon, Eugene, OR 97403, USA}
\email{brundan@uoregon.edu}
\address{Department of Mathematics, University of Bonn, 53115 Bonn, Germany}
\email{stroppel@math.uni-bonn.de}

\thanks{2010 {\it Mathematics Subject Classification}: 17B10, 16S37.}
\thanks{First author supported in part by NSF grant no. DMS-0654147.}

\begin{abstract}
We introduce some $\Z$-graded versions of the walled Brauer algebra
$B_{r,s}(\delta)$,
working over a field of characteristic zero. 
This allows us to prove that $B_{r,s}(\delta)$ is Morita equivalent to
an idempotent truncation of a certain
infinite dimensional version of Khovanov's arc algebra.
We deduce that the walled Brauer algebra is Koszul
whenever $\delta \neq 0$.
\end{abstract}
\maketitle
\tableofcontents

\section{Introduction and weight dictionary}

Let $\delta \in \C$ be a fixed parameter.
The {\em walled Brauer algebra} $B_{r,s}(\delta)$ is
a subalgebra of the classical 
Brauer algebra $B_{r+s}(\delta)$; see $\S$\ref{s2} for the precise definition.
It was introduced independently by Turaev \cite{T} and Koike \cite{K}
in the late 1980s
motivated in part by a Schur-Weyl duality between
$B_{r,s}(m)$ and the general linear group $GL_m(\C)$ 
arising from
mutually commuting actions on the ``mixed'' tensor space
$V^{\otimes r} \otimes W^{\otimes s}$, where $V$
is the natural representation of $GL_m(\C)$ and $W := V^*$;
see also \cite{BCH}.

If $\delta \notin \Z$ then the algebra $B_{r,s}(\delta)$ is semisimple,
and its representation theory 
can be described using character-theoretic methods 
analogous to ones used in the study of the complex representation theory
of the symmetric group; see e.g. \cite{H}, \cite{N}.
In this article we are mainly interested in the non-semisimple cases,
when 
the representation theory is still surprisingly tractable.
For example, it is known that $B_{r,s}(\delta)$ is a quasi-hereditary algebra
if $\delta \neq 0$, and in these cases
explicit closed formulae for 
the composition multiplicites of standard modules
and the corresponding Kazhdan-Lusztig polynomials
have been worked out by Cox and De Visscher; see \cite[Theorem 4.10]{CD} and
\cite[Corollary 9.3]{CD}.

Assuming $\delta \in \Z$, the work just mentioned reveals 
some striking combinatorial coincidences between the
representation theory of $B_{r,s}(\delta)$ and
a certain infinite dimensional 
generalisation $K(\delta)$ of Khovanov's arc algebra. 
The latter algebra was introduced 
originally in a special case in \cite{K2}
where it plays a role in Khovanov's categorification of the Jones polynomial.
Khovanov's definition was extended in \cite{BS1} to include various infinite dimensional algebras like $K(\delta)$. 

The main goal of this article is to explain these combinatorial coincidences
by proving that $B_{r,s}(\delta)$ is Morita equivalent to an idempotent
truncation of $K(\delta)$.
This adds a new twist to the results of \cite{BS3}, \cite{BS4} which
construct
equivalences between 
other generalisations of Khovanov's arc algebra from \cite{BS1} and 
various naturally occuring algebras and categories including 
level two cyclotomic Hecke algebras and Khovanov-Lauda-Rouquier algebras,
perverse sheaves on Grassmannians,
and rational representations of
the supergroup
$GL_{m|n}(\C)$.

To formulate our results in more detail, and to make the definition of the algebra $K(\delta)$ precise, we must recall the key combinatorial 
observation from \cite[$\S$4]{CD}.
This gives a ``weight dictionary''
between the 
bipartitions which appear when studying the representation theory of $B_{r,s}(\delta)$ and the weight diagrams of \cite{BS1} which 
appear in the definition of the algebra $K(\delta)$.
Recall a {\em bipartition}
means a pair $\lambda = (\lambda^{\mathrm L},\lambda^{\mathrm R})$ of partitions.
Let $\Lambda$ be the set of all bipartitions and 
\begin{align}\label{ladef}
\Lambda_{r,s} &:= \left\{\lambda \in \Lambda\:\big|\:|\lambda^{\mathrm L}|
=r-t \text{ and } |\lambda^{\mathrm R}| = s-t\text{ for }0 \leq t \leq \min(r,s)\right\},\\
\dot\Lambda_{r,s} &:= \left\{\begin{array}{ll}
\Lambda_{r,s}&\text{if $\delta \neq 0$ or $r \neq s$ or $r=s=0$,}\\
\Lambda_{r,s}\setminus\{(\varnothing,\varnothing)\}&\text{if $\delta
  =0$
and $r=s > 0$.}
\end{array}\right.\label{hing}
\end{align}
For each $\la \in \La_{r,s}$, there is a 
{\em cell module} $C_{r,s}(\lambda)$ for
$B_{r,s}(\delta)$; see $\S$\ref{s2} below.
Moreover $C_{r,s}(\lambda)$ has irreducible head denoted
$D_{r,s}(\lambda)$ whenever $\la \in \dot \La_{r,s}$, and 
the modules $\{D_{r,s}(\lambda)\:|\:\lambda \in \dot \La_{r,s}\}$
give a complete set of pairwise non-isomorphic irreducible $B_{r,s}(\delta)$-modules.

On the other hand a {\em weight diagram} in the sense of \cite{BS1} is a 
horizontal ``number line'' with vertices 
at all integers labelled by 
one of the symbols $\circ, \up, \down$ and $\cross$; we require
moreover that it is impossible to find a vertex labelled $\down$ to the left of a vertex labelled $\up$ outside of some finite subset of the vertices.
To identify bipartitions with certain weight diagrams,
take $\lambda \in \Lambda$ and let
\begin{align}
I_\up(\la)&:= \{\lambda^{\mathrm L}_1,\lambda^{\mathrm L}_2-1,\lambda^{\mathrm L}_3-2,\dots\},\label{Iu}\\
 I_\down(\la)&:=\{1-\delta-\lambda^{\mathrm R}_1, 2-\delta-\lambda^{\mathrm R}_2,3-\delta-\lambda^{\mathrm R}_3,\dots\},\label{Id}
\end{align}
where 
$\lambda^{\mathrm L}_1 \geq \lambda^{\mathrm L}_2 \geq \cdots$ are the parts of $\lambda^{\mathrm L}$
and
$\lambda^{\mathrm R}_1 \geq \lambda^{\mathrm R}_2 \geq \cdots$ are the parts of $\lambda^{\mathrm R}$.
Then we identify $\la$ with the weight diagram whose $i$th vertex is labelled
\begin{equation}\label{dict}
\left\{
\begin{array}{ll}
\circ&\text{if $i$ does not belong to either $I_\down(\la)$ or $I_\up(\la)$,}\\
{\up} &\text{if $i$ belongs to $I_\up(\la)$ but not to
$I_\down(\la)$,}\\
{\down} &\text{if $i$ belongs to $I_\down(\la)$ but not to
$I_\up(\la)$,}\\
\cross&\text{if $i$ belongs to both $I_\down(\la)$ and $I_\up(\la)$,}
\end{array}\right.
\end{equation}
for each $i \in \Z$.
Of course this depends implicitly on the fixed parameter $\delta \in \Z$.
For example, if $\delta=-2$ then
\begin{align*}
(\varnothing,\varnothing)&=
\!\!\!\begin{picture}(240,12)
\put(90.5,14){$_0$}
\put(144.5,14){$_{1-\delta}$}
\put(8,-.3){$\cdots$}
\put(227,-.3){$\cdots$}
\put(25,2.3){\line(1,0){198}}
\put(30,-2.4){$\up$}
\put(50,-2.4){$\up$}
\put(70,-2.4){$\up$}
\put(90,-2.4){$\up$}
\put(130,-.4){$\circ$}
\put(110,-.4){$\circ$}
\put(150,2.4){$\down$}
\put(170,2.4){$\down$}
\put(190,2.4){$\down$}
\put(210,2.4){$\down$}
\end{picture}\,
\\ 
((2^2,1),(3,2))&=
\!\!\!\begin{picture}(240,0)
\put(8,-.3){$\cdots$}
\put(227,-.3){$\cdots$}
\put(25,2.3){\line(1,0){198}}
\put(30,-2.4){$\up$}
\put(70,-2.4){$\up$}
\put(110,-2.4){$\up$}
\put(50,-.4){$\circ$}
\put(130,.4){$\cross$}
\put(150,-.4){$\circ$}
\put(170,-.4){$\circ$}
\put(90,2.4){$\down$}
\put(190,2.4){$\down$}
\put(210,2.4){$\down$}
\end{picture}\,
\end{align*}
 (where all the omitted vertices on the left are labelled $\up$
and the ones on the right are labelled $\down$).
In this way,
the set $\Lambda$ of bipartitions is identified with a 
set of weight diagrams in the sense of \cite{BS1}. 
Now the general construction from \cite[$\S$4]{BS1} attaches 
a positively graded, locally unital algebra $K(\delta)$ to this
set 
$\Lambda$ of weight diagrams; see $\S\S$\ref{s4}--\ref{s4b} below.
Finally for each $\la \in \La$
there are some {\em standard modules} 
$V(\lambda)$ for $K(\delta)$,
each of which has one-dimensional head denoted $L(\la)$.
The modules $\{L(\la)\:|\:\la \in \La\}$ give a complete set of pairwise
non-isomorphic irreducible $K(\delta)$-modules.

The main result of the article can now be formulated as follows.

\begin{Theorem}\label{fin}
There exists a Morita equivalence between the walled Brauer algebra $B_{r,s}(\delta)$ and the finite dimensional algebra $K_{r,s}(\delta) := e_{r,s} K(\delta) e_{r,s}$, where
\begin{equation}\label{ers}
e_{r,s} := 
\sum_{\lambda \in \dot\Lambda_{r,s}} e_\lambda\in K(\delta).
\end{equation}
The equivalence can be chosen so that the cell module
$C_{r,s}(\lambda)$ 
corresponds to the $K_{r,s}(\delta)$-module
$V_{r,s}(\lambda) := 
e_{r,s} V(\lambda)$ for each $\lambda \in \Lambda_{r,s}$. Hence the irreducible
module $D_{r,s}(\lambda)$ corresponds to 
$L_{r,s}(\lambda) := e_{r,s} L(\lambda)$ for each $\lambda \in \dot\Lambda_{r,s}$.
\end{Theorem}

The set $\La_{r,s}$ is an upper set in $\La$ with respect to the
natural Bruhat ordering. Combined with known properties of 
$K(\delta)$ from \cite{BS2}, it follows that
the idempotent truncation 
$K_{r,s}(\delta)$ 
is a standard Koszul algebra when $\La_{r,s}
= \dot\La_{r,s}$, i.e. it is both graded quasi-hereditary
and Koszul; see Theorem~\ref{iscell}.
So we obtain the following corollary.

\begin{Corollary}
If $\Lambda_{r,s} = \dot\Lambda_{r,s}$ then
$B_{r,s}(\delta)$ admits a Koszul grading.
In particular, $B_{r,s}(\delta)$ is standard Koszul when $\delta \neq 0$.
\end{Corollary}

The Koszul grading on $B_{r,s}(\delta)$ is only
unique up to automorphism (see \cite[$\S$2.5]{BGS}), 
and we don't expect that there is any distinguished way to realise it explicitly.
Instead, we construct some graded algebras $B_R(\delta)$, which we call
{\em graded walled Brauer algebras},
indexed by elements $R$ of the set $\Seq_{r,s}$ of sequences consisting 
of $r$ $E$'s and $s$ $F$'s.
We show that $B_R(\delta)$ is
graded Morita
equivalent 
to $K_{r,s}(\delta)$ in a canonical way; see Theorem~\ref{morita}.
Moreover we show that $B_R(\delta)$ is 
isomorphic (though not canonically)
to $B_{r,s}(\delta)$ as an ungraded algebra; 
see Corollary~\ref{isocor}.
Putting these two things together, we obtain a family of
equivalences of categories
\begin{equation}\label{gs}
\G_R:\mod{K_{r,s}(\delta)} \rightarrow \mod{B_{r,s}(\delta)},
\end{equation}
one for each $R \in \Seq_{r,s}$,
each of which satisfies the hypotheses of Theorem~\ref{fin};
see Theorem~\ref{mt}.

To give a little more idea about the definition of $B_R(\delta)$,
recall
for the walled Brauer algebra that there are two basic branching rules, one
from $B_{r,s}(\delta)$ to $B_{r-1,s}(\delta)$ (``$E$-restriction'') and the other 
from $B_{r,s}(\delta)$ to
$B_{r,s-1}(\delta)$ (``$F$-restriction''). 
The sequences $R \in \Seq_{r,s}$ also parametrise
the $\binom{r+s}{r}$ different ways
to restrict from $B_{r,s}(\delta)$ all the way down to the trivial subalgebra
$B_{0,0}(\delta)$ by iterating these basic branching rules,
e.g. 
$E^r F^s$ and $F^s E^r$ encode 
\begin{align*}
B_{r,s}(\delta) \downarrow B_{r,s-1}(\delta) \downarrow \cdots \downarrow B_{r,0}(\delta)
\downarrow B_{r-1,0}(\delta) \downarrow \cdots \downarrow
B_{0,0}(\delta),\\
B_{r,s}(\delta) \downarrow B_{r-1,s}(\delta) \downarrow \cdots \downarrow B_{0,s}(\delta)
\downarrow B_{0,s-1}(\delta) \downarrow \cdots \downarrow
B_{0,0}(\delta),
\end{align*}
respectively.
The graded walled Brauer 
algebra $B_R(\delta)$ has a distinguished basis indexed by
certain paths
in the Bratelli diagram which describes the branching of cell
modules with respect to the sequence $R$;
see $\S\ref{s5}$ for some 
examples.
Multiplication is defined by representing the basis elements
diagrammatically then applying an explicit combinatorial procedure
similar in spirit to the multiplication in Khovanov's arc algebra;
this diagrammatics is quite different from the sort that appears in the
definition of $B_{r,s}(\delta)$.
We also construct a graded cellular basis for each
$B_R(\delta)$, hence defining 
{\em graded cell modules}
$C_R(\la)$; see Theorem~\ref{maincell}
and Corollary~\ref{othercellid}.
Moreover we compute the graded dimensions of the irreducible representations;
see Theorem~\ref{irreddimensions}.

Once the algebras $B_R(\delta)$ have been constructed, 
the main problem is to
prove that $B_{r,s}(\delta) \cong B_R(\delta)$.
We do this by
exploiting the
super analogue of the mixed Schur-Weyl duality mentioned at the
beginning of the introduction. Assume for this that
$\delta = m-n$ for 
integers $m,n \geq 0$.
Let $G$ denote the general linear supergroup $GL_{m|n}(\C)$ with
natural
module $V$ and dual natural module $W$.
There is a natural right action of $B_{r,s}(\delta)$ on
$V^{\otimes r} \otimes W^{\otimes s}$ by $G$-module endomorphisms.
We first show that
the induced homomorphism
\begin{equation}
\Psi_{r,s}^{m,n}:B_{r,s}(\delta) \rightarrow 
\End_G(V^{\otimes r} \otimes W^{\otimes s})^{\op}
\end{equation}
is surjective; see Theorem~\ref{main1}. Moreover it is an isomorphism if and
only if \begin{equation}\label{prefer}
r+s < (m+1)(n+1).
\end{equation}
Then we apply the results of \cite{BS4} to 
describe
$\End_G(V^{\otimes r} \otimes W^{\otimes s})^{\op}$ in purely
diagrammatic terms. Assuming (\ref{prefer}) holds, we get exactly
the algebra $B_R(\delta)$, hence 
obtain
the desired 
isomorphism
$B_{r,s}(\delta)
\cong
B_R(\delta)$.
We expect that this isomorphism is 
independent of the particular choice of $m$ and $n$, but 
are unable to prove this at present; see Remark~\ref{drawback} for further
discussion of this issue.

Our approach has several other consequences. 
First, 
if either $m = 0$ or $n=0$, i.e. $G$ is just the general linear group
not the general linear supergroup, we are able to give an explicit
description of 
$\ker \Psi_{r,s}^{m,n}$, showing it is generated by an idempotent; see
Theorem~\ref{lehrerzhang}. This result is an analogue for the walled Brauer algebra
of a result about Brauer algebras 
obtained recently by Lehrer and Zhang \cite{LZ}.
For $m,n > 0$, the ideal $\ker \Psi_{r,s}^{m,n}$ need not be generated by an
idempotent; see Remark~\ref{tempt}.

We also determine the irreducible
representations of the quotient algebra $B_{r,s}(\delta) / \ker
\Psi_{r,s}^{m,n}$, showing up to
isomorphism that they are modules
$D_{r,s}(\la)$ for $\la \in \dot\La_{r,s}$ such that $\la$ is an
$(m,n)$-cross bipartition in the sense of Comes and Wilson \cite{CW};
see Theorem~\ref{irrclass2}.
This allows us to recover the classification of indecomposable summands of
$V^{\otimes r} \otimes W^{\otimes s}$ obtained originally in
\cite{CW} by a different approach; see Theorem~\ref{soupedup}. 
Our approach gives a little more information about these interesting modules,
e.g. we compute their irreducible socles and heads explicitly.

The work of Comes and Wilson
is motivated by the study of Deligne's
tensor category
$\underline{\operatorname{Re}}\!\operatorname{p}(GL_\delta)$;
see also \cite{Deligne}, \cite{Knop}.
For $\delta \in \Z$,
we conjecture 
that Deligne's category
is equivalent (on forgetting the
tensor structure) to the
category $\Proj{K(\delta)}$ of (locally unital) finitely generated projective
$K(\delta)$-modules.
The main obstacle preventing us from proving this conjecture using the
techniques of the article is related to the 
question of independence of $m$ and $n$ mentioned earlier.

\vspace{2mm}
\noindent{\em Acknowledgements.}
The results in this article were obtained during the program
``On the Interaction of Representation Theory with Geometry and Combinatorics''
at the Hausdorff Institute for Mathematics in 
Bonn in January, 2011, and written up in part during a visit by the first
author to the University of Sydney in March--May, 2011.
We thank J. Comes for some helpful discussions
about $\underline{\operatorname{Re}}\!\operatorname{p}(GL_\delta)$ and for the proof of Lemma~\ref{comeslem}.

\section{Refined branching rules for the walled Brauer algebra}\label{s2}

In this preliminary section, we fix $r, s \geq 0$ and $\delta \in \C$.
We are going to recall the definition of the walled
Brauer algebra $B_{r,s}(\delta)$
and establish some basic facts about its representation theory.
In particular, we will define some refined induction and restriction functors
and describe their effect on cell modules.

\phantomsubsection{The walled Brauer algebra}
As a $\C$-vector space, $B_{r,s}(\delta)$ has
a basis consisting of isotopy classes of {\em walled Brauer diagrams}.
These are certain diagrams drawn in a rectangle with $(r+s)$ vertices 
on its top and bottom edges,
numbered 
$1,\dots,r+s$ in order from left to right.
Each vertex must be connected
to exactly one other vertex by a smooth curve drawn in the interior of
the rectangle; curves can cross transversally, no triple intersections. 
We refer to the curves whose endpoints are on different edges of the rectangle
as {\em vertical strands}, and the curves whose endpoints are on
the same edge of the rectangle are {\em horizontal strands}.
In addition, there is a vertical wall separating vertices $1,\dots,r$
from vertices $r+1,\dots,r+s$. We require that
the endpoints of each vertical strand are on the same side of the wall
and the endpoints of each horizontal strand are on opposite sides of the wall.
For example here are two 
basis vectors in $B_{2,2}(\delta)$:
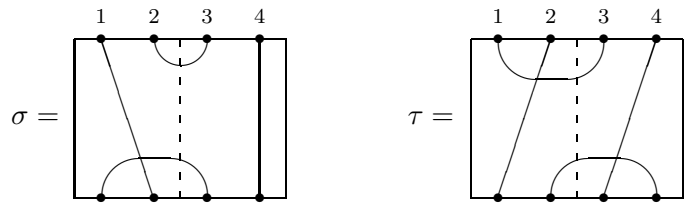
\begin{equation*}
\begin{picture}(150,76)
\put(-24,34){$\sigma=$}
\put(8,74){$_1$}
\put(28,74){$_2$}
\put(48,74){$_3$}
\put(68,74){$_4$}
\put(8,4.2){$\scriptstyle\bullet$}
\put(28,4.2){$\scriptstyle\bullet$}
\put(48,4.2){$\scriptstyle\bullet$}
\put(68,4.2){$\scriptstyle\bullet$}
\put(8,64.2){$\scriptstyle\bullet$}
\put(28,64.2){$\scriptstyle\bullet$}
\put(48,64.2){$\scriptstyle\bullet$}
\put(68,64.2){$\scriptstyle\bullet$}
\put(0,6){\line(1,0){80}}
\put(0,6){\line(0,1){60}}
\put(0,66){\line(1,0){80}}
\put(80,66){\line(0,-1){60}}
\dashline{3}(40,6)(40,66)
\put(30,6){\line(-1,3){20}}
\put(70,6){\line(0,1){60}}
\put(30.3,6){\oval(40,30)[t]}
\put(40.3,66){\oval(20,20)[b]}
\end{picture}
\begin{picture}(80,76)
\put(-24,34){$\tau=$}
\put(8,74){$_1$}
\put(28,74){$_2$}
\put(48,74){$_3$}
\put(68,74){$_4$}
\put(8,4.2){$\scriptstyle\bullet$}
\put(28,4.2){$\scriptstyle\bullet$}
\put(48,4.2){$\scriptstyle\bullet$}
\put(68,4.2){$\scriptstyle\bullet$}
\put(8,64.2){$\scriptstyle\bullet$}
\put(28,64.2){$\scriptstyle\bullet$}
\put(48,64.2){$\scriptstyle\bullet$}
\put(68,64.2){$\scriptstyle\bullet$}
\put(0,6){\line(1,0){80}}
\put(0,6){\line(0,1){60}}
\put(0,66){\line(1,0){80}}
\put(80,66){\line(0,-1){60}}
\dashline{3}(40,6)(40,66)
\put(10,6){\line(1,3){20}}
\put(50,6){\line(1,3){20}}
\put(50.3,6){\oval(40,30)[t]}
\put(30.3,66){\oval(40,30)[b]}
\end{picture}
\end{equation*}
Multiplication is by the following concatenation procedure: $\sigma\tau$ is obtained by putting $\sigma$ {\em under} $\tau$, removing internal circles, then
multiplying the resulting basis vector
by $\delta^t$ where $t$ is the number of internal circles removed.
For example, for $\sigma$ and $\tau$ as above, we have:
$$
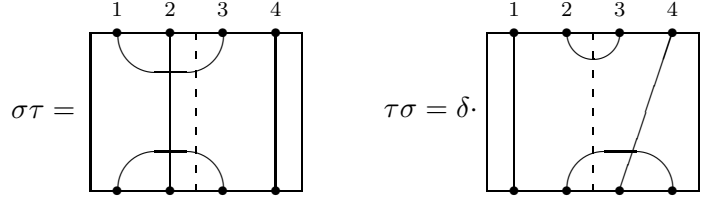
\begin{picture}(150,76)
\put(-30,34){$\sigma\tau=$}
\put(8,74){$_1$}
\put(28,74){$_2$}
\put(48,74){$_3$}
\put(68,74){$_4$}
\put(8,4.2){$\scriptstyle\bullet$}
\put(28,4.2){$\scriptstyle\bullet$}
\put(48,4.2){$\scriptstyle\bullet$}
\put(68,4.2){$\scriptstyle\bullet$}
\put(8,64.2){$\scriptstyle\bullet$}
\put(28,64.2){$\scriptstyle\bullet$}
\put(48,64.2){$\scriptstyle\bullet$}
\put(68,64.2){$\scriptstyle\bullet$}
\put(0,6){\line(1,0){80}}
\put(0,6){\line(0,1){60}}
\put(0,66){\line(1,0){80}}
\put(80,66){\line(0,-1){60}}
\dashline{3}(40,6)(40,66)
\put(70,6){\line(0,1){60}}
\put(30,6){\line(0,1){60}}
\put(30.3,6){\oval(40,30)[t]}
\put(30.3,66){\oval(40,30)[b]}
\end{picture}
\begin{picture}(80,76)
\put(-39,34){$\tau\sigma=\delta \cdot $}
\put(8,74){$_1$}
\put(28,74){$_2$}
\put(48,74){$_3$}
\put(68,74){$_4$}
\put(8,4.2){$\scriptstyle\bullet$}
\put(28,4.2){$\scriptstyle\bullet$}
\put(48,4.2){$\scriptstyle\bullet$}
\put(68,4.2){$\scriptstyle\bullet$}
\put(8,64.2){$\scriptstyle\bullet$}
\put(28,64.2){$\scriptstyle\bullet$}
\put(48,64.2){$\scriptstyle\bullet$}
\put(68,64.2){$\scriptstyle\bullet$}
\put(0,6){\line(1,0){80}}
\put(0,6){\line(0,1){60}}
\put(0,66){\line(1,0){80}}
\put(80,66){\line(0,-1){60}}
\dashline{3}(40,6)(40,66)
\put(10,6){\line(0,1){60}}
\put(50,6){\line(1,3){20}}
\put(50.3,6){\oval(40,30)[t]}
\put(40.3,66){\oval(20,20)[b]}
\end{picture}
$$

It is useful to compare $B_{r,s}(\delta)$ with the group algebra $\C \Sigma_{r+s}$ of the symmetric group. Of course the latter can also be viewed diagrammatically if we identify $\sigma \in \Sigma_{r+s}$ 
with its {\em permutation diagram}, 
the diagram drawn in the same rectangle as before with a vertical 
strand connecting the $a$th vertex on the top to the $\sigma(a)$th vertex on the bottom for each $1 \leq a \leq r+s$.
It is clear from the definitions that there is {vector space} isomorphism
\begin{equation}\label{oldflip}
\flip_{r,s}:\C \Sigma_{r+s} \stackrel{\sim}{\rightarrow} B_{r,s}(\delta)
\end{equation}
mapping a permutation diagram to the walled Brauer diagram 
obtained by adding a
wall between the $r$th and $(r+1)$th vertices, 
then
flipping the part of the diagram that is to the right 
of the wall in its horizontal axis without disconnecting any strands. 
For example the walled Brauer diagrams $\sigma$ and $\tau$ above 
arise by applying $\flip_{2,2}$ to the following permutation diagrams, respectively:
$$
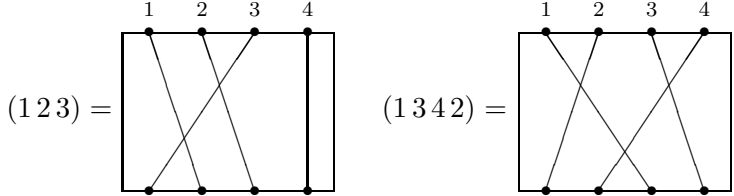
\begin{picture}(150,76)
\put(-44,34){$(1\,2\,3)=$}
\put(8,74){$_1$}
\put(28,74){$_2$}
\put(48,74){$_3$}
\put(68,74){$_4$}
\put(8,4.2){$\scriptstyle\bullet$}
\put(28,4.2){$\scriptstyle\bullet$}
\put(48,4.2){$\scriptstyle\bullet$}
\put(68,4.2){$\scriptstyle\bullet$}
\put(8,64.2){$\scriptstyle\bullet$}
\put(28,64.2){$\scriptstyle\bullet$}
\put(48,64.2){$\scriptstyle\bullet$}
\put(68,64.2){$\scriptstyle\bullet$}
\put(0,6){\line(1,0){80}}
\put(0,6){\line(0,1){60}}
\put(0,66){\line(1,0){80}}
\put(80,66){\line(0,-1){60}}
\put(30,6){\line(-1,3){20}}
\put(70,6){\line(0,1){60}}
\put(50,6){\line(-1,3){20}}
\put(10,6){\line(2,3){40}}
\end{picture}
\begin{picture}(80,76)
\put(-52,34){$(1\,3\,4\,2)=$}
\put(8,74){$_1$}
\put(28,74){$_2$}
\put(48,74){$_3$}
\put(68,74){$_4$}
\put(8,4.2){$\scriptstyle\bullet$}
\put(28,4.2){$\scriptstyle\bullet$}
\put(48,4.2){$\scriptstyle\bullet$}
\put(68,4.2){$\scriptstyle\bullet$}
\put(8,64.2){$\scriptstyle\bullet$}
\put(28,64.2){$\scriptstyle\bullet$}
\put(48,64.2){$\scriptstyle\bullet$}
\put(68,64.2){$\scriptstyle\bullet$}
\put(0,6){\line(1,0){80}}
\put(0,6){\line(0,1){60}}
\put(0,66){\line(1,0){80}}
\put(80,66){\line(0,-1){60}}
\put(10,6){\line(1,3){20}}
\put(70,6){\line(-1,3){20}}
\put(50,6){\line(-2,3){40}}
\put(30,6){\line(2,3){40}}
\end{picture}
$$
It follows in particular that
\begin{equation}
\label{brauerdim}
\dim B_{r,s}(\delta) = \dim \C \Sigma_{r+s} = (r+s)!.
\end{equation}

\phantomsubsection{\boldmath Generators, relations, and the central element $z_{r,s}$}
The algebra $\C \Sigma_{r+s}$ is generated by the
basic transpositions $\sigma_a := (a\:a\!+\!1)$ for
$a=1,\dots,r+s-1$, subject only to the usual Coxeter relations.
Let $\tau_a := \flip_{r,s}(\sigma_a) \in B_{r,s}(\delta)$.
So for $a \neq r$ the diagram
$\tau_a$ is just the same as $\sigma_a$ (with the addition of the wall),
while $\tau_r$ is the diagram
$$
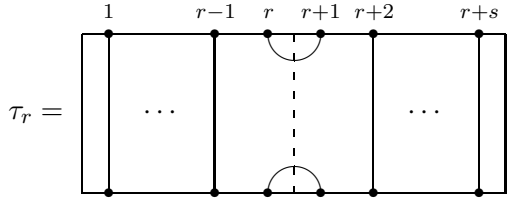
\begin{picture}(150,78)
\dashline{3}(80,6)(80,66)
\put(-28,34){$\tau_r=$}
\put(8,74){$_1$}
\put(23,34){${\cdots}$}
\put(43,74){$_{r-1}$}
\put(68.5,74){$_r$}
\put(82.7,74){$_{r+1}$}
\put(103,74){$_{r+2}$}
\put(123,34){${\cdots}$}
\put(143,74){$_{r+s}$}
\put(8,4.2){$\scriptstyle\bullet$}
\put(48,4.2){$\scriptstyle\bullet$}
\put(68,4.2){$\scriptstyle\bullet$}
\put(88,4.2){$\scriptstyle\bullet$}
\put(108,4.2){$\scriptstyle\bullet$}
\put(148,4.2){$\scriptstyle\bullet$}
\put(8,64.2){$\scriptstyle\bullet$}
\put(48,64.2){$\scriptstyle\bullet$}
\put(68,64.2){$\scriptstyle\bullet$}
\put(88,64.2){$\scriptstyle\bullet$}
\put(108,64.2){$\scriptstyle\bullet$}
\put(148,64.2){$\scriptstyle\bullet$}
\put(0,6){\line(1,0){160}}
\put(0,6){\line(0,1){60}}
\put(0,66){\line(1,0){160}}
\put(160,66){\line(0,-1){60}}
\put(10,6){\line(0,1){60}}
\put(50,6){\line(0,1){60}}
\put(110,6){\line(0,1){60}}
\put(150,6){\line(0,1){60}}
\put(80.3,66){\oval(20,20)[b]}
\put(80.3,6){\oval(20,20)[t]}
\end{picture}
$$
It is an exercise to see that $B_{r,s}(\delta)$
is generated by the elements $\tau_1,\dots,\tau_{r+s-1}$
and that the following relations hold whenever they make sense:
\begin{align}
\tau_a^2 &= 1 \quad(a \neq r), &\tau_a \tau_b &= \tau_b \tau_a\hspace{7mm}(|a-b| > 1),\\
\tau_r^2 &= \delta \tau_r,&
\tau_a\tau_{b} \tau_a &= \tau_{b}\tau_a \tau_{b}
\quad(|a-b|=1,a\neq r\neq b),\\
\tau_r &= \tau_r \tau_{r-1} \tau_r,&
\tau_r \tau_{r-1}\tau_{r+1}\tau_r &= 
\tau_{r-1}\tau_{r+1}
\tau_r \tau_{r-1}\tau_{r+1}\tau_r,\\
\tau_r &= \tau_r \tau_{r+1} \tau_r,&\tau_r \tau_{r-1}\tau_{r+1}\tau_r &= 
\tau_r \tau_{r-1}\tau_{r+1}\tau_r \tau_{r-1}\tau_{r+1}.
\end{align}
Although never needed in the present article,
it is known moreover 
that this is a full set of relations for the algebra $B_{r,s}(\delta)$;
see \cite[Corollary 4.5]{H} for a special case 
or \cite[Theorem 4.1]{N} in general.

More generally, for $1 \leq a, b \leq r+s$ with $a \neq b$,
we introduce the following ``transpositions'' in $B_{r,s}(\delta)$:
\begin{equation*}
\overline{(a\ b)} := \left\{
\begin{array}{rl}
\flip_{r,s}((a\ b))&\text{if $a$ and $b$ are on the same side of the wall,}\\
-\flip_{r,s}((a\ b))&\text{if $a$ and $b$ are on opposite sides of the wall.}
\end{array}
\right.
\end{equation*}
We stress the presence of the sign in the second case; it means that
$\tau_a = \overline{(a\ a\!+\!1)}$ for $a \neq r$ but
$\tau_r = -\overline{(r\ r\!+\!1)}$.
Then let
\begin{align}\label{zrs}
z_{r,s} &:= 
\sum_{1 \leq a < b \leq r+s} \overline{(a\ b)}.
\end{align}

\begin{Lemma}\label{centprops}
The element $z_{r,s}$ belongs to the center of $B_{r,s}(\delta)$.
\end{Lemma}

\begin{proof}
It suffices to show that $z_{r,s}$ commutes with the generators of $B_{r,s}(\delta)$, which is a routine exercise. 
\end{proof}

\phantomsubsection{Cell modules}
Recall the set $\La_{r,s}$ of bipartitions from (\ref{ladef}).
For each $\la \in \La_{r,s}$, there is a corresponding 
$B_{r,s}(\delta)$-module $C_{r,s}(\la)$,
which is a
{\em cell module} in the sense of \cite{GL}
arising from a certain cellular algebra structure on 
$B_{r,s}(\delta)$ 
discussed in \cite[$\S$2]{CDDM}.
We won't recall the definition of this
cellular algebra structure here,
as it is enough for our purposes to work with the explicit
construction of $C_{r,s}(\la)$ given in \cite[$\S$3]{CDDM}.

Let $J_{r,s}^t$ be the two-sided ideal
of $B_{r,s}(\delta)$
spanned by the
walled Brauer diagrams with at 
least $t$ horizontal strands at the top and bottom, giving a chain
$$
B_{r,s}(\delta) = J_{r,s}^0 \supset J_{r,s}^1 \supset \cdots \supset 
J_{r,s}^{\min(r,s)} \supset \{0\}.
$$
The images of the walled Brauer diagrams with exactly $t$ horizontal
strands at the top and bottom give a basis for the quotient $J_{r,s}^t / J_{r,s}^{t+1}$.

Now take $\la \in \La_{r,s}$ and let $t := r-|\la^{\mathrm L}| = s - |\la^{\mathrm R}|$.
Let $I_{r,s}^t$ be the subspace of $J_{r,s}^t / J_{r,s}^{t+1}$
spanned by the images of the 
diagrams with 
exactly $t$ horizontal strands at the top 
connecting the 
$(r+1-k)$th vertex to the $(r+k)$th vertex for each $k=1,\dots,t$.
Identify $\C (\Sigma_{r-t} \times \Sigma_{s-t})$ with the subalgebra of $B_{r,s}(\delta)$
generated by
$\tau_1,\dots,\tau_{r-t-1}, \tau_{r+t+1},\dots,\tau_{r+s-1}$ in the obvious way.
Then $I_{r,s}^t$ is invariant under left multiplication by elements of
$B_{r,s}(\delta)$ and right multiplication by elements of $\C (\Sigma_{r-t} \times\Sigma_{s-t})$, hence it is a
$\left(B_{r,s}(\delta), \C (\Sigma_{r-t} \times \Sigma_{s-t})\right)$-bimodule.
Let $S(\la^{\mathrm L})$ 
be the irreducible $\Sigma_{r-t}$-module parametrised
by $\la^{\mathrm L}$
and
$S(\la^{\mathrm R})$ be the irreducible $\Sigma_{s-t}$-module parametrised by
$\la^{\mathrm R}$. Let
$S(\la)$ denote their outer tensor product 
$S(\la^{\mathrm L}) \boxtimes S(\la^{\mathrm R})$, which is an
irreducible $\Sigma_{r-t} \times \Sigma_{s-t}$-module.
Then we have by definition that
\begin{equation}\label{cell}
C_{r,s}(\la) := I_{r,s}^t \otimes_{\C (\Sigma_{r-t} \times \Sigma_{s-t})} 
S(\la).
\end{equation}

The bimodule $I_{r,s}^t$ is free as a 
right $\C (\Sigma_{r-t} \times \Sigma_{s-t})$-module, with basis $X_{r,s}^t$
given by images of the walled Brauer diagrams 
with exactly $t$ horizontal strands
at the top connecting the 
$(r+1-k)$th vertex to the $(r+k)$th vertex for $k=1,\dots,t$
in which no two vertical strands cross.
So as a vector space we have that
\begin{equation}\label{vecd}
C_{r,s}(\la) = \bigoplus_{\tau \in X_{r,s}^t} \tau \otimes S(\la).
\end{equation}
In particular, this implies that
\begin{equation}
\label{celldim}
\dim C_{r,s}(\la) = 
\frac{r!s!}{(r-t)!t!(s-t)!} \dim S(\la^{\mathrm L}) \dim S(\la^{\mathrm R}).
\end{equation}

To compute the action of
a walled Brauer diagram $\sigma  \in B_{r,s}(\delta)$ on a vector 
$\tau \otimes v$ for $\tau \in X_{r,s}^t$ and $v \in 
S(\la)$ we proceed as follows.
If $\sigma$ has more than $t$ horizontal strands
then $\sigma(\tau\otimes v) = 0$; otherwise 
$\sigma (\tau \otimes v) = \tau' \otimes (c\sigma' v)$ 
where $\sigma\tau= c \tau'\sigma'$ 
for $c\in \C, \tau' \in X_{r,s}^t$ and $\sigma' \in \Sigma_{r-t}\times
\Sigma_{s-t}$.
It follows easily from this that
$C_{r,s}(\la)$ is a cyclic $B_{r,s}(\delta)$-module generated
by any vector of the form $\tau \otimes v$ for $\tau \in X_{r,s}^t$
and $0 \neq v \in S(\la)$.

\begin{Lemma}\label{schur}
For any $\la \in \La_{r,s}$ we have that
$\End_{B_{r,s}(\delta)}(C_{r,s}(\la)) \cong \C$.
\end{Lemma}

\begin{proof}
Set $t := r-|\la^{\mathrm L}| = s - |\la^{\mathrm R}|$.
Let $\tau$ be the unique walled Brauer diagram
in $X^t_{r,s}$
with horizontal strands
connecting the 
$(r+1-k)$th vertex to the $(r+k)$th vertex for $k=1,\dots,t$
at both the top and the bottom.

We first treat the case that $t < r$.
Let $\sigma := \tau \tau_{r-t} \cdots \tau_{r-2} \tau_{r-1}$.
On drawing the diagrams, it is easy to check that $\sigma \tau = \tau$, and that $\sigma C_{r,s}(\la) = \tau
\otimes S(\la)$ in the decomposition (\ref{vecd}).
It follows for any $f \in \End_{B_{r,s}(\delta)}(C_{r,s}(\la))$ 
and $v \in S(\la)$
that
$f(\tau \otimes v) = f(\sigma\tau \otimes v) = \sigma f(\tau \otimes
v)
= \tau \otimes \bar f(v)$ for 
a unique $\bar f(v) \in S(\la)$.
In other words, $f$ induces a linear map $\bar f:S(\la) \rightarrow
S(\la)$, which is easily seen to be a 
$\Sigma_{r-t} \times \Sigma_{s-t}$-module
homomorphism.
Hence $\bar f(v) = c v$ for a unique scalar $c \in \C$ by
Schur's lemma,
i.e. $f(\tau \otimes v) = c (\tau \otimes v)$.
Since $\tau \otimes v$ generates
$C_{r,s}(\la)$, we deduce that $f = c\cdot \operatorname{id}$.

The case that $t < s$ is similar.

It remains to treat the case that $t = r = s$. 
This is trivial if $t = 0$, so assume moreover that $t > 0$.
Let $\sigma := \tau' \tau_1 \cdots \tau_{r-1}$
where $\tau'$ is the 
 unique walled Brauer diagram
in $X^{t-1}_{r,s}$
with horizontal strands
connecting the 
$(r+1-k)$th vertex to the $(r+k)$th vertex for $k=1,\dots,t-1$
at both the top and the bottom.
Then again we have that $\sigma \tau = \tau$
and $\sigma C_{r,s}(\la) = \tau \otimes S(\la)$.
Now the proof can be completed as before.
\end{proof}

\phantomsubsection{\boldmath Action of $z_{r,s}$ on cell modules}
We draw
the Young diagram of a partition $\la$ in the usual English way, and
let $[\la]$ denote its set of boxes.
For $A \in [\la]$, its {\em content} $\residue(A)$ is the integer $(j-i)$
if $A$ is in row $i$ and column $j$.
For example, here is the Young diagram of the partition $\la = (5,4,4)$
with boxes labelled by their contents:
$$
\diagram{0&1&2&3&4\cr -1&0&1&2\cr -2 & -1 &0 & 1\cr}.
$$
For $\la \in \La_{r,s}$, let
\begin{equation*}
z_{r,s}(\la) := 
\sum_{A \in [\la^{\mathrm L}]} \residue(A) + \sum_{A \in [\la^{\mathrm R}]} \residue(A)
- t \delta \in \C,
\end{equation*}
where $t := r-|\la^{\mathrm L}| = s - |\la^{\mathrm R}|$.

\begin{Lemma}\label{eigenvalue}
For $\la \in \La_{r,s}$,
the central element $z_{r,s} \in B_{r,s}(\delta)$ acts on the cell module
$C_{r,s}(\la)$ by multiplication by the scalar $z_{r,s}(\la)$.
\end{Lemma}

\begin{proof}
Let $t := r-|\la^{\mathrm L}| = s - |\la^{\mathrm R}|$,
$I := \{1,\dots,r-t\}$, $J := \{r+t+1,\dots,r+s\}$, and
$K := \{r-t+1,\dots,r\}$.
For $a \in K$, let $\tilde a := 2r+1-a$.
Recalling (\ref{cell}), let $\tau$ be the image in $I_{r,s}^t$ of 
the walled Brauer diagram with 
a vertical strand
connecting the $a$th vertex at the top to the $a$th vertex at the bottom
for each $a \in I \cup J$,
and a horizontal strand connecting $a$ to $\tilde a$
both at the top and the bottom for each $a \in K$.
Also pick any non-zero vector $v \in S(\la)$.
As $C_{r,s}(\la)$ is generated by
$\tau \otimes v$ and $z_{r,s}$ is central, it suffices to show that
$z_{r,s}(\tau \otimes v) = z_{r,s}(\la) (\tau \otimes v)$.
We have that
$z_{r,s} = z_{I} + z_J + z_K + z'_{K}
+ z_K'' +z_{IJ} + z_{IK} + z_{JK}$
where
\begin{align*}
z_{I} &:= \sum_{\substack{a, b \in I \\ a < b}} \overline{(a\ b)},
&z_{IJ} &:= \sum_{\substack{a \in I\\b \in J}} \overline{(a\ b)},\\
z_{J} &:= \sum_{\substack{a, b \in J\\ a < b}} \overline{(a\,b)},
&z_{IK} &:=\sum_{\substack{a \in I \\  b\in K}} \left(\overline{(a\ b)} +
  \overline{(a\,\tilde b)}\right),\\
z_K &:= \sum_{a \in K} \overline{(a\ \tilde a)},&
z_{JK} &:=
\sum_{\substack{a \in J\\ b\in K}} \left(\overline{(a\ b)} +
  \overline{(a\,\tilde b)}\right),\\
z_{K}' &:= 
\sum_{\substack{a, b\in K \\ a <b}}
\left(\overline{(a\ b)} + \overline{(a\ \tilde b)}
\right),
&z_K'' &:= 
\sum_{\substack{a, b\in K \\ a <b}}
\left(\overline{(\tilde a\ b)} + \overline{(\tilde a\ \tilde b)}
\right).
\end{align*}
Now we observe that 
$$
z_{IJ} (\tau \otimes v) = z_{IK} (\tau \otimes v)
= z_{JK} (\tau \otimes v) = z_{K}' (\tau \otimes v)
= z_K'' (\tau \otimes v) = 0
$$
because each of the terms in the summations above already act as zero.
Moreover 
$z_{I} (\tau \otimes v) = \sum_{A \in [\la^{\mathrm L}]} \residue(A) (\tau
\otimes v)$, 
$z_{J} (\tau \otimes v) = \sum_{A \in [\la^{\mathrm R}]} \residue(A) (\tau
\otimes v)$,
and $z_K (\tau \otimes v) = -t \delta (\tau
\otimes v)$.
The first two equalities follow because $z_{I}$ and $z_{J}$ are the sums of all transpositions
in $\Sigma_{r-t}$ and $\Sigma_{s-t}$, respectively, and it is well known
(e.g. as a consequence of Young's orthogonal form or by the Murnaghan-Nakayama rule)
that these act on $S(\la^{\mathrm L})$ and $S(\la^{\mathrm R})$ by the scalars indicated.
\end{proof}

\phantomsubsection{Refined induction and restriction functors}
There are 
embeddings
\begin{equation}\label{iotas}
\iota_{r,s}^{r+1,s}:B_{r,s}(\delta) \hookrightarrow B_{r+1,s}(\delta),
\qquad
\iota_{r,s}^{r,s+1}:B_{r,s}(\delta) \hookrightarrow B_{r,s+1}(\delta)
\end{equation}
defined by inserting a vertical strand immediately to the left or right
of the wall, respectively.
Thus 
$\iota_{r,s}^{r+1,s}(\tau_r) = \tau_{r} \tau_{r+1} \tau_{r}$,
$\iota_{r,s}^{r,s+1}(\tau_r) = \tau_{r+1} \tau_{r} \tau_{r+1}$,
 and both $\iota_{r,s}^{r+1,s}$ and $\iota_{r,s}^{r,s+1}$ map $\tau_a \mapsto \tau_a$
for $1 \leq a \leq r-1$
and $\tau_b \mapsto \tau_{b+1}$ for $r+1 \leq b \leq r+s-1$.
Let 
$\big(\ind^{r+1,s}_{r,s}, \res^{r+1,s}_{r,s}\big)$
and
$\big(\ind^{r,s+1}_{r,s}, \res^{r,s+1}_{r,s}\big)$ be the adjoint pairs
of induction and restriction functors
attached to the embeddings $\iota_{r,s}^{r+1,s}$ and $\iota_{r,s}^{r,s+1}$, respectively.

We want to define refined versions of these 
functors. Fix $i \in \C$. Let
\begin{align}
x^{r+1,s}_{r,s} &:= \label{xl}
z_{r+1,s} - \iota_{r,s}^{r+1,s}(z_{r,s})
=\sum_{\substack{1 \leq a \leq r+s+1 \\ a \neq r}} \overline{(a\ r)}\hspace{-5mm} &&\in B_{r+1,s}(\delta),\\
x^{r,s+1}_{r,s} &:= z_{r,s+1} - \iota_{r,s}^{r,s+1}(z_{r,s})
=\sum_{\substack{1 \leq a \leq r+s+1 \\ a \neq r+1}} \overline{(a\ r\!+\!1)}
 \hspace{-5mm}&&\in B_{r,s+1}(\delta).\label{xr}
\end{align}
The elements $x^{r+1,s}_{r,s}$ and $x^{r,s+1}_{r,s}$
centralise
$\iota_{r,s}^{r+1,s}(B_{r,s}(\delta))$
and $\iota_{r,s}^{r,s+1}(B_{r,s}(\delta))$, respectively. 
In the finite dimensional commutative subalgebra of $B_{r+1,s}(\delta)$
generated by $x_{r,s}^{r+1,s}$, there is
a unique idempotent $1_{r,s;i}^{r+1,s}$ 
that acts on any module as the projection onto the generalised $i$-eigenspace of $x_{r,s}^{r+1,s}$.
Similarly in the subalgebra of $B_{r,s+1}(\delta)$
generated by $x_{r,s}^{r,s+1}$, there is
a unique idempotent $1_{r,s;i}^{r,s+1}$ 
that acts as the projection onto the generalised $(-i-\delta)$-eigenspace of $x_{r,s}^{r,s+1}$.
Let
\begin{align}\label{iotai}
\iota_{r,s;i}^{r+1,s}&:B_{r,s}(\delta) \rightarrow B_{r+1,s}(\delta),
&\iota_{r,s;i}^{r,s+1}&:B_{r,s}(\delta) \rightarrow B_{r,s+1}(\delta)
\end{align}
be the non-unital 
algebra homomorphisms defined by first applying the embedding
$\iota_{r,s}^{r+1,s}$ then multiplying by the centralising idempotent $1_{r,s;i}^{r+1,s}$,
or by applying $\iota_{r,s}^{r,s+1}$ then multiplying by $1_{r,s;i}^{r,s+1}$, respectively.
Let 
$(\iind^{r+1,s}_{r,s},\ires_{r,s}^{r+1,s})$
and
$(\iind^{r,s+1}_{r,s}, \ires^{r,s+1}_{r,s})$
be the adjoint pairs of {\em $i$-induction} and {\em $i$-restriction}
functors attached to the homomorphisms
$\iota_{r,s;i}^{r+1,s}$ and $\iota_{r,s;i}^{r,s+1}$, respectively.
Thus,
\begin{align}\label{fred1}
\ires^{r+1,s}_{r,s} 
&:\mod{B_{r+1,s}(\delta)} \rightarrow \mod{B_{r,s}(\delta)},\\
\ires^{r,s+1}_{r,s} 
&:\mod{B_{r,s+1}(\delta)} \rightarrow \mod{B_{r,s}(\delta)}\label{fred2}
\end{align}
are the exact functors defined by multiplication by
the idempotents $1_{r,s;i}^{r+1,s}$ and $1_{r,s;i}^{r,s+1}$, respectively, viewing the
result as a left $B_{r,s}(\delta)$-module via the homomorphisms (\ref{iotai}).
Also 
\begin{align}\label{fred3}
\iind^{r+1,s}_{r,s} 
&:\mod{B_{r,s}(\delta)} \rightarrow \mod{B_{r+1,s}(\delta)},\\
\iind^{r,s+1}_{r,s} 
&:\mod{B_{r,s}(\delta)} \rightarrow \mod{B_{r,s+1}(\delta)}\label{fred4}
\end{align}
are the functors
$B_{r+1,s}(\delta) 1_{r,s;i}^{r+1,s} \ \otimes_{B_{r,s}(\delta)}?$
and 
$B_{r,s+1}(\delta) 1_{r,s;i}^{r,s+1} \ \otimes_{B_{r,s}(\delta)}?$, respectively,
where the right $B_{r,s}(\delta)$-module structure is again defined
using (\ref{iotai}).
Since $\sum_{i \in \C} 1_{r,s;i}^{r+1,s} = 1 \in B_{r+1,s}(\delta)$
and 
$\sum_{i \in \C} 1_{r,s;i}^{r,s+1} \in B_{r,s+1}(\delta)$,
we have that
\begin{align*}
\res^{r+1,s}_{r,s} &= \bigoplus_{i \in \C}
\ires^{r+1,s}_{r,s},&
\ind^{r+1,s}_{r,s} &= \bigoplus_{i \in \C}
\iind^{r+1,s}_{r,s},\\
\res^{r,s+1}_{r,s} &= \bigoplus_{i \in \C}
\ires^{r,s+1}_{r,s},&
\ind^{r,s+1}_{r,s} &= \bigoplus_{i \in \C}
\iind^{r,s+1}_{r,s}.
\end{align*}
The following lemma 
realises the $i$-restriction and $i$-induction functors 
in terms of taking certain generalised eigenspaces.

\begin{Lemma}\label{altdef}
Let $M$ be a finite dimensional indecomposable $B_{r,s}(\delta)$-module
and $\la_M \in \C$ be the unique eigenvalue of the (not necessarily
diagonalisable) endomorphism of $M$ defined by multiplication by $z_{r,s}$.
\begin{itemize}
\item[(1)] The module $\ires^{r,s}_{r-1,s} M$
is the generalised $(\la_M-i)$-eigenspace of $z_{r-1,s}$
on $\res^{r,s}_{r-1,s} M$.
\item[(2)] The module 
$\ires^{r,s}_{r,s-1} M$
is the generalised $(\la_M+i+\delta)$-eigenspace of $z_{r,s-1}$
on $\res^{r,s}_{r,s-1} M$.
\item[(3)] The module $\iind^{r+1,s}_{r,s} M$
is the generalised $(\la_M+i)$-eigenspace of $z_{r+1,s}$
on $\ind^{r+1,s}_{r,s} M$.
\item[(4)] The module $\iind^{r,s+1}_{r,s} M$
is the generalised $(\la_M-i-\delta)$-eigenspace of $z_{r,s+1}$
on $\ind^{r+1,s}_{r,s} M$.
\end{itemize}
\end{Lemma}

\begin{proof}
This follows from the definitions.
\end{proof}

We give the set $\La$ of all bipartitions 
the structure of a labelled
directed graph,
with a directed edge $\la \stackrel{i}{\rightarrow} \mu$
labelled by $i \in \C$
if $\mu$ is obtained from $\la$ {either} by adding a box of content $i$
to the Young diagram of $\la^{\mathrm L}$
{or} by removing a box of content $-i-\delta$ from the Young diagram of $\la^{\mathrm R}$.
Equivalently, 
$\la \stackrel{i}{\leftarrow} \mu$ if $\mu$ is obtained from $\la$
by removing a box of content $i$ from the Young diagram of $\la^{\mathrm L}$ or
 by adding a box of content $-i-\delta$ to the Young diagram of $\la^{\mathrm R}$.
The following theorem is a refinement of \cite[Theorem 3.3]{CDDM}.

\begin{Theorem}\label{ibranch}
The following hold for $\la \in \La_{r,s}$ and $i \in \C$.
\begin{itemize}
\item[(1)]
The module $\ires^{r,s}_{r-1,s} C_{r,s}(\la)$
has a filtration with sections isomorphic 
to 
$\{C_{r-1,s}(\mu)\:|\:\text{for all $\mu \in \La_{r-1,s}$ with }
\mu \stackrel{i}{\rightarrow} \la\}$.
\item[(2)]
The module $\ires^{r,s}_{r,s-1} C_{r,s}(\la)$ has a filtration with sections isomorphic to 
$\{C_{r,s-1}(\mu)\:|\:\text{for all $\mu \in \La_{r,s-1}$ with }
\mu \stackrel{i}{\leftarrow} \la\}$.
\item[(3)]
The module $\iind^{r+1,s}_{r,s} C_{r,s}(\la)$ has a filtration with sections isomorphic 
to 
$\{C_{r+1,s}(\mu)\:|\:\text{for all $\mu \in \La_{r+1,s}$ with }
\la \stackrel{i}{\rightarrow} \mu\}$
\item[(4)]
The module $\iind^{r,s+1}_{r,s} C_{r,s}(\la)$ has a filtration with sections isomorphic to 
$\{C_{r,s+1}(\mu)\:|\:\text{for all $\mu \in \La_{r,s+1}$ with }
\la \stackrel{i}{\leftarrow} \mu\}$.
\end{itemize}
In all cases, the filtration is of length at most two. 
When it is of length exactly two, the sections 
are
indexed by
bipartitions $\mu'$ and $\mu''$ such that 
$|(\mu')^{\mathrm L}| < |(\mu'')^{\mathrm L}|$ and
$|(\mu')^{\mathrm R}| < |(\mu'')^{\mathrm R}|$, and the following hold:
\begin{itemize}
\item[(a)]
for $M = \ires^{r,s}_{r-1,s} C_{r,s}(\la)$ or $M = \ires^{r,s}_{r,s-1} C_{r,s}(\la)$, there is 
a {unique} submodule $M'$ isomorphic to the cell module indexed by $\mu'$;
the quotient $M / M'$ is isomorphic to the cell module indexed by $\mu''$;
\item[(b)]
for $M = \iind^{r+1,s}_{r,s} C_{r,s}(\la)$ or $M = \iind^{r,s+1}_{r,s} C_{r,s}(\la)$, there is 
a {unique} submodule $M'$ such that $M / M'$ is isomorphic to the cell module indexed by $\mu''$; the submodule $M'$ is isomorphic to the cell module indexed by $\mu'$.
\end{itemize}
\end{Theorem}

\begin{proof}
The existence of the filtrations
follows from 
\cite[Theorem 3.3]{CDDM}
 and Lemmas~\ref{eigenvalue}--\ref{altdef}.
The fact that the filtrations are of length at most two follows because $\la^{\mathrm L}$ has at most one addable or removable box of content $i$
and $\la^{\mathrm R}$ has at most one addable or removable box of content $-i-\delta$.

For the final uniqueness statements (a)--(b), we just explain the proof of (a)
for $M = \ires^{r,s}_{r-1,s} C_{r,s}(\la)$, since the 
other cases are similar. 
So 
assume both that $\la^{\mathrm L}$ has an removable box of content $i$ (hence $r > 0$) and $\la^{\mathrm R}$ has an addable box of content $-i-\delta$. Let $\mu'$ and $\mu''$
be obtained from $\la$ by removing the former box from $\la^{\mathrm L}$ or by adding the latter box to $\la^{\mathrm R}$, respectively.
The existence argument indicated in the previous paragraph
actually shows that $M$ has a submodule $M' \cong C_{r-1,s}(\mu')$
such that $M / M' \cong C_{r-1,s}(\mu'')$.
To get the uniqueness, it remains to observe that
$$
\dim \hom_{B_{r-1,s}(\delta)}(C_{r-1,s}(\mu'), \ires^{r,s}_{r-1,s} C_{r,s}(\la)) = 1,
$$
as follows from Lemma~\ref{schur}
using 
the fact that
$\iind^{r,s}_{r-1,s} C_{r-1,s}(\mu') \cong C_{r,s}(\la)$
and adjointness.
\end{proof}

\phantomsubsection{Jucys-Murphy elements}
The elements (\ref{xl})--(\ref{xr}) are examples of {\em Jucys-Murphy
elements} in the walled Brauer algebra. 
The definition of these elements in general depends also on a fixed choice
of a sequence $R  = R^{(1)} R^{(2)} \cdots R^{(r+s)}$
consisting of $r$ $E$'s and $s$ $F$'s, i.e. an element of the set $\Seq_{r,s}$
from the introduction.
The corresponding Jucys-Murphy elements
are denoted $x_1^{R},\dots,x_{r+s}^{R}$ and are defined recursively
as follows.
The base case $r+s=0$ is vacuous.
For the induction step,
let $RE$ and $RF$ be the sequences obtained from $R$
by adding an $E$ or an $F$ to the end.
Assuming $x_1^R,\dots,x_{r+s}^R \in B_{r,s}(\delta)$
have been defined already, let
$x_1^{RE},\dots,x_{r+s+1}^{RE} \in B_{r+1,s}(\delta)$
and
$x_1^{RF},\dots,x_{r+s+1}^{RF} \in B_{r,s+1}(\delta)$
be defined by setting
\begin{align}\label{jm1}
x_a^{RE} := \iota_{r,s}^{r+1,s}(x_a^{R}),
\qquad
x^{RE}_{r+s+1} := x_{r,s}^{r+1,s},\\
x_a^{RF} := \iota_{r,s}^{r,s+1}(x_a^{R}),
\qquad
x^{RF}_{r+s+1} := x_{r,s}^{r,s+1},
\label{jm2}\end{align}
for $a=1,\dots,r+s$.
When it is not ambiguous,
we drop the superscript $R$, denoting the Jucys-Murphy
elements in $B_{r,s}(\delta)$ simply by $x_1,\dots,x_{r+s}$.
Note always that $x_1=0$.

It is clear from the definition that $x_1,\dots,x_{r+s}$
generate a finite dimensional commutative subalgebra 
of $B_{r,s}(\delta)$.
For $\bi \in \C^{r+s}$, define the {\em weight idempotent}
$e(\bi)$ to be the unique idempotent in this commutative
subalgebra with the property that
$e(\bi) M = M_\bi$ for any finite dimensional module $M$,
where
\begin{equation}\label{jmid}
M_\bi = \left\{v\in M\:\bigg|\:
\begin{array}{l}
\text{$(x_a - i_a)^N v = 0$ for $N \gg 0$ if $R^{(a)} = E$,}\\
\text{$(x_a + i_a+\delta)^N v = 0$ for $N \gg 0$ if $R^{(a)} = F$}
\end{array}\right\}.
\end{equation}
The idempotents 
$\{e(\bi)\:|\:\bi \in \C^{r+s}\}$
are mutually orthogonal, all but finitely many of them are zero,
and they sum to $1$.

\begin{Remark}\rm
One can check that any symmetric polynomial in the Jucys-Murphy
elements
$x_1,\dots,x_{r+s}$ belongs to the center of $B_{r,s}(\delta)$.
For example, the central element $z_{r,s}$ from (\ref{zrs}) is
$x_1+\cdots+x_{r+s}$.
We conjecture by analogy with a well-known result about the symmetric
group
that every element of the center of $B_{r,s}(\delta)$ 
can be expressed as a symmetric polynomial in $x_1,\dots,x_{r+s}$.
\end{Remark}

\section{Arc algebras and special projective functors}\label{s4}

Next we review the construction and basic properties of 
Khovanov's arc algebras and their generalisations from \cite{K2}, \cite{BS1}.
Actually for notational convenience 
we work only with 
the universal arc algebra $K$, which
is the (infinite) direct sum of all the indecomposable arc algebras considered in \cite{BS1}.
We also recall the definition of projective
functors following \cite{BS2},
and introduce certain 
{special
projective
functors} which play a key role in all the subsequent applications.

\phantomsubsection{\boldmath The universal arc algebra}
Recall the definition of a {\em weight diagram} from the introduction.
There is a partial order $\leq$ on the set of weight diagrams called
the {\em Bruhat order}, which is
generated by the basic operation of switching a pair of (not
necessarily neighbouring) labels $\down\: \up$;
getting bigger in the Bruhat order means $\down$'s move to the right.
Also let $\sim$ be the equivalence relation 
defined by declaring that $\lambda\sim\mu$ if $\mu$ can be obtained
from $\lambda$ by permuting $\down$'s and $\up$'s (and doing nothing
to $\circ$'s and $\cross$'s).
We refer to $\sim$-equivalence classes of weight diagrams as {\em blocks}.

Given any weight diagram $\lambda$, there is a general procedure to close
$\lambda$ below to obtain its {\em cup diagram} $\underline{\lambda}$; see \cite[$\S$2]{BS1}.
Briefly, to construct $\underline{\lambda}$, 
one connects as many
$\down\: \up$ pairs of vertices as possible together 
by anti-clockwise cups then adds
vertical rays down to infinity at all other vertices labelled $\up$ or 
$\down$,
without allowing any crossings of cups and/or rays.
For example, if
$$
\begin{picture}(53,10)
\put(-108,0.5){$\la=\cdots$}
\put(152,0.5){$\cdots$}
\put(-68.8,3){\line(1,0){215}}
\put(-67.7,3.1){$\down$}
\put(-44.7,-1.6){$\up$}
\put(46.8,1){$\cross$}
\put(93.2,.2){$\circ$}
\put(-21.7,-1.6){$\up$}
\put(70.3,-1.6){$\up$}
\put(1.3,3.1){$\down$}
\put(116.3,-1.6){$\up$}
\put(139.3,3.1){$\down$}
\put(24.3,3.1){$\down$}
\end{picture}
$$
(where all the omitted vertices to the left are labelled $\up$ and
the ones to the right are labelled $\down$)
then 
$$
\begin{picture}(53,30)
\put(-108,21){$\underline{\la}=\cdots$}
\put(152,21){$\cdots$}
\put(-68.8,24){\line(1,0){215}}
\put(142,24){\line(0,-1){28}}
\put(-53.5,24){\oval(23,23)[b]} 
\put(61.5,24){\oval(115,52)[b]}
\put(50,24){\oval(46,23)[b]}
\put(-19,24){\line(0,-1){28}} 
\end{picture}
$$
(with a ray at 
all other 
vertices).
There is a similar procedure to close a weight diagram 
$\mu$ above to obtain its {\em cap diagram} $\overline{\mu}$; 
it is just the mirror image of $\underline{\mu}$ in the horizontal axis.
We stress that the vertices in $\underline{\lambda}$
and $\overline{\mu}$
are not labelled. 

Given weight diagrams $\la,\mu\sim\al$,
we can concatenate to obtain 
the labelled diagrams $\underline{\la} \al$ and
$\al \overline{\mu}$.
We say that these diagrams 
are {\em consistently oriented}
if each cup or cap 
has exactly one endpoint labelled $\down$
and one labelled $\up$,
and moreover all the rays labelled $\up$ are to the left of
all the rays labelled $\down$.
Introduce the shorthand $\lambda \subset \alpha$
to indicate
that $\lambda \sim \alpha$ and $\underline{\la} \alpha$
is consistently oriented.
By \cite[Lemma 2.3]{BS1}, this implies automatically that $\lambda \leq \alpha$
in the Bruhat order, indeed, $\leq$ is the transitive closure of the relation $\subset$.
Similarly we write $\alpha \supset \mu$ if $\alpha\sim\mu $
and $\alpha \overline{\mu}$ is 
consistently oriented, which implies that $\alpha \geq \mu$.
Finally
if $\lambda \subset \alpha \supset \mu$ we can 
draw the oriented cup diagram $\underline{\la}\alpha$
{under} the oriented cap diagram $\alpha \overline{\mu}$
to obtain the {\em oriented circle diagram} $\underline{\la}\al\overline{\mu}$. Here is an example of such a diagram:
$$
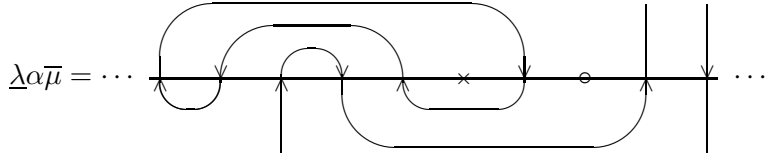
\begin{picture}(53,56)
\put(-122,21){$\underline{\la}\al\overline{\mu}=\cdots$}
\put(152,21){$\cdots$}
\put(-44.7,23.1){$\down$}
\put(-67.7,18.4){$\up$}
\put(-21.7,18.4){$\up$}
\put(1.3,23.1){$\down$}
\put(70.3,23.1){$\down$}
\put(46.8,21){$\cross$}
\put(24.3,18.4){$\up$}
\put(93.2,20.4){$\circ$}
\put(116.3,18.4){$\up$}
\put(139.3,23.1){$\down$}
\put(-68.8,23){\line(1,0){215}}
\put(-53.5,23){\oval(23,23)[b]} 
\put(61.5,23){\oval(115,52)[b]}
\put(50,23){\oval(46,23)[b]}
\put(-19,23){\line(0,-1){23}} 
\put(-53.5,23){\oval(23,23)[b]} 
\put(-7.5,23){\oval(23,23)[t]} 
\put(-7.5,23){\oval(69,40)[t]}
\put(4,23){\oval(138,57)[t]}
\put(-19,23){\line(0,-1){28}} 
\put(119,23){\line(0,1){28}}
\put(142,-5){\line(0,1){56}}
\end{picture}
$$
We define the {\em degree}
of any consistently oriented diagram to be the total number of clockwise cups and caps that it contains; the example above is of degree $4$.

Now we follow the construction in \cite{BS1}
to define a positively graded algebra $K$, which we call the {\em
  universal arc algebra}.
As a graded vector space, $K$ has the homogeneous basis
\begin{equation}\label{base1}
\left\{(\underline{\la}\alpha\overline{\mu})\:\big|\:
\text{for all weight diagrams }\alpha\text{ and }\la\subset\alpha\supset\mu
\right\},
\end{equation}
and the degree of the
vector $(\underline{\la} \al \overline{\mu})$ 
is the degree of the underlying diagram.
For $\la \subset \al \supset \mu$ and $\nu \subset \be \supset \omega$,
the product 
$(\underline{\la}\al\overline{\mu})
(\underline{\nu}\be\overline{\omega})$
is zero whenever $\mu \neq \nu$. If $\mu = \nu$
the product
is computed
by drawing $\underline{\la} \alpha \overline{\mu}$ under $\underline{\nu}\beta \overline{\omega}$, joining up corresponding pairs of rays,
then iterating a certain (generalised)
{\em surgery procedure} to contract the symmetric
$\overline{\mu}\underline{\nu}$-section 
of the diagram to obtain a (possibly zero) sum of new basis vectors
of the form $(\underline{\la} \gamma \overline{\omega})$ with $\al \leq 
\ga \geq \be$.
This surgery procedure is explained in detail in \cite[$\S$6]{BS1}; see especially
the example \cite[(6.3)]{BS1} and also \cite[Corollary 4.5]{BS1}.
The fact that the multiplication in $K$ is well defined
independent of the order chosen for the surgery procedures,
and that it is associative, is established in \cite[$\S$4]{BS1} by comparing with
Khovanov's topological definition of the multiplication in terms
of the two-dimensional TQFT 
associated to the commutative Frobenius algebra
$H^*(\mathbb{P}^1,\C) \cong \C[x] / (x^2)$.

The degree zero subalgebra $K_0$ has a basis
consisting of mutually orthogonal primitive idempotents
\begin{equation}\label{locun}
\{e_\lambda\:|\:\text{for all weight diagrams }\lambda\}
\end{equation}
defined from $e_\lambda := (\underline{\lambda}\lambda \overline{\lambda})$.
For $\mu \subset \al \supset \nu$, we have that
\begin{align}\label{idem1}
e_\la ( \underline{\mu} \al \overline{\nu} )&=
\begin{cases}
( \underline{\mu}\al\overline{\nu} )&\text{if $\la=\mu$,}\\
0&\text{otherwise,}
\end{cases}
&( \underline{\mu}\al \overline{\nu} ) e_\la
&=
\begin{cases}
( \underline{\mu}\al\overline{\nu} ) &\text{if $\la = \nu$,}\\
0&\text{otherwise.}
\end{cases}
\end{align}
The algebra $K$ is not unital. Instead the
idempotents
(\ref{locun}) form a system of local units, i.e. we have that
$K = \bigoplus_{\la,\mu} e_\la K e_\mu$ summing over all weight
diagrams $\la$ and $\mu$.
This is clear from (\ref{idem1}).

For any block $\Ga$, we let
$K_\Ga := \bigoplus_{\la,\mu \in \Ga} e_\la K e_\mu$, which is a
subalgebra of $K$.
We then have that
\begin{equation}\label{blockdec}
K = \bigoplus_{\text{all blocks }\Ga} K_\Ga
\end{equation}
as an algebra.
Each of the subalgebras $K_\Ga$ is indecomposable; this follows by
\cite[(5.9)]{BS1} as $\subset$ generates the partial order $\sim$. 
So (\ref{blockdec}) is
the decomposition of the algebra $K$ into blocks.
If $\Ga$ is {finite} then $K_\Ga$ is a finite dimensional
unital algebra with $1 = \sum_{\la \in \Ga} e_\la$. If $\Ga$ is infinite
then $K_\Ga$ is only locally unital, and the idempotents
$\{e_\la\:|\:\la \in \Ga\}$ provide a system of local units.

\phantomsubsection{\boldmath Graded modules}
Let $\Mod{K}$ be the category
of graded left $K$-modules $M$ such that
\begin{itemize}
\item[(1)]
$M$ is locally finite dimensional, i.e. 
$\dim M_i < \infty$ for each $i \in \Z$;
\item[(2)]
$M$ is bounded below, i.e. $M_i = \{0\}$ for $i \ll 0$;
\item[(3)]
$M$ is 
locally unital
in the sense that $M = \bigoplus_{\lambda} e_\lambda M$ summing over
all weight diagrams $\la$.
\end{itemize}
We write $M \langle j \rangle$ for the module $M$ shifted up in degree 
by $j$, so $M \langle j\rangle_i = M_{i-j}$.
Also for $M, N \in \Mod{K}$ we write 
$\hom_{K}(M, N)_j$ for the space of all homogeneous
$K$-module homomorphisms of degree $j$, that is, the
ones that map $M_i$ to $N_{i+j}$ for each $i \in \Z$, then define the
graded vector space
\begin{equation*}
\hom_{K}(M,N) = \bigoplus_{j \in \Z} \hom_{K}(M,N)_j.
\end{equation*}

For a weight diagram $\mu$, 
we have the {\em projective indecomposable module}
$P(\mu) := K e_\mu$,
which has basis
\begin{equation}\label{pbj}
\{(\underline\la \alpha \overline{\mu})\:|\:
\text{for all 
$\la \subset \alpha$ such that $\alpha \supset \mu$}\}.
\end{equation}
The 
vectors $(\underline{\la} \al \overline{\mu})$ from this basis
with $\alpha \neq \mu$
span a proper submodule of $P(\mu)$,
and the corresponding quotient is the
{\em standard module}
$V(\mu)$; see \cite[Theorem 5.1]{BS1}.
Finally we have the one-dimensional 
module $L(\mu) := P(\mu) / \rad P(\mu)
= V(\mu) / \rad V(\mu)$. The modules $\{L(\mu)\}$ give a full set of pairwise non-isomorphic irreducible $K$-modules (up to grading shift).
For later use, let \begin{equation}\label{proj}
p_\mu:P(\mu) \twoheadrightarrow L(\mu)
\end{equation} 
denote the canonical 
quotient map.

\begin{Theorem}[{\cite{BS1},\cite{BS2}}]\label{basict}
The following hold for weight diagrams $\lambda$ and $\mu$.
\begin{itemize}
\item[(1)]
The graded composition multiplicity $d_{\la,\mu}(q)
= \sum_{i \in \Z} d_{\la,\mu}^{(i)} q^i$
defined from 
$d_{\la,\mu}^{(i)} := [V(\mu):L(\la)\langle i 
\rangle]$ 
is given explicitly by
$$
d_{\la,\mu}(q) = \left\{
\begin{array}{ll}
q^{\deg(\underline{\la}\mu)}&\text{if $\la \subset \mu$,}\\
0&\text{otherwise.}
\end{array}
\right.
$$
\item[(2)]
The projective indecomposable module $P(\la)$ has a 
finite length filtration whose sections are
isomorphic to degree-shifted standard modules, with
$V(\la)$ 
at the top and 
$V(\mu)\langle j \rangle$ appearing with 
multiplicity $d_{\la,\mu}^{(j)}$.
\item[(3)]
Let $p_{\lambda,\mu}(q) = \sum_{i \geq 0}
p_{\lambda,\mu}^{(i)} q^i \in \mathbb{N}[q]$ 
be the polynomial
from \cite[(5.4)]{BS2}, which
is a certain Kazhdan-Lusztig polynomial for a
Grassmannian (see \cite[Remark 5.1]{BS2})
with $p_{\lambda,\mu}(q) = 0$
if $\la \not\leq \mu$.
Then the matrices $(d_{\lambda,\mu}(q))_{\lambda,\mu}$
and $(p_{\lambda,\mu}(-q))_{\lambda,\mu}$
are inverse to each other.
\item[(4)]
The standard module $V(\lambda)$ has a canonical linear projective resolution
$$
\cdots \rightarrow P_1(\lambda) \rightarrow P_0(\lambda) \rightarrow V(\lambda)
\rightarrow 0
$$
in which $P_0(\lambda) = P(\lambda)$ and $P_i(\lambda) =
\bigoplus_{\mu \geq \la} p_{\lambda,\mu}^{(i)} P(\mu) \langle i\rangle.$
Moreover the
irreducible modules $L(\lambda)$
also have linear projective resolutions, hence the algebra $K$ is Koszul.
\end{itemize}
\end{Theorem}

\begin{proof}
See \cite[Theorem 5.2]{BS1}, \cite[Theorem 5.1]{BS1}, \cite[Corollary 5.4]{BS2}
and \cite[Theorem 5.3, Corollary 5.13]{BS2}, respectively.
\end{proof}

\phantomsubsection{Geometric bimodules and projective functors}
Recall from \cite[$\S$2]{BS2} that a {\em crossingless matching}
is a diagram obtained by drawing 
a cap diagram underneath a cup diagram,
then joining 
pairs of rays according to some order-preserving bijection between 
the vertices at their endpoints. This means that $t$ is 
a union of finitely many cups and caps and (possibly infinitely many)
line segments.
Let $\cups(t)$ and $\caps(t)$ denote the number of cups and caps
in $t$, respectively, and let 
$t^*$ be the mirror image of $t$ in its horizontal axis.
Given blocks $\De$ and $\Ga$,
a {\em $\De\Ga$-matching} means a crossingless matching
$t$ such that the vertices 
at the bottom (resp.\ top) of $t$ that are not at the ends of 
cups, caps or line segments
are in exactly the same positions as the vertices labelled $\circ$ or $\cross$
in the weight diagrams from $\De$ (resp. $\Ga$).

Assume we are 
given a $\De\Ga$-matching $t$.
For $\alpha \in \Delta$ and $\beta \in \Gamma$,
we can form the labelled diagram $\al  t \beta$ (so $\al$ labels the bottom number line and $\beta$ labels the top).
We say $\alpha t \beta$ is {\em consistently oriented}
if each cup and cap has exactly one endpoint labelled $\down$ and one labelled
$\up$, and the endpoints of each line segment are labelled by the same 
symbol.
We write $\alpha \stackrel{t}{\longline} \beta$
to indicate 
 that $\alpha \in \Delta, \beta \in \Gamma$
and $\alpha t \beta$ is consistently oriented.
This is a slight abuse of notation:
we are assuming that $t$ remembers the blocks $\De$ and $\Ga$.
For $\la \subset \alpha \stackrel{t}{\longline} \beta \supset \mu$, we 
can concatenate (in order from bottom to top as usual)
to get the composite diagram $\underline{\la}\alpha t \beta \overline{\mu}$.

The {\em geometric bimodule} 
$K^t_{\De\Ga}$
is a certain graded 
$(K,K)$-bimodule constructed explicitly in \cite[$\S$3]{BS2}
(extending Khovanov's definition in a special case from \cite{K2}).
As a graded vector space, it has homogeneous basis
\begin{equation*}
\big\{
(\underline{\la} \alpha t \beta \overline{\mu})\:\big|\:
\text{for all }
\la \subset \alpha \stackrel{t}{\longline} \beta \supset \mu
\big\}
\end{equation*}
where the degree of a basis vector is the degree of the diagram as 
above.
The bimodule structure is defined as follows.
Given basis vectors
$(\underline{\nu} \alpha \overline{\omega}) \in K$ 
and
$(\underline{\la} \ga t \beta \overline{\mu}) \in K^t_{\De\Ga}$,
the product
$(\underline{\nu} \ga \overline{\omega}) (\underline{\la} \al t \beta \overline{\mu})$
is zero unless $\omega = \la$, in which case it is computed by
drawing $\underline{\nu} \ga \overline{\omega}$ underneath
$\underline{\la}\al t \beta\overline{\mu}$
then applying the surgery procedure to contract
the symmetric $\overline{\omega}\underline{\la}$-section of the diagram.
The product
$(\underline{\la} \al t \beta \overline{\mu})
(\underline{\nu} \alpha \overline{\omega})$
is defined similarly, this time putting
$\underline{\nu} \alpha \overline{\omega}$
on top.
Note in particular that all blocks from (\ref{blockdec})
other than $K_\Delta$
act on the left by zero; similarly 
all blocks other than $K_\Ga$ act on the right by zero.

The next two lemmas summarise some basic facts about geometric bimodules.
For the first, suppose we are given 
another block $\Pi$ and a $\Pi\De$-matching $s$.
We can glue $s$ under $t$ 
to obtain the {\em composite matching} 
$st$, which has 
three number lines corresponding in order from bottom to top to the blocks
$\Pi, \De$ and $\Ga$.
Its {\em reduction} $\red(st)$ is the $\Pi\Ga$-matching obtained 
by removing the middle number line and all internal circles; see \cite[$\S$2]{BS2}
for an example.
Let
$K^{st}_{\Pi\De\Ga}$ be the graded vector space with homogeneous basis
\begin{equation*}
\big\{(\underline{\la} \alpha s \beta t \gamma \overline{\mu})\:\big|\:\text{for all }\la \subset \alpha \stackrel{s}{\longline} \beta \stackrel{t}{\longline} \gamma \supset \mu\big\}.
\end{equation*}
Mimicking the construction in the previous paragraph,
we make $K^{st}_{\Pi\De\Ga}$ into a graded $(K,K)$-bimodule.

\begin{Lemma}[{\cite[Theorem 3.5(iii), Theorem 3.6]{BS2}}]\label{composition}
Let 
$s$ be a $\Pi\De$-matching,
$t$ be a $\De\Ga$-matching
and $u := \red(st)$.
There are canonical graded bimodule isomorphisms
$$
K_{\Pi\De}^s \otimes_{K} K_{\De\Ga}^t\cong K_{\Pi\De\Ga}^{st}
\cong \bigoplus_{\beta\:\text{with}\:\al \stackrel{s}{\,\longline\,} \beta \stackrel{t}{\,\longline\,} \ga}
K_{\Pi\Ga}^{u}\langle 
\deg(\al s \be t \ga) - \deg(\al u \ga)\rangle
$$
where in the direct sum 
$\al$ and $\ga$ are fixed weight diagrams satisfying
$\al \stackrel{u}{\longline} \ga$; the sum should be 
interpreted as zero if no such $\al$ and $\ga$ exist.
\end{Lemma}

\begin{Lemma}[{\cite[Theorem 4.7]{BS2}}]\label{adunctions}
There is a canonical graded bimodule isomorphism 
$K^{t}_{\De\Ga}\langle -\caps(t) \rangle 
\cong
\hom_{K}(K^{t^*}_{\Ga\De}\langle-\caps(t)\rangle, K).$
\end{Lemma}

Now we introduce the {\em indecomposable projective functors}
\begin{align}\label{legg1}
G_{\De\Ga}^t &:= K_{\De\Ga}^t \langle -\caps(t)\rangle \otimes_{K} ?
:\Mod{K} \rightarrow \Mod{K},\\
\lonestar G_{\De\Ga}^t &:= 
K_{\Ga\De}^{t^*} \langle -\caps(t)\rangle \otimes_{K} ?
:\Mod{K} \rightarrow \Mod{K}.\label{legg2}
\end{align}
More generally, a {\em projective functor} means any endofunctor 
of $\Mod{K}$ that is isomorphic
to a direct sum of such functors (possibly shifted in degree).
By
Lemma~\ref{composition}, the composition of two projective functors is itself
a projective functor.

By Lemma~\ref{adunctions}, the indecomposable projective functor $G^{t}_{\De\Ga}$ is 
isomorphic to the functor
$\hom_{K}(K^{t^*}_{\Ga\De}\langle-\caps(t)\rangle, ?)$,
which is right adjoint to 
$\lonestar G^t_{\De\Ga}$ by adjointness of tensor and hom.
Hence there is a canonical adjunction making
$(\lonestar G_{\De\Ga}^t, G^{t}_{\De\Ga})$
into an adjoint pair of functors.
We deduce that $G^t_{\De\Ga}$ is exact.
Similarly so is 
\begin{equation}\label{other}
\lonestar G^t_{\De\Ga} \cong G^{t^*}_{\Ga\De} \langle \cups(t)-\caps(t)\rangle,
\end{equation} 
and both functors send projectives to projectives.

We note finally that projective functors map finite dimensional modules
to finite dimensional modules; see \cite[Corollary 4.12]{BS2}.

\phantomsubsection{Projective functors on irreducible modules}
Continue with $t$ being a 
$\De\Ga$-matching.
For $\la \in \De$ and $\mu \in \Ga$,
we can concatenate in order from bottom to top
as usual
to obtain the composite diagrams $\underline{\la} t$
and $t \overline{\mu}$.
The {\em lower reduction} $\lowerred(\underline{\la} t)$ 
is the ordinary cup diagram
obtained from $\underline{\la} t$ by removing the bottom number line
and all 
connected components 
that do not extend up to
the top number line.
Similarly the {\em upper reduction} $\upperred(t \overline{\mu})$
is the ordinary cap diagram obtained from $t \overline{\mu}$ 
by removing the top number line.

For $\mu \in \Ga$,
the $K$-module $G^t_{\De\Ga} P(\mu)$ is identified with 
$K^t_{\De\Ga} e_\mu \langle -\caps(t)\rangle$, hence it has the diagram basis
\begin{equation}\label{pfp}
\big\{(\underline{\lambda} \al t \beta \overline{\mu})\:\big|\:
\text{for all $\la\subset\al \stackrel{t}{\longline}\be$
such that $\be \supset \mu$}
\big\}.
\end{equation}
Applying $G^t_{\De\Ga}$ 
to (\ref{proj}),
we 
get a surjection
$G^t_{\De\Ga}(p_\mu):G^t_{\De\Ga} P(\mu) \twoheadrightarrow G^t_{\De\Ga} L(\mu)$
which identifies $G^t_{\De\Ga} L(\mu)$ 
with a quotient of $G^t_{\De\Ga} P(\mu)$.
The following theorem
describes this quotient explicitly.

\begin{Theorem}\label{ipf}
Let $t$ be a $\De\Ga$-matching and $\mu \in \Ga$.
Then $G^t_{\De\Ga} L(\mu)$
is the quotient of $G^t_{\De\Ga} P(\mu)$ by the submodule spanned by all the 
vectors from
(\ref{pfp}) in which either $\lowerred(\underline{\la}t) \neq \underline{\mu}$ or $\be \neq \mu$.
Hence $G^t_{\De\Ga} L(\mu)$ is finite dimensional
with basis given by 
the image of the vectors
\begin{equation}\label{klot}
\big\{(\underline{\la} \al t \mu \overline{\mu})\:\big|\:
\text{for all $\la \subset \al \stackrel{t}{\longline} \mu$
such that $\lowerred(\underline{\la}t) = \underline{\mu}$}\big\}.
\end{equation}
Moreover, $G^t_{\De\Ga} L(\mu)$ is non-zero if and only if there exists a 
(necessarily unique) $\lambda$ such that $\lambda \stackrel{t}{\longline} \mu$ and $\deg(\la t \mu) = 0$, in which case
$G^t_{\De\Ga} L(\mu)$ is a self-dual indecomposable module with
head 
$L(\la)\langle-\caps(t)\rangle$ and socle
$L(\la)\langle\caps(t)\rangle$.
\end{Theorem}

\begin{proof}
The final statement is \cite[Theorem 4.11(ii)--(iii)]{BS2}.
For the other parts, \cite[Theorem 4.11(i)]{BS2} 
and \cite[Corollary 4.12]{BS2}
show that
$G^t_{\De\Ga} L(\mu)$ is finite dimensional with
\begin{equation}\label{ld}
[G^t_{\De\Ga} L(\mu)]
= \sum_{\substack{\la \subset \al \stackrel{t}{\text{---}} \mu\\
\lowerred(\underline{\la} t) = \underline{\mu}
}} 
[L(\la)\langle \deg(\underline{\la} \alpha t \mu)
- \caps(t) \rangle],
\end{equation}
equality written 
in the Grothendieck group of $\Mod{K}$.
In particular the dimension of 
$G^t_{\De\Ga} L(\mu)$ is exactly the
number of vectors listed in (\ref{klot}).
It remains to show that
$(\underline{\la}\al t \be \overline{\mu}) \in G^t_{\De\Ga} P(\mu)$ 
belongs to 
the kernel of the canonical
map $G^t_{\De\Ga}(p_\mu) :G^t_{\De\Ga} P(\mu)\twoheadrightarrow G^t_{\De\Ga} L(\mu)$
either if $\be \neq \mu$ or if
$\lowerred(\underline{\la}t) \neq \underline{\mu}$.
The first paragraph of the proof of \cite[Theorem 4.5]{BS2}
shows that 
$G^t_{\De\Ga} V(\mu)$
is the quotient of $G^t_{\De\Ga}P(\mu)$
by the submodule spanned by all the vectors 
$(\underline{\la}\al t \be \overline{\mu})$
with $\be \neq \mu$.
Hence all of these vectors lie in $\ker G^t_{\De\Ga} (p_\mu)$.
Moreover (\ref{ld}) implies
$e_\la G^t_{\De\Ga} L(\mu) = \{0\}$ for all $\la \in \De$
with $\lowerred(\underline{\la} t) \neq \underline{\mu}$, hence all the vectors
$(\underline{\la}\al t \be \overline{\mu})$
with 
$\lowerred(\underline{\la}t) \neq \underline{\mu}$ lie in
the kernel too.
\end{proof}

\begin{Corollary}\label{socbase}
Let $t$ be a $\De\Ga$-matching and $\mu \in \Ga$
such that
$G^t_{\De\Ga} L(\mu)$ is non-zero.
Let $\la$ be the unique weight diagram
with $\la \stackrel{t}{\longline} \mu$
and $\deg(\la t \mu) = 0$,
and let $\la'$
be obtained from $\la$ by 
reversing the orientation of all the caps in $\la t \mu$.
Then $G^t_{\De\Ga} L(\mu)$ is a cyclic module generated by the image of
the vector $(\underline{\la} \la t \mu \overline{\mu})$, and its irreducible
socle is generated by the image of the vector
$(\underline{\la} \la' t \mu \overline{\mu})$.
\end{Corollary}

\begin{proof}
This follows from the descriptions of the 
head and socle of
$G^t_{\De\Ga} L(\mu)$
in Theorem~\ref{ipf}, on noting that the vectors 
$(\underline{\la}\la t \mu \overline{\mu})$
and
$(\underline{\la}\la' t \mu \overline{\mu})$
in the basis (\ref{klot})
are of degrees $-\caps(t)$ and $\caps(t)$, respectively.
\end{proof}

\phantomsubsection{\boldmath Special projective functors}
For $i \in \Z$, we introduce the following {\em special projective functors}:
\begin{align}\label{proje}
E_i &:= \bigoplus_{\Ga,\De,s} G^s_{\Ga\De},
&F_i &:= \bigoplus_{\De,\Ga,t} G^{t}_{\De\Ga},
\\
\lonestar E_i &:= \bigoplus_{\Ga,\De,s} \lonestar G^s_{\Ga\De},
&
\lonestar F_i &:= \bigoplus_{\De,\Ga,t} \lonestar G^{t}_{\De\Ga},\label{projf}
\end{align}
where the direct sums are over all blocks $\Ga,\De$
and crossingless matchings $s, t$ such that
\begin{itemize}
\item[(1)] the labels $\circ$ and $\cross$ on the 
$i$th and $(i+1)$th vertices of weights in $\Ga$ and $\De$ match one of the configurations displayed in (\ref{here});
\item[(2)] the labels $\circ$ and $\cross$ on all other vertices are in the exactly same positions in weights from $\Ga$ and from $\De$;
\item[(3)] $s$ is the $\Ga\De$-matching
and $t$ is the $\De\Ga$-matching
displayed in the strip between the $i$th and $(i+1)$th vertices in 
(\ref{here}), and $s$ and $t$ have vertical line segments
connecting all other vertices not labelled $\circ$ or $\cross$
in weights from $\Ga$ and $\De$.
\end{itemize}

\begin{equation*}
\begin{picture}(75,20)
\put(152.6,12){$\De$}
\put(153.5,0.8){$s$}
\put(152,-11.4){$\Ga$}
\put(-84,15){\line(1,0){33}}
\put(-84,-9){\line(1,0){33}}
\put(-67.5,-9){\oval(23,23)[t]}
\put(-59.3,13.1){$\cross$}
\put(-81.8,12.4){$\circ$}
\put(-24,15){\line(1,0){33}}
\put(-24,-9){\line(1,0){33}}
\put(-7.5,15){\oval(23,23)[b]}
\put(1.2,-11.6){$\circ$}
\put(-22.2,-10.9){$\cross$}
\put(36,15){\line(1,0){33}}
\put(36,-9){\line(1,0){33}}
\put(38.2,12.4){$\circ$}
\put(61.2,-11.6){$\circ$}
\qbezier(41,-9)(41,3)(52.5,3)
\qbezier(64,15)(64,3)(52.5,3)
\put(96,15){\line(1,0){33}}
\put(96,-9){\line(1,0){33}}
\put(97.7,-10.9){$\cross$}
\put(120.7,13.1){$\cross$}
\qbezier(101,15)(101,3)(112.5,3)
\qbezier(124,-9)(124,3)(112.5,3)
\end{picture}
\end{equation*}
\begin{equation}
\begin{picture}(75,22)
\put(-81,2){$\text{$_i$\quad\:\,$_{i+1}$}$}
\put(-21,2){$\text{$_i$\quad\:\,$_{i+1}$}$}
\put(36,2){$\text{$_i$\quad\:\,$_{i+1}$}$}
\put(99.5,2){$\text{$_i$\quad\:\,$_{i+1}$}$}
\end{picture}
\label{here}\end{equation}
\begin{equation*}
\begin{picture}(75,21)
\put(152,12){$\Ga$}
\put(153.5,0.8){$t$}
\put(152.6,-11.4){$\De$}
\put(-84,15){\line(1,0){33}}
\put(-84,-9){\line(1,0){33}}
\put(-67.5,15){\oval(23,23)[b]}
\put(-59.3,-10.9){$\cross$}
\put(-81.8,-11.6){$\circ$}
\put(-24,15){\line(1,0){33}}
\put(-24,-9){\line(1,0){33}}
\put(-7.5,-9){\oval(23,23)[t]}
\put(1.2,12.4){$\circ$}
\put(-22.2,13.1){$\cross$}
\put(36,15){\line(1,0){33}}
\put(36,-9){\line(1,0){33}}
\put(61.2,12.4){$\circ$}
\put(38.2,-11.6){$\circ$}
\qbezier(64,-9)(64,3)(52.5,3)
\qbezier(41,15)(41,3)(52.5,3)
\put(96,15){\line(1,0){33}}
\put(96,-9){\line(1,0){33}}
\put(120.7,-10.9){$\cross$}
\put(97.7,13.1){$\cross$}
\qbezier(124,15)(124,3)(112.5,3)
\qbezier(101,-9)(101,3)(112.5,3)
\end{picture}
\end{equation*}
\vspace{1mm}

\noindent
So $E_i, F_i, \lonestar E_i$ and $\lonestar F_i$
are the endofunctors of $\Mod{K_\La}$ defined by tensoring with the graded $(K_\La,K_\La)$-bimodules
\begin{align}\label{bim}
\widetilde{E}_i &:= \bigoplus_{\Ga,\De,s} K^s_{\Ga\De}\langle-\caps(s)\rangle,
&\widetilde{F}_i &:= \bigoplus_{\De,\Ga, t} K^{t}_{\De\Ga}\langle-\caps(t)\rangle,\\
\lonestar\widetilde{E}_i &:= \bigoplus_{\Ga,\De, s} K^{s^*}_{\De\Ga}\langle-\caps(s)\rangle,
&\lonestar\widetilde{F}_i &:= \bigoplus_{\De,\Ga,t} K^{t^*}_{\Ga\De}\langle-\caps(t)\rangle,
\label{bim2}
\end{align}
respectively,
summing over the same sets as before.
By the general theory of projective functors
reviewed earlier,
there are canonical adjunctions making
$(\lonestar E_i, E_i)$ and $(\lonestar F_i, F_i)$ into adjoint pairs.
All four functors are exact,
they map projectives to projectives, and they map finite dimensional
modules to finite dimensional modules.

\phantomsubsection{Special projective functors on standard modules}
Given weight diagrams $\la$ and $\mu$ and $i \in \Z$, we write
$\la \stackrel{i}{\rightarrow} \mu$ if $\la$ and $\mu$
have the same labels on all but the $i$th and $(i+1)$th vertices,
and the $i$th and $(i+1)$th vertices of $\la$ and $\mu$ are labelled
as in one of the following cases:
\begin{align}
\label{break1}
\la \stackrel{i}{\rightarrow} \mu:\qquad
\begin{array}{|c|c|c|c|c|c|c|c|c|c|c|}
\hline
\la&\down\circ&\up\circ&\cross\down&\cross\up&\cross\circ&\cross\circ&\down\up&\up\down\\
\hline
\mu&\circ\down&\circ\up&\down\cross&\up\cross&\down\up&\up\down&\circ\cross&\circ\cross
\\\hline
\text{degree}
&0&0&0&0&0&1&-1&0\\
\hline
\end{array}
\end{align}
Similarly, we write
$\la \stackrel{i}{\leftarrow} \mu$ if $\la$ and $\mu$
have the same labels on all but the $i$th and $(i+1)$th vertices,
and the $i$th and $(i+1)$th vertices of $\la$ and $\mu$ are labelled
as in one of the following cases:
\begin{align}\label{break2}
\la \stackrel{i}{\leftarrow} \mu:
\qquad
\begin{array}{|c|c|c|c|c|c|c|c|c|c|c|}
\hline
\la&\circ\down&\circ\up&\down\cross&\up\cross&\down\up&\up\down&\circ\cross&\circ\cross\\
\hline
\mu&\down\circ&\up\circ&\cross\down&\cross\up&\cross\circ&\cross\circ&\down\up&\up\down\\
\hline
\text{degree}&0&0&0&0&-1&0&0&1\\
\hline
\end{array}
\end{align}
Then define the {\em edge degrees}
\begin{align}\label{edgedeg1}
\deg(\la \stackrel{i}{\rightarrow} \mu) &:=
\#\{\mu' < \mu\:|\:\la \stackrel{i}{\rightarrow} \mu'\}
-\#\{\la' > \la\:|\:\la' \stackrel{i}{\rightarrow} \mu\},\\
\deg(\la \stackrel{i}{\leftarrow} \mu) &:=
\#\{\mu' < \mu\:|\:\la \stackrel{i}{\leftarrow} \mu'\}
-\#\{\la' > \la\:|\:\la' \stackrel{i}{\leftarrow} \mu\}.\label{edgedeg2}
\end{align}
The explicit values of these degrees are the numbers
listed in the third rows of (\ref{break1}) and (\ref{break2}),
respectively.
Note that $\la \stackrel{i}{\rightarrow} \mu$
if and only if $\mu \stackrel{i}{\leftarrow} \la$; however
the degrees $\deg(\la \stackrel{i}{\rightarrow} \mu)$
and $\deg(\mu \stackrel{i}{\leftarrow} \la)$ are in general different.

\begin{Theorem}\label{wasaremark}
The following hold for any weight diagram $\la$ and $i \in \Z$.
\begin{itemize}
\item[(1)]
The module $E_i V(\la)$ has a filtration with sections isomorphic 
to the modules
$\{V(\mu) \langle \deg(\mu \stackrel{i}{\rightarrow}
\la)\rangle\:|\:\text{for all }\mu\text{ such that }\mu \stackrel{i}{\rightarrow} \la\}$.
\item[(2)]
The module $F_i V(\la)$ has a filtration with sections isomorphic to the modules
$\{V(\mu) \langle \deg(\mu \stackrel{i}{\leftarrow}
\la)\rangle\:|\:\text{for all }\mu\text{ such that }\mu \stackrel{i}{\leftarrow} \la\}$.
\item[(3)]
The module $\lonestar E_i V(\la)$ has a filtration with sections isomorphic 
to the modules
$\{V(\mu) \langle \deg(\la \stackrel{i}{\rightarrow}
\mu)\rangle\:|\:\text{for all }\mu\text{ such that }\la \stackrel{i}{\rightarrow} \mu\}$.
\item[(4)]
The module $\lonestar F_i V(\la)$ has a filtration with sections isomorphic to the modules
$\{V(\mu) \langle \deg(\la \stackrel{i}{\leftarrow}
\mu)\rangle\:|\:\text{for all }\mu\text{ such that }\la \stackrel{i}{\leftarrow} \mu\}$.
\end{itemize}
In all cases, the filtration is of length at most two; when it is
of length exactly two the standard module indexed by the biggest weight in the Bruhat ordering $\geq$ appears as a submodule and the other is a quotient.
\end{Theorem}

\begin{proof}
Parts (1) and (2) follow from 
\cite[Theorem 4.5]{BS2}.
Then (3) and (4) follow from (2) and (1), respectively, using (\ref{other}).
\end{proof}

\begin{Remark}\rm\label{newremark}
To make the connection with Theorem~\ref{ibranch} in the previous
section, fix $\delta \in \Z$, and suppose that $\La$ is the 
set of all bipartitions identified with a set of weight diagrams
according to the weight dictionary 
(\ref{Iu})--(\ref{dict}).
Then $\La$ is a connected component of the labelled directed graph
just defined, and the labelled directed edges 
between elements of $\La$ are exactly the same as in the labelled
directed graph introduced
just before Theorem~\ref{ibranch}.
\end{Remark}

\phantomsubsection{Special projective functors on PIMs}
For a weight diagram $\la, i \in \Z$
and a symbol $x \in \{\circ,\up,\down,\cross\}$
we write $\la^i_x$ for the weight diagram obtained from $\la$
by relabelling the $i$th vertex by $x$.
Then for $i < j$ and symbols $x,y$
we set $\la^{ij}_{xy} := (\la^i_x)^j_y$, and so on.

\begin{Theorem}\label{ga}
The following hold for any weight diagram $\la$ and $i \in \Z$.
\begin{itemize}
\item[(1)]
If $\la = \la^{i(i+1)}_{{\down}{\textstyle\circ}}$ then
$\lonestar E_i P(\la) \cong P(\la^{i(i+1)}_{{\textstyle\circ}{\down}})$.
\item[(2)]
If $\la = \la^{i(i+1)}_{{\scriptscriptstyle\up}{\textstyle\circ}}$ then
$\lonestar E_i P(\la) \cong P(\la^{i(i+1)}_{{\textstyle\circ}{\up}})$.
\item[(3)]
If $\la = \la^{i(i+1)}_{\cross\down}$ then
$\lonestar E_i P(\la) \cong P(\la^{i(i+1)}_{\down\cross})$.
\item[(4)]
If $\la = \la^{i(i+1)}_{\cross\up}$ then
$\lonestar E_i P(\la) \cong P(\la^{i(i+1)}_{\up\cross})$.
\item[(5)]
If $\la = \la^{i(i+1)}_{\cross{\textstyle\circ}}$ then
$\lonestar E_i P(\la) \cong P(\la^{i(i+1)}_{\down\up})$.
\item[(6)]
If $\la = \la^{i(i+1)}_{{\down\up}}$ then
$\lonestar E_i P(\la) \cong
P(\la^{i(i+1)}_{{\textstyle\circ}\cross})\langle 1 \rangle\oplus
P(\la^{i(i+1)}_{{\textstyle\circ}\cross})\langle -1\rangle$.
\item[(7)]
If $\la = \la^{i(i+1)}_{{\up\down}}$ then
$\lonestar E_i P(\la)\cong P(\la^{i(i+1)}_{{\textstyle\circ}\cross})$.
\item[(8)]
If
$\la = \la^{i(i+1)}_{{\down\down}}$
and there is a cap connecting vertex $i+1$ to vertex $j > i+1$ in
$\overline{\la}$
then
$\lonestar E_i P(\la) \cong P(\la^{i(i+1)j}_{{\textstyle\circ}\cross\down})$.
\item[(9)]
If
$\la = \la^{i(i+1)}_{{\up\up}}$
and there is a cap connecting vertex $i$ to vertex $j < i$ in $\overline{\la}$
then
$\lonestar E_i P(\la) \cong P(\la^{ji(i+1)}_{\up{\textstyle\circ}\cross})$.
\item[(10)]
For all other $\la$ we have that
$\lonestar E_i P(\la) = \{0\}$.
\end{itemize}
For the analogous results with $\lonestar E_i$ replaced by $\lonestar F_i$,
interchange all occurrences of $\circ$ and $\cross$.
\end{Theorem}

\begin{proof}
This can be deduced using \cite[Theorem 4.2]{BS2}.
\end{proof}

\section{Idempotent truncations of $K(\delta)$}\label{s4b}

In this section, we fix $\delta \in \Z$.
We are going to establish some basic properties of 
the algebra $K_{r,s}(\delta)$ from 
Theorem~\ref{fin}, i.e. the algebra which 
will eventually be shown to be Morita
equivalent to the walled Brauer algebra $B_{r,s}(\delta)$.
In particular we introduce counterparts for $K_{r,s}(\delta)$
of the 
walled Brauer algebras's $i$-induction and
$i$-restriction
functors,
culminating in a result which is parallel to
Theorem~\ref{ibranch}.

\phantomsubsection{Graded cellular algebras}
We begin the section by recalling briefly the definition of a {\em graded cellular
algebra}
from \cite{GL}, \cite{HM}.
This means an
associative unital algebra $H$
together with a {\em cell datum} $(X, I, C, \deg)$ such that
\begin{itemize}
\item[(1)]
$X$ is a finite partially ordered set;
\item[(2)] 
$I(\la)$
is a finite set
 for each $\la \in X$;
\item[(3)]
$C:\dot\bigcup_{\la \in X} I(\la)
\times I(\la) \rightarrow H,
(i,j) \mapsto C^\la_{i,j}$ is an injective map
whose image is a basis for $H$;
\item[(4)]
the map $H \rightarrow H, C^\la_{i,j} \mapsto C_{j,i}^\la$
is an algebra anti-automorphism;
\item[(5)]
if $\la \in X$ and $i,j \in I(\la)$
then for any $x \in H$ we have that
$$
x\, C_{i,j}^\la
\equiv \sum_{i' \in I(\la)} r_x(i',i) C_{i',j}^\la
\pmod{H_{> \la}}$$
where the scalar $r_x(i',i)$ is independent of $j$ and
$H_{> \la}$ denotes the subspace of $H$
spanned by $\{C_{k,l}^{\mu}\:|\:\mu > \la,
k,l \in I(\mu)\}$;
\item[(6)]
$\deg:\dot\bigcup_{\la \in X}
I(\la) \rightarrow \Z, i \mapsto \deg^\la_i$
is a function such that
the $\Z$-grading on $H$ defined by declaring each
$C^\la_{i,j}$ is of degree $\deg^\la_i+\deg^\la_j$
makes $H$ into a graded algebra.
\end{itemize}
Assuming $H$ is a graded cellular algebra,
we get for $\la \in X$ an associated {\em graded cell
  module} $V(\la)$ with 
basis $\{v_i\:|\:i \in I(\la)\}$.
It is a graded left $H$-module with 
$\deg(v_i) := \deg^\la_i$
and the action of $x \in H$ defined from
$x\, v_i := \sum_{i' \in I(\la)} r_x(i',i)
v_{i'}$.
Moreover there is a canonically defined 
homogeneous symmetric bilinear form $(.,.)$ 
on
$V(\la)$ with radical $\rad V(\la)$.
The quotient
$L(\la) := V(\la) / \rad V(\la)$ 
is either zero or an irreducible graded $H$-module,
and the non-zero $L(\la)$'s give a complete set of pairwise
inequivalent irreducible graded $H$-modules (up to degree shift).
See \cite{GL}, \cite{HM} for details.

\phantomsubsection{\boldmath Graded cellular structure on $K_{r,s}(\delta)$}
Let $\La$ be the set of all bipartitions
identified with a set of weight diagrams
using the weight dictionary 
(\ref{Iu})--(\ref{dict}). 
Let $K(\delta)$ be the subalgebra $\bigoplus_{\la,\mu \in \La} e_\la K
e_\mu$
of the universal arc algebra $K$ from the previous section.
Note that $\La$ is a union of blocks.
So the algebra $K(\delta)$ 
is a sum of some of the blocks of $K$ from (\ref{blockdec}):
\begin{equation}\label{blockdec2}
K(\delta) = \bigoplus_{\Ga\in \La / \sim} K_\Ga.
\end{equation}
Let $\Mod{K(\delta)}$ be the full subcategory of $\Mod{K}$
consisting of all modules $M$ such that $M = \bigoplus_{\la \in
  \La} e_\la M$. 
For $\la \in \La$, the $K$-modules
$P(\la)$, $V(\la)$ and $L(\la)$ 
from the previous section belong to $\Mod{K(\delta)}$. 
Moreover Theorem~\ref{basict} specialises in an
obvious way.

Let $e_{r,s} =
\sum_{\lambda \in \dot\Lambda_{r,s}} e_\lambda
\in K(\delta)$ be the idempotent from (\ref{ers})
and set
\begin{equation}\label{krs}
K_{r,s}(\delta) := e_{r,s} K(\delta) e_{r,s}.
\end{equation}
It is immediate from (\ref{idem1}) that 
$K_{r,s}(\delta)$ has a basis consisting of diagrams
$(\underline{\la} \al \overline{\mu})$ for all $\la \subset \al \supset \mu$
with $\la,\mu \in \dot \La_{r,s}$ and $\al \in \La$.
The weight $\al$ here automatically belongs to $\La_{r,s}$; this follows
because $\La_{r,s}$ is an upper set in the poset $(\La, \leq)$,
i.e. $\la \in \La_{r,s}, \mu \in \La, \la \leq \mu \Rightarrow \mu \in \La_{r,s}$.
So the basis for $K_{r,s}(\delta)$ can be denoted
\begin{equation*}
\{(\underline{\la} \al \overline{\mu}\:|\:\text{for all }
\la \subset \al \supset \mu\text{ with }\la,\mu \in \dot\La_{r,s}\text{ and } \al \in \La_{r,s}\}.
\end{equation*}
In particular $K_{r,s}(\delta)$ is a 
positively graded, finite dimensional algebra with identity $e_{r,s}$.

Multiplication by the idempotent $e_{r,s}$ defines an exact
functor
\begin{equation*}
e_{r,s}:\Mod{K(\delta)} \rightarrow \Mod{K_{r,s}(\delta)},
\end{equation*}
where $\Mod{K_{r,s}(\delta)}$ is the category of graded left
$K_{r,s}(\delta)$-modules that are locally finite dimensional
and bounded below.

\begin{Lemma}\label{stupidone}
Let $M$ and $N$ be left $K(\delta)$-modules.
If $M$ is a finite direct sum of summands of
$K(\delta) e_{r,s}$ then the natural 
map 
$$
\hom_{K(\delta)}(M,N) \rightarrow \hom_{K_{r,s}(\delta)}(e_{r,s} M, e_{r,s} N)
$$
defined by the functor $e_{r,s}$ is an isomorphism.
\end{Lemma}

\begin{proof}
It is obviously an isomorphism for $M = K(\delta) e_{r,s}$.
\end{proof}

\begin{Lemma}\label{stupidtwo}
Let $M$ be a right $K(\delta)$-module and $N$ be a left $K(\delta)$-module.
If $M$ is a direct sum of summands of $e_{r,s}K(\delta)$
or $N$ is a direct sum of summands of $K(\delta) e_{r,s}$,
then the natural map
$$
M e_{r,s} \otimes_{K_{r,s}(\delta)} e_{r,s} N
\rightarrow
M \otimes_{K(\delta)} N,
\qquad
m \otimes n \mapsto m \otimes n
$$
is an isomorphism.
\end{Lemma}

\begin{proof}
It is obviously true if $M = e_{r,s} K(\delta)$ or $N = K(\delta) e_{r,s}$.
\end{proof}

For $\la\in \La$, let
\begin{align}
P_{r,s}(\la) &:= e_{r,s} P(\la),\\
V_{r,s}(\la) &:= e_{r,s} V(\la),\label{as}\\
L_{r,s}(\mu) &:= e_{r,s}L(\mu).
\end{align}

\begin{Lemma}\label{zeroic}
The following statements hold for any $\la \in \La$.
\begin{itemize}
\item[(1)] $L_{r,s}(\la) \neq \{0\}$ if and only if $\la \in \dot\La_{r,s}$.
\item[(2)] $V_{r,s}(\la) \neq \{0\}$ if and only if $\la \in \La_{r,s}$.
\end{itemize}
\end{Lemma}

\begin{proof}
(1) This is immediate from the definition of the idempotent $e_{r,s}$.

\noindent
(2) Suppose first that $\la \in \La \setminus \La_{r,s}$.
All composition factors of $V(\la)$
are of the form $L(\mu)\langle j \rangle$ for $\mu \leq \la$,
hence $\mu \notin \La_{r,s}$, so we get that $e_{r,s} V(\la) = \{0\}$ by (1).
Conversely if $\la \in \dot\La_{r,s}$
then $V(\la)$ has head $L(\la)$, so $e_{r,s} V(\la) \neq \{0\}$ by (1).
Finally if $\la = (\varnothing,\varnothing) \in \La_{r,s}\setminus \dot\La_{r,s}$,
then $V(\la)$ has a composition factor isomorphic to $L(\mu) \langle 1 \rangle$
where $\mu = ((1),(1)) \in \dot\La_{r,s}$ by Theorem~\ref{basict}(1),
hence again $e_{r,s} V(\la) \neq \{0\}$.
\end{proof}

\begin{Theorem}\label{iscell}
The algebra $K_{r,s}(\delta)$ is a
graded cellular algebra 
with cell datum $(X, I, C,\deg)$
defined by
\begin{align*}
X &:= \La_{r,s},
&I(\alpha) &:= 
\{\la \in 
\dot\La_{r,s}\:|\: \la \subset \alpha\},\\
C^\alpha_{\la,\mu} &:= (\underline{\la} \alpha \overline{\mu}),
&\deg^\alpha_\la &:=
\deg(\underline{\la} \alpha)
\end{align*}
 for
$\alpha \in \La_{r,s}$ and
$\la,\mu \in I(\alpha)$.
Moreover the following hold.
\begin{itemize}
\item[(1)]
The graded cell modules 
are the modules
$\{V_{r,s}(\la)\:|\:\la \in \La_{r,s}\}$.
\item[(2)]
For $\la \in \dot\La_{r,s}$ the cell module 
$V_{r,s}(\la)$ has one-dimensional head $L_{r,s}(\la)$,
and
the modules $\{L_{r,s}(\la)\:|\:\la \in \dot\La_{r,s}\}$ give a complete set of
pairwise non-isomorphic irreducible $K_{r,s}(\delta)$-modules (up to degree shift).
\item[(3)]
The projective cover of $L_{r,s}(\la)$
is $P_{r,s}(\la) = K_{r,s}(\delta) e_\la$ for each $\la \in \dot \La_{r,s}$.
\item[(4)]
For $\la \in \dot\La_{r,s}$ and $\mu \in \La_{r,s}$,
the graded composition multiplicity 
$$
d_{\la,\mu}(q)
= \sum_{i \in \Z} [V_{r,s}(\mu):L_{r,s}(\la)\langle i 
\rangle] q^i
$$
is given explicitly by the formula from Theorem~\ref{basict}(1).
\item[(5)]
Irreducible modules $L_{r,s}(\la)$ and $L_{r,s}(\mu)$
for $\la, \mu \in \dot \La_{r,s}$ belong to the same block if and only if
$\la \sim \mu$.
\item[(6)]
For each $\la \in \La_{r,s}$, there is a resolution
$$
\cdots \rightarrow e_{r,s} P_1(\lambda) \rightarrow e_{r,s}P_0(\lambda) \rightarrow V_{r,s}(\lambda)
\rightarrow 0
$$
in which $e_{r,s}P_i(\lambda) = \bigoplus_{\mu \in \Lambda_{r,s}} p_{\lambda,\mu}^{(i)} P_{r,s}(\mu) \langle i\rangle$,
and $p_{\lambda,\mu}(q) = \sum_{i \geq 0}
p_{\lambda,\mu}^{(i)} q^i$ is as in Theorem~\ref{basict}(3).
\item[(7)]
If $\La_{r,s} = \dot \La_{r,s}$, then $K_{r,s}(\delta)$ is a 
graded quasi-hereditary algebra in the sense of \cite{CPS3},
and the $V_{r,s}(\la)$'s are its graded standard modules.
Moreover the resolution in (6) is a linear projective resolution, hence
$K_{r,s}(\delta)$ is standard Koszul in the sense of \cite{ADL}.
\end{itemize}
\end{Theorem}

\begin{proof}
The cellularity is deduced from \cite[Theorem 4.4]{BS1}
in exactly the same way as \cite[Corollary 4.6]{BS1}.

\noindent
(1) By definition,
the graded cell module 
indexed
by $\la \in \La_{r,s}$ has
a distinguished homogeneous 
basis $$
\{v_\mu\:|\:\text{for all }\mu \subset \la\text{ with }\mu \in \dot\La_{r,s}\}.
$$
Moreover the action of a basis vector
$(\underline{\nu} \alpha \overline{\omega}) \in K_{r,s}(\delta)$
on $v_\mu$
is computed as follows.
First compute the product $(\underline{\nu} \alpha \overline{\omega})(\underline{\mu} \la \overline{\la})$
in the algebra $K_{r,s}(\delta)$
to get a (possibly zero) sum of basis vectors of the form
$(\underline{\nu} \beta \overline{\la})$.
Then
$(\underline{\nu} \alpha \overline{\omega}) v_\mu$
is equal to $v_\nu$ if $(\underline{\nu} \la \overline{\la})$
appears in this sum, and 
$(\underline{\nu} \alpha \overline{\omega}) v_\mu = 0$ otherwise.
On the other hand, using the original definition of $V(\la)$
as a quotient of $P(\la)$ given just after (\ref{pbj}),
we see that $V_{r,s}(\la)$
has a basis given by the canonical images of the 
vectors $$
\{(\underline{\mu} \la \overline{\la})\:|\:\text{for all }\mu \subset \la
\text{ with }\mu\in\dot\La_{r,s}\},
$$ and the action is defined by the algorithm just explained.
Hence $V_{r,s}(\la)$ is isomorphic to the graded cell module parametrised by $\la$.

\noindent
(2), (3) These are standard consequences of Lemma~\ref{zeroic}(1),
since $L(\la)$ is the irreducible head of both $V(\la)$ and $P(\la)$.

\noindent
 (4) This follows from Theorem~\ref{basict}(1).

\noindent
 (5) This is a combinatorial exercise using
(4) and the usual description of the blocks of a cellular algebra
in terms of its decomposition matrix.

\noindent
 (6) This follows by applying 
$e_{r,s}$ to the resolution in Theorem~\ref{basict}(4),
using also Lemma~\ref{zeroic}(2).

\noindent
 (7) The fact that $K_{r,s}(\delta)$ is quasi-hereditary follows
from the general theory of cellular algebras 
since the number of isomorphism classes of irreducible modules
is equal to the number of cell modules when $\La_{r,s} = \dot\La_{r,s}$.
The fact that the resolution is a linear projective resolution
follows from (3).
\end{proof}

\phantomsubsection{Idempotent truncations of special projective functors}
In view of Remark~\ref{newremark}, the special projective functors 
from (\ref{proje})--(\ref{projf}) restrict to well-defined functors
\begin{equation*}
E_i, F_i, \lonestar E_i, \lonestar F_i: \Mod{K(\delta)} \rightarrow \Mod{K(\delta)}.
\end{equation*}
By definition, these functors are defined by tensoring with the graded
bimodules from (\ref{bim})--(\ref{bim2}).
Tensoring with various truncations of 
these bimodules
gives us the following right exact functors:
\begin{align}\label{never1}
\ires^{r+1,s}_{r,s}
:= e_{r,s} \widetilde{E}_i e_{r+1,s}\, \otimes_{K_{r+1,s}(\delta)}\, ?&:
\Mod{K_{r+1,s}(\delta)} \rightarrow \Mod{K_{r,s}(\delta)},\\
\ires^{r,s+1}_{r,s}
:= e_{r,s} \widetilde{F}_i e_{r,s+1}\, \otimes_{K_{r,s+1}(\delta)}\, ?&:
\Mod{K_{r,s+1}(\delta)} \rightarrow \Mod{K_{r,s}(\delta)},\\
\iind^{r+1,s}_{r,s}
:= e_{r+1,s} \lonestar\widetilde{E}_i e_{r,s}\, \otimes_{K_{r,s}(\delta)}\, ?&:
\Mod{K_{r,s}(\delta)} \rightarrow \Mod{K_{r+1,s}(\delta)},
\\
\iind^{r,s+1}_{r,s}
:= e_{r,s+1} \lonestar\widetilde{F}_i e_{r,s}\, \otimes_{K_{r,s}(\delta)}\, ?&:
\Mod{K_{r,s}(\delta)} \rightarrow \Mod{K_{r,s+1}(\delta)}.\label{never4}
\end{align}

\begin{Lemma}\label{genee}
For $\la \in \dot \La_{r,s}$ and $i \in \Z$, 
$\lonestar E_i P(\la)$ is isomorphic to a direct sum of degree-shifted copies of 
$P(\mu)$
for some $\mu \in \dot\La_{r+1,s}$.
Conversely, given $\mu \in \dot \La_{r+1,s}$, it is possible to choose
$\la \in \dot \La_{r,s}$ and $i \in \Z$ such that
$\lonestar E_i P(\la)$ has a summand isomorphic (up to a degree shift) to $P(\mu)$.
The same two statements hold with $\lonestar E_i$ replaced by $\lonestar F_i$ and $\dot\La_{r+1,s}$
replaced by $\dot \La_{r,s+1}$.
\end{Lemma}

\begin{proof}
We just prove the lemma for $\lonestar E_i$.
First take $\la \in \dot \La_{r,s}$.
The fact that $\lonestar E_i P(\la)$ is a sum of copies of degree-shifted copies 
$P(\mu)$ for some 
$\mu \in \dot\La_{r+1,s}$
follows from Theorem~\ref{ga}.
Conversely, take $\mu \in \dot\La_{r+1,s}$.
We want to show there exists $\la \in \dot\La_{r,s}$ and $i \in \Z$
such that $P(\mu)$ is isomorphic (up to a degree shift) to a 
summand of $\lonestar E_i P(\la)$.
Some observations:
\begin{itemize}
\item[(1)]
If we can find $i \in \Z$ 
such that the $i$th and $(i+1)$th
vertices of the weight diagram of $\mu$ are labelled
$\circ \up$, $\down \cross$ or $\down\up$, then we are done by
Theorem~\ref{ga}(2), (3) or (5), respectively.
\item[(2)]
If we can find $i \in \Z$ such that the $i$th and $(i+1)$th vertices
of $\mu$ are labelled $\circ \down$, $\up\cross$ or $\circ\cross$
then
we are done
by Theorem~\ref{ga}(1), (4) or (6), respectively,
unless
$s = |\mu^{\mathrm R}|$.
\item[(3)]
If we can find $i \in \Z$ such that the $i$th and $(i+1)$th vertices
of $\mu$ are labelled
$\circ\cross$, we are done
by Theorem~\ref{ga}(7) unless $\mu = ((1),\varnothing)$,
$\delta = 0$ and $s > 0$; moreover we are done in the exceptional case by (2).
\end{itemize}
Using (1) and (3) 
we are reduced to the situation that
every $\cross$ or $\up$ in $\mu$
has either $\cross$ or $\up$ immediately to its left.
This implies that $\mu^{\mathrm L} = \varnothing$, hence $s > |\mu^{\mathrm R}|$.
Hence by (2) every $\cross$ must have actually $\cross$ to its left, i.e.
since the number of $\cross$'s is finite 
there are actually no $\cross$'s at all. 
Then (2) once again gives that $\mu = (\varnothing,\varnothing)$ and $\delta = 0$, which contradicts the assumption that $\mu \in \dot\La_{r+1,s}$.
\end{proof}

\begin{Lemma}\label{summands}
The following hold for all $r, s \geq 0$ and $i \in \Z$.
\begin{itemize}
\item[(1)] 
The graded left $K(\delta)$-module $\lonestar\widetilde{E}_i e_{r,s}$ is isomorphic to a finite direct
sum of degree-shifted summands of
$K(\delta) e_{r+1,s}$.
\item[(2)]
The graded left $K(\delta)$-module $\lonestar\widetilde{F}_i e_{r,s}$ 
is isomorphic to a finite direct sum of degree-shifted summands of
$K(\delta) e_{r,s+1}$.
\item[(3)]
The graded right $K(\delta)$-module $e_{r,s}\widetilde{E}_i$ is isomorphic to
a finite direct sum of degree-shifted summands 
of $e_{r+1,s} K(\delta)$.
\item[(4)]
The graded right $K(\delta)$-module $e_{r,s}\widetilde{F}_i$ is isomorphic to
a finite direct sum of degree-shifted summands of $e_{r,s+1}K(\delta)$.
\end{itemize}
\end{Lemma}

\begin{proof}
Parts (1) and (2) follow from Lemma~\ref{genee}.
 Parts (3) and (4) follow from (1) and (2), respectively,
on considering the mirror images of the diagrams in a horizontal axis.
\end{proof}

\begin{Lemma}\label{comm}
For $i \in \Z$ and $r, s \geq 0$, there are canonical
isomorphisms of functors
\begin{align*}
\ires^{r+1,s}_{r,s} \circ e_{r+1,s} \cong e_{r,s} \circ E_i&:
\Mod{K(\delta)} \rightarrow \Mod{K_{r,s}(\delta)},\\
\ires^{r,s+1}_{r,s} \circ e_{r,s+1} \cong e_{r,s} \circ F_i
&:
\Mod{K(\delta)} \rightarrow \Mod{K_{r,s}(\delta)}.
\end{align*}
\end{Lemma}

\begin{proof}
We just construct the first isomorphism, since the second is similar.
By Lemma~\ref{summands}(3) and Lemma~\ref{stupidtwo},
the natural multiplication map is a bimodule isomorphism
$e_{r,s} \widetilde{E}_i e_{r+1,s} \otimes_{K_{r+1,s}(\delta)} e_{r+1,s} K(\delta)
\cong e_{r,s} \widetilde{E}_i.$
It remains to observe that $\ires^{r+1,s}_{r,s} \circ e_{r+1,s}$ and $e_{r,s} \circ E_i$ are the functors defined by tensoring with these bimodules.
\end{proof}

\begin{Lemma}\label{adjun}
There are canonical
adjunctions making $(\iind^{r+1,s}_{r,s}, \ires^{r+1,s}_{r,s})$ and $(\iind^{r,s+1}_{r,s}, \ires^{r,s+1}_{r,s})$ 
into adjoint pairs of functors.
In particular, the functors $\ires^{r+1,s}_{r,s}$ and $\ires^{r,s+1}_{r,s}$ are exact.
\end{Lemma}

\begin{proof}
We just construct the adjunction for $(\iind^{r,s+1}_{r,s}, \ires^{r,s+1}_{r,s})$, since the
argument for $(\iind^{r+1,s}_{r,s}, \ires^{r+1,s}_{r,s})$ is similar.
The restriction of the isomorphism from Lemma~\ref{adunctions} gives a 
graded $(K_{r,s}(\delta), K_{r,s+1}(\delta))$-bimodule isomorphism
$$
e_{r,s} K^{t}_{\De\Ga} e_{r,s+1} \langle-\caps(t)\rangle \cong 
e_{r,s} \hom_{K(\delta)}(K^{t^*}_{\Ga\De}\langle-\caps(t)\rangle, K(\delta)) e_{r,s+1}
$$
for each $\De, \Ga$ and $t$ as in (\ref{bim2}).
The bimodule on the right
is identified with
$\hom_{K(\delta)}(K^{t^*}_{\Ga\De}\langle-\caps(t)\rangle
e_{r,s}, K(\delta) e_{r,s+1})$.
Summing over all $\De, \Ga$ and $t$, we get a graded bimodule isomorphism
$$
e_{r,s} \widetilde{F}_i e_{r,s+1}\cong 
\hom_{K(\delta)}(\lonestar\widetilde{F}_i e_{r,s}, K(\delta) e_{r,s+1}).
$$
By Lemma~\ref{summands}(2) and Lemma~\ref{stupidone},
the functor $e_{r,s+1}$ defines an isomorphism
$$
\hom_{K(\delta)}(\lonestar\widetilde{F}_i e_{r,s}, K(\delta) e_{r,s+1})
\stackrel{\sim}{\rightarrow} \hom_{K_{r,s+1}(\delta)} (e_{r,s+1} \lonestar\widetilde{F}_i e_{r,s}, e_{r,s+1} K(\delta) e_{r,s+1}).
$$
We have now constructed a graded 
$(K_{r,s}(\delta), K_{r,s+1}(\delta))$-bimodule
isomorphism
$$
e_{r,s} \widetilde{F}_i e_{r,s+1}
\cong
\hom_{K_{r,s+1}(\delta)} (e_{r,s+1} \lonestar\widetilde{F}_i e_{r,s}, 
K_{r,s+1}(\delta)).
$$
Hence the functor
$\ires^{r,s+1}_{r,s} = e_{r,s} \widetilde{F}_i e_{r,s+1} \otimes_{K_{r,s+1}(\delta)} ?$ 
is isomorphic to the functor
$\hom_{K_{r,s+1}(\delta)}(e_{r,s+1} \lonestar\widetilde{F}_i e_{r,s}, ?)$.
The latter functor is canonically right adjoint to
$\iind^{r,s+1}_{r,s} = e_{r,s+1} \lonestar\widetilde{F}_i e_{r,s} \otimes_{K_{r,s}(\delta)} ?$.
\end{proof}

\begin{Lemma}\label{globalisation}
For $r, s \geq 0$ such that $(r,s,\delta) \neq (0,0,0)$,
the right exact functor
$$
G_{r,s}:=K_{r+1,s+1}(\delta) e_{r,s} \otimes_{K_{r,s}(\delta)} ?:\Mod{K_{r,s}(\delta)} \rightarrow
\Mod{K_{r+1,s+1}(\delta)}
$$
has the following properties:
\begin{itemize}
\item[(1)] $G_{r,s} P_{r,s}(\la) \cong P_{r+1,s+1}(\la)$
for each $\la \in \dot \La_{r,s}$;
\item[(2)] $G_{r,s} V_{r,s}(\la)\cong V_{r+1,s+1}(\la)$ 
for each $\la \in \La_{r,s}$.
\end{itemize}
\end{Lemma}

\begin{proof}
(1) Take $\la \in \dot \La_{r,s}$ and recall from Theorem~\ref{iscell}(3) 
that $P_{r,s}(\la) = K_{r,s}(\delta) e_\la$.
Hence
$G_{r,s} P_{r,s}(\la) =
K_{r+1,s+1}(\delta) e_{r,s} \otimes_{K_{r,s}} P_{r,s}(\la) \cong
K_{r+1,s+1}(\delta) e_\la = P_{r+1,s+1}(\la).$

\noindent
(2) We first show that $G_{r,s} V_{r,s}(\la) \cong V_{r+1,s+1}(\la)$
in the exceptional case $\delta = 0$, $r=s \neq 0$
and $\la = (\varnothing,\varnothing)$.
Observe that $P(\la) = V(\la)$ by Theorem~\ref{basict}(2).
Let $\mu = ((1),(1))$, 
$\nu' = ((2),(1^2))$ and $\nu'' = ((1^2),(2))$.
Right multiplication by $(\underline{\mu} \la \overline{\la})$
defines a graded $K(\delta)$-module homomorphism
$P(\mu) \langle 1 \rangle \twoheadrightarrow \rad V(\la)$.
It is easy to see that the kernel of this is generated by the vectors
$(\underline{\nu}' \mu \overline{\mu})$ and
$(\underline{\nu}'' \mu \overline{\mu})$.
Hence there is an exact sequence
\begin{equation}\label{exact1}
(P(\nu')\oplus
P(\nu'')) \langle 2 \rangle
\rightarrow P(\mu) \langle 1 \rangle
\rightarrow \rad V(\la) \rightarrow 0.
\end{equation}
Multiplication defines a bimodule homorphism
\begin{equation}\label{bio}
e_{r+1,s+1} K(\delta) e_{r,s} \otimes_{K_{r,s}} e_{r,s} K(\delta)
\rightarrow e_{r+1,s+1} K(\delta).
\end{equation}
Tensoring (\ref{bio}) and (\ref{exact1}) together,
noting $e_{r,s}(\rad V(\la)) = e_{r,s} V(\la)$
because $e_{r,s} L(\la) = \{0\}$, we get
the following commutative diagram with exact rows:
$$
\begin{CD}
(G_{r,s} P_{r,s}(\nu')
\oplus
G_{r,s} P_{r,s}(\nu'')) \langle 2 \rangle
\!&@>>> \!G_{r,s} P_{r,s}(\mu) \langle 1 \rangle
\!&@>>> \!G_{r,s} V_{r,s}(\la) &\rightarrow 0\\
@VVV&@VVV&@VVV\\
(P_{r+1,s+1}(\nu')
\oplus
P_{r+1,s+1}(\nu'')) \langle 2 \rangle
\!&@>>> \!P_{r+1,s+1}(\mu) \langle 1 \rangle
\!&@>>> \!V_{r+1,s+1}(\la) &\rightarrow 0.
\end{CD}
$$
The two vertical arrows on the left are isomorphisms by the previous paragraph.
Hence the vertical arrow on the right is an isomorphism too by the five lemma.

Finally for the general case,
we may assume by the previous paragraph
that $\la \in \dot \La_{r,s}$.
By Theorem~\ref{basict}(3)--(4),
there is an exact sequence 
\begin{equation}\label{exact2}
\bigoplus_{\mu \in \La_{r,s}} 
p_{\la,\mu}^{(1)} P(\mu) \langle 1 \rangle \rightarrow P(\la) \rightarrow V(\la) \rightarrow 0.
\end{equation}
Tensoring (\ref{bio}) and (\ref{exact2}) together, we get a commutative diagram with exact rows:
$$
\begin{CD}
\displaystyle\bigoplus_{\mu \in \La_{r,s}} 
p_{\la,\mu}^{(1)} G_{r,s} P_{r,s}(\mu) \langle 1 \rangle
&@>>> G_{r,s} P_{r,s}(\la) 
&@>>> G_{r,s} V_{r,s}(\la) &\rightarrow 0\\
@VVV&@VVV&@VVV\\
\displaystyle\bigoplus_{\mu \in \La_{r,s}} 
p_{\la,\mu}^{(1)} P_{r+1,s+1}(\mu) \langle 1 \rangle
&@>>> P_{r+1,s+1}(\la)
&@>>> V_{r+1,s+1}(\la) &\rightarrow 0.
\end{CD}
$$
The two vertical arrows on the left are isomorphisms either by (1)
or by the previous paragraph in case
$\mu = (\varnothing,\varnothing)$. Hence the arrow on the right is an isomorphism too.
\end{proof}

The following theorem should be compared with Theorem~\ref{ibranch}.

\begin{Theorem}\label{finalbranch}
The following hold for $\la \in \La_{r,s}$ and $i \in \Z$.
\begin{itemize}
\item[(1)]
The module $\ires^{r,s}_{r-1,s} V_{r,s}(\la)$ has a filtration with sections isomorphic 
to 
$\{V_{r-1,s}(\mu) \langle \deg(\mu\stackrel{i}{\rightarrow} \la\rangle\:|\:\text{for all $\mu \in \La_{r-1,s}$ such that }\mu \stackrel{i}{\rightarrow} \la\}$.
\item[(2)]
The module $\ires^{r,s}_{r,s-1} V_{r,s}(\la)$ has a filtration with sections isomorphic to 
$\{V_{r,s-1}(\mu) \langle \deg(\mu \stackrel{i}{\leftarrow} \la)\rangle\:|\:\text{for all }\mu \in \La_{r,s-1}\text{ such that }\mu \stackrel{i}{\leftarrow} \la\}$.
\item[(3)]
The module $\iind^{r+1,s}_{r,s} V_{r,s}(\la)$
has a filtration with sections isomorphic 
to 
$\{V_{r+1,s}(\mu) \langle \deg(\la \stackrel{i}{\rightarrow} \mu)\rangle\:|\:\text{for all }\mu \in \La_{r+1,s}\text{ such that }\la \stackrel{i}{\rightarrow} \mu\}$.
\item[(4)]
The module $\iind^{r,s+1}_{r,s} V_{r,s}(\la)$ has a filtration with sections isomorphic to 
$\{V_{r,s+1}(\mu) \langle \deg(\la \stackrel{i}{\leftarrow} \mu)\rangle\:|\:\text{for all }\mu \in \La_{r,s+1}\text{ such that }\la \stackrel{i}{\leftarrow} \mu\}$.
\end{itemize}
In all cases, the filtration is of length at most two; when it is
of length exactly two the 
module indexed by the biggest weight in the Bruhat ordering $\geq$ appears as a submodule and the other is a quotient.
\end{Theorem}

\begin{proof}
We just explain in the first and last cases.
For the first case, we have by Lemma~\ref{comm} that $\ires^{r,s}_{r-1,s} V_{r,s}(\la)
= \ires^{r,s}_{r-1,s} (e_{r,s}V(\la))
\cong e_{r-1,s} (E_i V(\la))$.
Now apply 
Theorem~\ref{wasaremark}(1).
Instead consider $\iind^{r,s+1}_{r,s} V_{r,s}(\la)$.
The case $(r,s,\delta) = (0,0,0)$ can be checked directly
using the observation that
$\iind^{r,s+1}_{r,s} V_{r,s}(\la) = e_{0,1} \lonestar\widetilde{F}_i e_{0,0}$.
Assuming $(r,s,\delta) \neq (0,0,0)$, we can use Lemma~\ref{globalisation}
to get that
\begin{align*}
\iind^{r,s+1}_{r,s} V_{r,s}(\la) &= e_{r,s+1} \lonestar\widetilde{F}_i e_{r,s} \otimes_{K_{r,s}(\delta)} V_{r,s}(\la)\\
&\cong 
e_{r,s+1} \lonestar\widetilde{F}_i e_{r+1,s+1} \otimes_{K_{r+1,s+1}(\delta)} K_{r+1,s+1}(\delta) e_{r,s} \otimes_{K_{r,s}(\delta)} V_{r,s}(\la)\\
&\cong 
e_{r,s+1} \lonestar\widetilde{F}_i e_{r+1,s+1} \otimes_{K_{r+1,s+1}(\delta)} V_{r+1,s+1}(\la).
\end{align*}
Up to a
grading shift, this is the same as
$\ires^{r+1,s+1}_{r,s+1} V_{r+1,s+1}(\la)$,
which has the appropriate filtration by the first case.
\end{proof}

\section{Diagram bases for special endomorphism algebras}\label{sd}

We return once again to the general setting of
$\S$\ref{s4}, so $K$ is the universal arc algebra.
We are going to construct diagram bases for the 
endomorphism algebras of the modules obtained by
applying sequences of special projective functors either
to projective
indecomposable modules or to irreducible modules.

\phantomsubsection{Good homomorphisms}
Suppose
we are given blocks $\De,\Pi$ and $\Ga$.
Let $s$ be a $\De\Pi$-matching and $t$ be a $\De\Ga$-matching.
Also fix $\la \in \Pi$ and $\mu \in \Ga$. 
In this subsection, we construct bases for 
$\hom_{K}(G^s_{\De\Pi} P(\la),
G^t_{\De\Ga} P(\mu))$ and
$\hom_{K}(G^s_{\De\Pi} L(\la),
G^t_{\De\Ga} L(\mu))$.

Given $\alpha \stackrel{s^*}{\longline} \beta \stackrel{t}{\longline}
\gamma$ such that $\la \subset \al$ and $\ga \supset \mu$, 
define a homomorphism
\begin{equation}\label{allhom}
f_{\alpha\beta\gamma}:G^s_{\De\Pi} P(\la)
\rightarrow G^t_{\De\Ga} P(\mu)
\end{equation}
by declaring 
for
$(\underline{\nu}\omega s\kappa \overline{\la}) \in G^s_{\De\Pi} P(\la)$
that 
$f_{\alpha\beta\gamma}(\underline{\nu}\omega s \kappa \overline{\la})$
is the (possibly zero) sum of 
basis vectors of 
$G^t_{\De\Ga}$ obtained
by drawing the diagram
$\underline \la \alpha s^* \beta t \gamma \overline{\mu}$
on top of
$\underline{\nu} \omega s \kappa \overline{\la}$,
then applying the following {\em extended surgery procedure}
to contract the $\omega s \kappa \overline{\la}|\underline{\la} \alpha s^* \beta$-part
of the diagram.
\begin{itemize}
\item[(1)] First consider all the mirror image pairs of closed circles
from the diagrams $\omega s \kappa \overline{\la}$ and $\underline{\la}\al s^* \beta$. If the circles from at least one of these pairs are oriented in the same way 
(both anti-clockwise or both clockwise),
the extended surgery procedure gives $0$.
\item[(2)] Assuming all the mirror image pairs of circles are oppositely oriented,
we replace $s \kappa \overline{\la}$
by $\upperred(s \overline{\la})$
and $\underline{\la} \alpha s^*$ by
$\lowerred(\underline{\la} s^*)$,
join corresponding pairs of rays, then apply the usual surgery procedure
to contract the symmetric $\upperred(s\overline{\la})\lowerred(\underline{\la} s^*)$-section of the resulting diagram.
\end{itemize}

\begin{Theorem}\label{phom}
The homomorphisms
\begin{equation}\label{thehoms}
\big\{f_{\alpha\beta\gamma}\:\big|\:\text{for all 
$\alpha \stackrel{s^*}{\longline} \beta \stackrel{t}{\longline}
\gamma$ such that $\la \subset \al$ and $\ga \supset \mu$}\big\}
\end{equation}
give a basis for $
\hom_{K}(
G^s_{\De\Pi} P(\la),
G^t_{\De\Ga} P(\mu))$. Moreover
$f_{\alpha\beta\gamma}$ is of degree 
$$
\deg(\underline{\la} \alpha s^* \beta t \gamma \overline{\mu})
+\caps(s)-2\cups(s)-\caps(t).
$$
\end{Theorem}

\begin{proof}
Using (\ref{other}), Lemma~\ref{composition} and the canonical adjunction, we get graded vector space
isomorphisms
\begin{align*}
\hom_{K}(
&G^s_{\De\Pi} P(\la),
G^t_{\De\Ga} P(\mu))\\
&\cong
\hom_{K}(
\lonestar G^{s^*}_{\Pi\De}P(\la),
G^t_{\De\Ga} P(\mu))\langle\caps(s)-\cups(s)\rangle 
\\
&\cong
\hom_{K}(
P(\la),
G^{s^*}_{\Pi\De} G^t_{\De\Ga} P(\mu))\langle\caps(s)-\cups(s)\rangle\\
&\cong
\hom_{K}(
K e_\la,
K^{s^*t}_{\Pi\De\Ga}e_\mu
)\langle \caps(s) -\cups(s) -\caps(s^*)-\caps(t)\rangle\\
&\cong
e_\la K^{s^* t}_{\Pi\De\Ga} e_\mu\langle \caps(s)-2\cups(s)-\caps(t)\rangle.
\end{align*}
The space
$e_\la K^{s^*t}_{\Pi\De\Ga} e_\mu$ has basis
given by all 
$(\underline{\la} \alpha s^* \beta t \gamma \overline{\mu})$ for 
$\alpha \stackrel{s^*}{\longline} \beta \stackrel{t}{\longline} \gamma$
 with $\la \subset \al$ and $\ga \supset \mu$.
It remains to trace these basis vectors through the explicit definitions of 
above isomorphisms 
to see that they correspond exactly to the
homomorphisms $f_{\alpha\beta\gamma}$.
\end{proof}

\begin{Remark}\label{otherdegs}\rm
One can also work with the adjoint projective functors from (\ref{legg2}).
In those terms, Theorem~\ref{phom}
shows that the homomorphisms (\ref{thehoms})
give a basis for 
$\hom_{K}(\lonestar G^{s^*}_{\Pi\De} P(\la),
\lonestar G^{t^*}_{\Ga\De} P(\mu))$.
Moreover now
$f_{\alpha\beta\gamma}$ is
of degree 
$\deg(\underline{\la} \alpha s^* \beta t \gamma \overline{\mu})
-\caps(s^*) - \cups(t)$; this follows from (\ref{other}).
\end{Remark}

A {\em good homomorphism} is
map
$f:G^s_{\De\Pi} P(\la)
\rightarrow G^t_{\De\Ga} P(\mu)$
that
factors through the canonical quotient maps to induce
$\bar f: G^s_{\De\Pi} L(\la)
\rightarrow G^t_{\De\Ga} L(\mu)$
making the following diagram commute:
$$
\begin{CD}
G^s_{\De\Pi} P(\la)&@>f>>&G^t_{\De\Ga} P(\mu)\\
@VG^s_{\De\Pi}(p_\la)VV&&@VV G^t_{\De\Ga}(p_\mu)V\\
G^s_{\De\Pi} L(\la)&@>\bar f>>&G^t_{\De\Ga} L(\mu).
\end{CD}
$$

\begin{Theorem}\label{goodhoms}
Let $u := \red(s^* t)$. 
The space
$\hom_{K}(G^s_{\De\Pi} L(\la), G^t_{\De\Ga} L(\mu))$ is non-zero
if and only if
$\la \stackrel{u}{\longline} \mu$
and $\deg(\la u \mu) = 0$.
Assuming this condition holds,
let $\la'$
be obtained from $\la$ by 
reversing the orientation of all the caps in $\la u \mu$.
\begin{itemize}
\item[(1)] 
The homomorphisms
$$
\bigg\{
f_{\al\be\ga}\:\bigg|\:
\begin{array}{l}
\text{for all $\alpha\stackrel{s^*}{\longline}\beta \stackrel{t}{\longline}
\gamma$ such that $\la \subset \alpha$, $\gamma \supset \mu$,}\\
\text{and either $\al = \la'$ or $\gamma \neq \mu$}
\end{array}
\bigg\}
$$
give a basis for the subspace of 
$\hom_{K}(G^s_{\De\Pi} P(\la), G^t_{\De\Ga} P(\mu))$
consisting of all good homomorphisms.
\item
[(2)]
The induced homomorphisms
$\bar f_{\alpha\beta\gamma}:G^s_{\De\Pi}L(\la)\rightarrow G^t_{\De\Ga} L(\mu)$
are zero for
all $\alpha \stackrel{s^*}{\longline}\beta \stackrel{t}{\longline} \gamma$
such that $\la\subset \alpha, \gamma \supset \mu$ and $\gamma \neq \mu$.
\item[(3)]
The induced homomorphisms
$\bar f_{\lambda'\be\mu}:G^s_{\De\Pi}L(\la)\rightarrow G^t_{\De\Ga} L(\mu)$
for all $\be$ such that
$\lambda' \stackrel{s^*}{\longline} \be \stackrel{t}{\longline} \mu$
give a basis for $\hom_{K}(G^s_{\De\Pi} L(\la), G^t_{\De\Ga} L(\mu))$.
\end{itemize}
\end{Theorem}

\begin{proof}
We will ignore grading shifts throughout the proof.
Also we'll often use Lemma~\ref{composition} without explicit reference
to identify the composite functor
$G^{s^*}_{\Pi\De} \circ G^t_{\De \Ga}$
with the functor $K^{s^* t}_{\Pi\De\Ga} \otimes_{K} ?$ 
hence with a direct sum of copies of $G^u_{\Pi\Ga}$.
This allows us to apply Theorem~\ref{ipf} and Corollary~\ref{socbase}
to $G^{s^*}_{\Pi\De} G^t_{\De \Ga} L(\mu)$.

For the opening statement, 
$\hom_{K}(G^s_{\De\Pi} L(\la), G^t_{\De\Ga} L(\mu)) \neq \{0\}$
if and only if $L(\la)$
appears in the socle of $G^{s^*}_{\Pi\De} G^t_{\De\Ga} L(\mu)$.
Now write 
$G^{s^*}_{\Pi\De} G^t_{\De\Ga} L(\mu)$ as a 
direct sum of copies
of $G^u_{\Pi\Ga} L(\mu)$
and apply the last part of Theorem~\ref{ipf}.

For the rest of the proof, we assume that 
$\la \stackrel{u}{\longline} \mu$ and $\deg(\la u \mu) = 0$.
Let $\la'$ be defined as in the statement of the theorem.
Note the assumption on $\la$ implies automatically that
$\la \subset \la'$ and $\lowerred(\underline{\la} u) 
= \underline{\mu}$.
Let 
$Q := e_\la K^{s^* t}_{\Pi\De\Ga} e_\mu$
and $R$, $S$ be the subspaces
spanned by the vectors
\begin{align}\label{vec1}
&\hspace{-0.9mm}\big\{
(\underline{\la} \la' s^* \beta t \mu \overline{\mu})\:\big|\:
\text{for all $\beta$ 
such that $\la' \stackrel{s^*}{\longline}
\beta \stackrel{t}{\longline} \mu$}\big\},\\
&\big\{
(\underline{\la} \alpha s^* \beta t \gamma \overline{\mu})\:\big|\:
\text{for all $\alpha \stackrel{s^*}{\longline} \beta \stackrel{t}{\longline} \gamma$ 
such that $\la\subset \alpha$,
$\ga \supset \mu$ and
$\gamma \neq \mu$}\big\},\label{vec2}
\end{align}
respectively.
Let $p: Q \twoheadrightarrow Q / S$ be the quotient 
and $i:(R+S) / S \hookrightarrow Q / S$ be the inclusion.
Adopting the shorthand 
$\HHH := \hom_{K}$,
let $q,r, j$ and $k$ be the other vertical maps indicated in the following diagram:
\begin{equation*}
\begin{CD}
\HHH(G^s_{\De\Pi} P(\la), G^t_{\De\Ga} P(\mu))
&@>\sim>\adj>&\HHH(P(\la), 
G^{s^*}_{\Pi\De} G^t_{\De\Ga} P(\mu))&@>\sim> a> &Q\\
@VV r: f \mapsto G^t_{\De\Ga}(p_\mu) \circ f V&&@VV
q:f \mapsto G^{s^*}_{\Pi\De}G^t_{\De\Ga}(p_\mu) \circ f V&&@VVp V\\
\HHH(G^s_{\De\Pi} P(\la), G^t_{\De\Ga} L(\mu))
&@>\sim>\adj>&\HHH(P(\la), G^{s^*}_{\Pi\De} G^t_{\De\Ga} L(\mu))&@>\sim>b>&Q / S\\
@AA k:f \mapsto f \circ G^s_{\De\Pi}(p_\la)A&&@AA j:f \mapsto f \circ p_\la  A&&@AAi A\\
\HHH(G^s_{\De\Pi} L(\la), G^t_{\De\Ga} L(\mu))
&@>\sim>\adj>&\HHH(L(\la), G^{s^*}_{\Pi\De} G^t_{\De\Ga} L(\mu))&@>\sim>c>&(R+S)/ S.
\end{CD}
\end{equation*}
Horizontal isomorphisms
making the diagram commute are defined as follows.
\begin{itemize}
\item[$\adj$:]
The three maps $\adj$
are the isomorphisms coming from the canonical adjunction. 
The two squares on the left commute by naturality.

\item[$a$:]
We can identify
$G^{s^*}_{\Pi\De} G^t_{\De\Ga} P(\mu)$ with $K^{s^* t}_{\Pi\De\Ga} e_\mu$.
Then evaluation at $e_\la$ gives the isomorphism $a$.
\item[$b$:]
The vertical map $q$ in the diagram
is surjective by projectivity. Its kernel is the
set of all homomorphisms 
$P(\la) \rightarrow
 G^{s^*}_{\Pi\De} G^{t}_{\De\Ga}P(\mu)$ which map $e_\la$
to $\ker G^{s^*}_{\Pi\De} G^t_{\De\Ga}(p_\mu)$.
The first statement of Theorem~\ref{ipf} explains how to identify
$\ker G^{s^*}_{\Pi\De} G^t_{\De\Ga}(p_\mu)$
with an explicit subspace of 
$K^{s^* t}_{\Pi\De\Ga} e_\mu$. It follows 
that $a(\ker q) = S$. Hence the isomorphism $a$ factors 
uniquely to induce the isomorphism $b$ making the top right square commute.
\item[$c$:]
The vertical map $j$ in the diagram is obviously injective.
Its image is the set of all homomorphisms
$P(\la) \rightarrow
 G^{s^*}_{\Pi\De} G^{t}_{\De\Ga}L(\mu)$ which map $e_\la$
to 
$\soc G^{s^*}_{\Pi\De} G^t_{\De\Ga} L(\mu)$.
Using the explicit description of this socle 
from Corollary~\ref{socbase},
it follows that $b(\im\,j) = (R+S) / S$.
Hence the isomorphism 
$b$ restricts to an isomorphism $c$ making the bottom right square commute.
\end{itemize}
Recall also from the proof of Theorem~\ref{phom} that the
composite $a \circ \adj$ of the maps at the top of our commuting diagram
sends $f_{\alpha\beta\gamma}$ to $(\underline{\la} \al s^* \be t \ga \overline{\mu})$.

To complete the proof of the theorem,
we have that
$p^{-1}(\im\, i) = R+S$.
By definition, the space of good homomorphisms
is $r^{-1}(\im\, k)$. 
Hence by the commutativity of the diagram,
the space of good homomorphisms
is $(a \circ \adj)^{-1}(R+S)$. The union of the vectors (\ref{vec1})--(\ref{vec2})
gives a basis for $R+S$, and applying $(a \circ \adj)^{-1}$ 
gives the basis in (1).
To deduce (2), we need to show that $r(f_{\alpha\beta\gamma}) = 0$ 
for all
$\al \stackrel{s^*}{\longline} \beta \stackrel{t}{\longline} \gamma$
such that $\la \subset \al, \ga \supset \mu$ and $\ga \neq \mu$.
This follows because
$$
(b \circ \adj \circ r)(f_{\alpha\beta\gamma}) =
(p \circ a \circ \adj)(f_{\alpha\beta\gamma}) = (\underline{\la} \al s^* \be t \ga \overline{\mu}) + S = 0
$$
by the definition of $S$.
Finally for (3), the image of the vectors (\ref{vec1}) gives a basis for
$(R+S) / S$. Now apply $(c \circ \adj)^{-1}$ and use the commutativity of the diagram once again.
\end{proof}

\begin{Remark}\label{wee}\rm
Suppose we are 
given $\alpha \stackrel{s^*}{\longline} \beta \stackrel{t}{\longline}
\gamma$ such that $\la \subset \alpha$ and $\gamma \supset \mu$.
Then $f_{\alpha\beta\gamma}$ is a good homomorphism inducing a
non-zero $\bar f_{\alpha\beta\gamma}$ as in
Theorem~\ref{goodhoms}(3)
if and only if
$\underline{\la}\alpha s^* \beta t \gamma \overline{\mu}$
satisfies the following properties.
\begin{itemize}
\item[(1)] Each connected component in the diagram
crosses the top number line at most twice.
If it crosses the top number line exactly twice
  then it is an anti-clockwise circle.
\item[(2)] Each connected component crosses the bottom number line at most
  twice. If it crosses the bottom number line exactly twice and 
does not cross the top number line then it is a clockwise circle.
\end{itemize}
This follows because the properties (1) and
(2)
are equivalent to the assertions $\lambda \stackrel{u}{\longline} \mu$,
$\deg(\lambda u \mu) = 0$,
$\gamma = \mu$ and $\alpha = \lambda'$,
where $u$ and $\lambda'$ are as in Theorem~\ref{goodhoms}.
\end{Remark}

Assume finally that we are given another block $\Omega$, 
a $\De\Omega$-matching $u$, and $\nu \in \Omega$.
We want to explain an algorithm to compose 
the homomorphisms
\begin{equation}\label{gd}
f_{\alpha\beta\gamma}:
G^s_{\De\Pi} P(\la) \rightarrow G^t_{\De\Ga} P(\mu)\qquad
f_{\delta\epsilon\phi} :G^t_{\De\Ga} P(\mu) \rightarrow G^u_{\De\Omega} P(\nu)
\end{equation}
for 
$\la \subset \al\stackrel{s^*}{\longline}
\be\stackrel{t}{\longline} \ga\supset \mu \subset \delta\stackrel{t^*}{\longline}
\epsilon\stackrel{u}{\longline} \phi\supset \nu$.
To do this, draw the diagram
`$\underline{\la} \al s^* \be t \ga \overline{\mu}$
underneath the diagram $\underline{\mu} \delta t^* \epsilon u \phi \overline{\nu}$,
then apply the extended surgery procedure as above
to contract the $\be t \ga \overline{\mu}| \underline{\mu} \delta t^*\epsilon$-part of the diagram. This produces a (possibly zero) sum of 
diagrams of the form $\underline{\la} \rho s^* \sigma u \tau \overline{\nu}$.
The composition
$f_{\delta\epsilon\phi} \circ f_{\alpha\beta\gamma}$ is the
sum of the corresponding $f_{\rho \sigma \tau}
\in \hom_{K}(G^s_{\De\Pi} P(\la), G^u_{\De\Omega} P(\nu))$.

If $f_{\alpha\beta\gamma}$ and
$f_{\delta\epsilon\phi}$ are 
both good homomorphisms
with $\bar f_{\al\be\ga} \neq 0 \neq \bar f_{\delta\epsilon\phi}$,
so that the induced maps $\bar f_{\al\be\ga}$ and $\bar 
f_{\delta\epsilon\phi}$ are basis vectors 
like in Theorem~\ref{goodhoms}(3), 
then we can modify the algorithm in the previous paragraph to compute 
$\bar f_{\delta\epsilon\phi} \circ \bar f_{\alpha\beta\gamma}$ too.
We may as well assume for this that
$\hom_{K}(G^s_{\De\Pi} L(\la), G^u_{\De\Omega} L(\nu)) \neq \{0\}$. 
First 
write
$f_{\delta\epsilon\phi}  \circ f_{\al\be\ga}$
as a 
sum of maps
$f_{\rho\sigma\tau} \in 
\hom_{K}(G^s_{\De\Pi} P(\la), G^u_{\De\Omega} P(\nu))$.
Since the composition of two good homomorphisms is automatically good,
Theorem~\ref{goodhoms}(1) implies that 
the resulting sum only involves
$f_{\rho\sigma\tau}$'s that are themselves good homomorphisms.
Then descend to
$\hom_{K}(G^s_{\De\Pi} L(\la), G^u_{\De\Omega} L(\nu))$
by replacing each
$f_{\rho\sigma\tau}$ in the sum 
in which $\tau = \nu$ by 
$\bar f_{\rho\sigma\tau}$, and discarding all other $f_{\rho\sigma\tau}$
since these induce zero by Theorem~\ref{goodhoms}(2). This produces the required linear combination of basis vectors.

\phantomsubsection{Special endomorphism algebras}
Recall from the introduction that $\Seq_{r,s}$ denotes the set of
all sequences
consisting of $r$ $E$'s and $s$ $F$'s
shuffled together in some way.
For $R = R^{(1)} R^{(2)} \cdots R^{(r+s)}  \in \Seq_{r,s}$,
let
$\lonestar R := \lonestar R^{(r+s)} \cdots \lonestar R^{(2)} \lonestar
R^{(1)},$
which is a sequence 
consisting of $r$ $\lonestar E$'s and
$s$ $\lonestar F$'s.
We interpret
$R$ and $\lonestar R$ as
compositions of the functors 
\begin{equation}\label{head}
E := \bigoplus_{i \in \Z} E_i,
\qquad
F := \bigoplus_{i \in \Z} F_i,
\qquad
\lonestar E := \bigoplus_{i \in \Z} \lonestar E_i,
\qquad
\lonestar F := \bigoplus_{i \in \Z} \lonestar F_i.
\end{equation}
So $R$ and $\lonestar R$ are both
endofunctors
of $\Mod{K}$,
and there is a  canonical adjunction making
$(\lonestar R, R)$ into an adjoint pair. 
Moreover we have the natural decompositions
\begin{equation}\label{sdec}
R = \bigoplus_{\bi \in \Z^{r+s}} R_\bi,
\qquad
\lonestar R = \bigoplus_{\bi \in \Z^{r+s}} \lonestar R_\bi
\end{equation}
where $R_\bi := R^{(1)}_{i_1} R^{(2)}_{i_2} \cdots 
R^{(r+s)}_{i_{r+s}}$
and
$\lonestar R_\bi := \lonestar R^{(r+s)}_{i_{r+s}} \cdots \lonestar
R^{(2)}_{i_2} \lonestar R^{(1)}_{i_1}$.

Given $M \in \Mod{K}$ and $R \in \Seq_{r,s}$, we can consider the
endomorphism algebra $\End_{K}(\lonestar{R} \,M)^{\op}$.
However this might not have finite (or even countable) 
dimension in general, 
so instead we work with the
subalgebra
\begin{equation}\label{eb}
\End_K^{\fin}(\lonestar R\, M)^{\op} := 
\bigoplus_{\bi,\bj \in \Z^{r+s}} \hom_{K}(\lonestar
R_{\bi} M, \lonestar R_{\bj} M)
\end{equation}
of $\End_{K}(\lonestar{R} \,M)^{\op}$ consisting of {\em locally
 finite endomorphisms} of $\lonestar{R}\, M$.
In this subalgebra, we have an obvious system
$\{e(\bi)\:|\:\bi \in \Z^{r+s}\}$
of mutually orthogonal idempotents such that
\begin{equation}\label{idemps}
e(\bi)
\End_K^{\fin}(\lonestar R\,M)^{\op} 
e(\bj) = \hom_{K}(\lonestar R_{\bi} M, \lonestar R_{\bj} M).
\end{equation}
In all the applications we will meet later in the article, 
$\End_K^{\fin}(\lonestar R\,M)^{\op}$
will actually be finite dimensional. In that case, we have simply that
$\End_K^{\fin}(\lonestar R\,M)^{\op} = \End_{K}(\lonestar R \, M)^{\op}$, 
all but finitely many of the
$e(\bi)$'s are zero, and they sum to $1$.

We have that
$\lonestar(RE) = \lonestar E \lonestar R$ and
$\lonestar(RF) = \lonestar F \lonestar R$.
The functors $\lonestar E_i$ and $\lonestar F_i$ define
algebra homomorphisms
\begin{align}
\iota_{R;i}^{RE}:&\End_K^{\fin}(\lonestar R\,M)^{\op}
\rightarrow 
\End_K^{\fin}(\lonestar (RE)M)^{\op},\qquad
f \mapsto \lonestar E_i(f),\label{EP}\\
\iota_{R;i}^{RF}:&\End_K^{\fin}(\lonestar R\,M)^{\op}
\rightarrow \End_K^{\fin}(\lonestar (RF)M)^{\op}, \qquad
f \mapsto \lonestar F_i(f).\label{FP}
\end{align}
These maps send the idempotent $e(\bi)$ to
$e(\bi i)$ where $\bi i$ denotes 
$(i_1,\dots,i_{r+s},i)$.
In particular if $\End_K^{\fin}(\lonestar R\,M)$ is finite dimensional, then
$\iota_{R;i}^{RE}$
and $\iota_{R;i}^{RF}$ 
map
$1$ to the idempotents
\begin{align}\label{onese}
1^{RE}_{R;i} &:= 
\sum_{\substack{\bi \in \Z^{r+s+1}\\ i_{r+s+1}=i}} e(\bi)
\in \End_K^{\fin}(\lonestar (RE)M)^{\op},\\
1^{RF}_{R;i} &:= 
\sum_{\substack{\bi \in \Z^{r+s+1}\\ i_{r+s+1}=i}} e(\bi)
\in \End_K^{\fin}(\lonestar (RF)M)^{\op},\label{onesf}
\end{align}
respectively.

\phantomsubsection{Generalised tableaux}
Fix $R = R^{(1)} R^{(2)} \cdots R^{(r+s)} \in \Seq_{r,s}$ 
and a weight diagram $\eta$.
In the rest of the section, we are going to construct bases for the
algebras 
$\End_K^{\fin}(\lonestar R\,P(\eta))^{\op}$ 
and
$\End_K^{\fin}(\lonestar R\,L(\eta))^{\op}$.
We must first develop some combinatorics 
of paths in the labelled directed graph from (\ref{break1})--(\ref{break2}).

An {\em $R$-tableau} of {\em type} $\eta$ 
is a chain
$\tT = (\lambda^{(0)}, \lambda^{(1)}, \dots, \lambda^{(r+s)})$
of weight diagrams such that 
$\eta \subset \lambda^{(0)}$ and
the following hold for each $a=1,\dots,r+s$.
\begin{itemize}
\item[(1)]
If $R^{(a)} = E$ then
$\lambda^{(a-1)} \stackrel{i_a}{\rightarrow} \lambda^{(a)}$
for some $i_a \in \Z$.
\item[(2)]
If $R^{(a)} = F$ then
$\lambda^{(a-1)} \stackrel{i_a}{\leftarrow} \lambda^{(a)}$
for some $i_a \in \Z$.
\end{itemize}
Let $\mathscr T_R(\eta)$ denote the set of all $R$-tableaux
of type $\eta$.
The 
{\em shape}
$\sh(\tT)$ of $\tT$
is the final weight diagram $\la^{(r+s)}$ in the chain,
and the
{\em content} is the tuple
$\bi^\tT = (i_1,\dots,i_{r+s}) \in \Z^{r+s}$
arising from the labels on the edges in (1)--(2).
For example if
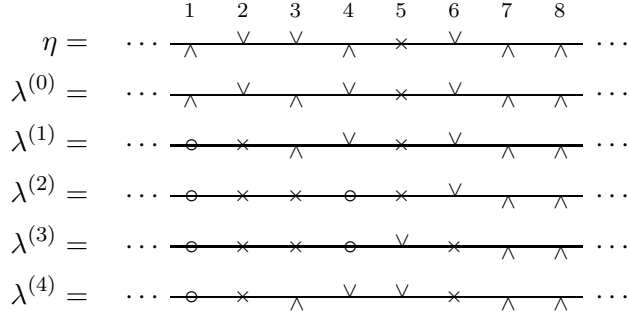
\begin{align*}
 \eta
&=
\hspace{-6mm}\begin{picture}(240,20)
\put(50.3,14){$_1$}
\put(70.3,14){$_2$}
\put(90.3,14){$_3$}
\put(110.3,14){$_4$}
\put(130.3,14){$_5$}
\put(150.3,14){$_6$}
\put(170.3,14){$_7$}
\put(190.3,14){$_8$}
\put(28,-.3){$\cdots$}
\put(206,-.3){$\cdots$}
\put(45,2.3){\line(1,0){156}}
\put(50,-2.4){$\up$}
\put(70,2.4){$\down$}
\put(90,2.4){$\down$}
\put(110,-2.4){$\up$}
\put(129.2,.5){$\cross$}
\put(150,2.4){$\down$}
\put(170,-2.4){$\up$}
\put(190,-2.4){$\up$}
\end{picture}\,
\\ 
\la^{(0)}&=
\hspace{-6mm}\begin{picture}(240,0)
\put(28,-.3){$\cdots$}
\put(206,-.3){$\cdots$}
\put(45,2.3){\line(1,0){156}}
\put(50,-2.4){$\up$}
\put(70,2.4){$\down$}
\put(90,-2.4){$\up$}
\put(110,2.4){$\down$}
\put(129.2,.5){$\cross$}
\put(150,2.4){$\down$}
\put(170,-2.4){$\up$}
\put(190,-2.4){$\up$}
\end{picture}\,
\\ 
\la^{(1)}&=
\hspace{-6mm}\begin{picture}(240,0)
\put(28,-.3){$\cdots$}
\put(206,-.3){$\cdots$}
\put(45,2.3){\line(1,0){156}}
\put(50.4,-0.4){$\circ$}
\put(69.2,.5){$\cross$}
\put(90,-2.4){$\up$}
\put(110,2.4){$\down$}
\put(129.2,.5){$\cross$}
\put(150,2.4){$\down$}
\put(170,-2.4){$\up$}
\put(190,-2.4){$\up$}
\end{picture}\,\\
\la^{(2)}&=
\hspace{-6mm}\begin{picture}(240,0)
\put(28,-.3){$\cdots$}
\put(206,-.3){$\cdots$}
\put(45,2.3){\line(1,0){156}}
\put(50.4,-0.4){$\circ$}
\put(69.2,.5){$\cross$}
\put(89.2,.5){$\cross$}
\put(110.4,-0.4){$\circ$}
\put(129.2,.5){$\cross$}
\put(150,2.4){$\down$}
\put(170,-2.4){$\up$}
\put(190,-2.4){$\up$}
\end{picture}\,\\
\la^{(3)}&=
\hspace{-6mm}\begin{picture}(240,0)
\put(28,-.3){$\cdots$}
\put(206,-.3){$\cdots$}
\put(45,2.3){\line(1,0){156}}
\put(50.4,-0.4){$\circ$}
\put(69.2,.5){$\cross$}
\put(89.2,.5){$\cross$}
\put(110.4,-0.4){$\circ$}
\put(130,2.4){$\down$}
\put(149.2,.5){$\cross$}
\put(170,-2.4){$\up$}
\put(190,-2.4){$\up$}
\end{picture}\,\\
\la^{(4)}&=
\hspace{-6mm}\begin{picture}(240,0)
\put(28,-.3){$\cdots$}
\put(206,-.3){$\cdots$}
\put(45,2.3){\line(1,0){156}}
\put(50.4,-0.4){$\circ$}
\put(69.2,.5){$\cross$}
\put(90.2,-2.4){$\up$}
\put(110.2,2.4){$\down$}
\put(130,2.4){$\down$}
\put(149.2,.5){$\cross$}
\put(170,-2.4){$\up$}
\put(190,-2.4){$\up$}
\end{picture}\,
\end{align*}
(and all other vertices are labelled $\circ$)
then $\tT = (\lambda^{(0)}, 
\lambda^{(1)},
\lambda^{(2)},
\lambda^{(3)},
\lambda^{(4)})$
is an $EFE^2$-tableau of type $\eta$ and
content $(1,3,5,3)$.

We need several other combinatorial notions
related to $R$-tableaux, all of which depend implicitly
on $\eta$.
Suppose $\tT = (\la^{(0)},\dots,\la^{(r+s)}) \in \mathscr{T}_R(\eta)$.
First we associate two
more diagrams denoted $\underline{\tT}$ and $\overline{\tT}$.
The former is defined as follows.
\begin{itemize}
\item[(1)]
Draw the weight diagrams of the bipartitions $\la^{(0)},\dots,\la^{(r+s)}$
in order {\em from bottom to top}, leaving some vertical space between them.
\item[(2)] For each $a=1,\dots,r+s$, 
connect the weight diagrams $\la^{(a-1)}$
and $\la^{(a)}$ together by
inserting 
into the space between them 
the unique crossingless matching
from (\ref{here}) that is consistent with their labels.
\item[(3)]
Finally glue $\underline{\eta}$ onto the bottom of the diagram.
\end{itemize}
The definition of $\overline{\tT}$ is similar: draw the bipartitions
$\la^{(0)},\dots,\la^{(r+s)}$ in order {\em from top to bottom}, connect pairs of
weight diagrams using the unique
crossingless matching from (\ref{here}) that fits,
then glue $\overline{\eta}$ onto the top of the diagram.
For the above example, these diagrams are as follows:
$$
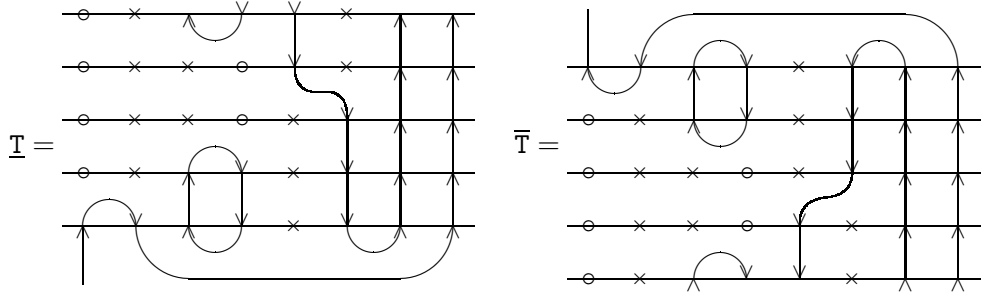
\hspace{-44mm}
\begin{picture}(100,110)
\put(25,50){$\underline{\tT} = $}

\put(45,22.3){\line(1,0){156}}
\put(50,17.6){$\up$}
\put(70,22.4){$\down$}
\put(90,17.6){$\up$}
\put(110,22.4){$\down$}
\put(129.2,20.5){$\cross$}
\put(150,22.4){$\down$}
\put(170,17.6){$\up$}
\put(190,17.6){$\up$}

\put(45,42.3){\line(1,0){156}}
\put(50.4,39.6){$\circ$}
\put(69.2,40.5){$\cross$}
\put(90,37.6){$\up$}
\put(110,42.4){$\down$}
\put(129.2,40.5){$\cross$}
\put(150,42.4){$\down$}
\put(170,37.6){$\up$}
\put(190,37.6){$\up$}

\put(45,62.3){\line(1,0){156}}
\put(50.4,59.6){$\circ$}
\put(69.2,60.5){$\cross$}
\put(89.2,60.5){$\cross$}
\put(110.4,59.6){$\circ$}
\put(129.2,60.5){$\cross$}
\put(150,62.4){$\down$}
\put(170,57.6){$\up$}
\put(190,57.6){$\up$}

\put(45,82.3){\line(1,0){156}}
\put(50.4,79.6){$\circ$}
\put(69.2,80.5){$\cross$}
\put(89.2,80.5){$\cross$}
\put(110.4,79.6){$\circ$}
\put(130,82.4){$\down$}
\put(149.2,80.5){$\cross$}
\put(170,77.6){$\up$}
\put(190,77.6){$\up$}

\put(45,102.3){\line(1,0){156}}
\put(50.4,99.6){$\circ$}
\put(69.2,100.5){$\cross$}
\put(90.2,97.6){$\up$}
\put(110.2,102.4){$\down$}
\put(130,102.4){$\down$}
\put(149.2,100.5){$\cross$}
\put(170,97.6){$\up$}
\put(190,97.6){$\up$}

\put(52.8,22.3){\line(0,-1){22}}
\put(102.8,22.3){\oval(20,20)[b]}
\put(162.8,22.3){\oval(20,20)[b]}
\put(132.8,22.3){\oval(120,40)[b]}

\put(62.8,22.3){\oval(20,20)[t]}
\put(102.8,42.3){\oval(20,20)[t]}
\put(102.8,102.3){\oval(20,20)[b]}
\put(192.8,22.3){\line(0,1){80}}
\put(172.8,22.3){\line(0,1){80}}
\put(152.8,22.3){\line(0,1){40}}
\put(132.8,82.3){\line(0,1){20}}
\put(112.8,22.3){\line(0,1){20}}
\put(92.8,22.3){\line(0,1){20}}

\qbezier(132.8,83)(132.8,72)(142.8,73)
\qbezier(152.8,63)(152.8,74)(142.8,73)
\end{picture}
\hspace{32mm}
\begin{picture}(100,110)

\qbezier(132.8,23)(132.8,32)(142.8,33)
\qbezier(152.8,43)(152.8,34)(142.8,33)

\put(192.8,2.3){\line(0,1){80}}
\put(172.8,2.3){\line(0,1){80}}
\put(152.8,42.3){\line(0,1){40}}
\put(132.8,2.3){\line(0,1){20}}
\put(112.8,62.3){\line(0,1){20}}
\put(92.8,62.3){\line(0,1){20}}

\put(62.8,82.3){\oval(20,20)[b]}
\put(102.8,62.3){\oval(20,20)[b]}
\put(102.8,2.3){\oval(20,20)[t]}

\put(52.8,82.3){\line(0,1){22}}
\put(102.8,82.3){\oval(20,20)[t]}
\put(162.8,82.3){\oval(20,20)[t]}
\put(132.8,82.3){\oval(120,40)[t]}

\put(25,50){$\overline{\tT} = $}

\put(45,82.3){\line(1,0){156}}
\put(50,77.6){$\up$}
\put(70,82.4){$\down$}
\put(90,77.6){$\up$}
\put(110,82.4){$\down$}
\put(129.2,80.5){$\cross$}
\put(150,82.4){$\down$}
\put(170,77.6){$\up$}
\put(190,77.6){$\up$}

\put(45,62.3){\line(1,0){156}}
\put(50.4,59.6){$\circ$}
\put(69.2,60.5){$\cross$}
\put(90,57.6){$\up$}
\put(110,62.4){$\down$}
\put(129.2,60.5){$\cross$}
\put(150,62.4){$\down$}
\put(170,57.6){$\up$}
\put(190,57.6){$\up$}

\put(45,42.3){\line(1,0){156}}
\put(50.4,39.6){$\circ$}
\put(69.2,40.5){$\cross$}
\put(89.2,40.5){$\cross$}
\put(110.4,39.6){$\circ$}
\put(129.2,40.5){$\cross$}
\put(150,42.4){$\down$}
\put(170,37.6){$\up$}
\put(190,37.6){$\up$}

\put(45,22.3){\line(1,0){156}}
\put(50.4,19.6){$\circ$}
\put(69.2,20.5){$\cross$}
\put(89.2,20.5){$\cross$}
\put(110.4,19.6){$\circ$}
\put(130,22.4){$\down$}
\put(149.2,20.5){$\cross$}
\put(170,17.6){$\up$}
\put(190,17.6){$\up$}

\put(45,2.3){\line(1,0){156}}
\put(50.4,-.4){$\circ$}
\put(69.2,0.5){$\cross$}
\put(89.9,-2.4){$\up$}
\put(109.9,2.4){$\down$}
\put(130,2.4){$\down$}
\put(149.2,.5){$\cross$}
\put(170,-2.4){$\up$}
\put(190,-2.4){$\up$}
\end{picture}
$$
We
let $\red(\underline{\tT})$ be the labelled cup diagram
obtained from $\underline{\tT}$ by removing all of the number lines
except for the top one, together with all connected components that do
not cross the top number line.
In the running example, we have that
$$
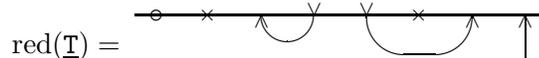
\red(\underline{\tT}) = 
\hspace{-15mm}\begin{picture}(180,20)
\put(45,14.3){\line(1,0){156}}
\put(50.4,11.6){$\circ$}
\put(69.2,12.5){$\cross$}
\put(90.2,9.6){$\up$}
\put(110.2,14.4){$\down$}
\put(130,14.4){$\down$}
\put(149.2,12.5){$\cross$}
\put(170,9.6){$\up$}
\put(190,9.6){$\up$}
\put(102.8,14.3){\oval(20,20)[b]}
\put(152.8,14.3){\oval(40,30)[b]}
\put(192.8,14.3){\line(0,-1){17}}
\end{picture}
$$
Similarly we define $\red(\overline{\tT})$, which is a labelled cap diagram.
Finally the {\em degree}
of $\tT$ is 
\begin{equation}\label{degrees}
\deg(\tT) := 
\deg(\underline{\tT}) - \caps(\underline{\tT})
=
\deg(\overline{\tT}) - \cups(\overline{\tT}),
\end{equation}
where $\caps(\underline{\tT})$ 
and $\cups(\overline{\tT})$ 
are the total numbers of caps and cups 
in the diagrams, respectively.
For the running example,
$\deg(\tT) = 2$.

\begin{Remark}\label{altdeg}\rm
Using (\ref{break1})--(\ref{break2}), one can check
that
\begin{equation*}
\deg(\tT) = 
\deg(\underline{\eta} \lambda^{(0)}) + 
\!\!\sum_{\substack{1 \leq a \leq r+s \\ R^{(a)} = E}} \deg(\la^{(a-1)} \stackrel{i_a}{\rightarrow} \la^{(a)})+
\!\!\sum_{\substack{1 \leq a \leq r+s \\ R^{(a)} = F}} \deg(\la^{(a-1)} \stackrel{i_a}{\leftarrow} \la^{(a)})
\end{equation*}
for $\tT \in \mathscr T_R(\eta)$ with $\bi^\tT = (i_1,\dots,i_{r+s})$.
\end{Remark}

\phantomsubsection{\boldmath Diagram bases for $\End^{\fin}_K(\lonestar R\,
  P(\eta))^{\op}$
and $\End^{\fin}_K(\lonestar R\,
  L(\eta))^{\op}$}
Continue with $R$ and $\eta$ fixed as in the previous subsection.

\begin{Lemma}\label{spbas}
The left $K$-module $\lonestar R\,P(\eta)$ has a distinguished basis
\begin{equation}\label{spb}
\{(\underline{\nu}\overline{\tU})\:|\:\text{for all 
$\tU \in \mathscr T_R(\eta)$ and $\nu \subset \sh(\tU)$}\}.
\end{equation}
The action of
$(\underline{\la} \alpha \overline{\mu}) \in K$
on $(\underline{\nu}\overline{\tU}) \in \lonestar R\,P(\eta)$ is as follows.
It acts as zero unless $\mu = \nu$.
Assuming $\mu = \nu$, draw $\underline{\la} \alpha \overline{\mu}$
underneath the diagram $\underline{\nu} 
\overline{\tU}$, connect corresponding pairs of
rays, then apply the surgery procedure to contract the
symmetric $\overline{\mu}\underline{\nu}$-section of the diagram.
\end{Lemma}

\begin{proof}
This is constructed in the same way as the basis (\ref{pfp}),
using
the iterated version of the first isomorphism in Lemma~\ref{composition} 
(see \cite[Theorem 3.5(iii)]{BS2}).
\end{proof}

Given $\tS,\tT \in \mathscr T_R(\eta)$ with 
$\sh(\tS)= \sh(\tT)$,
we can concatenate 
to obtain the composite diagram $\underline{\tS} \overline{\tT}$.
We associate the homomorphism
\begin{equation}\label{howtodoit}
f_{\tS\tT}:\lonestar R\,P(\eta) \rightarrow \lonestar R\,P(\eta)
\end{equation}
defined
on a basis vector 
$(\underline{\nu} \overline{\tU})$ from (\ref{spb}) as follows.
If $\bi^\tU \neq \bi^\tS$ then we set $f_{\tS \tT}(\underline{\nu}\overline{\tU}) := 0$.
If $\bi^\tU = \bi^\tS$,
the diagrams $\overline{\tU}$ and $\underline{\tS}$ are 
mirror images of each other 
(but the labels on their number lines can be different).
We compute
$f_{\tS\tT}(\underline{\nu} \overline{\tU})$ by
drawing $\underline{\tS} \overline{\tT}$ on top of
$\underline{\nu} \overline{\tU}$,
then applying the following {\em extended surgery procedure} 
to contract the $\overline{\tU}|\underline{\tS}$-part of the diagram.
\begin{itemize}
\item[(1)] First consider all the mirror image pairs of closed circles
from the diagrams $\overline{\tU}$ and $\underline{\tS}$. If the circles from at least one of these pairs are oriented in the same way 
(both anti-clockwise or both clockwise),
the extended surgery procedure gives $0$.
\item[(2)] Assuming all the mirror image pairs of circles are oppositely oriented,
we replace $\overline{\tU}$
by $\upperred(\overline{\tU})$
and $\underline{\tS}$ by 
$\lowerred(\underline{\tS})$,
join corresponding pairs of rays, then apply the usual surgery procedure
to contract the symmetric 
section of the resulting diagram.
\end{itemize}
This should be compared with the definition of $f_{\al\be\ga}$ from (\ref{allhom}).

Note that the vector $(\underline{\nu} \overline{\tU})$ from (\ref{spb}) lies
in the summand $\lonestar R_{\bi^\tU}\, P(\eta)$.
So $f_{\tS\tT}$ belongs to
$\hom_{K}(\lonestar R_{\bi^\tS}\, P(\eta),
\lonestar R_{\bi^\tT}\, P(\eta))$.
We deduce that
\begin{align}\label{ids}
e(\bi) f_{\tS\tT}&=
\begin{cases}
f_{\tS\tT}&\text{if $\bi=\bi^\tS$,}\\
0&\text{otherwise,}
\end{cases}
&f_{\tS\tT} e(\bi)
&=
\begin{cases}
f_{\tS\tT}&\text{if $\bi = \bi^\tT$,}\\
0&\text{otherwise.}
\end{cases}
\end{align}

\begin{Theorem}\label{myas}
The homomorphisms 
\begin{equation}\label{newallhom}
\{f_{\tS\tT}\:|\:\text{for all 
$\tS, \tT \in \mathscr{T}_R(\eta)$ with $\sh(\tS) = \sh(\tT)$}\}
\end{equation}
give a basis for $\End_K^{\fin}(\lonestar R\,P(\eta))^{\op}$.
Moreover $f_{\tS\tT}$ is of degree $\deg(\tS)+\deg(\tT)$.
\end{Theorem}

\begin{proof}
Using the iterated version of 
Lemma~\ref{composition} (see \cite[Theorem 3.6]{BS2} for the precise formulation),
we decompose the functor $\lonestar R$ into a direct sum of indecomposable
projective functors, then apply
Theorem~\ref{phom}. Actually Remark~\ref{otherdegs} 
is slightly more convenient since it gives directly that
$f_{\tS\tT}$ is of degree $\deg(\underline{\tS} \overline{\tT}) - 
\caps(\underline{\tS}) - \cups(\overline{\tT}) =
\deg(\tS) + \deg(\tT)$.
\end{proof}

Now we pass from 
$\End_K^{\fin}(\lonestar R\,P(\eta))^{\op}$
to $\End_K^{\fin}(\lonestar R\,L(\eta))^{\op}$.
A {\em good pair} $(\tS,\tT)$ of $R$-tableaux of type $\eta$
means a pair of tableaux
$\tS, \tT \in \mathscr T_R(\eta)$ of the same shape
such that the following hold.
\begin{itemize}
\item[(1)] Each connected component in the diagram $\underline{\tS}
  \overline{\tT}$ crosses the top number line at most twice. If it
  crosses the top number line exactly twice then it is an
  anti-clockwise circle.
\item[(2)] Each connected component in $\underline{\tS}
  \overline{\tT}$
crosses the bottom number line at most twice. If it crosses exactly
twice and it does not cross the top number line then it is a clockwise circle.
\end{itemize}
These are the same conditions as in Remark~\ref{wee}.

\begin{Theorem}\label{byas}
If $(\tS,\tT)$ is a good pair of $R$-tableaux of type $\eta$, then
the homomorphism $f_{\tS\tT}$ from (\ref{howtodoit})
factors through the quotients to induce a
non-zero
map $\bar f_{\tS\tT}:\lonestar R\, L(\eta) \rightarrow \lonestar R
\,L(\eta)$.
The homomorphisms
\begin{equation}\label{newirrhom}
\{\bar f_{\tS\tT}\:|\:\text{for all 
good pairs $(\tS,\tT)$ of $R$-tableaux of type $\eta$}\}
\end{equation}
give a basis for $\End_K^{\fin}(\lonestar R\,L(\eta))^{\op}$.
Moreover $\bar f_{\tS\tT}$ is of degree $\deg(\tS)+\deg(\tT)$.
\end{Theorem}

\begin{proof}
This is deduced from Theorem~\ref{goodhoms} (and Remark~\ref{wee}) 
in the same way that Theorem~\ref{myas} was deduced from Theorem~\ref{phom} above.
\end{proof}

\phantomsubsection{Algorithms to compute with the diagram bases}
We conclude the section by explaining various algorithms to compute
with the bases (\ref{newallhom}) and (\ref{newirrhom}).
We will give some concrete examples in a special case in the next section.

\vspace{2mm}

\noindent
{\em Products.}
We first explain how to compute products. For this, we adapt 
the algorithms for composing homomorphisms explained 
after Remark~\ref{wee}.
To compute products in 
$\End_K^{\fin}(\lonestar R\,P(\eta))^{\op}$, take two basis vectors 
$f_{\tU \tV}$ and $f_{\tS\tT}$ from (\ref{newallhom}).
We have that $f_{\tS\tT} f_{\tU\tV}= 0$
unless $\bi^\tT = \bi^\tU$. 
Assuming $\bi^\tT = \bi^\tU$,
we draw $\underline{\tS}\overline{\tT}$ underneath
$\underline{\tU}\overline{\tV}$, then apply the extended surgery procedure 
from (\ref{howtodoit}) to contract the $\overline{\tT}|\underline{\tU}$-part of the diagram.
This produces a (possibly zero) sum of diagrams of the form 
$(\underline{\tX} \overline{\tY})$
such that $\bi^\tX = \bi^\tS$ and $\bi^\tY = \bi^\tV$; the product
$f_{\tS\tT} f_{\tU\tV}$
is the corresponding sum of homomorphisms $f_{\tX\tY} \in  \End_K^{\fin}(\lonestar R\,P(\eta))^{\op}$.

Now suppose that $(\tU, \tV)$ and $(\tS, \tT)$ are both good pairs, so
that
$\bar f_{\tU\tV}$ and $\bar f_{\tS \tT}$ are basis vectors from
(\ref{newirrhom}).
To compute the product 
$\bar f_{\tU\tV} \bar f_{\tS\tT}$
in $\End_K^{\fin}(\lonestar R\,L(\eta))^{\op}$,
we expand $f_{\tU\tV} f_{\tS\tT}$ as a sum of $f_{\tX\tY}$ using
the algorithm in the previous paragraph.
Then we replace all $f_{\tX\tY}$ in which $(\tX,\tY)$ is not a good
pair by zero, and all remaining $f_{\tX\tY}$ by
$\bar f_{\tX\tY}$.
This gives
$\bar f_{\tU\tV} \bar f_{\tS\tT}$.

\vspace{2mm}

\noindent
{\em Homomorphisms.}
Next we explain how to compute the 
maps (\ref{EP})--(\ref{FP})
in terms of our diagram bases.
Take $\tS,\tT \in \mathscr T_R(\eta)$ of the same shape.
To compute $\iota_{R;i}^{RE}(f_{\tS\tT})$ 
(resp.\ $\iota_{R;i}^{RF}(f_{\tS\tT})$),
the rough idea is to 
insert two new levels at 
the middle number line of the diagram $\underline{\tS}
\overline{\tT}$
containing the unique composite matching
displayed in (\ref{CKLR2}) (resp.\ (\ref{CKLR3})) 
that is consistent with the labels on
the number line;
if none of the configurations match we set
$\iota_{R;i}^{RE} (f_{\tS\tT}) := 0$
(resp.\ $\iota_{R;i}^{RF} (f_{\tS\tT}) := 0$).
\begin{equation}
\begin{picture}(60,39)
\put(-130,3){$\iota_{R;i}^{RE}:$}
\put(-24,25){\line(1,0){33}}
\put(-24,5){\line(1,0){33}}
\put(-24,-15){\line(1,0){33}}
\put(-7.5,5){\oval(23,21)[b]}
\put(-7.5,5){\oval(23,21)[t]}
\put(-22.3,-16.9){$\cross$}
\put(-21.9,4.8){$\down$}
\put(1.2,.4){$\up$}
\put(1.2,-17.6){$\circ$}
\put(-22.3,23.1){$\cross$}
\put(1.2,22.4){$\circ$}

\put(-84,25){\line(1,0){33}}
\put(-67.5,25){\oval(23,23)[b]}
\put(-84,5){\line(1,0){33}}
\put(-84,-15){\line(1,0){33}}
\put(-67.5,-15){\oval(23,23)[t]}
\put(-81.8,2.4){$\circ$}
\put(-59.2,3.1){$\cross$}

\put(96,5){\line(1,0){33}}
\put(96,-15){\line(1,0){33}}
\put(120.7,3.1){$\cross$}
\put(97.7,-16.9){$\cross$}
\put(96,25){\line(1,0){33}}
\put(97.7,23.1){$\cross$}

\put(36,25){\line(1,0){33}}
\put(61.2,22.4){$\circ$}
\put(36,5){\line(1,0){33}}
\put(36,-15){\line(1,0){33}}
\put(61.2,-17.6){$\circ$}
\put(38.2,2.4){$\circ$}

\qbezier(41,-15)(41,-5)(52.5,-5)
\qbezier(64,5)(64,-5)(52.5,-5)

\qbezier(101,5)(101,15)(112.5,15)
\qbezier(124,25)(124,15)(112.5,15)

\qbezier(101,5)(101,-5)(112.5,-5)
\qbezier(124,-15)(124,-5)(112.5,-5)

\qbezier(41,25)(41,15)(52.5,15)
\qbezier(64,5)(64,15)(52.5,15)

\put(-81,-33){$\text{$_i$\quad\:\,$_{i+1}$}$}
\put(-21,-33){$\text{$_i$\quad\:\,$_{i+1}$}$}
\put(36,-33){$\text{$_i$\quad\:\,$_{i+1}$}$}
\put(99.5,-33){$\text{$_i$\quad\:\,$_{i+1}$}$}
\end{picture}
\vspace{8mm}\label{CKLR2}
\end{equation}
\begin{equation}
\begin{picture}(60,44)
\put(-130,3){$\iota_{R;i}^{RF}:$}
\put(-24,25){\line(1,0){33}}
\put(-24,5){\line(1,0){33}}
\put(-24,-15){\line(1,0){33}}
\put(-7.5,5){\oval(23,21)[b]}
\put(-7.5,5){\oval(23,21)[t]}
\put(.8,-16.9){$\cross$}
\put(-21.9,4.8){$\down$}
\put(1.2,.4){$\up$}
\put(-21.4,-17.6){$\circ$}
\put(.8,23.1){$\cross$}
\put(-21.4,22.4){$\circ$}

\put(-84,25){\line(1,0){33}}
\put(-67.5,25){\oval(23,23)[b]}
\put(-84,5){\line(1,0){33}}
\put(-84,-15){\line(1,0){33}}
\put(-67.5,-15){\oval(23,23)[t]}
\put(-58.8,2.4){$\circ$}
\put(-82.2,3.1){$\cross$}

\put(96,5){\line(1,0){33}}
\put(96,-15){\line(1,0){33}}
\put(121.2,2.4){$\circ$}
\put(98.2,-17.6){$\circ$}
\put(96,25){\line(1,0){33}}
\put(98.2,22.4){$\circ$}

\put(36,25){\line(1,0){33}}
\put(60.8,23){$\cross$}
\put(36,5){\line(1,0){33}}
\put(36,-15){\line(1,0){33}}
\put(60.8,-17){$\cross$}
\put(37.8,3){$\cross$}

\qbezier(41,-15)(41,-7)(52.5,-7)
\qbezier(64,5)(64,-7)(52.5,-7)

\qbezier(101,5)(101,17)(112.5,17)
\qbezier(124,25)(124,17)(112.5,17)

\qbezier(101,5)(101,-5)(112.5,-5)
\qbezier(124,-15)(124,-5)(112.5,-5)

\qbezier(41,25)(41,15)(52.5,15)
\qbezier(64,5)(64,15)(52.5,15)

\end{picture}
\vspace{8mm}\label{CKLR3}
\end{equation}
Here, we display only the strip between the $i$th
and $(i+1)$th vertices, and there are only vertical line segments joining the vertices
labelled $\up$ or $\down$ outside of this strip.
For the second, third and fourth of these configurations, it is quite
clear how to do this: insert the new levels into
$\underline{\tS}\overline{\tT}$, complete the labels on the new middle
number line in the only possible way to get a consistently oriented
diagram,
and this produces the desired basis vector 
in $\End_{K}^{\fin}(\lonestar (RE) P(\eta))^{\op}$
(resp.\ $\End_{K}^{\fin}(\lonestar (RF) P(\eta))^{\op}$).
The first configuration is a little more subtle and involves
applying one iteration of the surgery procedure. We refer to \cite[(6.34)]{BS3} where this is
explained in more detail in an analogous situation.
The proof that this is correct 
follows the same argument as 
\cite[Theorem 6.11]{BS3}.

Finally we assume in addition that $(\tS,\tT)$ is a good pair and modify
the algorithm in the previous paragraph to compute
$\iota_{R;i}^{RE}(\bar f_{\tS\tT})$ (resp.\ $\iota_{R;i}^{RF}(\bar
f_{\tS\tT})$).
First one computes $\iota_{R;i}^{RE}(f_{\tS\tT})$ (resp.\ 
$\iota_{R;i}^{RF}(f_{\tS\tT})$) exactly as in the previous paragraph to obtain a sum of basis
vectors
for the form $f_{\tX\tY}$. 
The result is automatically a good
homomorphism, so it remains to replace all $f_{\tX\tY}$ in which
$(\tX,\tY)$
is not a good pair by zero, and the remaining ones by $\bar
f_{\tX\tY}$.

\section{The graded walled Brauer algebra}\label{s5}

In this section $\delta \in \Z$ 
 and $R = R^{(1)} \cdots
R^{(r+s)} \in \Seq_{r,s}$ are fixed. 
We are ready to define the graded version $B_R(\delta)$ of the walled
Brauer algebra $B_{r,s}(\delta)$ and
to show that it is graded Morita
equivalent to the basic algebra $K_{r,s}(\delta)$ from $\S$\ref{s4b}.
The proof that $B_R(\delta)$ is isomorphic to $B_{r,s}(\delta)$ is 
deferred to $\S$\ref{smain}.

\phantomsubsection{\boldmath Definition of the algebra $B_R(\delta)$}
According to the weight dictionary (\ref{Iu})--(\ref{dict}),
the empty bipartition $(\varnothing,\varnothing)$ is identified with
the weight diagram
\begin{equation}\label{othergroundstate}
\eta := \Bigg\{
\begin{array}{ll}
\hspace{80mm}
&\text{if $\delta \geq 0$}\\\\
&\text{if $\delta \leq 0$}\\
\end{array}
\begin{picture}(0,40)
\put(-276,12.5){$\cdots$}
\put(-69,12.5){$\cdots$}
\put(-219,14.2){$\overbrace{\phantom{hellow orl}}^{\delta}$}
\put(-259,15.2){\line(1,0){187}}
\put(-258.7,10.6){$\scriptstyle\up$}
\put(-238.7,10.6){$\scriptstyle\up$}
\put(-219.2,13.3){$\scriptstyle\times$}
\put(-199.2,13.3){$\scriptstyle\times$}
\put(-179.2,13.3){$\scriptstyle\times$}
\put(-158.7,15.3){$\scriptstyle\down$}
\put(-138.7,15.3){$\scriptstyle\down$}
\put(-118.7,15.3){$\scriptstyle\down$}
\put(-98.7,15.3){$\scriptstyle\down$}
\put(-78.7,15.3){$\scriptstyle\down$}
\end{picture}
\begin{picture}(0,40)
\put(-276,-12.6){$\cdots$}
\put(-69,-12.6){$\cdots$}
\put(-177.5,2){$_0$}
\put(-159,-17){$\underbrace{\phantom{hellow orl}}_{-\delta}$}
\put(-259,-10){\line(1,0){187}}
\put(-258.7,-14.6){$\scriptstyle\up$}
\put(-238.7,-14.6){$\scriptstyle\up$}
\put(-218.7,-14.6){$\scriptstyle\up$}
\put(-198.7,-14.6){$\scriptstyle\up$}
\put(-178.7,-14.6){$\scriptstyle\up$}
\put(-158,-12.6){$\circ$}
\put(-138,-12.6){$\circ$}
\put(-118,-12.6){$\circ$}
\put(-98.7,-9.9){$\scriptstyle\down$}
\put(-78.7,-9.9){$\scriptstyle\down$}
\end{picture}
\end{equation}
\vspace{4mm}

\noindent
(where there are infinitely many vertices labelled $\up$ to the left
and $\down$ to the right).
In this section we refer to the generalised $R$-tableaux of type
$\eta$ 
as defined in the previous section
simply as {\em
  $R$-tableaux}, and 
denote the set of all such $R$-tableaux by $\mathscr T_R(\delta)$.
As $\eta$ is maximal in the Bruhat order,
any $\tT = (\lambda^{(0)},\dots,\lambda^{(r+s)}) \in \mathscr T_R(\delta)$
necessarily has $\lambda^{(0)} = \eta$.
For $\tS,\tT \in \mathscr T_R(\delta)$ with $\sh(\tS) = \sh(\tT)$,
we often talk about {\em boundary cups} in
the diagrams $\underline{\tS}$, $\overline{\tT}$ or
$\underline{\tS}\overline{\tT}$, 
meaning connected components
that intersect 
the top number line twice.
Similarly we define {\em boundary caps} to be components that
intersect the bottom number line twice.

It is often convenient to represent elements of
$\mathscr T_R(\delta)$ as chains of bipartitions in the graph 
defined just before Theorem~\ref{ibranch}
rather than as chains of
weight diagrams; this 
is justified by Remark~\ref{newremark}.
Let 
\begin{align*}
r_a &:= \#\{b=1,\dots,a\:|\:R^{(b)} = E\},
&s_a &:= \#\{b=1,\dots,a\:|\:R^{(b)} = F\}.
\end{align*}
Then elements of $\mathscr T_R(\delta)$
are paths $\tT = (\la^{(0)}, \la^{(1)},\dots,\la^{(r+s)})$
in the Bratelli diagram that describes the branching of
cell modules of $B_{r,s}(\delta)$
with respect to the chain of subalgebras 
$B_{0,0}(\delta)
< B_{r_1,s_1}(\delta) < B_{r_2,s_2}(\delta) < \cdots <
B_{r,s}(\delta)$;
in particular $\la^{(0)} = (\varnothing,\varnothing)$
and $\la^{(r+s)} \in \La_{r,s}$.
Note also for $\tT \in \mathscr T_R(\delta)$ 
that the alternative formula for $\deg(\tT)$ from Remark~\ref{altdeg}
can be written 
as
\begin{equation}
\deg(\tT) = 
\!\!\sum_{\substack{1 \leq a \leq r+s \\ R^{(a)} = E}} \deg(\la^{(a-1)} \stackrel{i_a}{\rightarrow} \la^{(a)})+
\!\!\sum_{\substack{1 \leq a \leq r+s \\ R^{(a)} = F}} \deg(\la^{(a-1)} \stackrel{i_a}{\leftarrow} \la^{(a)})
\end{equation}
for $\bi^\tT = (i_1,\dots,i_{r+s})$.
Combined with (\ref{edgedeg1})--(\ref{edgedeg2}), this formula allows
the degree of an $R$-tableau 
to be computed directly from the underlying chain of bipartitions (and
the value of $\delta$).
We will give some examples in the next subsection.

Since the set $\mathscr T_R(\delta)$ is finite for any $R \in \Seq_{r,s}$,
Theorem~\ref{myas} implies that the graded algebra
\begin{equation}
B_R(\delta) := \End^{\fin}_{K(\delta)}(\lonestar R\,P(\eta))^{\op}
\end{equation}
is finite dimensional.
This means we can drop the ``$\fin$'' in our notation here.
We call $B_R(\delta)$ the {\em graded walled Brauer algebra} of type
$R$. 
This name will be justified later in the article, when we show 
(on forgetting the grading) that $B_R(\delta)$ 
is isomorphic to the walled Brauer algebra $B_{r,s}(\delta)$.
By Theorem~\ref{myas}, $B_R(\delta)$ has a homogeneous basis
\begin{equation}\label{bbas}
\{f_{\tS\tT}\:|\:\text{for all }\tS,\tT \in \mathscr T_R(\delta)\text{
  with }
\sh(\tS) = \sh(\tT)\}
\end{equation}
such that $f_{\tS\tT}$ is of degree $\deg(\tS)+\deg(\tT)$.
Products can be computed in this basis using the algorithm
explained at the end of $\S$\ref{sd}.
Recall also that there is system $\{e(\bi)\:|\:\bi \in \Z^{r+s}\}$ of mutually
orthogonal idempotents in $B_R(\delta)$ characterised by
(\ref{ids}).
All but finitely many of the $e(\bi)$'s are zero and their sum is $1$.
In fact $e(\bi) \neq 0$ if and
only if there exists an $R$-tableau $\tT$ with $\bi^\tT =
\bi$.
Again we refer the reader to the next subsection for some examples.

The non-unital homomorphisms
\begin{equation}\label{newiota}
\iota_{R;i}^{RE}:B_R(\delta) \rightarrow B_{RE}(\delta),
\qquad
\iota_{R;i}^{RF}:B_R(\delta) \rightarrow B_{RF}(\delta)
\end{equation}
from (\ref{EP})--(\ref{FP}) map $1$ to $1_{R;i}^{RE} \in B_{RE}(\delta)$
and $1_{R;i}^{RF} \in B_{RF}(\delta)$, respectively.
The effect of these maps on the basis elements (\ref{bbas})
can be computed explicitly by the algorithm explained at the end of $\S$\ref{sd}.
Associated to these maps,
we have adjoint pairs
$(\iind_{R}^{RE}, \ires_{R}^{RE})$ and
$(\iind_R^{RF}, \ires_R^{RF})$ of induction and restriction functors
defined exactly like in (\ref{fred1})--(\ref{fred4}).
Thus
\begin{align}\label{bever1}
\ires^{RE}_{R} 
&:\Mod{B_{RE}(\delta)} \rightarrow \Mod{B_R(\delta)},\\
\ires^{RF}_{R} 
&:\Mod{B_{RF}(\delta)} \rightarrow \Mod{B_R(\delta)}\\\intertext{are the exact functors 
defined by multiplication by the idempotents $1^{RE}_{R;i}$
and $1^{RF}_{R;i}$, respectively,
and}
\iind^{RE}_{R} 
&:\Mod{B_R(\delta)} \rightarrow \Mod{B_{RE}(\delta)},\\
\iind^{RF}_{R} 
&:\Mod{B_R(\delta)} \rightarrow \Mod{B_{RF}(\delta)}\label{bever4}
\end{align}
are the right exact functors
$B_{RE}(\delta)1^{RE}_{R;i} \otimes_{B_R(\delta)} ?$,
and $B_{RF}(\delta)1^{RF}_{R;i} \otimes_{B_R(\delta)} ?$,
respectively.
Later in the article we 
will also relate these {\em graded $i$-induction} and {\em $i$-restriction
functors}
to the ungraded ones defined earlier.

\phantomsubsection{Examples}
In this subsection we give some examples to illustrate the 
definition of $B_R(\delta)$, taking $\delta := 0$ throughout.

We first describe $B_{E^2 F}(0)$.
Here are all the $E^2F$-tableaux:
\begin{align*}
\tS&=((\varnothing,\varnothing) \stackrel{0}{\rightarrow}
((1),\varnothing) \stackrel{-1}{\rightarrow}
((1^2),\varnothing) \stackrel{-1}{\leftarrow} ((1),\varnothing)),\\
\tT&=((\varnothing,\varnothing) \stackrel{0}{\rightarrow}
((1),\varnothing) \stackrel{1}{\rightarrow}
((2),\varnothing) \stackrel{1}{\leftarrow} ((1),\varnothing)),\\
\tU&=((\varnothing,\varnothing) \stackrel{0}{\rightarrow}
((1),\varnothing) \stackrel{-1}{\rightarrow}
((1^2),\varnothing) \stackrel{0}{\leftarrow} ((1^2),(1))),\\
\tV&=((\varnothing,\varnothing) \stackrel{0}{\rightarrow}
((1),\varnothing) \stackrel{1}{\rightarrow}
((2),\varnothing) \stackrel{0}{\leftarrow} ((2),(1))).
\end{align*}
The algebra $B_{E^2 F}(0)$ has basis
$\{f_{\tS\tS}, f_{\tT\tT}, f_{\tS\tT}, f_{\tT\tS}\} \cup
\{f_{\tU\tU}\} \cup \{f_{\tV\tV}\}$. These elements are represented by
the diagrams
$$
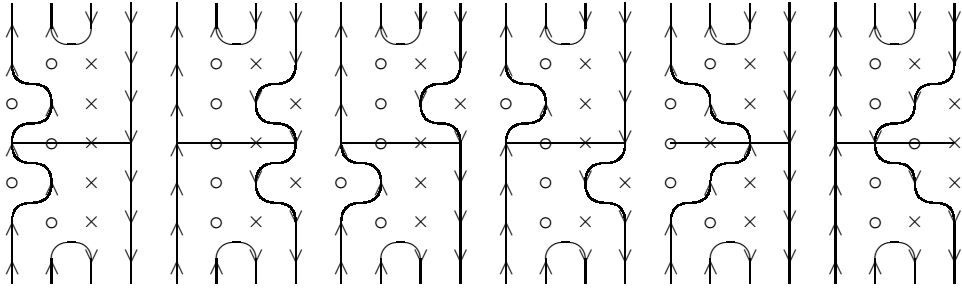
\begin{array}{llllll}
\begin{picture}(45,105)
\put(0,50){\line(1,0){45}}
\put(-2.5,90.4){$\up$}
\put(12.5,90.4){$\up$}
\put(27.5,95.1){$\down$}
\put(42.5,95.1){$\down$}
\put(-2.5,0.4){$\up$}
\put(12.5,0.4){$\up$}
\put(27.5,5.1){$\down$}
\put(42.5,5.1){$\down$}


\put(12.5,60.4){$\up$}
\put(42.5,65.1){$\down$}

\put(-2.5,75.4){$\up$}
\put(42.5,80.1){$\down$}

\put(-2.5,45.4){$\up$}
\put(12.5,47.2){$\circ$}
\put(27,48.2){$\cross$}
\put(42.5,50.1){$\down$}

\put(-2.5,62.2){$\circ$}
\put(27,63.2){$\cross$}
\put(12.5,77.2){$\circ$}
\put(27,78.2){$\cross$}
\put(45.2,50){\line(0,1){53}}
\put(15.2,95){\line(0,1){8}}
\put(30.2,95){\line(0,1){8}}
\put(.2,80){\line(0,1){23}}
\put(22.7,95){\oval(15,15)[b]}
\qbezier(.2,80)\qbezier(.2,72.5)\qbezier(7.7,72.5)
\qbezier(15.2,65)\qbezier(15.2,72.5)\qbezier(7.7,72.5)
\qbezier(.2,50)\qbezier(.2,57.5)\qbezier(7.7,57.5)
\qbezier(15.2,65)\qbezier(15.2,57.5)\qbezier(7.7,57.5)


\put(12.5,30.4){$\up$}
\put(42.5,35.1){$\down$}
\put(-2.5,32.2){$\circ$}
\put(27,33.2){$\cross$}

\put(-2.5,15.4){$\up$}
\put(42.5,20.1){$\down$}
\put(12.5,17.2){$\circ$}
\put(27,18.2){$\cross$}

\put(45.2,50){\line(0,-1){53}}
\put(15.2,5){\line(0,-1){8}}
\put(30.2,5){\line(0,-1){8}}
\put(.2,20){\line(0,-1){23}}
\put(22.7,5){\oval(15,15)[t]}
\qbezier(0.2,20)\qbezier(0.2,27.5)\qbezier(7.7,27.5)
\qbezier(15.2,35)\qbezier(15.2,27.5)\qbezier(7.7,27.5)
\qbezier(0.2,50)\qbezier(0.2,42.5)\qbezier(7.7,42.5)
\qbezier(15.2,35)\qbezier(15.2,42.5)\qbezier(7.7,42.5)
\end{picture}\:\:\:&
\begin{picture}(45,100)
\put(0,50){\line(1,0){45}}
\put(-2.5,90.4){$\up$}
\put(12.5,90.4){$\up$}
\put(27.5,95.1){$\down$}
\put(42.5,95.1){$\down$}
\put(-2.5,0.4){$\up$}
\put(12.5,0.4){$\up$}
\put(27.5,5.1){$\down$}
\put(42.5,5.1){$\down$}


\put(-2.5,45.4){$\up$}
\put(-2.5,60.4){$\up$}
\put(27.5,65.1){$\down$}
\put(-2.5,75.4){$\up$}
\put(42.5,80.1){$\down$}
\put(42.5,50.1){$\down$}

\put(12.5,47.2){$\circ$}
\put(27,48.2){$\cross$}
\put(12.5,62.2){$\circ$}
\put(42,63.2){$\cross$}
\put(12.5,77.2){$\circ$}
\put(27,78.2){$\cross$}
\put(.2,50){\line(0,1){53}}
\put(15.2,95){\line(0,1){8}}
\put(30.2,95){\line(0,1){8}}
\put(45.2,80){\line(0,1){23}}
\put(22.7,95){\oval(15,15)[b]}
\qbezier(45.2,80)\qbezier(45.2,72.5)\qbezier(37.7,72.5)
\qbezier(30.2,65)\qbezier(30.2,72.5)\qbezier(37.7,72.5)
\qbezier(45.2,50)\qbezier(45.2,57.5)\qbezier(37.7,57.5)
\qbezier(30.2,65)\qbezier(30.2,57.5)\qbezier(37.7,57.5)


\put(-2.5,30.4){$\up$}
\put(27.5,35.1){$\down$}
\put(-2.5,15.4){$\up$}
\put(42.5,20.1){$\down$}
\put(12.5,32.2){$\circ$}
\put(42,33.2){$\cross$}
\put(12.5,17.2){$\circ$}
\put(27,18.2){$\cross$}
\put(0.2,50){\line(0,-1){53}}
\put(15.2,5){\line(0,-1){8}}
\put(30.2,5){\line(0,-1){8}}
\put(45.2,20){\line(0,-1){23}}
\put(22.7,5){\oval(15,15)[t]}
\qbezier(45.2,20)\qbezier(45.2,27.5)\qbezier(37.7,27.5)
\qbezier(30.2,35)\qbezier(30.2,27.5)\qbezier(37.7,27.5)
\qbezier(45.2,50)\qbezier(45.2,42.5)\qbezier(37.7,42.5)
\qbezier(30.2,35)\qbezier(30.2,42.5)\qbezier(37.7,42.5)
\end{picture}\:\:\:&
\begin{picture}(45,100)
\put(0,50){\line(1,0){45}}
\put(-2.5,90.4){$\up$}
\put(12.5,90.4){$\up$}
\put(27.5,95.1){$\down$}
\put(42.5,95.1){$\down$}
\put(-2.5,0.4){$\up$}
\put(12.5,0.4){$\up$}
\put(27.5,5.1){$\down$}
\put(42.5,5.1){$\down$}


\put(-2.5,45.4){$\up$}
\put(-2.5,60.4){$\up$}
\put(27.5,65.1){$\down$}
\put(-2.5,75.4){$\up$}
\put(42.5,80.1){$\down$}
\put(42.5,50.1){$\down$}

\put(12.5,47.2){$\circ$}
\put(27,48.2){$\cross$}
\put(12.5,62.2){$\circ$}
\put(42,63.2){$\cross$}
\put(12.5,77.2){$\circ$}
\put(27,78.2){$\cross$}
\put(.2,50){\line(0,1){53}}
\put(15.2,95){\line(0,1){8}}
\put(30.2,95){\line(0,1){8}}
\put(45.2,80){\line(0,1){23}}
\put(22.7,95){\oval(15,15)[b]}
\qbezier(45.2,80)\qbezier(45.2,72.5)\qbezier(37.7,72.5)
\qbezier(30.2,65)\qbezier(30.2,72.5)\qbezier(37.7,72.5)
\qbezier(45.2,50)\qbezier(45.2,57.5)\qbezier(37.7,57.5)
\qbezier(30.2,65)\qbezier(30.2,57.5)\qbezier(37.7,57.5)

\put(12.5,30.4){$\up$}
\put(42.5,35.1){$\down$}
\put(-2.5,32.2){$\circ$}
\put(27,33.2){$\cross$}

\put(-2.5,15.4){$\up$}
\put(42.5,20.1){$\down$}
\put(12.5,17.2){$\circ$}
\put(27,18.2){$\cross$}

\put(45.2,50){\line(0,-1){53}}
\put(15.2,5){\line(0,-1){8}}
\put(30.2,5){\line(0,-1){8}}
\put(.2,20){\line(0,-1){23}}
\put(22.7,5){\oval(15,15)[t]}
\qbezier(0.2,20)\qbezier(0.2,27.5)\qbezier(7.7,27.5)
\qbezier(15.2,35)\qbezier(15.2,27.5)\qbezier(7.7,27.5)
\qbezier(0.2,50)\qbezier(0.2,42.5)\qbezier(7.7,42.5)
\qbezier(15.2,35)\qbezier(15.2,42.5)\qbezier(7.7,42.5)

\end{picture}\:\:\:&
\begin{picture}(45,100)
\put(0,50){\line(1,0){45}}
\put(-2.5,90.4){$\up$}
\put(12.5,90.4){$\up$}
\put(27.5,95.1){$\down$}
\put(42.5,95.1){$\down$}
\put(-2.5,0.4){$\up$}
\put(12.5,0.4){$\up$}
\put(27.5,5.1){$\down$}
\put(42.5,5.1){$\down$}


\put(-2.5,30.4){$\up$}
\put(27.5,35.1){$\down$}
\put(-2.5,15.4){$\up$}
\put(42.5,20.1){$\down$}
\put(12.5,32.2){$\circ$}
\put(42,33.2){$\cross$}
\put(12.5,17.2){$\circ$}
\put(27,18.2){$\cross$}
\put(0.2,50){\line(0,-1){53}}
\put(15.2,5){\line(0,-1){8}}
\put(30.2,5){\line(0,-1){8}}
\put(45.2,20){\line(0,-1){23}}
\put(22.7,5){\oval(15,15)[t]}
\qbezier(45.2,20)\qbezier(45.2,27.5)\qbezier(37.7,27.5)
\qbezier(30.2,35)\qbezier(30.2,27.5)\qbezier(37.7,27.5)
\qbezier(45.2,50)\qbezier(45.2,42.5)\qbezier(37.7,42.5)
\qbezier(30.2,35)\qbezier(30.2,42.5)\qbezier(37.7,42.5)


\put(12.5,60.4){$\up$}
\put(42.5,65.1){$\down$}

\put(-2.5,75.4){$\up$}
\put(42.5,80.1){$\down$}

\put(-2.5,45.4){$\up$}
\put(12.5,47.2){$\circ$}
\put(27,48.2){$\cross$}
\put(42.5,50.1){$\down$}

\put(-2.5,62.2){$\circ$}
\put(27,63.2){$\cross$}
\put(12.5,77.2){$\circ$}
\put(27,78.2){$\cross$}
\put(45.2,50){\line(0,1){53}}
\put(15.2,95){\line(0,1){8}}
\put(30.2,95){\line(0,1){8}}
\put(.2,80){\line(0,1){23}}
\put(22.7,95){\oval(15,15)[b]}
\qbezier(.2,80)\qbezier(.2,72.5)\qbezier(7.7,72.5)
\qbezier(15.2,65)\qbezier(15.2,72.5)\qbezier(7.7,72.5)
\qbezier(.2,50)\qbezier(.2,57.5)\qbezier(7.7,57.5)
\qbezier(15.2,65)\qbezier(15.2,57.5)\qbezier(7.7,57.5)

\end{picture}\:\:\:&
\begin{picture}(45,100)
\put(0,50){\line(1,0){45}}
\put(-2.5,90.4){$\up$}
\put(12.5,90.4){$\up$}
\put(27.5,95.1){$\down$}
\put(42.5,95.1){$\down$}
\put(-2.5,0.4){$\up$}
\put(12.5,0.4){$\up$}
\put(27.5,5.1){$\down$}
\put(42.5,5.1){$\down$}


\put(12.5,60.4){$\up$}
\put(42.5,65.1){$\down$}
\put(-2.5,75.4){$\up$}
\put(42.5,80.1){$\down$}

\put(27.5,45.4){$\up$}
\put(-2.5,47.2){$\circ$}
\put(12,48.2){$\cross$}
\put(42.5,50.1){$\down$}

\put(-2.5,62.2){$\circ$}
\put(27,63.2){$\cross$}
\put(12.5,77.2){$\circ$}
\put(27,78.2){$\cross$}
\put(45.2,50){\line(0,1){53}}
\put(15.2,95){\line(0,1){8}}
\put(30.2,95){\line(0,1){8}}
\put(.2,80){\line(0,1){23}}
\put(22.7,95){\oval(15,15)[b]}
\qbezier(.2,80)\qbezier(.2,72.5)\qbezier(7.7,72.5)
\qbezier(15.2,65)\qbezier(15.2,72.5)\qbezier(7.7,72.5)
\qbezier(15.2,65)\qbezier(15.2,57.5)\qbezier(22.7,57.5)
\qbezier(30.2,50)\qbezier(30.2,57.5)\qbezier(22.7,57.5)


\put(12.5,30.4){$\up$}
\put(42.5,35.1){$\down$}
\put(-2.5,32.2){$\circ$}
\put(27,33.2){$\cross$}

\put(-2.5,15.4){$\up$}
\put(42.5,20.1){$\down$}
\put(12.5,17.2){$\circ$}
\put(27,18.2){$\cross$}

\put(45.2,50){\line(0,-1){53}}
\put(15.2,5){\line(0,-1){8}}
\put(30.2,5){\line(0,-1){8}}
\put(0.2,20){\line(0,-1){23}}
\put(22.7,5){\oval(15,15)[t]}
\qbezier(0.2,20)\qbezier(0.2,27.5)\qbezier(7.7,27.5)
\qbezier(15.2,35)\qbezier(15.2,27.5)\qbezier(7.7,27.5)
\qbezier(15.2,35)\qbezier(15.2,42.5)\qbezier(22.7,42.5)
\qbezier(30.2,50)\qbezier(30.2,42.5)\qbezier(22.7,42.5)
\end{picture}\:\:\:&
\begin{picture}(45,100)


\put(0,50){\line(1,0){45}}
\put(-2.5,90.4){$\up$}
\put(12.5,90.4){$\up$}
\put(27.5,95.1){$\down$}
\put(42.5,95.1){$\down$}
\put(-2.5,0.4){$\up$}
\put(12.5,0.4){$\up$}
\put(27.5,5.1){$\down$}
\put(42.5,5.1){$\down$}


\put(-2.5,75.4){$\up$}
\put(42.5,80.1){$\down$}
\put(12.5,77.2){$\circ$}
\put(27,78.2){$\cross$}

\put(-2.5,45.4){$\up$}
\put(12.5,50.1){$\down$}
\put(27.5,47.2){$\circ$}
\put(42,48.2){$\cross$}

\put(-2.5,60.4){$\up$}
\put(27.5,65.1){$\down$}
\put(12.5,62.2){$\circ$}
\put(42,63.2){$\cross$}

\put(.2,50){\line(0,1){53}}
\put(15.2,95){\line(0,1){8}}
\put(30.2,95){\line(0,1){8}}
\put(45.2,80){\line(0,1){23}}
\put(22.7,95){\oval(15,15)[b]}
\qbezier(45.2,80)\qbezier(45.2,72.5)\qbezier(37.7,72.5)
\qbezier(30.2,65)\qbezier(30.2,72.5)\qbezier(37.7,72.5)
\qbezier(30.2,65)\qbezier(30.2,57.5)\qbezier(22.7,57.5)
\qbezier(15.2,50)\qbezier(15.2,57.5)\qbezier(22.7,57.5)


\put(-2.5,30.4){$\up$}
\put(27.5,35.1){$\down$}
\put(12.5,32.2){$\circ$}
\put(42,33.2){$\cross$}

\put(-2.5,15.4){$\up$}
\put(42.5,20.1){$\down$}
\put(12.5,17.2){$\circ$}
\put(27,18.2){$\cross$}

\put(0.2,50){\line(0,-1){53}}
\put(15.2,5){\line(0,-1){8}}
\put(30.2,5){\line(0,-1){8}}
\put(45.2,20){\line(0,-1){23}}
\put(22.7,5){\oval(15,15)[t]}
\qbezier(45.2,20)\qbezier(45.2,27.5)\qbezier(37.7,27.5)
\qbezier(30.2,35)\qbezier(30.2,27.5)\qbezier(37.7,27.5)
\qbezier(30.2,35)\qbezier(30.2,42.5)\qbezier(22.7,42.5)
\qbezier(15.2,50)\qbezier(15.2,42.5)\qbezier(22.7,42.5)
\end{picture}
\end{array}
$$
Using the explicit
multiplication rule, it follows easily 
that $$
B_{E^2 F}(0) \cong M_2(\C) \oplus \C \oplus
\C,$$ 
and the diagram basis consists of matrix units; 
in particular the idempotents $f_{\tS\tS}$ and $f_{\tT\tT}$ give the diagonal matrix
units in $M_2(\C)$.
Moreover, recalling that the degree of $f_{\tS\tT}$
is the total number of clockwise cups and caps in $\underline{\tS}
\overline{\tT}$ minus the number of caps in $\underline{\tS}$ minus
the number of cups in $\overline{\tT}$, 
the grading on $B_{E^2 F}(0)$ is trivial (concentrated in
degree zero).

Consider instead the algebra $B_{F E^2}(0)$.
There are also four $F E^2$-tableaux:
\begin{align*}
\tS'&=
((\varnothing,\varnothing) \stackrel{0}{\leftarrow}
(\varnothing,(1)) \stackrel{0}{\rightarrow}
((1),(1)) \stackrel{0}{\rightarrow} ((1),\varnothing)),\\
\tT'&=((\varnothing,\varnothing) \stackrel{0}{\leftarrow}
(\varnothing,(1)) \stackrel{0}{\rightarrow}
(\varnothing,\varnothing) \stackrel{0}{\rightarrow} ((1),\varnothing)),\\
\tU'&=
((\varnothing,\varnothing) \stackrel{0}{\leftarrow}
(\varnothing,(1)) \stackrel{0}{\rightarrow}
((1),(1)) \stackrel{-1}{\rightarrow} ((1^2),(1))),\\
\tV'&=((\varnothing,\varnothing) \stackrel{0}{\leftarrow}
(\varnothing,(1)) \stackrel{0}{\rightarrow}
((1),(1)) \stackrel{1}{\rightarrow} ((2),(1))).
\end{align*}
The algebra $B_{F E^2}(0)$ has basis
$\{f_{\tS'\tT'}, f_{\tT'\tS'}, f_{\tS'\tS'}, f_{\tT'\tT'}\} \cup
\{f_{\tU'\tU'}\} \cup \{f_{\tV'\tV'}\}$:
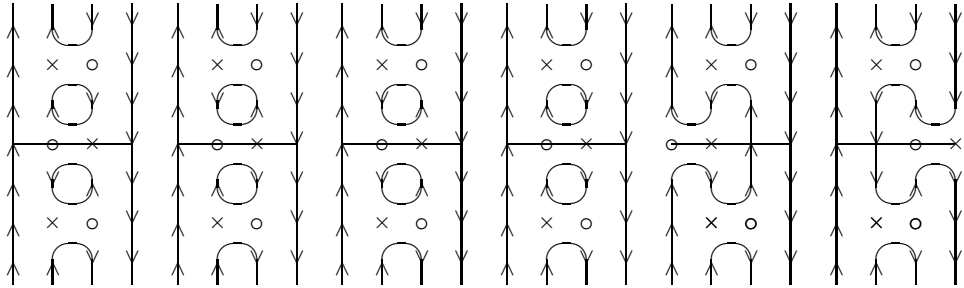
$$
\begin{array}{llllll}
\begin{picture}(45,105)
\put(0,50){\line(1,0){45}}
\put(-2.5,90.4){$\up$}
\put(12.5,90.4){$\up$}
\put(27.5,95.1){$\down$}
\put(42.5,95.1){$\down$}
\put(-2.5,0.4){$\up$}
\put(12.5,0.4){$\up$}
\put(27.5,5.1){$\down$}
\put(42.5,5.1){$\down$}
\put(30.2,95){\line(0,1){8}}
\put(15.2,95){\line(0,1){8}}
\put(30.2,5){\line(0,-1){8}}
\put(15.2,5){\line(0,-1){8}}


\put(-2.5,60.4){$\up$}
\put(27.5,65.1){$\down$}
\put(12.5,60.4){$\up$}
\put(42.5,65.1){$\down$}

\put(-2.5,75.4){$\up$}
\put(42.5,80.1){$\down$}
\put(27.5,77.2){$\circ$}
\put(12,78.2){$\cross$}

\put(-2.5,45.4){$\up$}
\put(12.5,47.2){$\circ$}
\put(27,48.2){$\cross$}
\put(42.5,50.1){$\down$}

\put(45.2,50){\line(0,1){53}}
\put(0.2,50){\line(0,1){53}}
\put(22.7,95){\oval(15,15)[b]}
\put(22.7,65){\oval(15,15)[t]}
\put(22.7,65){\oval(15,15)[b]}


\put(-2.5,30.4){$\up$}
\put(12.5,35.1){$\down$}
\put(27.5,30.4){$\up$}
\put(42.5,35.1){$\down$}

\put(-2.5,15.4){$\up$}
\put(42.5,20.1){$\down$}
\put(27.5,17.2){$\circ$}
\put(12,18.2){$\cross$}

\put(45.2,50){\line(0,-1){53}}
\put(0.2,50){\line(0,-1){53}}
\put(22.7,5){\oval(15,15)[t]}
\put(22.7,35){\oval(15,15)[b]}
\put(22.7,35){\oval(15,15)[t]}

\end{picture}\:\:\:&
\begin{picture}(45,100)
\put(0,50){\line(1,0){45}}
\put(-2.5,90.4){$\up$}
\put(12.5,90.4){$\up$}
\put(27.5,95.1){$\down$}
\put(42.5,95.1){$\down$}
\put(-2.5,0.4){$\up$}
\put(12.5,0.4){$\up$}
\put(27.5,5.1){$\down$}
\put(42.5,5.1){$\down$}
\put(30.2,95){\line(0,1){8}}
\put(15.2,95){\line(0,1){8}}
\put(30.2,5){\line(0,-1){8}}
\put(15.2,5){\line(0,-1){8}}


\put(-2.5,60.4){$\up$}
\put(12.5,65.1){$\down$}
\put(27.5,60.4){$\up$}
\put(42.5,65.1){$\down$}

\put(-2.5,75.4){$\up$}
\put(42.5,80.1){$\down$}
\put(27.5,77.2){$\circ$}
\put(12,78.2){$\cross$}

\put(-2.5,45.4){$\up$}
\put(12.5,47.2){$\circ$}
\put(27,48.2){$\cross$}
\put(42.5,50.1){$\down$}

\put(45.2,50){\line(0,1){53}}
\put(0.2,50){\line(0,1){53}}
\put(22.7,95){\oval(15,15)[b]}
\put(22.7,65){\oval(15,15)[t]}
\put(22.7,65){\oval(15,15)[b]}


\put(-2.5,30.4){$\up$}
\put(27.5,35.1){$\down$}
\put(12.5,30.4){$\up$}
\put(42.5,35.1){$\down$}

\put(-2.5,15.4){$\up$}
\put(42.5,20.1){$\down$}
\put(27.5,17.2){$\circ$}
\put(12,18.2){$\cross$}

\put(45.2,50){\line(0,-1){53}}
\put(0.2,50){\line(0,-1){53}}
\put(22.7,5){\oval(15,15)[t]}
\put(22.7,35){\oval(15,15)[b]}
\put(22.7,35){\oval(15,15)[t]}


\end{picture}\:\:\:&
\begin{picture}(45,100)
\put(0,50){\line(1,0){45}}
\put(-2.5,90.4){$\up$}
\put(12.5,90.4){$\up$}
\put(27.5,95.1){$\down$}
\put(42.5,95.1){$\down$}
\put(-2.5,0.4){$\up$}
\put(12.5,0.4){$\up$}
\put(27.5,5.1){$\down$}
\put(42.5,5.1){$\down$}



\put(-2.5,60.4){$\up$}
\put(12.5,65.1){$\down$}
\put(27.5,60.4){$\up$}
\put(42.5,65.1){$\down$}

\put(-2.5,75.4){$\up$}
\put(42.5,80.1){$\down$}
\put(27.5,77.2){$\circ$}
\put(12,78.2){$\cross$}

\put(-2.5,45.4){$\up$}
\put(12.5,47.2){$\circ$}
\put(27,48.2){$\cross$}
\put(42.5,50.1){$\down$}

\put(45.2,50){\line(0,1){53}}
\put(0.2,50){\line(0,1){53}}
\put(22.7,95){\oval(15,15)[b]}
\put(22.7,65){\oval(15,15)[t]}
\put(22.7,65){\oval(15,15)[b]}


\put(-2.5,30.4){$\up$}
\put(12.5,35.1){$\down$}
\put(27.5,30.4){$\up$}
\put(42.5,35.1){$\down$}

\put(-2.5,15.4){$\up$}
\put(42.5,20.1){$\down$}
\put(27.5,17.2){$\circ$}
\put(12,18.2){$\cross$}

\put(45.2,50){\line(0,-1){53}}
\put(30.2,5){\line(0,-1){8}}
\put(15.2,5){\line(0,-1){8}}
\put(15.2,95){\line(0,1){8}}
\put(30.2,95){\line(0,1){8}}
\put(0.2,50){\line(0,-1){53}}
\put(22.7,5){\oval(15,15)[t]}
\put(22.7,35){\oval(15,15)[b]}
\put(22.7,35){\oval(15,15)[t]}

\end{picture}\:\:\:&
\begin{picture}(45,100)
\put(0,50){\line(1,0){45}}
\put(-2.5,90.4){$\up$}
\put(12.5,90.4){$\up$}
\put(27.5,95.1){$\down$}
\put(42.5,95.1){$\down$}
\put(-2.5,0.4){$\up$}
\put(12.5,0.4){$\up$}
\put(27.5,5.1){$\down$}
\put(42.5,5.1){$\down$}
\put(30.2,95){\line(0,1){8}}
\put(15.2,95){\line(0,1){8}}
\put(30.2,5){\line(0,-1){8}}
\put(15.2,5){\line(0,-1){8}}



\put(-2.5,60.4){$\up$}
\put(27.5,65.1){$\down$}
\put(12.5,60.4){$\up$}
\put(42.5,65.1){$\down$}

\put(-2.5,75.4){$\up$}
\put(42.5,80.1){$\down$}
\put(27.5,77.2){$\circ$}
\put(12,78.2){$\cross$}

\put(-2.5,45.4){$\up$}
\put(12.5,47.2){$\circ$}
\put(27,48.2){$\cross$}
\put(42.5,50.1){$\down$}

\put(45.2,50){\line(0,1){53}}
\put(0.2,50){\line(0,1){53}}
\put(22.7,95){\oval(15,15)[b]}
\put(22.7,65){\oval(15,15)[t]}
\put(22.7,65){\oval(15,15)[b]}


\put(-2.5,30.4){$\up$}
\put(27.5,35.1){$\down$}
\put(12.5,30.4){$\up$}
\put(42.5,35.1){$\down$}

\put(-2.5,15.4){$\up$}
\put(42.5,20.1){$\down$}
\put(27.5,17.2){$\circ$}
\put(12,18.2){$\cross$}

\put(45.2,50){\line(0,-1){53}}
\put(0.2,50){\line(0,-1){53}}
\put(22.7,5){\oval(15,15)[t]}
\put(22.7,35){\oval(15,15)[b]}
\put(22.7,35){\oval(15,15)[t]}

\end{picture}\:\:\:&
\begin{picture}(45,100)
\put(30.2,95){\line(0,1){8}}
\put(15.2,95){\line(0,1){8}}
\put(30.2,5){\line(0,-1){8}}
\put(15.2,5){\line(0,-1){8}}
\put(0,50){\line(1,0){45}}
\put(-2.5,90.4){$\up$}
\put(12.5,90.4){$\up$}
\put(27.5,95.1){$\down$}
\put(42.5,95.1){$\down$}
\put(-2.5,0.4){$\up$}
\put(12.5,0.4){$\up$}
\put(27.5,5.1){$\down$}
\put(42.5,5.1){$\down$}

\put(-2.5,60.4){$\up$}
\put(12.5,65.1){$\down$}
\put(27.5,60.4){$\up$}
\put(42.5,65.1){$\down$}

\put(-2.5,75.4){$\up$}
\put(42.5,80.1){$\down$}
\put(27.5,77.2){$\circ$}
\put(12,78.2){$\cross$}

\put(27.5,45.4){$\up$}
\put(42.5,50.1){$\down$}
\put(-2.5,47.2){$\circ$}
\put(12,48.2){$\cross$}

\put(45.2,50){\line(0,1){53}}
\put(30.2,50){\line(0,1){15}}
\put(0.2,65){\line(0,1){38}}
\put(22.7,95){\oval(15,15)[b]}
\put(22.7,65){\oval(15,15)[t]}
\put(7.7,65){\oval(15,15)[b]}


\put(-2.5,30.4){$\up$}
\put(12.5,35.1){$\down$}
\put(27.5,30.4){$\up$}
\put(42.5,35.1){$\down$}

\put(-2.5,15.4){$\up$}
\put(42.5,20.1){$\down$}
\put(27.5,17.2){$\circ$}
\put(12,18.2){$\cross$}

\put(27.5,17.2){$\circ$}
\put(12,18.2){$\cross$}

\put(45.2,50){\line(0,-1){53}}
\put(30.2,50){\line(0,-1){15}}
\put(0.2,35){\line(0,-1){38}}
\put(22.7,5){\oval(15,15)[t]}
\put(22.7,35){\oval(15,15)[b]}
\put(7.7,35){\oval(15,15)[t]}

\end{picture}\:\:\:&
\begin{picture}(45,100)


\put(30.2,95){\line(0,1){8}}
\put(15.2,95){\line(0,1){8}}
\put(30.2,5){\line(0,-1){8}}
\put(15.2,5){\line(0,-1){8}}
\put(0,50){\line(1,0){45}}
\put(-2.5,90.4){$\up$}
\put(12.5,90.4){$\up$}
\put(27.5,95.1){$\down$}
\put(42.5,95.1){$\down$}
\put(-2.5,0.4){$\up$}
\put(12.5,0.4){$\up$}
\put(27.5,5.1){$\down$}
\put(42.5,5.1){$\down$}

\put(-2.5,60.4){$\up$}
\put(12.5,65.1){$\down$}
\put(27.5,60.4){$\up$}
\put(42.5,65.1){$\down$}

\put(-2.5,75.4){$\up$}
\put(42.5,80.1){$\down$}
\put(27.5,77.2){$\circ$}
\put(12,78.2){$\cross$}

\put(-2.5,45.4){$\up$}
\put(12.5,50.1){$\down$}
\put(27.5,47.2){$\circ$}
\put(42,48.2){$\cross$}

\put(0.2,50){\line(0,1){53}}
\put(15.2,50){\line(0,1){15}}
\put(45.2,65){\line(0,1){38}}
\put(22.7,95){\oval(15,15)[b]}
\put(22.7,65){\oval(15,15)[t]}
\put(37.7,65){\oval(15,15)[b]}


\put(-2.5,30.4){$\up$}
\put(12.5,35.1){$\down$}
\put(27.5,30.4){$\up$}
\put(42.5,35.1){$\down$}

\put(-2.5,15.4){$\up$}
\put(42.5,20.1){$\down$}
\put(27.5,17.2){$\circ$}
\put(12,18.2){$\cross$}

\put(27.5,17.2){$\circ$}
\put(12,18.2){$\cross$}

\put(0.2,50){\line(0,-1){53}}
\put(15.2,50){\line(0,-1){15}}
\put(45.2,35){\line(0,-1){38}}
\put(22.7,5){\oval(15,15)[t]}
\put(22.7,35){\oval(15,15)[b]}
\put(37.7,35){\oval(15,15)[t]}
\end{picture}
\end{array}
$$
This is also a semisimple algebra isomorphic to
$M_2(\C) \oplus \C \oplus \C$, with the idempotents
$f_{\tS'\tT'}$ and $f_{\tT'\tS'}$ giving the diagonal matrix
units in $M_2(\C)$.
But the grading is no longer trivial as the off-diagonal matrix units
$f_{\tS'\tS'}$ and $f_{\tT'\tT'}$ 
in the copy of $M_2(\C)$ are
of degrees $-2$ and $2$, respectively.
Comparing with the preceding example, we see 
that $B_{E^2 F}(0)$ and $B_{F E^2}(0)$ are isomorphic 
as ungraded algebras but not as graded algebras.

Finally we give a less trivial example which is not semisimple.
Consider $R = E^2 F^2$. There are ten $R$-tableaux, namely,
the
paths in the following Bratelli diagram.
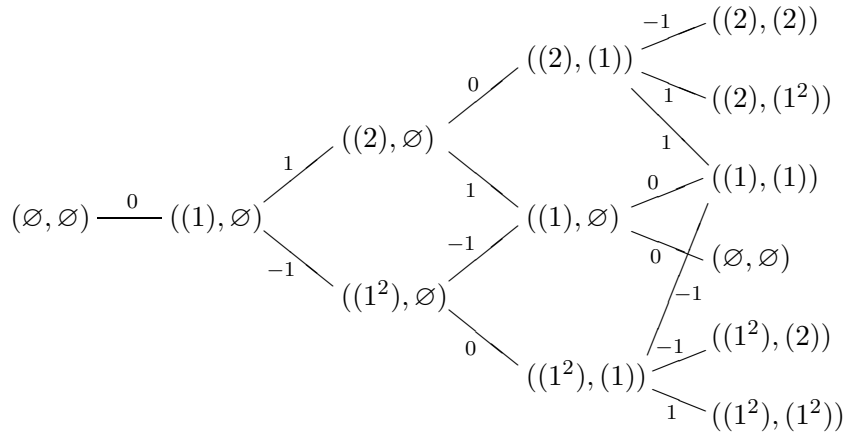
$$
\begin{picture}(320,165)
\put(-5,75){$(\varnothing,\varnothing)$}
\put(55,75){$((1),\varnothing)$}
\put(120,105){$((2),\varnothing)$}
\put(120,45){$((1^2),\varnothing)$}
\put(190,135){$((2),(1))$}
\put(190,75){$((1),\varnothing)$}
\put(190,15){$((1^2),(1))$}
\put(260,150){$((2),(2))$}
\put(260,120){$((2),(1^2))$}
\put(260,90){$((1),(1))$}
\put(260,60){$(\varnothing,\varnothing)$}
\put(260,30){$((1^2),(2))$}
\put(260,0){$((1^2),(1^2))$}

\put(28,78){\line(1,0){24}}
\put(39,78){$^0$}

\put(91,84){\line(5,4){26}}
\put(98,92){$^1$}
\put(91,73){\line(5,-4){26}}
\put(92,52){$^{-1}$}

\put(161,114){\line(5,4){26}}
\put(168,122){$^0$}
\put(161,103){\line(5,-4){26}}
\put(167,82){$^{1}$}
\put(161,54){\line(5,4){26}}
\put(160,62){$^{-1}$}
\put(161,43){\line(5,-4){26}}
\put(167,22){$^{0}$}

\put(234,142){\line(5,2){24}}
\put(234,146){$^{-1}$}
\put(234,133){\line(5,-2){24}}
\put(242,118){$^{1}$}

\put(230,82){\line(5,2){27}}
\put(236,85){$^{0}$}
\put(230,73){\line(5,-2){27}}
\put(237,57){$^{0}$}

\put(236,26){\line(2,5){23}}
\put(246,43){$^{-1}$}
\put(238,20){\line(5,2){20}}
\put(239,23){$^{-1}$}
\put(238,13){\line(5,-2){20}}
\put(243,-2){$^{1}$}

\put(230,127){\line(1,-1){29}}
\put(241,100){$^{1}$}

\end{picture}
$$
The algebra $B_{E^2 F^2}(0)$ is of dimension 24, and the diagrams
representing its basis are displayed in Figure~\ref{fig1}.
Using the multiplication rule, one can check directly in this example
that the Jacobson radical is of dimension 16 and is spanned by all the basis vectors $\{f_{\tS\tT}\}$ such that either
$\underline{\tS}$ has a clockwise boundary cup or $\overline{\tT}$
has a clockwise boundary cap; cf. Corollary~\ref{jcor} below.
Moreover the quotient of $B_{E^2 F^2}(0)$ by its Jacobson radical is
isomorphic to $M_2(\C) \oplus \C \oplus \C \oplus \C \oplus \C$;
cf. Theorem~\ref{irreddimensions} below.
\begin{figure}
$$
\begin{array}{llllll}
\begin{picture}(45,130)

\put(0,65){\line(1,0){45}}
\put(-2.5,120.4){$\up$}
\put(12.5,120.4){$\up$}
\put(27.5,125.1){$\down$}
\put(42.5,125.1){$\down$}
\put(-2.5,0.4){$\up$}
\put(12.5,0.4){$\up$}
\put(27.5,5.1){$\down$}
\put(42.5,5.1){$\down$}


\put(-2.5,75.4){$\up$}
\put(42.5,80.1){$\down$}
\put(-2.5,90.4){$\up$}
\put(27.5,95.1){$\down$}
\put(-2.5,105.4){$\up$}
\put(42.5,110.1){$\down$}

\put(-2.5,60.4){$\up$}
\put(27.5,60.4){$\up$}
\put(12.5,65.1){$\down$}
\put(42.5,65.1){$\down$}
\put(12.5,77.2){$\circ$}
\put(27,78.2){$\cross$}
\put(12.5,92.2){$\circ$}
\put(42,93.2){$\cross$}
\put(12.5,107.2){$\circ$}
\put(27,108.2){$\cross$}
\put(0.2,65){\line(0,1){68}}
\put(15.2,125){\line(0,1){8}}
\put(30.2,125){\line(0,1){8}}
\put(45.2,110){\line(0,1){23}}
\put(45.2,65){\line(0,1){15}}
\put(22.7,125){\oval(15,15)[b]}
\put(22.7,65){\oval(15,15)[t]}
\qbezier(45.2,110)\qbezier(45.2,102.5)\qbezier(37.7,102.5)
\qbezier(30.2,95)\qbezier(30.2,102.5)\qbezier(37.7,102.5)
\qbezier(45.2,80)\qbezier(45.2,87.5)\qbezier(37.7,87.5)
\qbezier(30.2,95)\qbezier(30.2,87.5)\qbezier(37.7,87.5)


\put(-2.5,45.4){$\up$}
\put(42.5,50.1){$\down$}
\put(-2.5,30.4){$\up$}
\put(27.5,35.1){$\down$}
\put(-2.5,15.4){$\up$}
\put(42.5,20.1){$\down$}

\put(12.5,47.2){$\circ$}
\put(27,48.2){$\cross$}
\put(12.5,32.2){$\circ$}
\put(42,33.2){$\cross$}
\put(12.5,17.2){$\circ$}
\put(27,18.2){$\cross$}
\put(0.2,65){\line(0,-1){68}}
\put(15.2,5){\line(0,-1){8}}
\put(30.2,5){\line(0,-1){8}}
\put(45.2,20){\line(0,-1){23}}
\put(45.2,65){\line(0,-1){15}}
\put(22.7,5){\oval(15,15)[t]}
\put(22.7,65){\oval(15,15)[b]}
\qbezier(45.2,20)\qbezier(45.2,27.5)\qbezier(37.7,27.5)
\qbezier(30.2,35)\qbezier(30.2,27.5)\qbezier(37.7,27.5)
\qbezier(45.2,50)\qbezier(45.2,42.5)\qbezier(37.7,42.5)
\qbezier(30.2,35)\qbezier(30.2,42.5)\qbezier(37.7,42.5)
\end{picture}\:\:\:
&
\begin{picture}(45,130)

\put(0,65){\line(1,0){45}}
\put(-2.5,120.4){$\up$}
\put(12.5,120.4){$\up$}
\put(27.5,125.1){$\down$}
\put(42.5,125.1){$\down$}
\put(-2.5,0.4){$\up$}
\put(12.5,0.4){$\up$}
\put(27.5,5.1){$\down$}
\put(42.5,5.1){$\down$}


\put(-2.5,75.4){$\up$}
\put(42.5,80.1){$\down$}
\put(-2.5,90.4){$\up$}
\put(27.5,95.1){$\down$}
\put(-2.5,105.4){$\up$}
\put(42.5,110.1){$\down$}

\put(-2.5,60.4){$\up$}
\put(12.5,60.4){$\up$}
\put(27.5,65.1){$\down$}
\put(42.5,65.1){$\down$}
\put(12.5,77.2){$\circ$}
\put(27,78.2){$\cross$}
\put(12.5,92.2){$\circ$}
\put(42,93.2){$\cross$}
\put(12.5,107.2){$\circ$}
\put(27,108.2){$\cross$}
\put(0.2,65){\line(0,1){68}}
\put(15.2,125){\line(0,1){8}}
\put(30.2,125){\line(0,1){8}}
\put(45.2,110){\line(0,1){23}}
\put(45.2,65){\line(0,1){15}}
\put(22.7,125){\oval(15,15)[b]}
\put(22.7,65){\oval(15,15)[t]}
\qbezier(45.2,110)\qbezier(45.2,102.5)\qbezier(37.7,102.5)
\qbezier(30.2,95)\qbezier(30.2,102.5)\qbezier(37.7,102.5)
\qbezier(45.2,80)\qbezier(45.2,87.5)\qbezier(37.7,87.5)
\qbezier(30.2,95)\qbezier(30.2,87.5)\qbezier(37.7,87.5)

\put(-2.5,45.4){$\up$}
\put(42.5,50.1){$\down$}
\put(-2.5,30.4){$\up$}
\put(27.5,35.1){$\down$}
\put(-2.5,15.4){$\up$}
\put(42.5,20.1){$\down$}

\put(12.5,47.2){$\circ$}
\put(27,48.2){$\cross$}
\put(12.5,32.2){$\circ$}
\put(42,33.2){$\cross$}
\put(12.5,17.2){$\circ$}
\put(27,18.2){$\cross$}
\put(0.2,65){\line(0,-1){68}}
\put(15.2,5){\line(0,-1){8}}
\put(30.2,5){\line(0,-1){8}}
\put(45.2,20){\line(0,-1){23}}
\put(45.2,65){\line(0,-1){15}}
\put(22.7,5){\oval(15,15)[t]}
\put(22.7,65){\oval(15,15)[b]}
\qbezier(45.2,20)\qbezier(45.2,27.5)\qbezier(37.7,27.5)
\qbezier(30.2,35)\qbezier(30.2,27.5)\qbezier(37.7,27.5)
\qbezier(45.2,50)\qbezier(45.2,42.5)\qbezier(37.7,42.5)
\qbezier(30.2,35)\qbezier(30.2,42.5)\qbezier(37.7,42.5)
\end{picture}\:\:\:
&
\begin{picture}(45,130)

\put(0,65){\line(1,0){45}}
\put(-2.5,120.4){$\up$}
\put(12.5,120.4){$\up$}
\put(27.5,125.1){$\down$}
\put(42.5,125.1){$\down$}
\put(-2.5,0.4){$\up$}
\put(12.5,0.4){$\up$}
\put(27.5,5.1){$\down$}
\put(42.5,5.1){$\down$}


\put(-2.5,75.4){$\up$}
\put(42.5,80.1){$\down$}
\put(12.5,90.4){$\up$}
\put(42.5,95.1){$\down$}
\put(-2.5,105.4){$\up$}
\put(42.5,110.1){$\down$}

\put(-2.5,60.4){$\up$}
\put(27.5,60.4){$\up$}
\put(12.5,65.1){$\down$}
\put(42.5,65.1){$\down$}
\put(12.5,77.2){$\circ$}
\put(27,78.2){$\cross$}
\put(-2.5,92.2){$\circ$}
\put(27,93.2){$\cross$}
\put(12.5,107.2){$\circ$}
\put(27,108.2){$\cross$}
\put(45.2,65){\line(0,1){68}}
\put(15.2,125){\line(0,1){8}}
\put(30.2,125){\line(0,1){8}}
\put(0.2,110){\line(0,1){23}}
\put(0.2,65){\line(0,1){15}}
\put(22.7,125){\oval(15,15)[b]}
\put(22.7,65){\oval(15,15)[t]}
\qbezier(0.2,110)\qbezier(0.2,102.5)\qbezier(7.7,102.5)
\qbezier(15.2,95)\qbezier(15.2,102.5)\qbezier(7.7,102.5)
\qbezier(0.2,80)\qbezier(0.2,87.5)\qbezier(7.7,87.5)
\qbezier(15.2,95)\qbezier(15.2,87.5)\qbezier(7.7,87.5)


\put(-2.5,45.4){$\up$}
\put(42.5,50.1){$\down$}
\put(12.5,30.4){$\up$}
\put(42.5,35.1){$\down$}
\put(-2.5,15.4){$\up$}
\put(42.5,20.1){$\down$}

\put(12.5,47.2){$\circ$}
\put(27,48.2){$\cross$}
\put(-2.5,32.2){$\circ$}
\put(27,33.2){$\cross$}
\put(12.5,17.2){$\circ$}
\put(27,18.2){$\cross$}
\put(45.2,65){\line(0,-1){68}}
\put(15.2,5){\line(0,-1){8}}
\put(30.2,5){\line(0,-1){8}}
\put(0.2,20){\line(0,-1){23}}
\put(0.2,65){\line(0,-1){15}}
\put(22.7,5){\oval(15,15)[t]}
\put(22.7,65){\oval(15,15)[b]}
\qbezier(0.2,20)\qbezier(0.2,27.5)\qbezier(7.7,27.5)
\qbezier(15.2,35)\qbezier(15.2,27.5)\qbezier(7.7,27.5)
\qbezier(0.2,50)\qbezier(0.2,42.5)\qbezier(7.7,42.5)
\qbezier(15.2,35)\qbezier(15.2,42.5)\qbezier(7.7,42.5)

\end{picture}\:\:\:
&
\begin{picture}(45,130)

\put(0,65){\line(1,0){45}}
\put(-2.5,120.4){$\up$}
\put(12.5,120.4){$\up$}
\put(27.5,125.1){$\down$}
\put(42.5,125.1){$\down$}
\put(-2.5,0.4){$\up$}
\put(12.5,0.4){$\up$}
\put(27.5,5.1){$\down$}
\put(42.5,5.1){$\down$}

\put(-2.5,75.4){$\up$}
\put(42.5,80.1){$\down$}
\put(12.5,90.4){$\up$}
\put(42.5,95.1){$\down$}
\put(-2.5,105.4){$\up$}
\put(42.5,110.1){$\down$}

\put(-2.5,60.4){$\up$}
\put(12.5,60.4){$\up$}
\put(27.5,65.1){$\down$}
\put(42.5,65.1){$\down$}
\put(12.5,77.2){$\circ$}
\put(27,78.2){$\cross$}
\put(-2.5,92.2){$\circ$}
\put(27,93.2){$\cross$}
\put(12.5,107.2){$\circ$}
\put(27,108.2){$\cross$}
\put(45.2,65){\line(0,1){68}}
\put(15.2,125){\line(0,1){8}}
\put(30.2,125){\line(0,1){8}}
\put(0.2,110){\line(0,1){23}}
\put(0.2,65){\line(0,1){15}}
\put(22.7,125){\oval(15,15)[b]}
\put(22.7,65){\oval(15,15)[t]}
\qbezier(0.2,110)\qbezier(0.2,102.5)\qbezier(7.7,102.5)
\qbezier(15.2,95)\qbezier(15.2,102.5)\qbezier(7.7,102.5)
\qbezier(0.2,80)\qbezier(0.2,87.5)\qbezier(7.7,87.5)
\qbezier(15.2,95)\qbezier(15.2,87.5)\qbezier(7.7,87.5)


\put(-2.5,45.4){$\up$}
\put(42.5,50.1){$\down$}
\put(12.5,30.4){$\up$}
\put(42.5,35.1){$\down$}
\put(-2.5,15.4){$\up$}
\put(42.5,20.1){$\down$}

\put(12.5,47.2){$\circ$}
\put(27,48.2){$\cross$}
\put(-2.5,32.2){$\circ$}
\put(27,33.2){$\cross$}
\put(12.5,17.2){$\circ$}
\put(27,18.2){$\cross$}
\put(45.2,65){\line(0,-1){68}}
\put(15.2,5){\line(0,-1){8}}
\put(30.2,5){\line(0,-1){8}}
\put(0.2,20){\line(0,-1){23}}
\put(0.2,65){\line(0,-1){15}}
\put(22.7,5){\oval(15,15)[t]}
\put(22.7,65){\oval(15,15)[b]}
\qbezier(0.2,20)\qbezier(0.2,27.5)\qbezier(7.7,27.5)
\qbezier(15.2,35)\qbezier(15.2,27.5)\qbezier(7.7,27.5)
\qbezier(0.2,50)\qbezier(0.2,42.5)\qbezier(7.7,42.5)
\qbezier(15.2,35)\qbezier(15.2,42.5)\qbezier(7.7,42.5)
\end{picture}\:\:\:
&
\begin{picture}(45,130)

\put(0,65){\line(1,0){45}}
\put(-2.5,120.4){$\up$}
\put(12.5,120.4){$\up$}
\put(27.5,125.1){$\down$}
\put(42.5,125.1){$\down$}
\put(-2.5,0.4){$\up$}
\put(12.5,0.4){$\up$}
\put(27.5,5.1){$\down$}
\put(42.5,5.1){$\down$}


\put(-2.5,75.4){$\up$}
\put(12.5,80.1){$\down$}
\put(-2.5,90.4){$\up$}
\put(27.5,95.1){$\down$}
\put(-2.5,105.4){$\up$}
\put(42.5,110.1){$\down$}

\put(-2.5,60.4){$\up$}
\put(12.5,65.1){$\down$}
\put(27.5,65.1){$\down$}
\put(42.5,60.4){$\up$}
\put(27.5,77.2){$\circ$}
\put(42,78.2){$\cross$}
\put(12.5,92.2){$\circ$}
\put(42,93.2){$\cross$}
\put(12.5,107.2){$\circ$}
\put(27,108.2){$\cross$}
\put(0.2,65){\line(0,1){68}}
\put(15.2,125){\line(0,1){8}}
\put(30.2,125){\line(0,1){8}}
\put(45.2,110){\line(0,1){23}}
\put(15.2,65){\line(0,1){15}}
\put(22.7,125){\oval(15,15)[b]}
\put(37.7,65){\oval(15,15)[t]}
\qbezier(45.2,110)\qbezier(45.2,102.5)\qbezier(37.7,102.5)
\qbezier(30.2,95)\qbezier(30.2,102.5)\qbezier(37.7,102.5)
\qbezier(15.2,80)\qbezier(15.2,87.5)\qbezier(22.7,87.5)
\qbezier(30.2,95)\qbezier(30.2,87.5)\qbezier(22.7,87.5)


\put(-2.5,45.4){$\up$}
\put(12.5,50.1){$\down$}
\put(-2.5,30.4){$\up$}
\put(27.5,35.1){$\down$}
\put(-2.5,15.4){$\up$}
\put(42.5,20.1){$\down$}

\put(27.5,47.2){$\circ$}
\put(42,48.2){$\cross$}
\put(12.5,32.2){$\circ$}
\put(42,33.2){$\cross$}
\put(12.5,17.2){$\circ$}
\put(27,18.2){$\cross$}
\put(0.2,65){\line(0,-1){68}}
\put(15.2,5){\line(0,-1){8}}
\put(30.2,5){\line(0,-1){8}}
\put(45.2,20){\line(0,-1){23}}
\put(15.2,65){\line(0,-1){15}}
\put(22.7,5){\oval(15,15)[t]}
\put(37.7,65){\oval(15,15)[b]}
\qbezier(45.2,20)\qbezier(45.2,27.5)\qbezier(37.7,27.5)
\qbezier(30.2,35)\qbezier(30.2,27.5)\qbezier(37.7,27.5)
\qbezier(15.2,50)\qbezier(15.2,42.5)\qbezier(22.7,42.5)
\qbezier(30.2,35)\qbezier(30.2,42.5)\qbezier(22.7,42.5)
\end{picture}\:\:\:
&
\begin{picture}(45,130)
\put(0,65){\line(1,0){45}}
\put(-2.5,120.4){$\up$}
\put(12.5,120.4){$\up$}
\put(27.5,125.1){$\down$}
\put(42.5,125.1){$\down$}
\put(-2.5,0.4){$\up$}
\put(12.5,0.4){$\up$}
\put(27.5,5.1){$\down$}
\put(42.5,5.1){$\down$}

\put(-2.5,75.4){$\up$}
\put(12.5,80.1){$\down$}
\put(-2.5,90.4){$\up$}
\put(27.5,95.1){$\down$}
\put(-2.5,105.4){$\up$}
\put(42.5,110.1){$\down$}

\put(-2.5,60.4){$\up$}
\put(12.5,65.1){$\down$}
\put(42.5,65.1){$\down$}
\put(27.5,60.4){$\up$}
\put(27.5,77.2){$\circ$}
\put(42,78.2){$\cross$}
\put(12.5,92.2){$\circ$}
\put(42,93.2){$\cross$}
\put(12.5,107.2){$\circ$}
\put(27,108.2){$\cross$}
\put(0.2,65){\line(0,1){68}}
\put(15.2,125){\line(0,1){8}}
\put(30.2,125){\line(0,1){8}}
\put(45.2,110){\line(0,1){23}}
\put(15.2,65){\line(0,1){15}}
\put(22.7,125){\oval(15,15)[b]}
\put(37.7,65){\oval(15,15)[t]}
\qbezier(45.2,110)\qbezier(45.2,102.5)\qbezier(37.7,102.5)
\qbezier(30.2,95)\qbezier(30.2,102.5)\qbezier(37.7,102.5)
\qbezier(15.2,80)\qbezier(15.2,87.5)\qbezier(22.7,87.5)
\qbezier(30.2,95)\qbezier(30.2,87.5)\qbezier(22.7,87.5)


\put(-2.5,45.4){$\up$}
\put(12.5,50.1){$\down$}
\put(-2.5,30.4){$\up$}
\put(27.5,35.1){$\down$}
\put(-2.5,15.4){$\up$}
\put(42.5,20.1){$\down$}

\put(27.5,47.2){$\circ$}
\put(42,48.2){$\cross$}
\put(12.5,32.2){$\circ$}
\put(42,33.2){$\cross$}
\put(12.5,17.2){$\circ$}
\put(27,18.2){$\cross$}
\put(0.2,65){\line(0,-1){68}}
\put(15.2,5){\line(0,-1){8}}
\put(30.2,5){\line(0,-1){8}}
\put(45.2,20){\line(0,-1){23}}
\put(15.2,65){\line(0,-1){15}}
\put(22.7,5){\oval(15,15)[t]}
\put(37.7,65){\oval(15,15)[b]}
\qbezier(45.2,20)\qbezier(45.2,27.5)\qbezier(37.7,27.5)
\qbezier(30.2,35)\qbezier(30.2,27.5)\qbezier(37.7,27.5)
\qbezier(15.2,50)\qbezier(15.2,42.5)\qbezier(22.7,42.5)
\qbezier(30.2,35)\qbezier(30.2,42.5)\qbezier(22.7,42.5)

\end{picture}
\\\\
\begin{picture}(45,130)
\put(0,65){\line(1,0){45}}
\put(-2.5,120.4){$\up$}
\put(12.5,120.4){$\up$}
\put(27.5,125.1){$\down$}
\put(42.5,125.1){$\down$}
\put(-2.5,0.4){$\up$}
\put(12.5,0.4){$\up$}
\put(27.5,5.1){$\down$}
\put(42.5,5.1){$\down$}


\put(27.5,75.4){$\up$}
\put(42.5,80.1){$\down$}
\put(12.5,90.4){$\up$}
\put(42.5,95.1){$\down$}
\put(-2.5,105.4){$\up$}
\put(42.5,110.1){$\down$}

\put(-2.5,65.1){$\down$}
\put(12.5,60.4){$\up$}
\put(27.5,60.4){$\up$}
\put(42.5,65.1){$\down$}

\put(-2.5,77.2){$\circ$}
\put(12,78.2){$\cross$}

\put(-2.5,92.2){$\circ$}
\put(27.5,93.2){$\cross$}

\put(12.5,107.2){$\circ$}
\put(27,108.2){$\cross$}
\put(45.2,65){\line(0,1){68}}
\put(30.2,65){\line(0,1){15}}
\put(15.2,125){\line(0,1){8}}
\put(30.2,125){\line(0,1){8}}
\put(0.2,110){\line(0,1){23}}
\put(22.7,125){\oval(15,15)[b]}
\put(7.7,65){\oval(15,15)[t]}
\qbezier(0.2,110)\qbezier(0.2,102.5)\qbezier(7.7,102.5)
\qbezier(15.2,95)\qbezier(15.2,102.5)\qbezier(7.7,102.5)
\qbezier(30.2,80)\qbezier(30.2,87.5)\qbezier(22.7,87.5)
\qbezier(15.2,95)\qbezier(15.2,87.5)\qbezier(22.7,87.5)


\put(27.5,45.4){$\up$}
\put(42.5,50.1){$\down$}
\put(12.5,30.4){$\up$}
\put(42.5,35.1){$\down$}
\put(-2.5,15.4){$\up$}
\put(42.5,20.1){$\down$}

\put(-2.5,47.2){$\circ$}
\put(12,48.2){$\cross$}

\put(-2.5,32.2){$\circ$}
\put(27.5,33.2){$\cross$}

\put(12.5,17.2){$\circ$}
\put(27,18.2){$\cross$}
\put(45.2,65){\line(0,-1){68}}
\put(30.2,65){\line(0,-1){15}}
\put(15.2,5){\line(0,-1){8}}
\put(30.2,5){\line(0,-1){8}}
\put(0.2,20){\line(0,-1){23}}
\put(22.7,5){\oval(15,15)[t]}
\put(7.7,65){\oval(15,15)[b]}
\qbezier(0.2,20)\qbezier(0.2,27.5)\qbezier(7.7,27.5)
\qbezier(15.2,35)\qbezier(15.2,27.5)\qbezier(7.7,27.5)
\qbezier(30.2,50)\qbezier(30.2,42.5)\qbezier(22.7,42.5)
\qbezier(15.2,35)\qbezier(15.2,42.5)\qbezier(22.7,42.5)

\end{picture}
&
\begin{picture}(45,130)
\put(0,65){\line(1,0){45}}
\put(-2.5,120.4){$\up$}
\put(12.5,120.4){$\up$}
\put(27.5,125.1){$\down$}
\put(42.5,125.1){$\down$}
\put(-2.5,0.4){$\up$}
\put(12.5,0.4){$\up$}
\put(27.5,5.1){$\down$}
\put(42.5,5.1){$\down$}


\put(27.5,75.4){$\up$}
\put(42.5,80.1){$\down$}
\put(12.5,90.4){$\up$}
\put(42.5,95.1){$\down$}
\put(-2.5,105.4){$\up$}
\put(42.5,110.1){$\down$}

\put(12.5,65.1){$\down$}
\put(-2.5,60.4){$\up$}
\put(27.5,60.4){$\up$}
\put(42.5,65.1){$\down$}

\put(-2.5,77.2){$\circ$}
\put(12,78.2){$\cross$}

\put(-2.5,92.2){$\circ$}
\put(27.5,93.2){$\cross$}

\put(12.5,107.2){$\circ$}
\put(27,108.2){$\cross$}
\put(45.2,65){\line(0,1){68}}
\put(30.2,65){\line(0,1){15}}
\put(15.2,125){\line(0,1){8}}
\put(30.2,125){\line(0,1){8}}
\put(0.2,110){\line(0,1){23}}
\put(22.7,125){\oval(15,15)[b]}
\put(7.7,65){\oval(15,15)[t]}
\qbezier(0.2,110)\qbezier(0.2,102.5)\qbezier(7.7,102.5)
\qbezier(15.2,95)\qbezier(15.2,102.5)\qbezier(7.7,102.5)
\qbezier(30.2,80)\qbezier(30.2,87.5)\qbezier(22.7,87.5)
\qbezier(15.2,95)\qbezier(15.2,87.5)\qbezier(22.7,87.5)


\put(27.5,45.4){$\up$}
\put(42.5,50.1){$\down$}
\put(12.5,30.4){$\up$}
\put(42.5,35.1){$\down$}
\put(-2.5,15.4){$\up$}
\put(42.5,20.1){$\down$}

\put(-2.5,47.2){$\circ$}
\put(12,48.2){$\cross$}

\put(-2.5,32.2){$\circ$}
\put(27.5,33.2){$\cross$}

\put(12.5,17.2){$\circ$}
\put(27,18.2){$\cross$}
\put(45.2,65){\line(0,-1){68}}
\put(30.2,65){\line(0,-1){15}}
\put(15.2,5){\line(0,-1){8}}
\put(30.2,5){\line(0,-1){8}}
\put(0.2,20){\line(0,-1){23}}
\put(22.7,5){\oval(15,15)[t]}
\put(7.7,65){\oval(15,15)[b]}
\qbezier(0.2,20)\qbezier(0.2,27.5)\qbezier(7.7,27.5)
\qbezier(15.2,35)\qbezier(15.2,27.5)\qbezier(7.7,27.5)
\qbezier(30.2,50)\qbezier(30.2,42.5)\qbezier(22.7,42.5)
\qbezier(15.2,35)\qbezier(15.2,42.5)\qbezier(22.7,42.5)
\end{picture}
&
\begin{picture}(45,130)

\put(0,65){\line(1,0){45}}
\put(-2.5,120.4){$\up$}
\put(12.5,120.4){$\up$}
\put(27.5,125.1){$\down$}
\put(42.5,125.1){$\down$}
\put(-2.5,0.4){$\up$}
\put(12.5,0.4){$\up$}
\put(27.5,5.1){$\down$}
\put(42.5,5.1){$\down$}


\put(-2.5,75.4){$\up$}
\put(12.5,80.1){$\down$}
\put(-2.5,90.4){$\up$}
\put(27.5,95.1){$\down$}
\put(-2.5,105.4){$\up$}
\put(42.5,110.1){$\down$}

\put(-3,63.2){$\cross$}
\put(12.5,62.2){$\circ$}
\put(27.5,62.2){$\circ$}
\put(42,63.2){$\cross$}

\put(27.5,77.2){$\circ$}
\put(42,78.2){$\cross$}
\put(12.5,92.2){$\circ$}
\put(42,93.2){$\cross$}
\put(12.5,107.2){$\circ$}
\put(27,108.2){$\cross$}
\put(0.2,80){\line(0,1){53}}
\put(15.2,125){\line(0,1){8}}
\put(30.2,125){\line(0,1){8}}
\put(45.2,110){\line(0,1){23}}
\put(22.7,125){\oval(15,15)[b]}
\put(7.7,80){\oval(15,15)[b]}
\qbezier(45.2,110)\qbezier(45.2,102.5)\qbezier(37.7,102.5)
\qbezier(30.2,95)\qbezier(30.2,102.5)\qbezier(37.7,102.5)
\qbezier(15.2,80)\qbezier(15.2,87.5)\qbezier(22.7,87.5)
\qbezier(30.2,95)\qbezier(30.2,87.5)\qbezier(22.7,87.5)


\put(-2.5,45.4){$\up$}
\put(12.5,50.1){$\down$}
\put(-2.5,30.4){$\up$}
\put(27.5,35.1){$\down$}
\put(-2.5,15.4){$\up$}
\put(42.5,20.1){$\down$}

\put(27.5,47.2){$\circ$}
\put(42,48.2){$\cross$}
\put(12.5,32.2){$\circ$}
\put(42,33.2){$\cross$}
\put(12.5,17.2){$\circ$}
\put(27,18.2){$\cross$}
\put(0.2,50){\line(0,-1){53}}
\put(15.2,5){\line(0,-1){8}}
\put(30.2,5){\line(0,-1){8}}
\put(45.2,20){\line(0,-1){23}}
\put(22.7,5){\oval(15,15)[t]}
\put(7.7,50){\oval(15,15)[t]}
\qbezier(45.2,20)\qbezier(45.2,27.5)\qbezier(37.7,27.5)
\qbezier(30.2,35)\qbezier(30.2,27.5)\qbezier(37.7,27.5)
\qbezier(15.2,50)\qbezier(15.2,42.5)\qbezier(22.7,42.5)
\qbezier(30.2,35)\qbezier(30.2,42.5)\qbezier(22.7,42.5)
\end{picture}
&
\begin{picture}(45,130)

\put(0,65){\line(1,0){45}}
\put(-2.5,120.4){$\up$}
\put(12.5,120.4){$\up$}
\put(27.5,125.1){$\down$}
\put(42.5,125.1){$\down$}
\put(-2.5,0.4){$\up$}
\put(12.5,0.4){$\up$}
\put(27.5,5.1){$\down$}
\put(42.5,5.1){$\down$}


\put(27.5,75.4){$\up$}
\put(42.5,80.1){$\down$}
\put(12.5,90.4){$\up$}
\put(42.5,95.1){$\down$}
\put(-2.5,105.4){$\up$}
\put(42.5,110.1){$\down$}

\put(-2.5,62.2){$\circ$}
\put(42.5,62.2){$\circ$}
\put(12,63.2){$\cross$}
\put(27,63.2){$\cross$}

\put(-2.5,77.2){$\circ$}
\put(12,78.2){$\cross$}

\put(-2.5,92.2){$\circ$}
\put(27.5,93.2){$\cross$}

\put(12.5,107.2){$\circ$}
\put(27,108.2){$\cross$}
\put(45.2,80){\line(0,1){53}}
\put(15.2,125){\line(0,1){8}}
\put(30.2,125){\line(0,1){8}}
\put(0.2,110){\line(0,1){23}}
\put(22.7,125){\oval(15,15)[b]}
\put(37.7,80){\oval(15,15)[b]}
\qbezier(0.2,110)\qbezier(0.2,102.5)\qbezier(7.7,102.5)
\qbezier(15.2,95)\qbezier(15.2,102.5)\qbezier(7.7,102.5)
\qbezier(30.2,80)\qbezier(30.2,87.5)\qbezier(22.7,87.5)
\qbezier(15.2,95)\qbezier(15.2,87.5)\qbezier(22.7,87.5)


\put(27.5,45.4){$\up$}
\put(42.5,50.1){$\down$}
\put(12.5,30.4){$\up$}
\put(42.5,35.1){$\down$}
\put(-2.5,15.4){$\up$}
\put(42.5,20.1){$\down$}

\put(-2.5,47.2){$\circ$}
\put(12,48.2){$\cross$}

\put(-2.5,32.2){$\circ$}
\put(27.5,33.2){$\cross$}

\put(12.5,17.2){$\circ$}
\put(27,18.2){$\cross$}
\put(45.2,50){\line(0,-1){53}}
\put(15.2,5){\line(0,-1){8}}
\put(30.2,5){\line(0,-1){8}}
\put(0.2,20){\line(0,-1){23}}
\put(22.7,5){\oval(15,15)[t]}
\put(37.7,50){\oval(15,15)[t]}
\qbezier(0.2,20)\qbezier(0.2,27.5)\qbezier(7.7,27.5)
\qbezier(15.2,35)\qbezier(15.2,27.5)\qbezier(7.7,27.5)
\qbezier(30.2,50)\qbezier(30.2,42.5)\qbezier(22.7,42.5)
\qbezier(15.2,35)\qbezier(15.2,42.5)\qbezier(22.7,42.5)

\end{picture}
&
\begin{picture}(45,130)
\put(0,65){\line(1,0){45}}
\put(-2.5,120.4){$\up$}
\put(12.5,120.4){$\up$}
\put(27.5,125.1){$\down$}
\put(42.5,125.1){$\down$}
\put(-2.5,0.4){$\up$}
\put(12.5,0.4){$\up$}
\put(27.5,5.1){$\down$}
\put(42.5,5.1){$\down$}


\put(27.5,75.4){$\up$}
\put(42.5,80.1){$\down$}
\put(12.5,90.4){$\up$}
\put(42.5,95.1){$\down$}
\put(-2.5,105.4){$\up$}
\put(42.5,110.1){$\down$}

\put(12.5,65.1){$\down$}
\put(-2.5,60.4){$\up$}
\put(27.5,60.4){$\up$}
\put(42.5,65.1){$\down$}

\put(-2.5,77.2){$\circ$}
\put(12,78.2){$\cross$}

\put(-2.5,92.2){$\circ$}
\put(27.5,93.2){$\cross$}

\put(12.5,107.2){$\circ$}
\put(27,108.2){$\cross$}
\put(45.2,65){\line(0,1){68}}
\put(30.2,65){\line(0,1){15}}
\put(15.2,125){\line(0,1){8}}
\put(30.2,125){\line(0,1){8}}
\put(0.2,110){\line(0,1){23}}
\put(22.7,125){\oval(15,15)[b]}
\put(7.7,65){\oval(15,15)[t]}
\qbezier(0.2,110)\qbezier(0.2,102.5)\qbezier(7.7,102.5)
\qbezier(15.2,95)\qbezier(15.2,102.5)\qbezier(7.7,102.5)
\qbezier(30.2,80)\qbezier(30.2,87.5)\qbezier(22.7,87.5)
\qbezier(15.2,95)\qbezier(15.2,87.5)\qbezier(22.7,87.5)


\put(-2.5,45.4){$\up$}
\put(12.5,50.1){$\down$}
\put(-2.5,30.4){$\up$}
\put(27.5,35.1){$\down$}
\put(-2.5,15.4){$\up$}
\put(42.5,20.1){$\down$}

\put(27.5,47.2){$\circ$}
\put(42,48.2){$\cross$}
\put(12.5,32.2){$\circ$}
\put(42,33.2){$\cross$}
\put(12.5,17.2){$\circ$}
\put(27,18.2){$\cross$}
\put(0.2,65){\line(0,-1){68}}
\put(15.2,5){\line(0,-1){8}}
\put(30.2,5){\line(0,-1){8}}
\put(45.2,20){\line(0,-1){23}}
\put(15.2,65){\line(0,-1){15}}
\put(22.7,5){\oval(15,15)[t]}
\put(37.7,65){\oval(15,15)[b]}
\qbezier(45.2,20)\qbezier(45.2,27.5)\qbezier(37.7,27.5)
\qbezier(30.2,35)\qbezier(30.2,27.5)\qbezier(37.7,27.5)
\qbezier(15.2,50)\qbezier(15.2,42.5)\qbezier(22.7,42.5)
\qbezier(30.2,35)\qbezier(30.2,42.5)\qbezier(22.7,42.5)
\end{picture}
&
\begin{picture}(45,130)
\put(0,65){\line(1,0){45}}
\put(-2.5,120.4){$\up$}
\put(12.5,120.4){$\up$}
\put(27.5,125.1){$\down$}
\put(42.5,125.1){$\down$}
\put(-2.5,0.4){$\up$}
\put(12.5,0.4){$\up$}
\put(27.5,5.1){$\down$}
\put(42.5,5.1){$\down$}

\put(-2.5,75.4){$\up$}
\put(12.5,80.1){$\down$}
\put(-2.5,90.4){$\up$}
\put(27.5,95.1){$\down$}
\put(-2.5,105.4){$\up$}
\put(42.5,110.1){$\down$}

\put(-2.5,60.4){$\up$}
\put(12.5,65.1){$\down$}
\put(42.5,65.1){$\down$}
\put(27.5,60.4){$\up$}
\put(27.5,77.2){$\circ$}
\put(42,78.2){$\cross$}
\put(12.5,92.2){$\circ$}
\put(42,93.2){$\cross$}
\put(12.5,107.2){$\circ$}
\put(27,108.2){$\cross$}
\put(0.2,65){\line(0,1){68}}
\put(15.2,125){\line(0,1){8}}
\put(30.2,125){\line(0,1){8}}
\put(45.2,110){\line(0,1){23}}
\put(15.2,65){\line(0,1){15}}
\put(22.7,125){\oval(15,15)[b]}
\put(37.7,65){\oval(15,15)[t]}
\qbezier(45.2,110)\qbezier(45.2,102.5)\qbezier(37.7,102.5)
\qbezier(30.2,95)\qbezier(30.2,102.5)\qbezier(37.7,102.5)
\qbezier(15.2,80)\qbezier(15.2,87.5)\qbezier(22.7,87.5)
\qbezier(30.2,95)\qbezier(30.2,87.5)\qbezier(22.7,87.5)


\put(27.5,45.4){$\up$}
\put(42.5,50.1){$\down$}
\put(12.5,30.4){$\up$}
\put(42.5,35.1){$\down$}
\put(-2.5,15.4){$\up$}
\put(42.5,20.1){$\down$}

\put(-2.5,47.2){$\circ$}
\put(12,48.2){$\cross$}

\put(-2.5,32.2){$\circ$}
\put(27.5,33.2){$\cross$}

\put(12.5,17.2){$\circ$}
\put(27,18.2){$\cross$}
\put(45.2,65){\line(0,-1){68}}
\put(30.2,65){\line(0,-1){15}}
\put(15.2,5){\line(0,-1){8}}
\put(30.2,5){\line(0,-1){8}}
\put(0.2,20){\line(0,-1){23}}
\put(22.7,5){\oval(15,15)[t]}
\put(7.7,65){\oval(15,15)[b]}
\qbezier(0.2,20)\qbezier(0.2,27.5)\qbezier(7.7,27.5)
\qbezier(15.2,35)\qbezier(15.2,27.5)\qbezier(7.7,27.5)
\qbezier(30.2,50)\qbezier(30.2,42.5)\qbezier(22.7,42.5)
\qbezier(15.2,35)\qbezier(15.2,42.5)\qbezier(22.7,42.5)
\end{picture}\\\\
\begin{picture}(45,130)
\put(0,65){\line(1,0){45}}
\put(-2.5,120.4){$\up$}
\put(12.5,120.4){$\up$}
\put(27.5,125.1){$\down$}
\put(42.5,125.1){$\down$}
\put(-2.5,0.4){$\up$}
\put(12.5,0.4){$\up$}
\put(27.5,5.1){$\down$}
\put(42.5,5.1){$\down$}

\put(-2.5,75.4){$\up$}
\put(42.5,80.1){$\down$}
\put(-2.5,90.4){$\up$}
\put(27.5,95.1){$\down$}
\put(-2.5,105.4){$\up$}
\put(42.5,110.1){$\down$}

\put(-2.5,60.4){$\up$}
\put(27.5,60.4){$\up$}
\put(12.5,65.1){$\down$}
\put(42.5,65.1){$\down$}
\put(12.5,77.2){$\circ$}
\put(27,78.2){$\cross$}
\put(12.5,92.2){$\circ$}
\put(42,93.2){$\cross$}
\put(12.5,107.2){$\circ$}
\put(27,108.2){$\cross$}
\put(0.2,65){\line(0,1){68}}
\put(15.2,125){\line(0,1){8}}
\put(30.2,125){\line(0,1){8}}
\put(45.2,110){\line(0,1){23}}
\put(45.2,65){\line(0,1){15}}
\put(22.7,125){\oval(15,15)[b]}
\put(22.7,65){\oval(15,15)[t]}
\qbezier(45.2,110)\qbezier(45.2,102.5)\qbezier(37.7,102.5)
\qbezier(30.2,95)\qbezier(30.2,102.5)\qbezier(37.7,102.5)
\qbezier(45.2,80)\qbezier(45.2,87.5)\qbezier(37.7,87.5)
\qbezier(30.2,95)\qbezier(30.2,87.5)\qbezier(37.7,87.5)


\put(-2.5,45.4){$\up$}
\put(42.5,50.1){$\down$}
\put(12.5,30.4){$\up$}
\put(42.5,35.1){$\down$}
\put(-2.5,15.4){$\up$}
\put(42.5,20.1){$\down$}

\put(12.5,47.2){$\circ$}
\put(27,48.2){$\cross$}
\put(-2.5,32.2){$\circ$}
\put(27,33.2){$\cross$}
\put(12.5,17.2){$\circ$}
\put(27,18.2){$\cross$}
\put(45.2,65){\line(0,-1){68}}
\put(15.2,5){\line(0,-1){8}}
\put(30.2,5){\line(0,-1){8}}
\put(0.2,20){\line(0,-1){23}}
\put(0.2,65){\line(0,-1){15}}
\put(22.7,5){\oval(15,15)[t]}
\put(22.7,65){\oval(15,15)[b]}
\qbezier(0.2,20)\qbezier(0.2,27.5)\qbezier(7.7,27.5)
\qbezier(15.2,35)\qbezier(15.2,27.5)\qbezier(7.7,27.5)
\qbezier(0.2,50)\qbezier(0.2,42.5)\qbezier(7.7,42.5)
\qbezier(15.2,35)\qbezier(15.2,42.5)\qbezier(7.7,42.5)
\end{picture}
&
\begin{picture}(45,130)
\put(0,65){\line(1,0){45}}
\put(-2.5,120.4){$\up$}
\put(12.5,120.4){$\up$}
\put(27.5,125.1){$\down$}
\put(42.5,125.1){$\down$}
\put(-2.5,0.4){$\up$}
\put(12.5,0.4){$\up$}
\put(27.5,5.1){$\down$}
\put(42.5,5.1){$\down$}

\put(-2.5,75.4){$\up$}
\put(42.5,80.1){$\down$}
\put(-2.5,90.4){$\up$}
\put(27.5,95.1){$\down$}
\put(-2.5,105.4){$\up$}
\put(42.5,110.1){$\down$}

\put(-2.5,60.4){$\up$}
\put(12.5,60.4){$\up$}
\put(27.5,65.1){$\down$}
\put(42.5,65.1){$\down$}
\put(12.5,77.2){$\circ$}
\put(27,78.2){$\cross$}
\put(12.5,92.2){$\circ$}
\put(42,93.2){$\cross$}
\put(12.5,107.2){$\circ$}
\put(27,108.2){$\cross$}
\put(0.2,65){\line(0,1){68}}
\put(15.2,125){\line(0,1){8}}
\put(30.2,125){\line(0,1){8}}
\put(45.2,110){\line(0,1){23}}
\put(45.2,65){\line(0,1){15}}
\put(22.7,125){\oval(15,15)[b]}
\put(22.7,65){\oval(15,15)[t]}
\qbezier(45.2,110)\qbezier(45.2,102.5)\qbezier(37.7,102.5)
\qbezier(30.2,95)\qbezier(30.2,102.5)\qbezier(37.7,102.5)
\qbezier(45.2,80)\qbezier(45.2,87.5)\qbezier(37.7,87.5)
\qbezier(30.2,95)\qbezier(30.2,87.5)\qbezier(37.7,87.5)

\put(-2.5,45.4){$\up$}
\put(42.5,50.1){$\down$}
\put(12.5,30.4){$\up$}
\put(42.5,35.1){$\down$}
\put(-2.5,15.4){$\up$}
\put(42.5,20.1){$\down$}

\put(12.5,47.2){$\circ$}
\put(27,48.2){$\cross$}
\put(-2.5,32.2){$\circ$}
\put(27,33.2){$\cross$}
\put(12.5,17.2){$\circ$}
\put(27,18.2){$\cross$}
\put(45.2,65){\line(0,-1){68}}
\put(15.2,5){\line(0,-1){8}}
\put(30.2,5){\line(0,-1){8}}
\put(0.2,20){\line(0,-1){23}}
\put(0.2,65){\line(0,-1){15}}
\put(22.7,5){\oval(15,15)[t]}
\put(22.7,65){\oval(15,15)[b]}
\qbezier(0.2,20)\qbezier(0.2,27.5)\qbezier(7.7,27.5)
\qbezier(15.2,35)\qbezier(15.2,27.5)\qbezier(7.7,27.5)
\qbezier(0.2,50)\qbezier(0.2,42.5)\qbezier(7.7,42.5)
\qbezier(15.2,35)\qbezier(15.2,42.5)\qbezier(7.7,42.5)
\end{picture}
&
\begin{picture}(45,130)
\put(0,65){\line(1,0){45}}
\put(-2.5,120.4){$\up$}
\put(12.5,120.4){$\up$}
\put(27.5,125.1){$\down$}
\put(42.5,125.1){$\down$}
\put(-2.5,0.4){$\up$}
\put(12.5,0.4){$\up$}
\put(27.5,5.1){$\down$}
\put(42.5,5.1){$\down$}

\put(-2.5,75.4){$\up$}
\put(42.5,80.1){$\down$}
\put(12.5,90.4){$\up$}
\put(42.5,95.1){$\down$}
\put(-2.5,105.4){$\up$}
\put(42.5,110.1){$\down$}

\put(-2.5,60.4){$\up$}
\put(27.5,60.4){$\up$}
\put(12.5,65.1){$\down$}
\put(42.5,65.1){$\down$}
\put(12.5,77.2){$\circ$}
\put(27,78.2){$\cross$}
\put(-2.5,92.2){$\circ$}
\put(27,93.2){$\cross$}
\put(12.5,107.2){$\circ$}
\put(27,108.2){$\cross$}
\put(45.2,65){\line(0,1){68}}
\put(15.2,125){\line(0,1){8}}
\put(30.2,125){\line(0,1){8}}
\put(0.2,110){\line(0,1){23}}
\put(0.2,65){\line(0,1){15}}
\put(22.7,125){\oval(15,15)[b]}
\put(22.7,65){\oval(15,15)[t]}
\qbezier(0.2,110)\qbezier(0.2,102.5)\qbezier(7.7,102.5)
\qbezier(15.2,95)\qbezier(15.2,102.5)\qbezier(7.7,102.5)
\qbezier(0.2,80)\qbezier(0.2,87.5)\qbezier(7.7,87.5)
\qbezier(15.2,95)\qbezier(15.2,87.5)\qbezier(7.7,87.5)


\put(-2.5,45.4){$\up$}
\put(12.5,50.1){$\down$}
\put(-2.5,30.4){$\up$}
\put(27.5,35.1){$\down$}
\put(-2.5,15.4){$\up$}
\put(42.5,20.1){$\down$}

\put(27.5,47.2){$\circ$}
\put(42,48.2){$\cross$}
\put(12.5,32.2){$\circ$}
\put(42,33.2){$\cross$}
\put(12.5,17.2){$\circ$}
\put(27,18.2){$\cross$}
\put(0.2,65){\line(0,-1){68}}
\put(15.2,5){\line(0,-1){8}}
\put(30.2,5){\line(0,-1){8}}
\put(45.2,20){\line(0,-1){23}}
\put(15.2,65){\line(0,-1){15}}
\put(22.7,5){\oval(15,15)[t]}
\put(37.7,65){\oval(15,15)[b]}
\qbezier(45.2,20)\qbezier(45.2,27.5)\qbezier(37.7,27.5)
\qbezier(30.2,35)\qbezier(30.2,27.5)\qbezier(37.7,27.5)
\qbezier(15.2,50)\qbezier(15.2,42.5)\qbezier(22.7,42.5)
\qbezier(30.2,35)\qbezier(30.2,42.5)\qbezier(22.7,42.5)
\end{picture}
&
\begin{picture}(45,130)
\put(0,65){\line(1,0){45}}
\put(-2.5,120.4){$\up$}
\put(12.5,120.4){$\up$}
\put(27.5,125.1){$\down$}
\put(42.5,125.1){$\down$}
\put(-2.5,0.4){$\up$}
\put(12.5,0.4){$\up$}
\put(27.5,5.1){$\down$}
\put(42.5,5.1){$\down$}


\put(27.5,45.4){$\up$}
\put(42.5,50.1){$\down$}
\put(12.5,30.4){$\up$}
\put(42.5,35.1){$\down$}
\put(-2.5,15.4){$\up$}
\put(42.5,20.1){$\down$}

\put(-2.5,47.2){$\circ$}
\put(12,48.2){$\cross$}

\put(-2.5,32.2){$\circ$}
\put(27.5,33.2){$\cross$}

\put(12.5,17.2){$\circ$}
\put(27,18.2){$\cross$}
\put(45.2,65){\line(0,-1){68}}
\put(30.2,65){\line(0,-1){15}}
\put(15.2,5){\line(0,-1){8}}
\put(30.2,5){\line(0,-1){8}}
\put(0.2,20){\line(0,-1){23}}
\put(22.7,5){\oval(15,15)[t]}
\put(7.7,65){\oval(15,15)[b]}
\qbezier(0.2,20)\qbezier(0.2,27.5)\qbezier(7.7,27.5)
\qbezier(15.2,35)\qbezier(15.2,27.5)\qbezier(7.7,27.5)
\qbezier(30.2,50)\qbezier(30.2,42.5)\qbezier(22.7,42.5)
\qbezier(15.2,35)\qbezier(15.2,42.5)\qbezier(22.7,42.5)

\put(-2.5,75.4){$\up$}
\put(42.5,80.1){$\down$}
\put(12.5,90.4){$\up$}
\put(42.5,95.1){$\down$}
\put(-2.5,105.4){$\up$}
\put(42.5,110.1){$\down$}

\put(-2.5,60.4){$\up$}
\put(27.5,60.4){$\up$}
\put(12.5,65.1){$\down$}
\put(42.5,65.1){$\down$}
\put(12.5,77.2){$\circ$}
\put(27,78.2){$\cross$}
\put(-2.5,92.2){$\circ$}
\put(27,93.2){$\cross$}
\put(12.5,107.2){$\circ$}
\put(27,108.2){$\cross$}
\put(45.2,65){\line(0,1){68}}
\put(15.2,125){\line(0,1){8}}
\put(30.2,125){\line(0,1){8}}
\put(0.2,110){\line(0,1){23}}
\put(0.2,65){\line(0,1){15}}
\put(22.7,125){\oval(15,15)[b]}
\put(22.7,65){\oval(15,15)[t]}
\qbezier(0.2,110)\qbezier(0.2,102.5)\qbezier(7.7,102.5)
\qbezier(15.2,95)\qbezier(15.2,102.5)\qbezier(7.7,102.5)
\qbezier(0.2,80)\qbezier(0.2,87.5)\qbezier(7.7,87.5)
\qbezier(15.2,95)\qbezier(15.2,87.5)\qbezier(7.7,87.5)

\end{picture}
&
\begin{picture}(45,130)
\put(0,65){\line(1,0){45}}
\put(-2.5,120.4){$\up$}
\put(12.5,120.4){$\up$}
\put(27.5,125.1){$\down$}
\put(42.5,125.1){$\down$}
\put(-2.5,0.4){$\up$}
\put(12.5,0.4){$\up$}
\put(27.5,5.1){$\down$}
\put(42.5,5.1){$\down$}

\put(-2.5,75.4){$\up$}
\put(42.5,80.1){$\down$}
\put(-2.5,90.4){$\up$}
\put(27.5,95.1){$\down$}
\put(-2.5,105.4){$\up$}
\put(42.5,110.1){$\down$}

\put(-2.5,60.4){$\up$}
\put(27.5,60.4){$\up$}
\put(12.5,65.1){$\down$}
\put(42.5,65.1){$\down$}
\put(12.5,77.2){$\circ$}
\put(27,78.2){$\cross$}
\put(12.5,92.2){$\circ$}
\put(42,93.2){$\cross$}
\put(12.5,107.2){$\circ$}
\put(27,108.2){$\cross$}
\put(0.2,65){\line(0,1){68}}
\put(15.2,125){\line(0,1){8}}
\put(30.2,125){\line(0,1){8}}
\put(45.2,110){\line(0,1){23}}
\put(45.2,65){\line(0,1){15}}
\put(22.7,125){\oval(15,15)[b]}
\put(22.7,65){\oval(15,15)[t]}
\qbezier(45.2,110)\qbezier(45.2,102.5)\qbezier(37.7,102.5)
\qbezier(30.2,95)\qbezier(30.2,102.5)\qbezier(37.7,102.5)
\qbezier(45.2,80)\qbezier(45.2,87.5)\qbezier(37.7,87.5)
\qbezier(30.2,95)\qbezier(30.2,87.5)\qbezier(37.7,87.5)


\put(-2.5,45.4){$\up$}
\put(12.5,50.1){$\down$}
\put(-2.5,30.4){$\up$}
\put(27.5,35.1){$\down$}
\put(-2.5,15.4){$\up$}
\put(42.5,20.1){$\down$}

\put(27.5,47.2){$\circ$}
\put(42,48.2){$\cross$}
\put(12.5,32.2){$\circ$}
\put(42,33.2){$\cross$}
\put(12.5,17.2){$\circ$}
\put(27,18.2){$\cross$}
\put(0.2,65){\line(0,-1){68}}
\put(15.2,5){\line(0,-1){8}}
\put(30.2,5){\line(0,-1){8}}
\put(45.2,20){\line(0,-1){23}}
\put(15.2,65){\line(0,-1){15}}
\put(22.7,5){\oval(15,15)[t]}
\put(37.7,65){\oval(15,15)[b]}
\qbezier(45.2,20)\qbezier(45.2,27.5)\qbezier(37.7,27.5)
\qbezier(30.2,35)\qbezier(30.2,27.5)\qbezier(37.7,27.5)
\qbezier(15.2,50)\qbezier(15.2,42.5)\qbezier(22.7,42.5)
\qbezier(30.2,35)\qbezier(30.2,42.5)\qbezier(22.7,42.5)
\end{picture}
&
\begin{picture}(45,130)
\put(0,65){\line(1,0){45}}
\put(-2.5,120.4){$\up$}
\put(12.5,120.4){$\up$}
\put(27.5,125.1){$\down$}
\put(42.5,125.1){$\down$}
\put(-2.5,0.4){$\up$}
\put(12.5,0.4){$\up$}
\put(27.5,5.1){$\down$}
\put(42.5,5.1){$\down$}


\put(27.5,45.4){$\up$}
\put(42.5,50.1){$\down$}
\put(12.5,30.4){$\up$}
\put(42.5,35.1){$\down$}
\put(-2.5,15.4){$\up$}
\put(42.5,20.1){$\down$}

\put(-2.5,47.2){$\circ$}
\put(12,48.2){$\cross$}

\put(-2.5,32.2){$\circ$}
\put(27.5,33.2){$\cross$}

\put(12.5,17.2){$\circ$}
\put(27,18.2){$\cross$}
\put(45.2,65){\line(0,-1){68}}
\put(30.2,65){\line(0,-1){15}}
\put(15.2,5){\line(0,-1){8}}
\put(30.2,5){\line(0,-1){8}}
\put(0.2,20){\line(0,-1){23}}
\put(22.7,5){\oval(15,15)[t]}
\put(7.7,65){\oval(15,15)[b]}
\qbezier(0.2,20)\qbezier(0.2,27.5)\qbezier(7.7,27.5)
\qbezier(15.2,35)\qbezier(15.2,27.5)\qbezier(7.7,27.5)
\qbezier(30.2,50)\qbezier(30.2,42.5)\qbezier(22.7,42.5)
\qbezier(15.2,35)\qbezier(15.2,42.5)\qbezier(22.7,42.5)


\put(-2.5,75.4){$\up$}
\put(42.5,80.1){$\down$}
\put(-2.5,90.4){$\up$}
\put(27.5,95.1){$\down$}
\put(-2.5,105.4){$\up$}
\put(42.5,110.1){$\down$}

\put(-2.5,60.4){$\up$}
\put(27.5,60.4){$\up$}
\put(12.5,65.1){$\down$}
\put(42.5,65.1){$\down$}
\put(12.5,77.2){$\circ$}
\put(27,78.2){$\cross$}
\put(12.5,92.2){$\circ$}
\put(42,93.2){$\cross$}
\put(12.5,107.2){$\circ$}
\put(27,108.2){$\cross$}
\put(0.2,65){\line(0,1){68}}
\put(15.2,125){\line(0,1){8}}
\put(30.2,125){\line(0,1){8}}
\put(45.2,110){\line(0,1){23}}
\put(45.2,65){\line(0,1){15}}
\put(22.7,125){\oval(15,15)[b]}
\put(22.7,65){\oval(15,15)[t]}
\qbezier(45.2,110)\qbezier(45.2,102.5)\qbezier(37.7,102.5)
\qbezier(30.2,95)\qbezier(30.2,102.5)\qbezier(37.7,102.5)
\qbezier(45.2,80)\qbezier(45.2,87.5)\qbezier(37.7,87.5)
\qbezier(30.2,95)\qbezier(30.2,87.5)\qbezier(37.7,87.5)
\end{picture}\\\\
\begin{picture}(45,130)
\put(0,65){\line(1,0){45}}
\put(-2.5,120.4){$\up$}
\put(12.5,120.4){$\up$}
\put(27.5,125.1){$\down$}
\put(42.5,125.1){$\down$}
\put(-2.5,0.4){$\up$}
\put(12.5,0.4){$\up$}
\put(27.5,5.1){$\down$}
\put(42.5,5.1){$\down$}

\put(-2.5,75.4){$\up$}
\put(42.5,80.1){$\down$}
\put(12.5,90.4){$\up$}
\put(42.5,95.1){$\down$}
\put(-2.5,105.4){$\up$}
\put(42.5,110.1){$\down$}

\put(-2.5,60.4){$\up$}
\put(27.5,60.4){$\up$}
\put(12.5,65.1){$\down$}
\put(42.5,65.1){$\down$}
\put(12.5,77.2){$\circ$}
\put(27,78.2){$\cross$}
\put(-2.5,92.2){$\circ$}
\put(27,93.2){$\cross$}
\put(12.5,107.2){$\circ$}
\put(27,108.2){$\cross$}
\put(45.2,65){\line(0,1){68}}
\put(15.2,125){\line(0,1){8}}
\put(30.2,125){\line(0,1){8}}
\put(0.2,110){\line(0,1){23}}
\put(0.2,65){\line(0,1){15}}
\put(22.7,125){\oval(15,15)[b]}
\put(22.7,65){\oval(15,15)[t]}
\qbezier(0.2,110)\qbezier(0.2,102.5)\qbezier(7.7,102.5)
\qbezier(15.2,95)\qbezier(15.2,102.5)\qbezier(7.7,102.5)
\qbezier(0.2,80)\qbezier(0.2,87.5)\qbezier(7.7,87.5)
\qbezier(15.2,95)\qbezier(15.2,87.5)\qbezier(7.7,87.5)


\put(-2.5,45.4){$\up$}
\put(42.5,50.1){$\down$}
\put(-2.5,30.4){$\up$}
\put(27.5,35.1){$\down$}
\put(-2.5,15.4){$\up$}
\put(42.5,20.1){$\down$}

\put(12.5,47.2){$\circ$}
\put(27,48.2){$\cross$}
\put(12.5,32.2){$\circ$}
\put(42,33.2){$\cross$}
\put(12.5,17.2){$\circ$}
\put(27,18.2){$\cross$}
\put(0.2,65){\line(0,-1){68}}
\put(15.2,5){\line(0,-1){8}}
\put(30.2,5){\line(0,-1){8}}
\put(45.2,20){\line(0,-1){23}}
\put(45.2,65){\line(0,-1){15}}
\put(22.7,5){\oval(15,15)[t]}
\put(22.7,65){\oval(15,15)[b]}
\qbezier(45.2,20)\qbezier(45.2,27.5)\qbezier(37.7,27.5)
\qbezier(30.2,35)\qbezier(30.2,27.5)\qbezier(37.7,27.5)
\qbezier(45.2,50)\qbezier(45.2,42.5)\qbezier(37.7,42.5)
\qbezier(30.2,35)\qbezier(30.2,42.5)\qbezier(37.7,42.5)
\end{picture}
&
\begin{picture}(45,130)
\put(0,65){\line(1,0){45}}
\put(-2.5,120.4){$\up$}
\put(12.5,120.4){$\up$}
\put(27.5,125.1){$\down$}
\put(42.5,125.1){$\down$}
\put(-2.5,0.4){$\up$}
\put(12.5,0.4){$\up$}
\put(27.5,5.1){$\down$}
\put(42.5,5.1){$\down$}

\put(-2.5,75.4){$\up$}
\put(42.5,80.1){$\down$}
\put(12.5,90.4){$\up$}
\put(42.5,95.1){$\down$}
\put(-2.5,105.4){$\up$}
\put(42.5,110.1){$\down$}

\put(-2.5,60.4){$\up$}
\put(12.5,60.4){$\up$}
\put(27.5,65.1){$\down$}
\put(42.5,65.1){$\down$}
\put(12.5,77.2){$\circ$}
\put(27,78.2){$\cross$}
\put(-2.5,92.2){$\circ$}
\put(27,93.2){$\cross$}
\put(12.5,107.2){$\circ$}
\put(27,108.2){$\cross$}
\put(45.2,65){\line(0,1){68}}
\put(15.2,125){\line(0,1){8}}
\put(30.2,125){\line(0,1){8}}
\put(0.2,110){\line(0,1){23}}
\put(0.2,65){\line(0,1){15}}
\put(22.7,125){\oval(15,15)[b]}
\put(22.7,65){\oval(15,15)[t]}
\qbezier(0.2,110)\qbezier(0.2,102.5)\qbezier(7.7,102.5)
\qbezier(15.2,95)\qbezier(15.2,102.5)\qbezier(7.7,102.5)
\qbezier(0.2,80)\qbezier(0.2,87.5)\qbezier(7.7,87.5)
\qbezier(15.2,95)\qbezier(15.2,87.5)\qbezier(7.7,87.5)


\put(-2.5,45.4){$\up$}
\put(42.5,50.1){$\down$}
\put(-2.5,30.4){$\up$}
\put(27.5,35.1){$\down$}
\put(-2.5,15.4){$\up$}
\put(42.5,20.1){$\down$}

\put(12.5,47.2){$\circ$}
\put(27,48.2){$\cross$}
\put(12.5,32.2){$\circ$}
\put(42,33.2){$\cross$}
\put(12.5,17.2){$\circ$}
\put(27,18.2){$\cross$}
\put(0.2,65){\line(0,-1){68}}
\put(15.2,5){\line(0,-1){8}}
\put(30.2,5){\line(0,-1){8}}
\put(45.2,20){\line(0,-1){23}}
\put(45.2,65){\line(0,-1){15}}
\put(22.7,5){\oval(15,15)[t]}
\put(22.7,65){\oval(15,15)[b]}
\qbezier(45.2,20)\qbezier(45.2,27.5)\qbezier(37.7,27.5)
\qbezier(30.2,35)\qbezier(30.2,27.5)\qbezier(37.7,27.5)
\qbezier(45.2,50)\qbezier(45.2,42.5)\qbezier(37.7,42.5)
\qbezier(30.2,35)\qbezier(30.2,42.5)\qbezier(37.7,42.5)
\end{picture}
&
\begin{picture}(45,130)
\put(0,65){\line(1,0){45}}
\put(-2.5,120.4){$\up$}
\put(12.5,120.4){$\up$}
\put(27.5,125.1){$\down$}
\put(42.5,125.1){$\down$}
\put(-2.5,0.4){$\up$}
\put(12.5,0.4){$\up$}
\put(27.5,5.1){$\down$}
\put(42.5,5.1){$\down$}


\put(-2.5,75.4){$\up$}
\put(12.5,80.1){$\down$}
\put(-2.5,90.4){$\up$}
\put(27.5,95.1){$\down$}
\put(-2.5,105.4){$\up$}
\put(42.5,110.1){$\down$}

\put(-2.5,60.4){$\up$}
\put(12.5,65.1){$\down$}
\put(42.5,65.1){$\down$}
\put(27.5,60.4){$\up$}
\put(27.5,77.2){$\circ$}
\put(42,78.2){$\cross$}
\put(12.5,92.2){$\circ$}
\put(42,93.2){$\cross$}
\put(12.5,107.2){$\circ$}
\put(27,108.2){$\cross$}
\put(0.2,65){\line(0,1){68}}
\put(15.2,125){\line(0,1){8}}
\put(30.2,125){\line(0,1){8}}
\put(45.2,110){\line(0,1){23}}
\put(15.2,65){\line(0,1){15}}
\put(22.7,125){\oval(15,15)[b]}
\put(37.7,65){\oval(15,15)[t]}
\qbezier(45.2,110)\qbezier(45.2,102.5)\qbezier(37.7,102.5)
\qbezier(30.2,95)\qbezier(30.2,102.5)\qbezier(37.7,102.5)
\qbezier(15.2,80)\qbezier(15.2,87.5)\qbezier(22.7,87.5)
\qbezier(30.2,95)\qbezier(30.2,87.5)\qbezier(22.7,87.5)

\put(-2.5,45.4){$\up$}
\put(42.5,50.1){$\down$}
\put(12.5,30.4){$\up$}
\put(42.5,35.1){$\down$}
\put(-2.5,15.4){$\up$}
\put(42.5,20.1){$\down$}

\put(12.5,47.2){$\circ$}
\put(27,48.2){$\cross$}
\put(-2.5,32.2){$\circ$}
\put(27,33.2){$\cross$}
\put(12.5,17.2){$\circ$}
\put(27,18.2){$\cross$}
\put(45.2,65){\line(0,-1){68}}
\put(15.2,5){\line(0,-1){8}}
\put(30.2,5){\line(0,-1){8}}
\put(0.2,20){\line(0,-1){23}}
\put(0.2,65){\line(0,-1){15}}
\put(22.7,5){\oval(15,15)[t]}
\put(22.7,65){\oval(15,15)[b]}
\qbezier(0.2,20)\qbezier(0.2,27.5)\qbezier(7.7,27.5)
\qbezier(15.2,35)\qbezier(15.2,27.5)\qbezier(7.7,27.5)
\qbezier(0.2,50)\qbezier(0.2,42.5)\qbezier(7.7,42.5)
\qbezier(15.2,35)\qbezier(15.2,42.5)\qbezier(7.7,42.5)

\end{picture}
&
\begin{picture}(45,130)
\put(0,65){\line(1,0){45}}
\put(-2.5,120.4){$\up$}
\put(12.5,120.4){$\up$}
\put(27.5,125.1){$\down$}
\put(42.5,125.1){$\down$}
\put(-2.5,0.4){$\up$}
\put(12.5,0.4){$\up$}
\put(27.5,5.1){$\down$}
\put(42.5,5.1){$\down$}

\put(27.5,75.4){$\up$}
\put(42.5,80.1){$\down$}
\put(12.5,90.4){$\up$}
\put(42.5,95.1){$\down$}
\put(-2.5,105.4){$\up$}
\put(42.5,110.1){$\down$}

\put(12.5,65.1){$\down$}
\put(-2.5,60.4){$\up$}
\put(27.5,60.4){$\up$}
\put(42.5,65.1){$\down$}

\put(-2.5,77.2){$\circ$}
\put(12,78.2){$\cross$}

\put(-2.5,92.2){$\circ$}
\put(27.5,93.2){$\cross$}

\put(12.5,107.2){$\circ$}
\put(27,108.2){$\cross$}
\put(45.2,65){\line(0,1){68}}
\put(30.2,65){\line(0,1){15}}
\put(15.2,125){\line(0,1){8}}
\put(30.2,125){\line(0,1){8}}
\put(0.2,110){\line(0,1){23}}
\put(22.7,125){\oval(15,15)[b]}
\put(7.7,65){\oval(15,15)[t]}
\qbezier(0.2,110)\qbezier(0.2,102.5)\qbezier(7.7,102.5)
\qbezier(15.2,95)\qbezier(15.2,102.5)\qbezier(7.7,102.5)
\qbezier(30.2,80)\qbezier(30.2,87.5)\qbezier(22.7,87.5)
\qbezier(15.2,95)\qbezier(15.2,87.5)\qbezier(22.7,87.5)


\put(-2.5,45.4){$\up$}
\put(42.5,50.1){$\down$}
\put(12.5,30.4){$\up$}
\put(42.5,35.1){$\down$}
\put(-2.5,15.4){$\up$}
\put(42.5,20.1){$\down$}

\put(12.5,47.2){$\circ$}
\put(27,48.2){$\cross$}
\put(-2.5,32.2){$\circ$}
\put(27,33.2){$\cross$}
\put(12.5,17.2){$\circ$}
\put(27,18.2){$\cross$}
\put(45.2,65){\line(0,-1){68}}
\put(15.2,5){\line(0,-1){8}}
\put(30.2,5){\line(0,-1){8}}
\put(0.2,20){\line(0,-1){23}}
\put(0.2,65){\line(0,-1){15}}
\put(22.7,5){\oval(15,15)[t]}
\put(22.7,65){\oval(15,15)[b]}
\qbezier(0.2,20)\qbezier(0.2,27.5)\qbezier(7.7,27.5)
\qbezier(15.2,35)\qbezier(15.2,27.5)\qbezier(7.7,27.5)
\qbezier(0.2,50)\qbezier(0.2,42.5)\qbezier(7.7,42.5)
\qbezier(15.2,35)\qbezier(15.2,42.5)\qbezier(7.7,42.5)

\end{picture}
&
\begin{picture}(45,130)
\put(0,65){\line(1,0){45}}
\put(-2.5,120.4){$\up$}
\put(12.5,120.4){$\up$}
\put(27.5,125.1){$\down$}
\put(42.5,125.1){$\down$}
\put(-2.5,0.4){$\up$}
\put(12.5,0.4){$\up$}
\put(27.5,5.1){$\down$}
\put(42.5,5.1){$\down$}

\put(-2.5,75.4){$\up$}
\put(12.5,80.1){$\down$}
\put(-2.5,90.4){$\up$}
\put(27.5,95.1){$\down$}
\put(-2.5,105.4){$\up$}
\put(42.5,110.1){$\down$}

\put(-2.5,60.4){$\up$}
\put(12.5,65.1){$\down$}
\put(42.5,65.1){$\down$}
\put(27.5,60.4){$\up$}
\put(27.5,77.2){$\circ$}
\put(42,78.2){$\cross$}
\put(12.5,92.2){$\circ$}
\put(42,93.2){$\cross$}
\put(12.5,107.2){$\circ$}
\put(27,108.2){$\cross$}
\put(0.2,65){\line(0,1){68}}
\put(15.2,125){\line(0,1){8}}
\put(30.2,125){\line(0,1){8}}
\put(45.2,110){\line(0,1){23}}
\put(15.2,65){\line(0,1){15}}
\put(22.7,125){\oval(15,15)[b]}
\put(37.7,65){\oval(15,15)[t]}
\qbezier(45.2,110)\qbezier(45.2,102.5)\qbezier(37.7,102.5)
\qbezier(30.2,95)\qbezier(30.2,102.5)\qbezier(37.7,102.5)
\qbezier(15.2,80)\qbezier(15.2,87.5)\qbezier(22.7,87.5)
\qbezier(30.2,95)\qbezier(30.2,87.5)\qbezier(22.7,87.5)

\put(-2.5,45.4){$\up$}
\put(42.5,50.1){$\down$}
\put(-2.5,30.4){$\up$}
\put(27.5,35.1){$\down$}
\put(-2.5,15.4){$\up$}
\put(42.5,20.1){$\down$}

\put(12.5,47.2){$\circ$}
\put(27,48.2){$\cross$}
\put(12.5,32.2){$\circ$}
\put(42,33.2){$\cross$}
\put(12.5,17.2){$\circ$}
\put(27,18.2){$\cross$}
\put(0.2,65){\line(0,-1){68}}
\put(15.2,5){\line(0,-1){8}}
\put(30.2,5){\line(0,-1){8}}
\put(45.2,20){\line(0,-1){23}}
\put(45.2,65){\line(0,-1){15}}
\put(22.7,5){\oval(15,15)[t]}
\put(22.7,65){\oval(15,15)[b]}
\qbezier(45.2,20)\qbezier(45.2,27.5)\qbezier(37.7,27.5)
\qbezier(30.2,35)\qbezier(30.2,27.5)\qbezier(37.7,27.5)
\qbezier(45.2,50)\qbezier(45.2,42.5)\qbezier(37.7,42.5)
\qbezier(30.2,35)\qbezier(30.2,42.5)\qbezier(37.7,42.5)

\end{picture}
&
\begin{picture}(45,130)
\put(0,65){\line(1,0){45}}
\put(-2.5,120.4){$\up$}
\put(12.5,120.4){$\up$}
\put(27.5,125.1){$\down$}
\put(42.5,125.1){$\down$}
\put(-2.5,0.4){$\up$}
\put(12.5,0.4){$\up$}
\put(27.5,5.1){$\down$}
\put(42.5,5.1){$\down$}

\put(-2.5,45.4){$\up$}
\put(42.5,50.1){$\down$}
\put(-2.5,30.4){$\up$}
\put(27.5,35.1){$\down$}
\put(-2.5,15.4){$\up$}
\put(42.5,20.1){$\down$}

\put(12.5,47.2){$\circ$}
\put(27,48.2){$\cross$}
\put(12.5,32.2){$\circ$}
\put(42,33.2){$\cross$}
\put(12.5,17.2){$\circ$}
\put(27,18.2){$\cross$}
\put(0.2,65){\line(0,-1){68}}
\put(15.2,5){\line(0,-1){8}}
\put(30.2,5){\line(0,-1){8}}
\put(45.2,20){\line(0,-1){23}}
\put(45.2,65){\line(0,-1){15}}
\put(22.7,5){\oval(15,15)[t]}
\put(22.7,65){\oval(15,15)[b]}
\qbezier(45.2,20)\qbezier(45.2,27.5)\qbezier(37.7,27.5)
\qbezier(30.2,35)\qbezier(30.2,27.5)\qbezier(37.7,27.5)
\qbezier(45.2,50)\qbezier(45.2,42.5)\qbezier(37.7,42.5)
\qbezier(30.2,35)\qbezier(30.2,42.5)\qbezier(37.7,42.5)


\put(27.5,75.4){$\up$}
\put(42.5,80.1){$\down$}
\put(12.5,90.4){$\up$}
\put(42.5,95.1){$\down$}
\put(-2.5,105.4){$\up$}
\put(42.5,110.1){$\down$}

\put(12.5,65.1){$\down$}
\put(-2.5,60.4){$\up$}
\put(27.5,60.4){$\up$}
\put(42.5,65.1){$\down$}

\put(-2.5,77.2){$\circ$}
\put(12,78.2){$\cross$}

\put(-2.5,92.2){$\circ$}
\put(27.5,93.2){$\cross$}

\put(12.5,107.2){$\circ$}
\put(27,108.2){$\cross$}
\put(45.2,65){\line(0,1){68}}
\put(30.2,65){\line(0,1){15}}
\put(15.2,125){\line(0,1){8}}
\put(30.2,125){\line(0,1){8}}
\put(0.2,110){\line(0,1){23}}
\put(22.7,125){\oval(15,15)[b]}
\put(7.7,65){\oval(15,15)[t]}
\qbezier(0.2,110)\qbezier(0.2,102.5)\qbezier(7.7,102.5)
\qbezier(15.2,95)\qbezier(15.2,102.5)\qbezier(7.7,102.5)
\qbezier(30.2,80)\qbezier(30.2,87.5)\qbezier(22.7,87.5)
\qbezier(15.2,95)\qbezier(15.2,87.5)\qbezier(22.7,87.5)

\end{picture}\end{array}
$$
\caption{The diagram basis for $B_{E^2 F^2}(0)$}\label{fig1}
\end{figure}
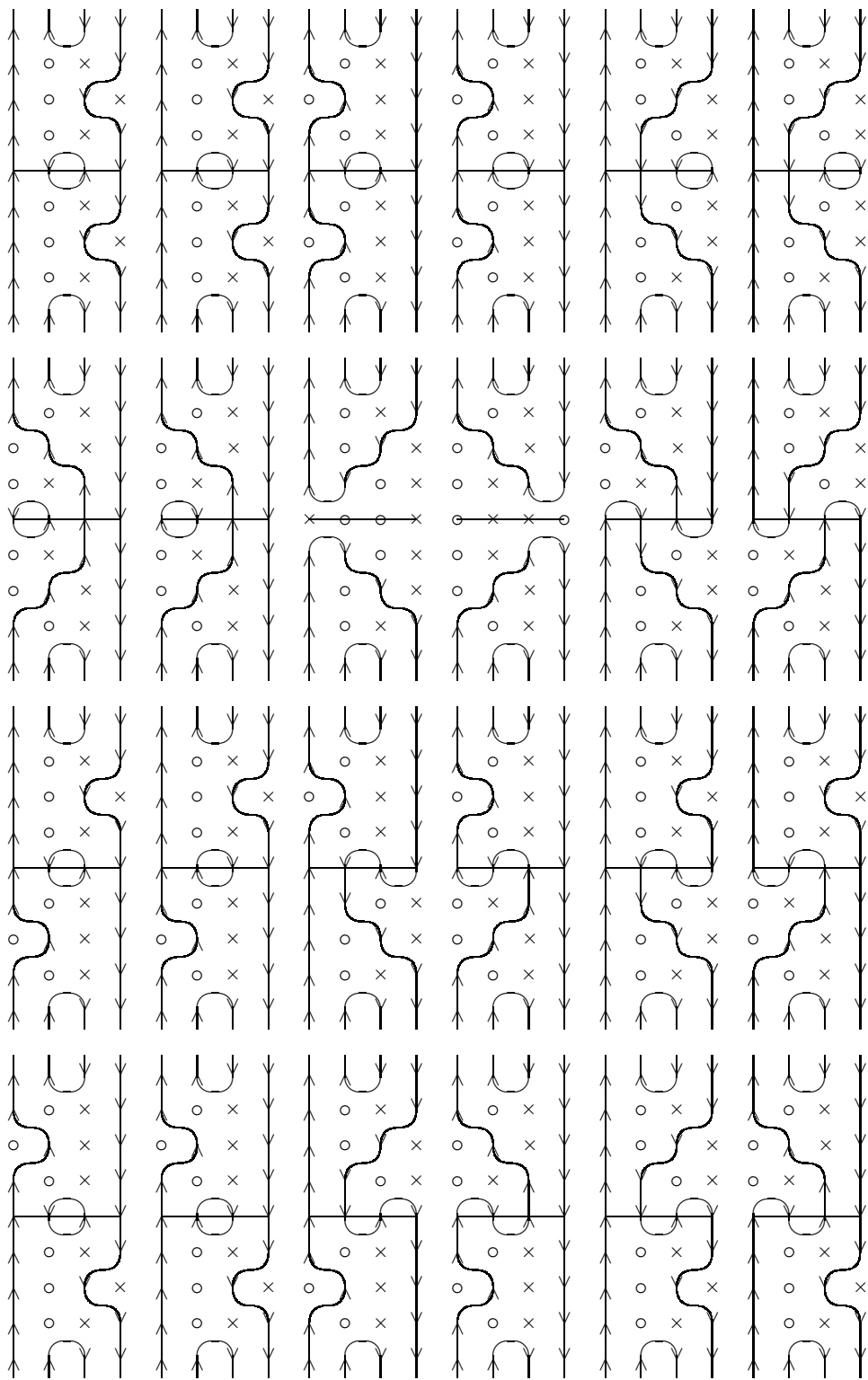

\phantomsubsection{Graded Morita equivalence}
In this subsection we prove that 
$B_R(\delta)$ is graded Morita equivalent to
$K_{r,s}(\delta)$.
Let $\eta$ be as in (\ref{othergroundstate})
and recall that $B_R(\delta)$ is the endomorphism algebra of the
$K(\delta)$-module $\lonestar R\,P(\eta)$,
while $K_{r,s}(\delta) = e_{r,s} K(\delta) e_{r,s}$
according to (\ref{krs}).

\begin{Lemma}\label{pims}
Every indecomposable summand of $\lonestar
R\,P(\eta)$ is isomorphic (up to grading shift) to one of 
the projective indecomposable modules
$\{P(\la)\:|\:\la \in \dot\La_{r,s}\}$.
Moreover every such $P(\la)$ appears
as a summand.
\end{Lemma}

\begin{proof}
Proceed by induction on $r+s$, the case $r+s=0$ being trivial.
For the induction step, suppose we have shown already
that the indecomposable summands
of $\lonestar R\,P(\eta)$ 
are the modules
$\{P(\la)\:|\:\la\in\dot\La_{r,s}\}$. We need to show that the
indecomposable summands of
$\lonestar E \lonestar R\,P(\eta)$ 
and $\lonestar F \lonestar R\,P(\eta)$ are the modules
$\{P(\la)\:|\:\la \in \dot \La_{r+!,s}\}$
and
$\{P(\la)\:|\:\la \in \dot \La_{r,s+1}\}$, respectively.
This follows from Lemma~\ref{genee}.
\end{proof}

Now we can construct the Morita equivalence. Set
\begin{equation}\label{aragorn}
P_R(\delta) := e_{r,s} \lonestar R\,P(\eta),
\end{equation}
which is a graded $(K_{r,s}(\delta), B_R(\delta))$-bimodule 
spanned by the vectors
$(\underline{\nu} \overline{\tU})$
from (\ref{spb}) with $\tU \in \mathscr T_R(\delta)$ and
$\nu \in \dot\La_{r,s}$ such that
$\nu \subset \sh(\tU)$.
Recall also the functors (\ref{never1})--(\ref{never4})
and (\ref{bever1})--(\ref{bever4}).

\begin{Theorem}\label{morita}
The functor
\begin{equation}\label{life}
\F_R := \hom_{K_{r,s}(\delta)}(P_{R}(\delta), ?):
\Mod{K_{r,s}(\delta)} \rightarrow \Mod{B_R(\delta)}
\end{equation}
is an equivalence of categories commuting with $i$-restriction and
$i$-induction, i.e. there are isomorphisms of functors 
\begin{align*}
\F_{R} \circ \ires^{r+1,s}_{r,s} 
\cong \ires^{RE}_R \circ \F_{RE}&:\Mod{K_{r+1,s}(\delta)} \rightarrow \Mod{B_R(\delta)},\\
\F_{R} \circ \ires^{r,s+1}_{r,s} 
\cong \ires^{RF}_R \circ \F_{RF}&:\Mod{K_{r,s+1}(\delta)} \rightarrow \Mod{B_R(\delta)},\\
\F_{RE} \circ \iind^{r+1,s}_{r,s}
\cong \iind^{RE}_{R} \circ \F_R&:\Mod{K_{r,s}(\delta)} \rightarrow \Mod{B_{RE}(\delta)},\\
\F_{RF} \circ \iind^{r,s+1}_{r,s} 
\cong \iind^{RF}_{R} \circ \F_R&:\Mod{K_{r,s}(\delta)} \rightarrow \Mod{B_{RF}(\delta)}.
\end{align*}
\end{Theorem}

\begin{proof}
By Lemma~\ref{pims} and Theorem~\ref{iscell}(3),
$P_R(\delta) = e_{r,s} \lonestar R\,P(\eta)$ is a graded projective generator
for the category $\Mod{K_{r,s}(\delta)}$.
Moreover the endomorphism algebra
$\End_{K_{r,s}(\delta)}(P_R(\delta))^{\op}$
is identified with $B_R(\delta) = \End_{K(\delta)}(\lonestar R\,P(\eta))^{\op}$ 
by Lemma~\ref{stupidone} (and
Lemma~\ref{pims} again).
Hence $\F_{R}$ is an equivalence of categories by general principles.

For the four isomorphisms of functors, it is enough to
prove the ones involving restriction, since the ones for induction
then follow by unicity of adjoints.
We just explain the argument for the first isomorphism of functors,
since the second is similar.
For this, we use the following chain
of graded $(B_R(\delta), K_{r+1,s}(\delta))$-bimodule isomorphisms.
\begin{align*}
\F_R (\ires^{r+1,s}_{r,s} K_{r+1,s}(\delta))
&\cong\hom_{K_{r,s}(\delta)}(e_{r,s} \lonestar R\,P(\eta),
e_{r,s} \widetilde{E}_i e_{r+1,s})\\
&\cong
\hom_{K(\delta)}(\lonestar R\,P(\eta), \widetilde{E}_i e_{r+1,s})\\
&\cong
\hom_{K(\delta)}(\lonestar R\,P(\eta), E_i K(\delta) e_{r+1,s})\\
&\cong
\hom_{K(\delta)}(\lonestar E_i\lonestar R\,P(\eta), K(\delta) e_{r+1,s})\\
&\cong
\ires^{RE}_R
(\hom_{K(\delta)}(\lonestar E\lonestar R\,P(\eta), K(\delta) e_{r+1,s}))\\
&\cong
\ires^{RE}_R
(\hom_{K_{r+1,s}(\delta)}(e_{r+1,s}\lonestar (RE)\,P(\eta),
K_{r+1,s}(\delta)))\\
&= \ires^{RE}_R (\F_{RE} K_{r+1,s}(\delta)).
\end{align*}
The second and the penultimate isomorphisms here follow from
Lemmas~\ref{stupidone}
and \ref{pims}.
\end{proof}

In the rest of the subsection, we give an
alternative 
description of the Morita equivalence $\F_{R}$
from Theorem~\ref{morita} which is sometimes easier for explicit
computations.
As well as the left $K(\delta)$-module $\lonestar{R}\,P(\eta) =
\lonestar{R}\, K(\delta) e_\eta$,
we can consider the right $K(\delta)$-module $e_\eta R\, K(\delta)$.
Just like in Lemma~\ref{spbas},
$e_\eta R\, K(\delta)$ has a distinguished basis
\begin{equation}\label{friend}
\{(\underline{\tU} \overline{\nu})\:|\:
\text{for all $\tU \in \mathscr T_R(\delta)$ and $\nu \subset \sh(\tU)$}\}.
\end{equation}
The right action of $K(\delta)$ on elements of this basis can be computed similarly to the
left action in Lemma~\ref{spbas}.
We make $e_\eta R\, K(\delta)$ into a left $B_R(\delta)$-module
by defining $f_{\tS\tT} (\underline{\tU} \overline{\nu})$
to be zero if $\bi^{\tT} \neq \bi^{\tU}$, while
if $\bi^{\tT} = \bi^{\tU}$
it is computed
by drawing $\underline{\tS} \overline{\tT}$ underneath
$\underline{\tU}\overline{\nu}$ and applying the extended surgery
procedure as usual.
Set
\begin{equation}
Q_{R}(\delta) := 
e_\eta R\, K(\delta) e_{r,s},
\end{equation}
which is a
graded $(B_R(\delta), K_{r,s}(\delta))$-bimodule
spanned by the vectors 
$(\underline{\tU} \overline{\nu})$
from (\ref{friend}) 
in which $\nu \in \dot\La_{r,s}$.
Recalling also the bimodule $P_R(\delta)$ from (\ref{aragorn}),
there are bimodule homomorphisms
\begin{equation}\label{bimaps}
P_{R}(\delta) \otimes_{B_R(\delta)} Q_{R}(\delta)
\rightarrow K_{r,s}(\delta),\quad
 Q_{R}(\delta) \otimes_{K_{r,s}(\delta)} P_{R}(\delta)
\rightarrow B_R(\delta),
\end{equation}
defined by the (by now) obvious multiplication maps; the first
involves the extended surgery procedure once again.

\begin{Lemma}\label{favbim}
There is a canonical isomorphism of
graded $(B_R(\delta), K_{r,s}(\delta))$-bimodules
$Q_R(\delta)
\cong
\hom_{K_{r,s}(\delta)}(P_R(\delta), K_{r,s}(\delta))$.
Hence the functor $\F_R$ from (\ref{life}) is isomorphic to the
functor $Q_{R}(\delta) \otimes_{K_{r,s}(\delta)} ?$.
Moreover the bimodule homomorphisms in (\ref{bimaps}) are both isomorphisms.
\end{Lemma}

\begin{proof}
Lemmas~\ref{composition} and
\ref{adunctions} give us a canonical
graded $(B_R(\delta), K(\delta))$-bimodule
isomorphism
$\hom_{K(\delta)}(\lonestar{R}\,P(\eta), K(\delta))
\cong
e_\eta R\, K(\delta)$.
It restricts to an isomorphism
of $(B_R(\delta), K_{r,s}(\delta))$-bimodules
$$
\hom_{K(\delta)}(\lonestar{R}\,P(\eta), K(\delta)e_{r,s})
\cong
e_\eta R\, K(\delta)e_{r,s} = Q_R(\delta).
$$
To establish the first statement, it remains to observe using Lemma~\ref{stupidone}
that
$$
\hom_{K(\delta)}(\lonestar{R}\,P(\eta), K(\delta)e_{r,s})
\cong \hom_{K_{r,s}(\delta)}(P_R(\delta), K_{r,s}(\delta)).
$$
The rest of the lemma now follows from Theorem~\ref{morita} and standard
Morita theory. One needs to use the explicit form of the isomorphism
in Lemma~\ref{adunctions} to check that the first map in (\ref{bimaps})
corresponds to the 
natural map $P_R(\delta) \otimes \hom_{K_{r,s}(\delta)}(P_R(\delta),
K_{r,s}(\delta))
\rightarrow K_{r,s}(\delta)$.
\end{proof}

\phantomsubsection{Graded cell modules}
Using Theorem~\ref{morita}, one can deduce a great deal of information about the
graded representation theory of the algebras $B_R(\delta)$
from Theorems~\ref{iscell} and \ref{finalbranch}.
Still it is desirable to give intrinsic descriptions of the $B_R(\delta)$-modules
$\mathbf{f}_R(V_{r,s}(\la))$
and
$\mathbf{f}_R(D_{r,s}(\la))$.
We do this in the next two subsections using the theory of graded
cellular algebras once again.

\begin{Theorem}\label{maincell}
The algebra $B_R(\delta)$ is a graded cellular algebra 
with cell datum $(X, I, C,\deg)$ defined by
\begin{align*}
X &:= \La_{r,s},
&I(\la) &:= \{\tT \in \mathscr T_R(\delta)\:|\:\sh(\tT) = \la\},\\
C^\la_{\tS, \tT} &:= f_{\tS\tT},&
\deg^\la_{\tT} &:= \deg(\tT)
\end{align*}
for $\la \in \La_{r,s}$ and $\tS, \tT \in I(\la)$.
\end{Theorem}

\begin{proof}
This follows from the algorithm for multiplication
by arguments similar to the proof of \cite[Corollary 3.3]{BS1}.
\end{proof}

We denote the graded cell modules attached to this graded cellular
structure
by $\{C_R(\la)\:|\:\la \in \La_{r,s}\}$.
The general construction gives easily that
$C_R(\la)$ has a homogeneous basis
\begin{equation}\label{cbas}
\left\{v_\tT\:\big|\:
\text{for all $\tT \in 
\mathscr T_R(\delta)$ with $\sh(\tT) = \la$}\right\}
\end{equation}
with $v_{\tT}$ of degree $\deg(\tT)$.
We have that $f_{\tU\tV} v_{\tT} = 0$
if $\bi^\tV \neq \bi^\tT$.
Assuming $\bi^\tV = \bi^\tT$, 
$f_{\tU\tV} v_{\tT}$
is computed as follows.
Draw
$\underline{\tU} \overline{\tV}$ underneath $\underline{\tT}
\overline{\la}$.
Then use the extended surgery procedure to contract the
$\overline{\tV}|\underline{\tT}$-part of the diagram. 
This produces a (possibly zero) sum of 
diagrams $(\underline{\tS} \overline{\la})$
for various $\tS$. Then we discard all the
ones
in which $\tS$ is not of shape $\la$, 
and replace the remaining $(\underline{\tS} \overline{\la})$'s
with the corresponding vectors $v_{\tS} \in C_R(\la)$.

\begin{Lemma}\label{sun}
For any $\la \in \La_{r,s}$
we have that
$\F_R(V_{r,s}(\la)) \cong C_R(\la)$
as graded $B_R(\delta)$-modules.
\end{Lemma}

\begin{proof}
By definition, $V_{r,s}(\la)$ is the quotient of
$K_{r,s}(\delta) e_\la$ by the submodule $S_1$ spanned by all vectors
$(\underline{\kappa} \alpha \overline{\la})$
for $\dot\La_{r,s} \ni \kappa \subset \alpha \supset \la$ with
$\alpha \neq \la$.
We claim that $S_1$ is equal to the submodule $S_2$ 
generated by the images of the
homomorphisms $K_{r,s}(\delta) \rightarrow K_{r,s}(\delta)e_\la$ 
defined by right multiplication by the
elements $\{(\underline \beta \beta \overline \la)\:|\:\beta \supset
\la, \beta \neq \la\}$. 
To see that $S_2 \subseteq S_1$, it suffices to observe that
$(\underline{\kappa} \alpha \overline{\beta})(\underline{\beta}
\beta \overline{\la})$ is a linear combination of vectors
$(\underline{\kappa} \gamma \overline{\la})$ for $\gamma \geq \beta$
by \cite[Theorem 4.4(i)]{BS1}.
Conversely to see that $S_1 \subseteq S_2$,
we take $(\underline{\kappa} \alpha \overline{\la}) \in S_1$
and compare with the product
$(\underline{\kappa} \alpha \overline{\alpha})(\underline{\alpha}
\alpha \overline{\la}) \in S_2$. 
By \cite[Theorem 4.4(iii)]{BS1}, this product is equal to
$(\underline{\kappa} \alpha \overline{\la})$ plus
higher terms of the form $(\underline{\kappa} \beta \overline{\la})$
for $\beta > \alpha$. We may assume that 
these higher terms belong to $S_2$ by induction.
Hence $(\underline{\kappa} \alpha \overline{\la}) \in S_2$, as required.

Using the description of the equivalence $\F_R$ as the functor $Q_R(\delta)
\otimes_{K_{r,s}(\delta)} ?$ given by Lemma~\ref{favbim},
we deduce that $\F_R(V_{r,s}(\la))$ is 
isomorphic to the quotient of $Q_R(\delta) \otimes_{K_{r,s}(\delta)}
 K_{r,s}(\delta) e_\la = Q_R(\delta) e_\la$
by the submodule $$
S := \sum_{\beta \supset \la, \beta \neq \la}
Q_R(\delta) (\underline{\beta} \beta \overline{\la}).
$$
In terms of the basis (\ref{friend}),
$Q_R(\delta) e_\la$ has basis
$$
\{(\underline{\tT} \overline{\la})\:|\:
\text{for all $\tT \in \mathscr T_R(\delta)$
with $\sh(\tT) \supset \la$}\}.
$$
Moreover by similar arguments to the previous paragraph,
the submodule $S$ is the span of all the vectors in this
basis in which $\sh(\tT) \neq \la$.
Now it is clear the map $Q_R(\delta) e_\la / S \rightarrow C_R(\la),
(\underline{\tT} \overline{\la}) + S \mapsto v_{\tT}$
is a graded $B_R(\delta)$-module isomorphism.
\end{proof}

\phantomsubsection{Characters of graded irreducible modules}
We say that $\tT \in \mathscr T_R(\delta)$ is {\em restricted}
if all the boundary cups in the diagram $\underline{\tT}$ are
anti-clockwise.
Recall by the general theory of graded cellular algebras that the graded cell
module
$C_R(\la)$ is equipped with a homogeneous symmetric bilinear form
$(.,.)$. 

\begin{Lemma}\label{easythm}
Let $\la \in \La_{r,s}$.
For $\tS, \tT \in \mathscr T_R(\delta)$
with $\sh(\tS)=\sh(\tT) = \la$,
the inner product
$(v_{\tS}, v_{\tT})$
is equal to $1$ if all of the following hold:
\begin{itemize}
\item[(1)]
$\tS$ and $\tT$ are both restricted;
\item[(2)]
$\bi^\tS = \bi^\tT$, i.e. the connected components in the diagrams
$\underline{\tS}$ and $\underline{\tT}$ are in exactly the same positions;
\item[(3)]
all matching pairs of circles in $\underline{\tS}$ and
$\underline{\tT}$ are oriented in opposite ways.
\end{itemize}
For all other $\tS$ and $\tT$ we have that $(v_{\tS}, v_{\tT}) = 0$.
\end{Lemma}

\begin{proof}
This follows from the explicit definition of the bilinear form on
$C_R(\la)$.
See \cite[Lemma 9.7]{BS3} where this is explained in an analogous situation.
\end{proof}

Let $D_R(\la)$ denote the quotient $C_R(\la) / \rad C_R(\la)$ 
by the radical of the bilinear form $(.,.)$.
The non-zero $D_R(\la)$'s give a complete set of pairwise
non-isomorphic irreducible $B_R(\delta)$-modules (up to degree shift).
The following theorem gives complete information about the
$D_R(\la)$'s
from a combinatorial viewpoint.

\begin{Theorem}\label{irreddimensions}
For $\la \in \La_{r,s}$,
let $\bar v_{\tT} \in D_R(\la)$ 
be the image of $v_{\tT} \in C_R(\la)$.
The vectors
$\left\{\bar v_{\tT}\:\big|\:
\text{for all restricted $\tT \in \mathscr T_R(\delta)$ with
$\sh(\tT) = \la$}
\right\}$
give a basis for $D_R(\la)$.
Moreover $D_R(\la)$ is non-zero if and only if $\la \in \dot
\La_{r,s}$,
in which case we have that
$D_R(\la) \cong \F_R(L_{r,s}(\la))$.
\end{Theorem}

\begin{proof}
The opening statement is a consequence of Lemma~\ref{easythm}
and the definition of $D_R(\la)$.
In particular it is clear from this that $D_R(\la) = \{0\}$
if $\la \notin \dot\La_{r,s}$, i.e. when $r=s> 0$,
$\delta = 0$ and $\la = (\varnothing,\varnothing)$.
If $\la \in \dot\La_{r,s}$ then
$D_R(\la)\cong\F_R(L_{r,s}(\la))$ by Lemma~\ref{sun} and Theorem~\ref{iscell}(2),
in particular, $D_R(\la)$ is non-zero in view of Theorem~\ref{morita}.
\end{proof}

\begin{Corollary}\label{jcor}
The Jacobson radical of $B_R(\delta)$ is spanned by all the basis
vectors
$f_{\tS\tT}$ from (\ref{bbas})
such that either $\tS$ or $\tT$ is {\em not} restricted.
\end{Corollary}

\begin{proof}
Let $I$ be the span of all the basis vectors $f_{\tS\tT}$
in which either $\tS$ or $\tT$ is not restricted.
If $\tS$ is not restricted, i.e. 
$\underline{\tS}$ has a clockwise boundary cup, it is clear from Theorem~\ref{irreddimensions}
and 
the multiplication rule that $f_{\tS\tT}$ annihilates all irreducible
left $B_R(\delta)$-modules, hence it lies in the Jacobson radical $J(B_R(\delta))$.
By symmetry this implies 
also that
$f_{\tS\tT} \in J(B_R(\delta))$ if $\tT$ is not restricted. Hence 
$I\subseteq J(B_R(\delta))$.
Moreover Theorem~\ref{irreddimensions} implies that 
$\dim B_R(\delta) / J(B_R(\delta))$ 
is equal to the number of
pairs $(\tS,\tT)$ of restricted $R$-tableaux of the same shape,
which is the same as $\dim B_R(\delta) / I$.
\end{proof}

\phantomsubsection{Primitive idempotents}
For a restricted $R$-tableau $\tT \in \mathscr T_R(\delta)$, we let
\begin{equation}\label{primid}
e_{\tT} := f_{\tT\tT^*} \in B_R(\delta)
\end{equation}
where $\tT^*$ is
 the unique $R$-tableau such that $\underline{\tT}^*$
is obtained from $\underline{\tT}$ by reversing the orientation of all
its circles.

\begin{Lemma}\label{primitivespre}
Suppose $\tT \in \mathscr T_R(\delta)$ is restricted and set
$\la := \sh(\tT)$.
Then $e_{\tT}$ is a homogeneous primitive idempotent.
Moreover the projective indecomposable module
$B_R(\delta) e_{\tT}$ 
has irreducible head isomorphic
to $D_R(\la) \langle -\deg(\tT)\rangle$.
\end{Lemma}

\begin{proof}
It is clear from the multiplication rule that $e_{\tT}$ is an idempotent.
For everything else, it suffices to show for $\mu \in \dot \La_{r,s}$ that 
$$
\hom_{B_R(\delta)}(B_R(\delta) e_{\tT}, D_R(\mu)) \cong 
e_{\tT} D_R(\mu)
$$
is one-dimensional concentrated in degree $\deg(\tT)$ if $\mu = \la$
and it is zero for all other $\mu$.
For this we consider the basis for $D_R(\mu)$ from 
Theorem~\ref{irreddimensions}.
If $\mu = \la$ the only basis vector on which $e_{\tT}$ is non-zero
is the vector $\bar v_{\tT}$, which is of degree $\deg(\tT)$.
If $\mu \neq \la$ all the basis vectors are annihilated
by $e_{\tT}$.
\end{proof}

Now we look again at the idempotents $e(\bi)$ for $\bi \in \Z^{r+s}$.

\begin{Lemma}\label{primitives}
Given $\bi\in \Z^{r+s}$,
the idempotent $e(\bi)$ decomposes into a sum of mutually 
orthogonal primitive idempotents as $e(\bi) = \sum_{\tT} e_{\tT}$
summing over all restricted $\tT \in \mathscr T_R(\delta)$
with $\bi^\tT = \bi$.
Moreover all the tableaux $\tT$ appearing in this sum are of the same shape.
\end{Lemma}

\begin{proof}
By the algorithm for multiplication, the given sum 
$\sum_{\tT} e_{\tT}$
satisfies the defining property (\ref{ids}) of the idempotent $e(\bi)$.
Hence $e(\bi) = \sum_{\tT} e_{\tT}$. 
It is clear that the $e_{\tT}$'s for
different $\tT$ are orthogonal,
and all the $\tT$'s appearing in the sum have the same shape.
Finally the $e_{\tT}$'s are primitive by Lemma~\ref{primitivespre}.
\end{proof}

\begin{Remark}\label{isotypic}\rm
Lemmas~\ref{primitivespre} and \ref{primitives} together show that the idempotents
$e(\bi)$ are {\em isotypic} in the sense that 
$B_R(\delta) e(\bi)$ is isomorphic to 
a direct sum of degree-shifted copies of 
a single projective indecomposable module, namely, the projective
cover of $D_R(\la)$ where $\la$ is the shape of any restricted
$R$-tableau $\tT$ with $\bi^\tT = \bi$.
\end{Remark}

\phantomsubsection{\boldmath The quotient algebra 
$B_{R}(\delta) / B_R(\delta)_{> k}$}
In this subsection we investigate certain quotients of $B_R(\delta)$.
Given $\la \in \La_{r,s}$, we 
let $\defect(\la)$ be the {\em defect} of $\la$, that is, the number of
caps in the cap diagram $\overline{\la}$.
Also let $\rk(\la)$ be the number of $\circ$'s or the number of
$\times$'s in $\la$, whichever is smaller; it is the number of
$\circ$'s if $\delta \geq 0$ and the number of $\times$'s if $\delta
\leq 0$.
Then set
\begin{equation}\label{kdefla}
k(\la) := \defect(\la) + \rk(\la).
\end{equation}

\begin{Lemma}\label{daft}
Suppose we are given $\tS, \tT \in \mathscr T_R(\delta)$ of the same
shape $\la$.
\begin{itemize}
\item[(1)] In the diagram $\underline{\tS}$, the 
number of boundary caps is equal to
$\rk(\la)$ plus the
number of boundary cups. Similarly in the diagram $\overline{\tT}$, the 
number of boundary cups is equal to
$\rk(\la)$ plus the
number of boundary caps.
\item[(2)] In the diagram $\underline{\tS} \overline{\tT}$, 
the number of boundary cups is equal to the number of boundary caps.
\end{itemize}
\end{Lemma}

\begin{proof}
(1)
In $\underline{\tS}$,
the number of boundary cups minus the number of boundary caps 
is equal
to the total number of cups minus the total number of caps.
Arguing by induction on $r+s$ using
(\ref{here}), this is equal to $\rk(\la)$.

(2)
Suppose there are $a$ boundary caps and $b$ boundary cups 
in $\underline{\tS}$.
Suppose there are $c$ boundary caps and $d$ boundary cups 
in $\overline{\tT}$.
Finally suppose that $e$ of the boundary cups in $\underline{\tS}$
belong to circles in $\underline{\tS} \overline{\tT}$. This
is also the number of boundary caps in $\overline{\tT}$
that belong to circles in $\underline{\tS} \overline{\tT}$.
The total number of boundary cups in $\underline{\tS} \overline{\tT}$
is $b + d - e$.
The total number of boundary caps in $\underline{\tS} \overline{\tT}$
is $a + c - e$.
These are equal
as
$a = b + \rk(\la)$ and $d = c + \rk(\la)$
by (1).
\end{proof}

Now we introduce various other ``$k$-values.''
For $\tS,\tT \in \mathscr T_R(\delta)$
of the same shape, let $k(\tS,\tT)$ be 
the number of boundary cups in the diagram
$\underline{\tS} \overline{\tT}$;
in view of Lemma~\ref{daft}(2) this is the same thing as the number of
boundary caps in
$\underline{\tS} \overline{\tT}$.
We write simply $k(\tT)$ for $k(\tT,\tT)$; this is just 
the number of boundary cups in $\overline{\tT}$ or the number of
boundary caps in $\underline{\tT}$.
Finally for $\bi \in \Z^{r+s}$ with $e(\bi) \neq 0$,
let $k(\bi) := k(\tT)$ for any $\tT \in \mathscr T_R(\delta)$
with $\bi^\tT = \bi$; if $e(\bi) = 0$ we set $k(\bi) := \infty$.

Let $B_R(\delta)_k$ be the span of all the
basis vectors $f_{\tS\tT} \in B_R(\delta)$ 
in which $k(\tS,\tT) = k$.
Thus we have the vector space decomposition
\begin{equation*}
B_R(\delta) = \bigoplus_{k \geq 0} B_R(\delta)_k.
\end{equation*}
For example, in Figure~\ref{fig1}, the vector spaces
$B_R(\delta)_1$ and $B_R(\delta)_2$ are of
dimensions $20$ and $4$, respectively, and all other $B_R(\delta)_k$
are zero.
Note also when non-zero that the idempotent
$e(\bi)$ belongs to $B_R(\delta)_{k(\bi)}$.

For $k \geq 0$, 
let
\begin{equation}\label{subalgbas}
B_R(\delta)_{\leq k} := \bigoplus_{j =0}^k B_R(\delta)_j,
\qquad
B_R(\delta)_{> k} := \bigoplus_{j > k} B_R(\delta)_j.
\end{equation}

\begin{Lemma}\label{howprod}
For any $k,l \geq 0$, we have
that $B_R(\delta)_k B_R(\delta)_l \subseteq B_R(\delta)_{\max(k,l)}$.
\end{Lemma}

\begin{proof}
Take basis vectors $f_{\tS\tT} \in B_R(\delta)_k$
and $f_{\tU\tV} \in B_R(\delta)_l$
such that $\bi^\tT = \bi^{\tU}$.
By symmetry, we may as well assume that $k \leq l$
and need to show that $f_{\tS\tT} f_{\tU\tV} \in B_R(\delta)_l$.
Of the $l$ boundary cups in $\underline{\tU} \overline{\tV}$,
some of them, say $l'$ of them, do not cross the middle number line.
These $l'$ boundary cups clearly remain as boundary cups in
all the diagrams produced when $f_{\tS\tT} f_{\tU\tV}$ is computed
using the extended surgery procedure.
The remaining $l'' := l - l'$ boundary cups 
cross the middle number line.
During the extended surgery procedure, these 
$l''$ boundary cups either remain as boundary cups in the product (when they
are joined with anti-clockwise circles) or we get zero (when anything
else happens). In this way 
we see there are at least $l = l'+l''$ boundary cups in all the diagrams produced
in the product.
The only way we could get more than $l$ boundary cups in the product
$f_{\tS\tT} f_{\tU\tV}$
is when some boundary cup from $\underline{\tS} \overline{\tT}$
gets joined to a pair of vertical line segments in $\underline{\tU}
\overline{\tV}$.
If this happens, there must be $l$ boundary caps in between these
vertical line segments in the diagram $\underline{\tS}
\overline{\tT}$, 
so that $\underline{\tS} \overline{\tT}$ has at least $(l+1)$ boundary
cups.
This contradicts the initial assumption that $k \leq l$.
\end{proof}

Lemma~\ref{howprod} implies that
$B_R(\delta)_{> k}$ is a two-sided ideal and
$B_R(\delta)_{\leq k}$ is a subalgebra of $B_R(\delta)$.
Moreover the canonical quotient map gives an algebra isomorphism
\begin{equation}\label{ks}
B_R(\delta)_{\leq k} \stackrel{\sim}{\rightarrow} 
B_R(\delta) / B_R(\delta)_{> k}.
\end{equation}
The identity element in $B_R(\delta)_{\leq k}$
is the sum of 
the idempotents $e(\bi)$
for $\bi \in \Z^{r+s}$ with $k(\bi) \leq k$.
Lemma~\ref{primitives} gives a canonical way to decompose
each such $e(\bi)$
into mutually orthogonal primitive idempotents in $B_R(\delta)_{\leq
  k}$.
In particular the identity element of $B_R(\delta)_{\leq k}$
decomposes as $\sum_{\tT} e_{\tT}$ summing over all restricted $\tT$
with $k(\tT) \leq k$.
Using (\ref{ks}), we get an analogous decomposition
of the identity
$B_R(\delta) / B_R(\delta)_{> k}$.

\begin{Theorem}\label{irrclass}
For $\la \in \dot \La_{r,s}$ with $k(\la) \leq k$, 
the ideal $B_R(\delta)_{> k}$ acts as zero on $D_R(\la)$,
hence $D_R(\la)$ can be viewed as an irreducible
$B_R(\delta) / B_R(\delta)_{> k}$-module.
The modules
$$\{D_R(\la)\:|\:\la \in \dot \La_{r,s}, k(\la)
\leq k\}$$ give
a complete set of pairwise non-isomorphic irreducible
$B_R(\delta) / B_R(\delta)_{> k}$-modules.
\end{Theorem}

\begin{proof}
The remarks just made show that
$$
B_R(\delta) / B_R(\delta)_{> k} = \bigoplus_{\tT} B_R(\delta) e_\tT
/ B_R(\delta)_{> k} e_\tT
$$
summing over all restricted $\tT \in \mathscr T_R(\delta)$
with $k(\tT) \leq k$.
Moreover each of the summands on the right hand side is a non-zero
projective indecomposable module.
Now suppose that $\tT$ is a restricted $R$-tableau with $k(\tT) \leq
k$,
and set $\la := \sh(\tT)$.
Lemma~\ref{primitivespre}
implies that
the head of
$B_R(\delta) e_{\tT} / B_R(\delta)_{> k} e_{\tT}$
is isomorphic (up to degree shift) to $D_R(\la)$.
Moreover Lemma~\ref{daft}(1) implies that
$k(\tT) = k(\la)$, hence $k(\la) \leq k$.
The theorem follows from these observations.
\end{proof}

We end with a couple more results about 
$B_R(\delta) / B_R(\delta)_{> k}$ which will be needed later.
The following lemma implies that the maps 
$\iota_{R;i}^{RE}$ and $\iota_{R;i}^{RF}$ from (\ref{newiota}) factor to induce well-defined maps
\begin{align}\label{leg1}
\iota_{R;i}^{RE}&: B_{R}(\delta) / B_R(\delta)_{> k} \rightarrow
B_{RE}(\delta) / B_{RE}(\delta)_{> k},\\
\iota_{R;i}^{RF}&: B_{R}(\delta)/  B_R(\delta)_{> k} \rightarrow
B_{RF}(\delta)/ B_{RF}(\delta)_{> k}.\label{leg2}
\end{align}

\begin{Lemma}\label{lehrerzhangsetup}
We have that 
$\iota_{R;i}^{RE}(B_R(\delta)_{> k})
\subseteq
B_{RE}(\delta)_{> k}$ and 
$\iota_{R;i}^{RF}(B_R(\delta)_{>
  k})\subseteq B_{RF}(\delta)_{> k}$.
\end{Lemma}

\begin{proof}
We just consider $\iota_{R;i}^{RE}$. Take $f_{\tS\tT} \in
B_R(\delta)_{> k}$, i.e. $\underline{\tS} \overline{\tT}$ has more
than $k$ boundary cups. 
To apply the map $\iota_{R;i}^{RE}$, we use the algorithm explained at
the end of the previous section, and need to show that the number of
boundary cups does not get smaller. This is immediate for all except the
first
configuration in (\ref{CKLR2}). For the first configuration it follows
from the explicit form of the surgery procedure in
\cite[(6.2)]{BS2}.
\end{proof}

\begin{Lemma}
We have that $B_R(\delta)_{> 0}=\langle e(\bi)\:|\:\text{for all }\bi \in \Z^{r+s}\text{ with }k(\bi) >
0\rangle$.
\end{Lemma}

\begin{proof}
All the $e(\bi)$'s with $k(\bi) > 0$ belong to $B_R(\delta)_{> 0}$.
Conversely if $f_{\tS\tT} \in B_R(\delta)_{> 0}$, then there is at
least one boundary cup in the diagram $\underline{\tS}
\overline{\tT}$.
This means either that $\overline{\tT}$ has a boundary cup 
or $\underline{\tS}$ has a boundary
cup.
In the former case, $\bi := \bi^{\tT}$ has $k(\bi) > 0$ and 
$f_{\tS\tT} \in B_R(\delta) e(\bi)$.
In the latter case we deduce from Lemma~\ref{daft}(1) that
$\underline{\tS}$
also has a boundary cap, so $\bi := \bi^{\tS}$ has $k(\bi) > 0$
and $f_{\tS\tT} \in e(\bi) B_R(\delta)$.
\end{proof}

\begin{Remark}\label{tempt}\rm
For $k > 0$, the ideal $B_R(\delta)_{> k}$ 
need not be generated by idempotents.
See Figure~\ref{fig1} taking $k = 1$ for an example
in which it is not.
\end{Remark}

\section{Mixed Schur-Weyl duality for $\mathfrak{gl}_{m|n}(\C)$}\label{s3}

Fix integers $m, n \geq 0$ and set $\delta := m-n$.
In order to identify the walled Brauer algebra $B_{r,s}(\delta)$
with the graded walled Brauer algebra $B_R(\delta)$,
we first need to establish a Schur-Weyl duality connecting
$B_{r,s}(\delta)$ to 
the general linear Lie superalgebra $\mathfrak{gl}_{m|n}(\C)$.
Our approach to this follows the same strategy as the proof of
\cite[Theorem 3.4]{N}
 (which treats $\mathfrak{gl}_m(\C)$ rather than $\mathfrak{gl}_{m|n}(\C)$).
The results in this section have also been obtained
independently by Comes and Wilson \cite{CW} by a similar trick;
actually their exposition is slicker since they exploit fully the tensor category
formalism. We also note a partial result along the same lines
was obtained earlier in \cite{SM}.

\phantomsubsection{The general linear Lie superalgebra}
As a general convention, for a homogeneous vector $v$ in a vector superspace, we will write $|v| \in \Z_2$ for its parity.
Let $V$ be the 
vector superspace with homogeneous basis $v_1,\dots,v_{m+n}$
and $|v_i| := |i| \in \Z_2$,
where 
$|i| := 0$ if $1 \leq i \leq m$ and
$|i| := 1$ if $m+1 \leq i \leq m+n$.
The Lie superalgebra $\mathfrak{g}:= \mathfrak{gl}_{m|n}(\C)$ is the 
vector 
superspace $\End(V)$ of all (not necessarily homogeneous) linear endomorphisms of $V$, with superbracket $[x,y] := xy - (-1)^{|x| | y|} yx$
for homogeneous $x,y\in \End(V)$.
Let $U(\mathfrak{g})$ be the universal enveloping superalgebra of $\mathfrak{g}$,
which is a Hopf superalgebra with comultiplication
$\Delta$ and antipode $S$ defined on $x \in \mathfrak{g}$
by $\Delta(x) = x\otimes 1 + 1\otimes x$ and $S(x) = -x$.
We recall in particular that the multiplication on $U(\mathfrak{g})\otimes U(\mathfrak{g})$
is by $(a \otimes b)(c \otimes d) = (-1)^{|b||c|}ac \otimes bd$
for homogeneous $a,b,c,d \in U(\mathfrak{g})$.

Let $W:=V^*$ be the dual of the tautological $\mathfrak{g}$-supermodule $V$, 
with basis
$w_1,\dots,w_{m+n}$ that is dual to the given basis for $V$,
so $w_i(v_j) = \delta_{i,j}$. 
It is a $\mathfrak{g}$-supermodule
via the usual rule $(xw)(v) := (-1)^{|x||w|} w(S(x)v)$ for homogeneous
$x \in \mathfrak{g}, v \in V, w \in W$.
In particular, the standard basis elements of $\mathfrak{g}$
consisting of matrix units 
act on the bases of $V$ and $W$ by the formulae
\begin{equation}\label{act1}
e_{i,j} v_k = \delta_{j,k} v_i,
\qquad
e_{i,j} w_k = -\delta_{i,k}(-1)^{(|i|+|j|) |i|} w_j
\end{equation}
for all $1\leq i,j,k \leq m+n$.
If $M$ and $N$ are $\mathfrak{g}$-supermodules,
we write $\hom_{\mathfrak{g}}(M,N)$
for the vector superspace of all
$\mathfrak{g}$-supermodule homomorphisms from $M$ to $N$. By definition, this is the
set of all (not necessarily homogeneous) linear maps 
$f \in \hom(M,N)$ annihilated by all $x \in \mathfrak{g}$, where
the linear action of $\mathfrak{g}$ on $\hom(M,N)$ is by
$(x f)(m) = x(f(m)) - (-1)^{|x||f|}f(xm)$ for homogeneous
$x \in \mathfrak{g}, f \in \hom(M,N)$ and $m \in M$.

An important role is played by the even element
\begin{equation}\label{om}
\Omega := \sum_{i,j=1}^{m+n}(-1)^{|j|} e_{i,j} \otimes e_{j,i}
\in U(\mathfrak{g}) \otimes U(\mathfrak{g}).
\end{equation}
This corresponds to the (invariant) supertrace form on $\mathfrak{g}$,
so $\Omega$ commutes with $\Delta(x) = x \otimes 1 + 1 \otimes x$ for all $x \in \mathfrak{g}$.
Left multiplication by $\Omega$ defines a $\mathfrak{g}$-supermodule endomorphism of the tensor product $M \otimes N$ of $\mathfrak{g}$-supermodules $M$ and $N$.
More generally, for $1 \leq a < b \leq r+s$, we let
\begin{align}\label{genom}
\hspace{-2mm}\Omega_{a,b} &:= \sum_{i,j=1}^{m+n}\!(-1)^{|j|}
1^{\otimes (a-1)} \otimes e_{i,j} \otimes 1^{\otimes (b-a-1)} \otimes e_{j,i} \otimes 1^{r+s-b} \hspace{-4mm}&&\in U(\mathfrak{g})^{\otimes (r+s)},
\end{align}
which defines a $\mathfrak{g}$-supermodule 
endomorphism of 
$M_1 \otimes \cdots\otimes M_{r+s}$ given any $\mathfrak{g}$-supermodules
$M_1,\dots,M_{r+s}$.

\begin{Lemma}\label{oml}
For $M, N \in \{V, W\}$, the endomorphism $\Omega$ of $M \otimes N$ satisfies
\begin{align*}
\Omega(v_i \otimes v_j) &= (-1)^{|i| |j|} v_j \otimes v_i,\\
\Omega(w_i \otimes w_j) &= (-1)^{|i| |j|} w_j \otimes w_i,\\
\Omega(v_i \otimes w_j) &= -\delta_{i,j} (-1)^{|i|} \sum_{k=1}^{m+n}
v_k \otimes w_k,\\
\Omega(w_i \otimes v_j) &= -\delta_{i,j} \sum_{k=1}^{m+n}
(-1)^{|k|}  w_k \otimes v_k,
\end{align*}
for all $1 \leq i,j \leq m+n$.
\end{Lemma}

\begin{proof}
Use (\ref{act1}), the definition of $\Omega$ from (\ref{om}), and recall that the action is by $(a \otimes b) (v \otimes w)
= (-1)^{|b||v|} av \otimes bw$.
\end{proof}

\phantomsubsection{Actions on tensor space}
The next goal is to define right actions of the algebras 
$\C \Sigma_{r+s}$ and $B_{r,s}(\delta)$
on the spaces $V^{\otimes (r+s)}$ and $V^{\otimes r} \otimes W^{\otimes s}$, respectively. 
One way to do this is to write down the actions of generators and then verify 
relations. For our purposes, it is better to take 
a different approach which gives explicit
closed formulae for the actions of
arbitrary basis elements, not just generators.
We need some multi-index notation.
Let $I := \{1,\dots,m+n\}$. 
For any $t \geq 0$, the symmetric group $\Sigma_t$
acts naturally on the right on the set
$I^t$ by place permutation, so 
$\bi \cdot\sigma = (i_{\sigma(1)},\dots,i_{\sigma(t)})$ for 
$\bi = (i_1,\dots,i_t) \in I^t$ and $\sigma \in \Sigma_t$.
For  $\bi\in I^t$, we let
\begin{equation}\label{pdef}
|\bi| := |i_1|+\cdots+|i_t|,
\qquad
p(\bi) := \sum_{1 \leq a < b \leq t} |i_a||i_b|.
\end{equation}
Given $\bi = (i_1,\dots,i_{r+s}) 
\in I^{r+s}$, we let $\bi^{\mathrm L} := (i_1,\dots,i_r) \in I^r$
and $\bi^{\mathrm R} := (i_{r+1},\dots,i_{r+s}) \in I^s$, so that 
$\bi = \bi^{\mathrm L} \bi^{\mathrm R}$ where the product on the right is by concatenation of tuples.
The $\mathfrak{g}$-supermodules
$V^{\otimes r}$ and $W^{\otimes s}$ 
have the obvious bases 
$\{v_\bi\:|\:\bi \in I^r\}$
and 
$\{w_\bj\:|\:\bj \in I^s\}$, respectively, 
writing $v_{\bi} := v_{i_1}\otimes\cdots \otimes v_{i_r}$ and
$w_{\bj} := w_{j_1}\otimes\cdots\otimes w_{j_s}$.
Similarly
$V^{\otimes (r+s)}$ has basis 
$\left\{v_\bi = v_{\bi^{\mathrm L}} \otimes v_{\bi^{\mathrm R}}\:\big|\:\bi \in I^{r+s}\right\}$ and 
$V^{\otimes r} \otimes W^{\otimes s}$ has basis
$\left\{v_{\bi^{\mathrm L}} \otimes w_{\bi^{\mathrm R}}\:\big|\:\bi \in I^{r+s}\right\}$.

Now suppose we are given a diagram $\sigma$, which could either be
a permutation diagram from 
$\Sigma_{r+s}$ or a walled Brauer diagram in $B_{r,s}(\delta)$.
For $\bi,\bj \in I^{r+s}$, let ${_\bi} \sigma_\bj$
be the {labelled diagram}
obtained by colouring the vertices at the bottom of $\sigma$ by
 $i_1,\dots,i_{r+s}$ and the vertices at the top of $\sigma$
by $j_1,\dots,j_{r+s}$ (in order from left to right as usual).
We say ${_\bi} \sigma_\bj$
is {\em consistently coloured}
if the vertices at the ends of each 
strand are coloured in the same way, in which case we get induced a well-defined colouring of the strands themselves.
In a consistently coloured {permutation diagram} ${_\bi}\sigma_\bj$
we of course have that $\bj = \bi \cdot \sigma$,
so the colouring
$\bi$ at the bottom uniquely determines the colouring $\bj$ at the top;
this is not in general true for walled Brauer diagrams.

For a permutation diagram or a walled Brauer diagram $\sigma$,
and $\bi,\bj \in I^{r+s}$, we 
introduce the {\em weight}
\begin{equation}\label{wtdef}
\wt({_\bi}\sigma_\bj)
:= \left\{
\begin{array}{ll}
\prod_{c}
(-1)^{|c|}
\cdot 
\prod_{h} (-1)^{|h|}
&\text{if ${_\bi}\sigma{_\bj}$ is consistently coloured,}\\
0&\text{otherwise,}
\end{array}
\right.
\end{equation}
where
\begin{itemize}
\item[(1)] the first product is over all {\em proper} coloured 
crossings
$c$ in ${_\bi}\sigma_\bj$, i.e.
the crossings that involve two different strands rather than self-intersections;
\item[(2)]
the second product is over all coloured horizontal strands $h$
in ${_\bi}\sigma_\bj$ whose endpoints 
are on the bottom edge 
(if $\sigma$ is a permutation diagram this product is empty so can be omitted);
\item[(3)]
the parity $|c|$ of a coloured crossing
$
\begin{picture}(32,10)
\put(0,0){$c=$}
\put(20,1){\line(1,1){10}}
\put(20,11){\line(1,-1){10}}
\put(18.7,-3.5){$\scriptscriptstyle{i}$}
\put(27.4,-3.5){$\scriptscriptstyle{j}$}
\end{picture}
$
of strands of colours $i$ and $j$
is $|i||j|$;
\item[(4)]
the parity $|h|$ of a coloured horizontal strand $
\begin{picture}(33,10)
\put(0,0){$h=$}
\put(26,1.7){\oval(9,16)[t]}
\put(19.7,-3.5){$\scriptscriptstyle{k}$}
\put(28.4,-3.5){$\scriptscriptstyle{k}$}
\end{picture}
$
of colour $k$ is $|k|$.
\end{itemize}
For example, for $\bi = (2,2,1,2)$ and $\bj = (3,2,3,1)$,
the following consistently coloured diagram has $\wt({_\bi}\sigma_\bj)= 1$
either if $m=1,n=2$ or if $m=2,n=1$:
$$
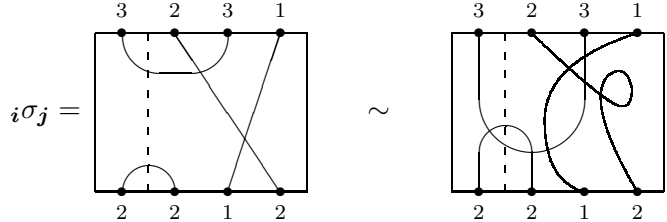
\begin{picture}(135,76)
\put(-32,34){${_\bi}\sigma_{\bj}=$}
\put(8,74){$_3$}
\put(28,74){$_2$}
\put(48,74){$_3$}
\put(68,74){$_1$}
\put(8,-2.5){$_2$}
\put(28,-2.5){$_2$}
\put(48,-2.5){$_1$}
\put(68,-2.5){$_2$}
\put(8,4.2){$\scriptstyle\bullet$}
\put(28,4.2){$\scriptstyle\bullet$}
\put(48,4.2){$\scriptstyle\bullet$}
\put(68,4.2){$\scriptstyle\bullet$}
\put(8,64.2){$\scriptstyle\bullet$}
\put(28,64.2){$\scriptstyle\bullet$}
\put(48,64.2){$\scriptstyle\bullet$}
\put(68,64.2){$\scriptstyle\bullet$}
\put(0,6){\line(1,0){80}}
\put(0,6){\line(0,1){60}}
\put(0,66){\line(1,0){80}}
\put(80,66){\line(0,-1){60}}
\dashline{3}(20,6)(20,66)
\put(50,6){\line(1,3){20}}
\put(70,6){\line(-2,3){40}}
\put(30.3,66){\oval(40,30)[b]}
\put(20.3,6){\oval(20,20)[t]}
\end{picture}
\begin{picture}(90,76)
\put(-32,34){$\sim$}
\put(8,74){$_3$}
\put(28,74){$_2$}
\put(48,74){$_3$}
\put(68,74){$_1$}
\put(8,-2.5){$_2$}
\put(28,-2.5){$_2$}
\put(48,-2.5){$_1$}
\put(68,-2.5){$_2$}
\put(8,4.2){$\scriptstyle\bullet$}
\put(28,4.2){$\scriptstyle\bullet$}
\put(48,4.2){$\scriptstyle\bullet$}
\put(68,4.2){$\scriptstyle\bullet$}
\put(8,64.2){$\scriptstyle\bullet$}
\put(28,64.2){$\scriptstyle\bullet$}
\put(48,64.2){$\scriptstyle\bullet$}
\put(68,64.2){$\scriptstyle\bullet$}
\put(0,6){\line(1,0){80}}
\put(0,6){\line(0,1){60}}
\put(0,66){\line(1,0){80}}
\put(80,66){\line(0,-1){60}}
\dashline{3}(20,6)(20,66)
\qbezier(70,6)(50,40)(59,50)
\qbezier(59,50)(66,55)(68,45)
\qbezier(68,45)(68,35)(60,40)
\qbezier(60,40)(52,45)(30,66)

\qbezier(50,6)(35,10)(35,33)
\qbezier(35,33)(35,55)(70,66)
\put(30.3,66){\oval(40,90)[b]}
\put(20.3,6){\oval(20,50)[t]}
\end{picture}
$$

\begin{Lemma}\label{reide}
The weight $\wt({_\bi}\sigma_\bj)$ depends only on the isotopy class of $\sigma$.
\end{Lemma}

\begin{proof}
Check that $\wt({_\bi}\sigma_\bj)$ does not change when any of the three Reidemeister moves are applied to $\sigma$.
\end{proof}

In the following two lemmas, we construct the required
actions of $\C \Sigma_{r+s}$ and $B_{r,s}(\delta)$
on $V^{\otimes(r+s)}$ and $V^{\otimes r} \otimes W^{\otimes s}$, respectively.
The first of these
is well known; see e.g. \cite[Corollary 1.6]{BR}.
We include the proof 
only to prepare the reader for the 
more complicated walled Brauer case in the second lemma.

\begin{Lemma}\label{easy}
There is a well-defined right $\C \Sigma_{r+s}$-module structure on 
$V^{\otimes (r+s)}$ such that
\begin{equation}\label{easyact}
v_\bi \cdot \sigma = \sum_{\bj \in I^{r+s}} \wt({_\bi}\sigma_\bj)\, v_\bj
\end{equation}
for all $\bi \in I^{r+s}$ and 
$\sigma \in \Sigma_{r+s}$.
Moreover for $1 \leq a < b \leq r+s$ the transposition
$(a\:b)$ acts in the same way as the operator
$\Omega_{a,b}$ from (\ref{genom}).
Hence the action of $\C \Sigma_{r+s}$ commutes with the action of 
$\mathfrak{g}$.
\end{Lemma}

\begin{proof}
Once the first statement is proved, the second statement is an easy consequence
using also Lemma~\ref{oml}. Then the final statement follows because
$\Omega_{a,b}$ acts as a $\mathfrak{g}$-supermodule endomorphism and the
transpositions generate $\Sigma_{r+s}$.
To prove the first statement, take $\sigma,\tau \in \Sigma_{r+s}$ and set $\pi := \sigma \tau$.
We need to check that (\ref{easyact}) is consistent with the associativity equation
$(v_\bi \cdot \sigma) \cdot \tau = v_\bi \cdot \pi$ for all $\bi \in I^{r+s}$, i.e.
$$
\sum_{\bj,\bk \in I^{r+s}} 
\wt({_\bi} \sigma_\bj)
\wt({_\bj}\tau_\bk)
\,v_\bk=
\sum_{\bk \in I^{r+s}} \wt({_\bi}\pi_\bk)\, v_\bk.
$$
Equating coefficients, we therefore need to see that
$$
\sum_{\bj \in I^{r+s}} 
\wt({_\bi}\sigma_\bj) \wt({_\bj}\sigma_\bk)
=\wt({_\bi}\pi_\bk)$$
for all $\bi,\bk \in I^{r+s}$.
Both sides are zero unless ${_\bi}\pi_\bk$ is consistently coloured, so we
can assume $\bk = \bi \cdot \pi$. Then the sum on the left involves just 
one non-zero term, so we can also assume that $\bj := \bi\cdot\sigma$, and 
are reduced to checking for this  $\bj$ and $\bk$ 
that
$\wt({_\bi}\sigma_{\bj}) \wt({_{\bj}}\tau_{\bk})=\wt({_\bi}\pi_{\bk})$.
This is clear from Lemma~\ref{reide}.
\end{proof}

\begin{Lemma}\label{act2}
There is a well-defined right $B_{r,s}(\delta)$-module structure on 
$V^{\otimes r} \otimes W^{\otimes s}$ such that
\begin{equation}\label{hardact}
v_{\bi^{\mathrm L}} \otimes w_{\bi^{\mathrm R}} 
\cdot \sigma := 
\sum_{\bj \in I^{r+s}} 
\wt(_{\bi}\sigma_\bj)\,
v_{\bj^{\mathrm L}} \otimes w_{\bj^{\mathrm R}}
\end{equation}
for all $\bi \in I^{r+s}$ and
all walled Brauer diagrams $\sigma \in B_{r,s}(\delta)$.
Moreover for $1 \leq a < b \leq r+s$ the ``transposition''
$\overline{(a\:b)}$ acts in the same way as the operator
$\Omega_{a,b}$ from (\ref{genom}).
Hence the action of $B_{r,s}(\delta)$ commutes with the action of 
$\mathfrak{g}$.
\end{Lemma}

\begin{proof}
We just check the first statement.
Let $\sigma, \tau$ and $\pi$ be walled Brauer diagrams such that
$\sigma \tau = \delta^t \pi$ for $t \geq 0$.
Proceeding exactly as in the proof of Lemma~\ref{easy}, one reduces to
checking the identity
\begin{equation}\label{iden}
\sum_{\bj \in I^{r+s}} \wt(_{\bi}\sigma_\bj) \wt({_\bj}\tau_\bk)
=\delta^t 
\wt({_\bi}\pi_\bk) 
\end{equation}
for all $\bi, \bk \in I^{r+s}$.
Both sides are zero unless
${_\bi}\pi_\bk$ is consistently coloured, so assume this is the case from now on.
It is convenient to think in terms of the 
composite diagram $\sigma|\tau$ that is $\sigma$ drawn under $\tau$ separated
by a horizontal boundary line on which the vertices at the top 
of $\sigma$ are
identified with the vertices at the bottom of $\tau$.
Let ${_\bi}\sigma_\bj\tau_\bk$ be the coloured diagram obtained 
by colouring the vertices
of $\sigma|\tau$ on the bottom edge by $\bi$, the vertices on the middle boundary
line by $\bj$,
and the vertices on the top edge by $\bk$.

The diagram $\pi$ is obtained from $\sigma|\tau$ by removing
the middle boundary line together with  
$t$ closed circles from the interior of the diagram. Enumerate these circles 
by $1,\dots,t$ in some order. Given $\bj = (j_1,\dots,j_t) \in I^t$,
there is a unique $\tilde\bmathj \in I^{r+s}$ such that
${_\bi}\sigma_{\tilde\bmathj}$ and ${_{\tilde\bmathj}}\tau_\bk$ are 
consistently coloured
and the $a$th circle in ${_\bi}\sigma_{\tilde\bmathj}\tau_{\bk}$
is of colour $j_a$ for each $a=1,\dots,t$.
In 
this notation, the identity we are trying to prove is equivalent to
\begin{equation}\label{simp}
\sum_{\bj \in I^t} \wt(_{\bi}\sigma_{\tilde\bmathj})\wt({_{\tilde\bmathj}}\tau_\bk)=
\delta^t 
\wt({_\bi}\pi_\bk),
\end{equation}
since all other summands on the left hand side of (\ref{iden})
are zero.
To prove (\ref{simp}), we claim for a fixed $\bj \in I^t$ that
\begin{equation}\label{claim}
\wt({_\bi}\sigma_{_{\tilde\bmathj}})\wt({_{\tilde\bmathj}}\tau_{\bk})=
(-1)^{|\bj|}
\wt({_\bi}\pi_\bk),
\end{equation}
which implies (\ref{simp})
since $\sum_{\bj \in I^t}(-1)^{|\bj|} = \delta^t$.
To prove (\ref{claim}), we compute the contribution to
$\wt({_\bi}\sigma_{_{\tilde\bmathj}})\wt({_{\tilde\bmathj}}\tau_{\bk})
\wt({_\bi}\pi_\bk)$ coming from 
each connected component of $\sigma|\tau$ in turn, showing for $1 \leq a \leq t$
that the $a$th circle contributes a factor of $(-1)^{|j_a|}$
and that all other components contribute the factor $1$.

To start with, consider the $a$th circle
for some $1 \leq a \leq t$, which of course 
is a connected component
of $\sigma|\tau$ which does not appear at all in $\pi$.
This circle is the concatenation $\alpha|\beta$
of two diagrams $\alpha$ and $\beta$, the first consisting of $z$ horizontal
strands from the top of $\sigma$ and the second consisting of $z$
horizontal strands from the bottom of $\tau$ for some $z \geq 1$.
Proceeding by induction on $z$, one checks that the number of proper crossings
in $\alpha$ plus the number of proper crossings in $\beta$ is equal to $(z-1)$.
Each of these $(z-1)$ proper crossings contributes a factor of $(-1)^{|j_a|}$ to
$\wt({_\bi}\sigma_{_{\tilde\bmathj}})\wt({_{\tilde\bmathj}}\tau_{\bk})\wt({_\bi}\pi_\bk)$, as
does 
each of the $z$ horizontal strands in $\beta$. 
Since all other connected components of $\sigma|\tau$ cross the circle $\alpha|\beta$
an even number of times,
we conclude that the total contribution to 
$\wt({_\bi}\sigma_{_{\tilde\bmathj}})\wt({_{\tilde\bmathj}}\tau_{\bk})
\wt({_\bi}\pi_{\bk})$
coming from {\em all} crossings or horizontal strands
involving at least one strand of $\alpha$ or $\beta$
is equal to $(-1)^{|j_a|}$, which is what we wanted.

There are two more types of connected component to consider, corresponding
to the horizontal and vertical strands in $\pi$. The argument for these is quite similar, so let us just
explain in the case of a 
horizontal strand $\gamma$ whose endpoints lie on the bottom edge of $\pi$.
The corresponding connected 
component of $\sigma|\tau$ is of the form $\alpha|\beta$,
where $\alpha$ consists of one or more strands from $\sigma$ and $\beta$
consists of zero or more strands from $\tau$.
Like in the previous paragraph, the number of proper crossings 
in $\alpha$ plus the number of proper crossings in $\beta$
is one less than the total number of horizontal strands in $\beta$.
These proper crossings become self-intersections in $\gamma$, and there is exactly one horizontal strand in $\gamma$ by definition, so all the contributions from these crossings and strands cancel out in the product
$\wt({_\bi}\sigma_{_{\tilde\bmathj}})\wt({_{\tilde\bmathj}}\tau_{\bk})
\wt({_\bi}\pi_{\bk})$.
Combined with Lemma~\ref{reide} this is enough to complete the proof.
\end{proof}

\phantomsubsection{Mixed Schur-Weyl duality}
The right actions constructed in 
Lemmas~\ref{easy}--\ref{act2} induce algebra homomorphisms
\begin{align}
\Phi_{r+s}^{m,n}:\C \Sigma_{r+s} &\rightarrow \End_{\mathfrak{g}}(V^{\otimes(r+s)})^{\op},\label{Phi}\\
\Psi_{r,s}^{m,n}:B_{r,s}(\delta) &\rightarrow \End_{\mathfrak{g}}(V^{\otimes r} \otimes W^{\otimes s})^{\op}.\label{Psi}
\end{align}
The first of these homomorphisms has been extensively studied in the literature
starting from the works of Sergeev \cite{Sergeev} and Berele-Regev \cite{BR}.
In particular, we have the following well-known result.

\begin{Theorem}[Sergeev, Berele-Regev]\label{onto}
The map 
$$
\Phi_{r,s}^{m,n}:\C \Sigma_{r+s} \rightarrow \End_{\mathfrak{g}}(V^{\otimes(r+s)})^{\op}
$$
from (\ref{Phi}) is surjective. It is injective if and only if
$r+s < (m+1)(n+1)$.
\end{Theorem}

\begin{proof}
The surjectivity follows from \cite[Remark 4.15]{BR}.
For the injectivity,
\cite[Theorem 3.20]{BR}
implies that $\Phi_{r,s}^{m,n}$ is injective
if and only if all partitions of $(r+s)$ are
{\em $(m,n)$-hook partitions} in the sense of \cite[Definition 2.3]{BR}, i.e.
their Young diagrams do not contain the box in
row $(m+1)$ and column $(n+1)$.
This is easily seen to be 
equivalent to the condition $r+s < (m+1)(n+1)$.
(All the results just cited from \cite{BR} were also established 
independently in \cite{Sergeev}.)
\end{proof}

We now want to prove that the analogue
of Theorem~\ref{onto} holds for the walled Brauer algebra (with $\Phi_{r+s}^{m,n}$ replaced by $\Psi_{r,s}^{m,n}$).
We begin by
recalling some standard facts about the 
tensor category of
finite dimensional $\mathfrak{g}$-supermodules.
Let $K, L, M$ and $N$ be finite dimensional vector superspaces.
We identify
$M \otimes N =  N \otimes M$
so that $m \otimes n = (-1)^{|m||n|} n \otimes m$ for homogeneous
$m\in M, n \in N$ (and $\otimes$ always means $\otimes_{\C}$).
Similarly, we identify
\begin{equation}\label{stupid}
\hom(M, N) \otimes \hom(K, L)
=
\hom(M \otimes K, N \otimes L)
\end{equation}
so that
$f \otimes g \in\hom(M, N) \otimes \hom(K, L)$
is identified with the homomorphism in $\hom(M\otimes K, N \otimes L)$
defined from 
$(f \otimes g)(m \otimes k) := (-1)^{|g||m|} f(m) \otimes g(k)$
for homogeneous $g$ and $m$.
Finally we identify $M$ 
with $M^{**}$, so a homogeneous vector $m \in M$
is identified with the function $m \in M^{**}$
defined from $m(\mu) = (-1)^{|m||\mu|}\mu(m)$ for all homogeneous $\mu \in M^*$.
As usual, for homogeneous $f \in \hom(M,N)$, we have 
the dual map $f^* \in \hom(N^*,M^*)$
defined from $f(\nu)(m) = (-1)^{|f||\nu|}\nu(f(m))$
for homogeneous $\nu \in N^*, m \in M$, and $f^{**} = f$.

The canonical isomorphism
$N \otimes L^* \stackrel{\sim}{\rightarrow} \hom(L,N),
f \mapsto \widetilde{f}$ defined from
$\widetilde{n \otimes \lambda}(l) = \lambda(l) n$
for $n \in N$, 
$\lambda \in L^*, l \in L$
induces the isomorphism
\begin{equation}\label{adj}
\operatorname{adj}:
\hom(M,N\otimes L^*) \stackrel{\sim}{\rightarrow}
\hom(M \otimes L, N),\qquad
f \mapsto \widehat{f}
\end{equation}
such that $\widehat{f}(m \otimes l) = \widetilde{f(m)}(l)$
for $m \in M, l \in L$. 
If
$M, N$ and $L$ are all $\mathfrak{g}$-supermodules then
this is a $\mathfrak{g}$-supermodule isomorphism, so
it restricts to an isomorphism
$\hom_{\mathfrak{g}}(M, N \otimes L^*) \cong \hom_{\mathfrak{g}}(M \otimes L, N)$.
This is the canonical adjunction making
$(? \otimes L, ? \otimes L^*)$ into an adjoint pair of functors on the category
of finite dimensional $\mathfrak{g}$-supermodules.

\begin{Lemma}\label{duals}
Let $K, L, M$ and $N$ be finite dimensional $\mathfrak{g}$-supermodules.
The map $\flip$ defined by the following commutative diagram
is a $\mathfrak{g}$-supermodule isomorphism:
$$
\begin{CD}
\hom(M\otimes K, N \otimes L)
&@>\flip>> &
\hom(M \otimes L^*, N \otimes K^*)\\
@|&&@|\\
\hom(M, N) \otimes \hom(K, L)&@>{f\otimes g \mapsto f \otimes g^*}>>&
\hom(M,N) \otimes \hom(L^*,K^*).
\end{CD}
$$
\end{Lemma}

\begin{proof}
Let us instead define $\flip$ so that the following diagram commutes:
$$
\begin{CD}
\hom(M\otimes K, N \otimes L)&@>\flip>> &\hom(M \otimes L^*, N \otimes K^*)\\
@|&&@VV\operatorname{adj}V\\
\hom(M\otimes K, N\otimes L^{**})&&&&\hom(M\otimes L^* \otimes K, N)\\
@V\operatorname{adj}VV&&@|\\
\hom(M\otimes K \otimes L^*, N)&@=&\hom(M\otimes K\otimes L^*, N).\\
\end{CD}
$$
Recalling (\ref{adj}),
all of the other maps in this diagram are $\mathfrak{g}$-supermodule isomorphisms,
hence $\flip$ is one too.
To complete the proof of the lemma, it remains to take
$f \in \hom(M,N)$ and $g \in \hom(K, L)$
and check that $\flip(f \otimes g) = f \otimes g^*$, which is a routine calculation.
\end{proof}

Like in (\ref{stupid}), it is natural to identify $W^{\otimes s}$ with 
$(V^{\otimes s})^*$ so that
\begin{equation}\label{woopdeedoo}
w_{\bi}(v_{\bj}) = (-1)^{p(\bi)} \delta_{\bi,\bj}
\end{equation}
for all $\bi,\bj \in I^s$, recalling (\ref{pdef}).
If we apply Lemma~\ref{duals} to
$M = N = V^{\otimes r}$ and $K = L = V^{\otimes s}$,
we see that the map $\flip_{r,s}$ defined by the following commutative
diagram is a $\mathfrak{g}$-supermodule isomorphism:
\begin{equation}\label{nuals}
\begin{CD}
\End(V^{\otimes (r+s)})
&@>\flip_{r,s}>> &
\End(V^{\otimes r} \otimes W^{\otimes s})
\\
@|&&@|\\
\End(V^{\otimes r}) \otimes \End(V^{\otimes s})
&@>{f\otimes g \mapsto f \otimes g^*}>>&
\End(V^{\otimes r}) \otimes \End(W^{\otimes s}).
\end{CD}
\end{equation}
The main point now is to see that this map is consistent
with the linear isomorphism 
$\flip_{r,s}:\C \Sigma_{r+s} \rightarrow B_{r,s}(\delta)$ from
(\ref{oldflip}) in the following sense.

\begin{Lemma}\label{lastc}
The following diagram commutes:
\begin{equation}\label{ad}
\begin{CD}
\C \Sigma_{r+s}&@>\flip_{r,s}>>&B_{r,s}(\delta)\\
@V\Phi_{r+s}^{m,n} VV&&@VV\Psi_{r,s}^{m,n} V\\
\End(V^{\otimes(r+s)})&@>\flip_{r,s}>>&\End(V^{\otimes r} \otimes W^{\otimes s}).
\end{CD}
\end{equation}
\end{Lemma}

\begin{proof}
This is just a direct calculation but the signs are 
tricky so we go through the details carefully. 
For $\bi,\bj \in I^r$, let $e_{\bi,\bj} \in  \End(V^{\otimes r})$
be the matrix unit defined from 
$e_{\bi,\bj}(v_\bk) = \delta_{\bj,\bk} v_\bi$.
Similarly for $\bi,\bj \in I^s$, let $f_{\bi,\bj} \in  \End(W^{\otimes s})$
be the matrix unit defined from 
$f_{\bi,\bj}(w_\bk) = \delta_{\bj,\bk} w_\bi$.
Using (\ref{woopdeedoo}), it is straightforward to check for $\bi,\bj \in I^s$ that
\begin{equation}\label{lastt}
(e_{\bi,\bj})^* = 
(-1)^{(|\bi|+|\bj|) |\bi|+p(\bi)+p(\bj)}
f_{\bj,\bi}.
\end{equation}
Now take $\sigma \in \Sigma_{r+s}$ and let $\tau := \flip_{r,s}(\sigma) \in B_{r,s}(\delta)$.
Using (\ref{easyact}) and the identification of
$\End(V^{\otimes (r+s)})$ with $\End(V^{\otimes r}) \otimes \End(V^{\otimes s})$,
we see that
\begin{equation}\label{first}
\Phi_{r+s}^{m,n}(\sigma) = 
\sum_{\bi,\bj \in I^{r+s}}
(-1)^{(|\bi^{\mathrm R}|+|\bj^{\mathrm R}|) |\bi^{\mathrm L}|}
\wt({_\bi}\sigma_\bj)\, e_{\bj^{\mathrm L}, \bi^{\mathrm L}} \otimes e_{\bj^{\mathrm R}, \bi^{\mathrm R}}.
\end{equation}
Similarly using (\ref{hardact}) we have that
\begin{equation}\label{second}
\Psi_{r,s}^{m,n}(\tau) = 
\sum_{\bi,\bj \in I^{r+s}}
(-1)^{(|\bi^{\mathrm R}|+|\bj^{\mathrm R}|) |\bi^{\mathrm L}|}
\wt({_\bi}\tau_\bj)\, e_{\bj^{\mathrm L}, \bi^{\mathrm L}} \otimes f_{\bj^{\mathrm R}, \bi^{\mathrm R}}.
\end{equation}
Reindex the summation on the right hand side of (\ref{first})
to switch the roles of $\bi^{\mathrm R}$ and $\bj^{\mathrm R}$, then apply the map $\flip_{r,s}$ from 
(\ref{nuals})
and use (\ref{lastt})
to deduce that
$$
\flip_{r,s}(\Phi_{r+s}^{m,n}(\sigma)) = \!\!\!\!
\sum_{\bi,\bj \in I^{r+s}}\!\!\!\!
(-1)^{(|\bi^{\mathrm R}|+|\bj^{\mathrm R}|) |\bi|+p(\bi^{\mathrm R})+p(\bj^{\mathrm R})}
\wt({_{\bi^{\mathrm L}\bj^{\mathrm R}}}\sigma_{\bj^{\mathrm L} \bi^{\mathrm R}})\, e_{\bj^{\mathrm L}, \bi^{\mathrm L}} \otimes f_{\bj^{\mathrm R}, \bi^{\mathrm R}}.
$$
Comparing with (\ref{second})
the proof of the lemma is now reduced to verifying the combinatorial identity
\begin{equation}\label{Reduce}
(-1)^{(|\bi^{\mathrm R}|+|\bj^{\mathrm R}|) |\bi^{\mathrm R}|+p(\bi^{\mathrm R})+p(\bj^{\mathrm R})}
\wt({_{\bi^{\mathrm L}\bj^{\mathrm R}}}\sigma_{\bj^{\mathrm L} \bi^{\mathrm R}})
=
\wt({_\bi}\tau_\bj)
\end{equation}
for any $\bi,\bj \in I^{r+s}$.
It is clear that both sides of (\ref{Reduce}) are zero
unless ${_\bi}\tau_\bj$ is consistently coloured.
Assume this is the case from now on.

Partition the set
$\{r+1,\dots,r+s\}$ as
$\{a_1 < \cdots < a_k\} \sqcup\{b_1 < \cdots < b_l\}$
so that the vertices $a_1,\dots,a_k$ 
on the bottom edge
are at the ends of Horizontal strands of $\tau$
and the vertices $b_1,\dots,b_l$ on the bottom edge
are at the ends of Vertical strands of $\tau$.
Let $\bi^{\mathrm{RH}}:= (i_{a_1},\dots,i_{a_k})$
and $\bi^{\mathrm{RV}} := (i_{b_1},\dots,i_{b_l})$.
So the tuple $\bi^{\mathrm{RH}}$ lists the colours of all the horizontal segments
in ${_\bi}\tau_\bj$
whose endpoints lie on the bottom edge
and we have that
\begin{equation}\label{id1}
|\bi^{\mathrm R}| = |\bi^{\mathrm{RH}}| + |\bi^{\mathrm{RV}}|,
\qquad
p(\bi^{\mathrm R}) = p(\bi^{\mathrm{RH}}) + |\bi^{\mathrm{RH}}| |\bi^{\mathrm{RV}}| + 
p(\bi^{\mathrm{RV}}).
\end{equation}
Similarly we define $\bj^{\mathrm{RH}} \in I^k$ and $\bj^{\mathrm{RV}} \in I^l$
so that $\bj^{\mathrm{RH}}$ lists the colours of all the horizontal line segments in ${_\bi}\tau_\bj$ whose endpoints lie on the top edge
and 
\begin{equation}\label{id2}
|\bj^{\mathrm R}| = |\bj^{\mathrm{RH}}| + |\bj^{\mathrm{RV}}|,
\qquad
p(\bj^{\mathrm R}) = p(\bj^{\mathrm{RH}}) + |\bj^{\mathrm{RH}}| |\bj^{\mathrm{RV}}| + 
p(\bj^{\mathrm{RV}}).
\end{equation} 
The assumption that ${_\bi}\tau_\bj$ is consistently
coloured implies that 
\begin{equation}\label{id3}
|\bi^{\mathrm{RV}}| = |\bj^{\mathrm{RV}}|.
\end{equation}

Finally let us assume the diagrams representing $\sigma$ and $\tau$
are reduced in the sense that the total number of crossings is 
as small as possible and each strand crosses the wall at most once.
Then by the definition (\ref{wtdef}) we have that
\begin{equation}\label{a}
\wt({_{\bi^{\mathrm L}\bj^{\mathrm R}}}\sigma_{\bj^{\mathrm L} \bi^{\mathrm R}}) = \prod_c (-1)^{|c|}
\end{equation}
where the product is over all coloured crossings $c$ in the diagram
${_{\bi^{\mathrm L}\bj^{\mathrm R}}}\sigma_{\bj^{\mathrm L} \bi^{\mathrm R}}$. 
Similarly we have that 
\begin{equation}\label{b}
(-1)^{|\bi^{\mathrm{RH}}|}
\wt({_\bi}\tau_{\bj}) = \prod_{c'} (-1)^{|c'|}
\end{equation}
where the product is over all coloured crossings $c'$ in the diagram
${_\bi}\tau_\bj$.
When flipping from $\sigma$ to $\tau$, any
strand of $\sigma$ which does not cross the wall 
transforms into a vertical strand in $\tau$ 
which has exactly the same crossings with other strands after the flip as it did before.
However given two strands in $\sigma$ 
which both cross the wall (transforming to two horizontal strands in $\tau$), 
they cross after the flip if and only if they did {\em not}
cross before the flip, and vice versa.
Combined with (\ref{a})--(\ref{b}), this is enough to see that
\begin{equation*}
\wt({_{\bi^{\mathrm L}\bj^{\mathrm R}}}\sigma_{\bj^{\mathrm L} \bi^{\mathrm R}})
=
(-1)^{p(\bi^{\mathrm{RH}}) + |\bi^{\mathrm{RH}}||\bj^{\mathrm{RH}}|+ p(\bj^{\mathrm{RH}})}
(-1)^{|\bi^{\mathrm{RH}}|} \wt({_\bi}\tau_\bj).
\end{equation*}
Comparing with (\ref{Reduce}),
it remains to observe that
$$
(|\bi^{\mathrm R}|+|\bj^{\mathrm R}|) |\bi^{\mathrm R}|+p(\bi^{\mathrm R})+p(\bj^{\mathrm R})
=p(\bi^{\mathrm{RH}}) + |\bi^{\mathrm{RH}}||\bj^{\mathrm{RH}}|+ p(\bj^{\mathrm{RH}})+|\bi^{\mathrm{RH}}|
$$
in $\Z_2$,
which follows formally from the identities
(\ref{id1})--(\ref{id3}).
\end{proof}

\begin{Theorem}\label{main1}
The map $$
\Psi_{r,s}^{m,n}:B_{r,s}(\delta) \rightarrow \End_{\mathfrak{g}}(V^{\otimes r}\otimes W^{\otimes s})^{\op}
$$ 
from (\ref{Psi}) is surjective. It is injective if and only if
$r+s < (m+1)(n+1)$.
\end{Theorem}

\begin{proof}
We know from Lemma~\ref{duals}
that the bottom map in the diagram (\ref{ad})
is a $\mathfrak{g}$-supermodule isomorphism, hence it
maps the subspace $\End_{\mathfrak{g}}(V^{\otimes (r+s)})$
isomorphically onto the subspace
$\End_{\mathfrak{g}}(V^{\otimes r} \otimes W^{\otimes s})$.
The theorem follows from this observation combined with Lemma~\ref{lastc} 
and Theorem~\ref{onto}.
\end{proof}

\section{Isomorphism theorem and applications}\label{smain}

In order to prove the main results of the article, we still need to
show that
the graded walled Brauer algebra $B_R(\delta)$ is isomorphic to
$B_{r,s}(\delta)$.
Throughout the section we fix integers $m,n \geq 0$, 
and set $\delta := m-n$, $k := \min(m,n)$. 
This is equivalent to fixing $\delta \in \Z$, $k \geq 0$, and letting
$m := k, n := k-\delta$ if $\delta \leq 0$
or $m := k +\delta, n := k$ if $\delta \geq 0$.

\phantomsubsection{The main isomorphism theorem}
Recall the algebra $B_R(\delta) / B_R(\delta)_{> k}$ from
(\ref{subalgbas}), and
the maps (\ref{iotai}) and (\ref{leg1})--(\ref{leg2}).

\begin{Theorem}\label{isotheorem}
For each $R \in \Seq_{r,s}$,
there is a surjective algebra homomorphism
\begin{equation}\label{thetaonto}
\Theta^{(k)}_{R}:B_{r,s}(\delta) \twoheadrightarrow
B_R(\delta) / B_R(\delta)_{> k}
\end{equation}
with the same kernel as the map $\Psi^{m,n}_{r,s}$ from
Theorem~\ref{main1}.
Moreover the following diagrams commute.
\begin{align*}
\begin{CD}
 B_{r,s}(\delta) &@>\iota_{r,s;i}^{r+1,s} >>&
 B_{r+1,s}(\delta) \\
@V\Theta_{R}^{(k)} VV&&@VV\Theta^{(k)}_{RE} V\\
 B_R(\delta) / B_R(\delta)_{> k}&@>\iota_{R;i}^{RE} >>& B_{RE}(\delta)/
 B_{RE}(\delta)_{> k}
\end{CD}
\end{align*}\begin{align*}
\begin{CD}
B_{r,s}(\delta) &@>\iota_{r,s;i}^{r,s+1} >>&
B_{r,s+1}(\delta) \\
@V\Theta^{(k)}_{R} VV&&@VV\Theta^{(k)}_{RF} V\\
B_{R}(\delta) / B_R(\delta)_{> k}&@>\iota_{R;i}^{RF} >>&
B_{RF}(\delta) / B_{RF}(\delta)_{> k}
\end{CD}
\end{align*}
\end{Theorem}

\begin{proof}
We assume some notation and results 
still to be formulated; see the next three subsections.
The homomorphism $\Theta^{(k)}_R$ 
is the composite $\Upsilon_R^{(k)}\circ \Pi_R^{(k)}\circ \Psi_R^{(k)}$
of certain maps
\begin{align}\label{i1}
\Psi_R^{(k)}:B_{r,s}(\delta) &\twoheadrightarrow
\End_{\mathfrak{g}}(\lonestar R\,\C)^{\op},\\\label{i2}
\Pi_R^{(k)}:
\End_{\mathfrak{g}}(\lonestar R\,\C)^{\op}
&\stackrel{\sim}{\rightarrow}
\End_{K(m|n)}(\lonestar R\,L(\zeta))^{\op},\\
\Upsilon_R^{(k)}:\End_{K(m|n)}(\lonestar R\,L(\zeta))^{\op}
&\stackrel{\sim}{\rightarrow}
B_{r,s}(\delta) / B_{r,s}(\delta)_{> k}\label{i3}
\end{align}
constructed in 
Theorems~\ref{iso1}, \ref{iso2} and \ref{iso3} below. 
The maps $\Upsilon_R^{(k)}$ and $\Pi_R^{(k)}$  are
isomorphisms and $\Psi_R^{(k)}$ is surjective, hence $\Theta_R^{(k)}$
is surjective.
Moreover it is clear
from its explicit definition below that
$\Psi_R^{(k)}$ has the same kernel as $\Psi^{m,n}_{r,s}$,
hence $\Theta^{(k)}_R$ has this as its kernel too.
Finally the commutativity of the diagrams follows immediately from the
commutativity
of the diagrams in the
statements of Theorems~\ref{iso1}, \ref{iso2} and \ref{iso3}.
\end{proof}

\begin{Corollary}\label{isocor}
For any $R \in \Seq_{r,s}$ 
we have that $\dim B_R(\delta)_{> k} = \dim \ker \Psi^{m,n}_{r,s}$.
So $B_R(\delta)_{> k}$ is zero if and only if
$r+s<(m+1)(n+1)$, in which case $\Theta_{R}^{(k)}$ is 
an {isomorphism}
\begin{equation}\label{thiso}
\Theta_R^{(k)}:B_{r,s}(\delta) \stackrel{\sim}{\rightarrow}
B_R(\delta).
\end{equation}
\end{Corollary}

\begin{proof}
It is clear from the definition that $B_R(\delta)_{> k} = \{0\}$
if $k$ is sufficiently large. In that case $m$ and $n$ are large too,
so $\Psi_{r,s}^{m,n}$ is injective by Theorem~\ref{main1}.
We deduce that $\Theta_R^{(k)}$ is an 
isomorphism 
$B_{r,s}(\delta) \stackrel{\sim}{\rightarrow} B_R(\delta)$
for $k \gg 0$.
This shows that $\dim B_{r,s}(\delta) = \dim B_R(\delta)$.
Combined with Theorem~\ref{isotheorem}
we deduce for any $k$ that $\dim B_{r,s}(\delta)_{> k} = \dim\ker
\Psi^{m,n}_{r,s}$.
\end{proof}

\begin{Remark}\rm\label{drawback}
There are some implicit choices involved in the definition of 
$\Theta^{(k)}_R$; these come from the isomorphism (\ref{i2})
which we will deduce from results in \cite{BS4}.
We conjecture that these choices can be made so that 
the homomorphisms $\Theta_R^{(k)}$ 
stabilise as $k \rightarrow \infty$. More precisely, 
we expect there are isomorphisms
$$
\Theta_R:B_{r,s}(\delta) \stackrel{\sim}{\rightarrow} B_R(\delta)
$$
for all $R$
such that $\Theta_R^{(k)}$ is the composition of $\Theta_R$ 
and the quotient map $B_R(\delta) \twoheadrightarrow B_R(\delta)
/ B_R(\delta)_{> k}$ for each $k \geq 0$.
The arguments in \cite{BS4} are not well-suited to allowing $k$ to vary,
so we are not able to see this using the present approach.
\end{Remark}

\phantomsubsection{\boldmath The epimorphism $\Psi_R^{(k)}$}
In this subsection we construct the first 
map
(\ref{i1}) needed in the proof of Theorem~\ref{isotheorem}.
Let $\mathfrak{g} := \mathfrak{gl}_{m|n}(\C)$ as in $\S$\ref{s3}.
We begin by recalling the definition of
the {\em special projective functors}
on a certain category of finite dimensional $\mathfrak{g}$-supermodules.
Let $\mathfrak{t}$ be the Cartan subalgebra of $\mathfrak{g}$
consisting of diagonal matrices.
Let $\eps_1,\dots,\eps_{m+n}$ be the usual basis for $\mathfrak{t}^*$, 
i.e. $\eps_r$ picks 
out the $r$th diagonal entry of a diagonal matrix.
Let 
$(.,.)$ be the symmetric bilinear form on $\mathfrak{t}^*$ defined by
$(\eps_i, \eps_j) = (-1)^{|i|} \delta_{i,j}$, and set
$X := \bigoplus_{i=1}^{m+n} \Z \eps_i \subset \mathfrak{t}^*$.
As in the introduction of \cite{BS4}, we 
let $\mathscr F(m|n)$ be the category of all $\mathfrak{g}$-supermodules $M$ such that
\begin{itemize}
\item[(1)] $M$ is finite dimensional;
\item[(2)] $M$ decomposes as
$M = \bigoplus_{\la \in X} M_\la$,
where $M_\la$ denotes the $\la$-weight space of $M$ with respect to $\mathfrak{t}$;
\item[(3)]
the $\Z_2$-grading on $M_\la$ is 
concentrated in degree $(\lambda,\eps_{m+1}+\cdots+\eps_{m+n})
\pmod{2}$
for each $\la \in X$.
\end{itemize}
The morphisms in $\mathscr F(m|n)$ are all (necessarily even) $\mathfrak{g}$-supermodule homomorphisms.
Note further that
$\sF(m|n)$ is an abelian category closed under tensor product, and it contains
both the natural supermodule $V$ and its dual $W$.
We refer to $\sF(m|n)$ as the category of {\em rational $\mathfrak{g}$-supermodules}.

Let $\lonestar E := ? \otimes V$ and $\lonestar F := ? \otimes W$,
viewed as endofunctors of $\mathscr F(m|n)$.
Observe that
multiplication by the element $\Omega$ from (\ref{om}) defines natural 
endomorphisms
of both $\lonestar E M$ and $\lonestar F M$ for every $M \in \mathscr
F(m|n)$.
Let $\lonestar E_i$ and $\lonestar F_i$ be the subfunctors
of $\lonestar E$ and $\lonestar F$, respectively, defined on
$M \in \mathscr F(m|n)$ 
by declaring that
\begin{itemize}
\item[(1)] 
$\lonestar E_i M$ is the generalised $i$-eigenspace of $\Omega$ on
$\lonestar E M$;
\item[(2)]
$\lonestar F_i M$ is the generalised
$(-i-\delta)$-eigenspace of $\Omega$ on $\lonestar F M$.
\end{itemize}
By \cite[Corollary 2.9, Lemma 2.10]{BS4}, we have that
\begin{equation}\label{Dec}
\lonestar E = \bigoplus_{i \in \Z} \lonestar E_i,
\qquad\qquad
\lonestar F = \bigoplus_{i \in \Z} \lonestar F_i.
\end{equation}

Let
\begin{align}\label{rhol}
b_{r,s}^{r+1,s}: V^{\otimes r} \otimes W^{\otimes s} \otimes V \stackrel{\sim}{\rightarrow} V^{\otimes (r+1)} \otimes W^{\otimes s},\\
b_{r,s}^{r,s+1}: V^{\otimes r} \otimes W^{\otimes s} \otimes W \stackrel{\sim}{\rightarrow} V^{\otimes r} \otimes W^{\otimes (s+1)}\label{rhor}
\end{align}
be the canonical isomorphisms defined by supercommuting
$W^{\otimes s}$ past the rightmost copy of $V$ or $W$, respectively.
We stress $b_{r,s}^{r,s+1}$ is {\em not} the identity (unless $s = 0$).
The following lemma connects the special projective functors on $\mathscr F(m|n)$
to the $i$-restriction functors from
(\ref{fred1})--(\ref{fred2}).

\begin{Lemma}\label{eek}
For each $i \in \Z$, the restrictions of the maps $b_{r,s}^{r+1,s}$ 
and $b_{r,s}^{r,s+1}$
just defined give
$(U(\mathfrak{g}), B_{r,s}(\delta))$-bimodule isomorphisms
\begin{align*}
b_{r,s}^{r+1,s}&:\lonestar E_i (V^{\otimes r} \otimes W^{\otimes s}) \stackrel{\sim}{\rightarrow}
(V^{\otimes (r+1)} \otimes W^{\otimes s}) 1_{r,s;i}^{r+1,s},\\
b_{r,s}^{r,s+1}&:\lonestar F_i (V^{\otimes r} \otimes W^{\otimes s}) \stackrel{\sim}{\rightarrow}
(V^{\otimes r} \otimes W^{\otimes (s+1)}) 1_{r,s;i}^{r,s+1}.
\end{align*}
\end{Lemma}

\begin{proof}
We just explain for $b_{r,s}^{r+1,s}$; the other case is similar.
It is clear that the map (\ref{rhol}) is a
$(U(\mathfrak{g}), B_{r,s}(\delta))$-module homomorphism,
where the $B_{r,s}(\delta)$-module structure on the
the left hand side is obtained by applying the functor $\lonestar E = ? \otimes V$
to its action on $V^{\otimes r} \otimes W^{\otimes s}$, and 
on the right hand side
it is the restriction of the $B_{r+1,s}(\delta)$-module structure
via the homomorphism $\iota_{r,s}^{r+1,s}$ from (\ref{iotas}).
By definition,
$\lonestar E_i(V^{\otimes r} \otimes W^{\otimes s})$ is the 
generalised $i$-eigenspace of the endomorphism 
of $V^{\otimes r} \otimes W^{\otimes s} \otimes V$
defined by multiplication by
$\sum_{1 \leq a \leq r+s} \Omega_{a, r+s+1}$. This
maps under $b_{r,s}^{r+1,s}$ to the generalised $i$-eigenspace of the
endomorphism of $V^{\otimes (r+1)} \otimes W^{\otimes s}$ defined by multiplication by
$\sum_{1 \leq a \leq r} \Omega_{a, r+1} + \sum_{r+2 \leq b \leq r+s+1} \Omega_{r+1,b}$.
The latter is the same as the endomorphism defined by
$x_{r,s}^{r+1,s} \in B_{r+1,s}(\delta)$ from (\ref{xl}); this follows by the second statement of Lemma~\ref{act2}.
Hence $b_{r,s}^{r+1,s}$ maps
$\lonestar E_i(V^{\otimes r} \otimes W^{\otimes s})$
isomorphically onto
the generalised $i$-eigenspace of $x_{r,s}^{r+1,s}$
on $V^{\otimes(r+1)} \otimes W^{\otimes s}$, which is exactly the subspace
$(V^{\otimes (r+1)} \otimes W^{\otimes s}) 1_{r,s;i}^{r+1,s}$.
\end{proof}

For $R =R^{(1)}\cdots R^{(r+s)} \in \Seq_{r,s}$, 
we can interpret
$\lonestar R
=
\lonestar R^{(r+s)} \cdots \lonestar R^{(1)}$
as an 
endofunctor
of $\mathscr F(m|n)$.
Let $\C$ be the trivial $\mathfrak{g}$-supermodule.
Mimicking the definitions (\ref{EP})--(\ref{FP}),
the functors $\lonestar E_i$ and $\lonestar F_i$ induce non-unital
algebra homomorphisms
\begin{align}
\iota_{R;i}^{RE}:&\End_{\mathfrak{g}}(\lonestar R\,\C)^{\op}
\rightarrow 
\End_{\mathfrak{g}}(\lonestar (RE)\,\C)^{\op},\qquad
f \mapsto \lonestar E_i(f),\label{EPg}\\
\iota_{R;i}^{RF}:&\End_{\mathfrak{g}}(\lonestar R\,\C)^{\op}
\rightarrow \End_{\mathfrak{g}}(\lonestar (RF)\,\C)^{\op}, \qquad
f \mapsto \lonestar F_i(f).\label{FPg}
\end{align}
We recursively define a $\mathfrak{g}$-supermodule isomorphism
\begin{equation*}
\psi_R:\lonestar 
R\, \C \stackrel{\sim}{\rightarrow} V^{\otimes r} \otimes W^{\otimes s}
\end{equation*}
as follows. If $r=s=0$ then we 
take $\psi_R$ to be the identity map $\C \rightarrow \C$.
For the induction step,
assume 
$\psi_R:\lonestar R\,\C\stackrel{\sim}{\rightarrow} V^{\otimes r} \otimes W^{\otimes s}$ has been defined already, then set
\begin{align*}
\psi_{RE} := b_{r,s}^{r+1,s} \circ \lonestar E(\psi_R):\lonestar(RE)\,\C\rightarrow V^{\otimes (r+1)} \otimes W^{\otimes s},\\
\psi_{RF} := b_{r,s}^{r,s+1}\circ \lonestar F(\psi_R):\lonestar(RF)\,\C\stackrel{\sim}{\rightarrow} V^{\otimes r} \otimes W^{\otimes (s+1)}.
\end{align*}

\begin{Theorem}\label{iso1}
For $R \in \Seq_{r,s}$, there is a surjective algebra homomorphism
\begin{equation*}
\Psi^{(k)}_{R}:B_{r,s}(\delta) \twoheadrightarrow
\End_{\mathfrak{g}}(\lonestar R\, \C)^{\op},
\qquad
\sigma \mapsto \psi_R^{-1} \circ \Psi_{r,s}^{m,n}(\sigma) \circ \psi_R.
\end{equation*}
Moreover 
the following diagrams commute for each $i \in \Z$.
\begin{align*}
\begin{CD}
B_{r,s}(\delta)&@>\iota_{r,s;i}^{r+1,s} >>&B_{r+1,s}(\delta)\\
@V\Psi^{(k)}_{R} VV&&@VV\Psi^{(k)}_{RE}V\\
\End_{\mathfrak{g}}(\lonestar R\, \C)^{\op}&@>\iota^{RE}_{R;i}>>&\End_{\mathfrak{g}}(\lonestar(RE)\, \C)^{\op}
\end{CD}\end{align*}\begin{align*}
\begin{CD}
B_{r,s}(\delta)&@>\iota_{r,s;i}^{r,s+1} >>&B_{r,s+1}(\delta)\\
@V\Psi^{(k)}_{R} VV&&@VV\Psi^{(k)}_{RF}V\\
\End_{\mathfrak{g}}(\lonestar R\, \C)^{\op}&@>\iota_{R;i}^{RF}>>&\End_{\mathfrak{g}}(\lonestar(RF)\, \C)^{\op}
\end{CD}
\end{align*}
\end{Theorem}

\begin{proof}
The first part follows from Theorem~\ref{main1}.
The commutative diagrams follow from Lemma~\ref{eek} and the definitions.
\end{proof}

\phantomsubsection{\boldmath The isomorphism $\Pi^{(k)}_R$}
To construct the second map (\ref{i2}) needed in the proof of
Theorem~\ref{isotheorem}, we apply the main results of
\cite{BS4}.
Recall that the {\em dominance ordering}
on $X$ is defined by $\la \preceq \mu$ if $\mu-\la$ is a sum of 
the positive roots $\{\eps_i-\eps_j\:|\:1 \leq i < j \leq m+n\}$.
Also let
\begin{equation*}
\rho :=
\sum_{r=1}^m (1-r) \eps_r + \sum_{s=1}^n (m-s) \eps_{m+s}.
\end{equation*}
It is well known that the isomorphism classes of irreducible supermodules
in the category $\mathscr F(m|n)$ from the previous subsection
are 
parametrised by their highest weights with respect to the dominance
ordering by 
the set
\begin{equation}
\Xi
:= \left\{
\la \in X\:\bigg|\:
\begin{array}{c}
\:\:\,(\la+\rho,\eps_1) > \cdots > (\la+\rho,\eps_m),\\
(\la+\rho,\eps_{m+1}) < \cdots < (\la+\rho,\eps_{m+n})
\end{array}
\right\}\label{xidef}
\end{equation}
of {\em dominant integral weights}.
We denote the irreducible corresponding to $\la \in \Xi$ by $L(\la)$.

There is another weight dictionary which allows us to identify $\Xi$
with a set of weight diagrams.
Given $\la \in \Xi$, let
\begin{align}\label{otherw}
I_{\down}(\la) &:= \{(\la+\rho,\eps_1),\dots,(\la+\rho,\eps_m)\},\\
I_\up(\la) &:= \Z \setminus \{(\la+\rho,\eps_{m+1}),\dots,(\la+\rho,\eps_{m+n})\}.\label{totherw}
\end{align}
Then identify $\la$ with the weight diagram associated to these
sets via (\ref{dict}).
For example, the zero weight (which parametrises the trivial
$\mathfrak{g}$-supermodule) is identified with the diagram
\begin{equation}\label{groundstate}
\!\!\zeta := \Bigg\{
\begin{array}{ll}
\hspace{88mm}
&\text{if $\delta \geq 0$}\\\\
&\text{if $\delta \leq 0$}\\
\end{array}
\begin{picture}(0,40)
\put(-300,12.5){$\cdots$}
\put(-74,12.5){$\cdots$}
\put(-263.5,14.2){$\overbrace{\phantom{hellow orl}}^{n}$}
\put(-204,14.2){$\overbrace{\phantom{hellow orl}}^{\delta=m-n}$}
\put(-284,15.2){\line(1,0){207}}
\put(-283.7,10.6){$\scriptstyle\up$}
\put(-263.7,15.3){$\scriptstyle\down$}
\put(-243.7,15.3){$\scriptstyle\down$}
\put(-223.7,15.3){$\scriptstyle\down$}
\put(-204.2,13.3){$\scriptstyle\times$}
\put(-184.2,13.3){$\scriptstyle\times$}
\put(-164.2,13.3){$\scriptstyle\times$}
\put(-143.7,10.6){$\scriptstyle\up$}
\put(-123.7,10.6){$\scriptstyle\up$}
\put(-103.7,10.6){$\scriptstyle\up$}
\put(-83.7,10.6){$\scriptstyle\up$}
\end{picture}
\begin{picture}(0,0)
\put(-300,-12.6){$\cdots$}
\put(-74,-12.6){$\cdots$}
\put(-162.5,3){$_0$}
\put(-203.5,-17){$\underbrace{\phantom{hellow orl}}_{m}$}
\put(-144,-17){$\underbrace{\phantom{hellow orl}}_{-\delta=n-m}$}
\put(-284,-10){\line(1,0){207}}
\put(-283.7,-14.6){$\scriptstyle\up$}
\put(-263.7,-14.6){$\scriptstyle\up$}
\put(-243.7,-14.6){$\scriptstyle\up$}
\put(-223.7,-14.6){$\scriptstyle\up$}
\put(-203.7,-9.9){$\scriptstyle\down$}
\put(-183.7,-9.9){$\scriptstyle\down$}
\put(-163.7,-9.9){$\scriptstyle\down$}
\put(-143,-12.6){$\circ$}
\put(-123,-12.6){$\circ$}
\put(-103,-12.6){$\circ$}
\put(-88.7,-14.6){$\scriptstyle\up$}
\end{picture}
\end{equation}
\vspace{4mm}

\noindent
(where there are infinitely many vertices labelled $\up$ to the left
and the right).

Let $K(m|n)$ be the subalgebra $\bigoplus_{\la,\mu \in \Xi} e_\la K
e_\mu$
of the universal arc algebra from $\S$\ref{s4}.
As $\Xi$ is a union of blocks, the algebra $K(m|n)$ is just the sum of
some blocks of $K$ from (\ref{blockdec}).
Let $\Mod{K(m|n)}$ be the full subcategory of $\Mod{K}$
consisting of all modules $M$ such that $M = \bigoplus_{\la \in \Xi}
e_\la M$.
Up to grading shift, the irreducible objects in this category are 
the one-dimensional modules $\{L(\la)\:|\:\la \in \Xi\}$.

The set $\Xi$ is a connected component in
the labelled directed graph defined by (\ref{break1})--(\ref{break2}). It follows
that the special projective functors $\lonestar E_i$ and $\lonestar
F_i$ from (\ref{projf}) restrict to well-defined endofunctors of $\Mod{K(m|n)}$.
These functors are defined by tensoring with certain bimodules.
Forgetting the gradings on these bimodules,
we can view them
instead as endofunctors of 
the category $\mod{K(m|n)}$ of
locally unital, finite dimensional
left $K(m|n)$-modules. 
The main result of \cite{BS4} can now be formulated as follows.

\begin{Theorem}[{\cite[Theorem 5.11, Theorem 5.12]{BS4}}]\label{BS4a}
There is an equivalence of categories
\begin{equation*}
\E:\mathscr F(m|n) \rightarrow \mod{K(m|n)}.
\end{equation*}
Under the equivalence, the irreducible $\mathfrak{g}$-supermodule $L(\la)$
of highest weight $\la \in \Xi$ corresponds to the irreducible
$K(m|n)$-module with the same name.
Moreover for each $i \in \Z$ there are canonical isomorphisms
of functors 
\begin{align}\label{lastlest1}
\lonestar E_i \circ \E&\cong \E \circ \lonestar E_i
:\mathscr F(m|n) \rightarrow \mod{K(m|n)},\\
\lonestar F_i\circ \E &\cong \E\circ \lonestar F_i
:\mathscr F(m|n) \rightarrow \mod{K(m|n)}.\label{lastlest2}
\end{align}
\end{Theorem}

Now recall for $R \in \Seq_{r,s}$
that $\lonestar R$ can be viewed either as an endofunctor of
$\mathscr F(m|n)$ as in the previous section, as an endofunctor
of $\Mod{K(m|n)}$ as in (\ref{head}), or even as an endofunctor
of $\mod{K(m|n)}$ on forgetting the gradings.
Using (\ref{lastlest1})--(\ref{lastlest2}) repeatedly, 
we get an isomorphism of functors
$\lonestar R \circ \E\cong \E\circ \lonestar R:
\mathscr F(m|n)\rightarrow \mod{K(m|n)}$.
Fixing also a (unique up to scalars) isomorphism $L(\zeta) \cong \E\,\C$,
we get from this an isomorphism
\begin{equation*}
\pi_R: \lonestar R\,L(\zeta) \stackrel{\sim}{\rightarrow}
\E \lonestar R\,\C.
\end{equation*}
Note although we made a choice of scalar here, the isomorphism 
$\Pi^{(k)}_R$ in the following theorem is independent of this choice.
For the statement, recall also the maps (\ref{EPg})--(\ref{FPg})
and (\ref{EP})--(\ref{FP}).

\begin{Theorem}\label{iso2}
For any $R \in \Seq_{r,s}$, there is an algebra isomorphism
\begin{equation*}
\Pi^{(k)}_R:
\End_{\mathfrak{g}}(\lonestar R\, \C)^{\op}\stackrel{\sim}{\rightarrow}
\End_{K(m|n)}(\lonestar R\, L(\zeta))^{\op},
\qquad
\theta \mapsto \pi_R^{-1} \circ \E(\theta) \circ \pi_R.
\end{equation*}
Moreover the following diagrams commute.
\begin{align*}
\begin{CD}
\End_{\mathfrak{g}}(\lonestar R\,\C)^{\op}&@>\iota_{R;i}^{RE} >>& \End_{\mathfrak{g}}(\lonestar(RE)\,\C)^{\op}\\
@V\Pi^{(k)}_R VV&&@VV\Pi^{(k)}_{RE} V\\
\End_{K(m|n)}(\lonestar R\,L(\zeta))^{\op}&@>\iota_{R;i}^{RE} >>& \End_{K(m|n)}(\lonestar(RE)\,L(\zeta))^{\op}
\end{CD}\end{align*}\begin{align*}
\begin{CD}
\End_{\mathfrak{g}}(\lonestar R\,\C)^{\op}&@>\iota_{R;i}^{RF} >>& \End_{\mathfrak{g}}(\lonestar(RF)\,\C)^{\op}\\
@V\Pi^{(k)}_R VV&&@VV\Pi^{(k)}_{RF} V\\
\End_{K(m|n)}(\lonestar R\,L(\zeta))^{\op}&@>\iota_{R;i}^{RF} >>& \End_{K(m|n)}(\lonestar(RF)\,L(\zeta))^{\op}
\end{CD}
\end{align*}
\end{Theorem}

\begin{proof}
This follows from Theorem~\ref{BS4a}.
\end{proof}

\phantomsubsection{\boldmath The isomorphism $\Upsilon^{(k)}_R$}
Let $R \in \mathscr R_{r,s}$ be fixed,
and $\eta$ and $\zeta$ be the weights from (\ref{othergroundstate}) 
and (\ref{groundstate}), respectively.
We can now complete the proof of Theorem~\ref{isotheorem} by constructing the
third map
(\ref{i3}). Here there is an interesting
diagrammatic trick.
Recall Theorem~\ref{byas}. It gives us a basis $\{\bar f_{\tS\tT}\}$ for
the special endomorphism algebra
$\End_{K(m|n)}(\lonestar R\,L(\zeta))^{\op}$
indexed by good pairs $(\tS,\tT)$ of $R$-tableaux of type $\zeta$.
We are going to compare this with the basis
$\{f_{\tS\tT}\}$ for the graded walled Brauer algebra
$B_R(\delta)$ from (\ref{bbas}), 
which is indexed by pairs $(\tS,\tT)$ of
$R$-tableaux of type $\eta$ with $\sh(\tS) = \sh(\tT)$.

\begin{Lemma}\label{trick}
There is a bijection
$$
\left\{
\begin{array}{ll}
\text{Good pairs $(\tS,\tT)$ of}\\
\text{$R$-tableaux of type $\zeta$}
\end{array}
\right\}
\stackrel{\sim}{\rightarrow}
\left\{
\begin{array}{ll}
\text{Pairs $(\tS',\tT')$ of $R$-tableaux of type $\eta$}\\
\text{with $\sh(\tS') = \sh(\tT')$ and $k(\tS',\tT') \leq k$}
\end{array}
\right\}.
$$
It maps $(\tS,\tT)$ to $(\tS',\tT')$ defined as follows.
\begin{itemize}
\item[(1)] 
Starting from the cup diagram $\underline{\tS} \overline{\tT}$,
replace the cap diagram $\overline{\zeta}$ 
on the top 
by
$\overline{\eta}$
and the cup diagram $\underline{\zeta}$ on the bottom by $\underline{\eta}$.
\item[(2)] Next adjust the labels of all the vertices lying on connected 
components passing through the top or bottom number lines so that the 
weight diagrams at the top and bottom are both equal to $\eta$.
\item[(3)]
Then $(\tS',\tT')$ is the unique pair of $R$-tableaux of type $\eta$ 
such that the resulting
diagram is equal to $\underline{\tS}'\overline{\tT}'$.
\end{itemize}
\end{Lemma}

\begin{proof}
Construct a two-sided inverse.
\end{proof}

Here are a couple of examples illustrating 
the definition in Lemma~\ref{trick}; actually we display the diagrams
$\underline{\tS} \overline{\tT} \rightsquigarrow \underline{\tS}'
\overline{\tT}'$
assuming $(\tS,\tT) \mapsto (\tS',\tT')$ under the bijection defined
in the lemma.
In the first example, $m=n=1$ (so $k=1$ and $\delta = 0$) and in the 
second example $m=n=2$ (so $k=2$ and $\delta = 0$).
$$
\begin{picture}(45,138)
\put(0,65){\line(1,0){45}}
\put(-2.5,120.4){$\up$}
\put(27.5,120.4){$\up$}
\put(42.5,120.4){$\up$}
\put(12.5,125.1){$\down$}
\put(-2.5,0.4){$\up$}
\put(12.5,0.4){$\up$}
\put(27.5,5.1){$\down$}
\put(42.5,0.4){$\up$}

\put(-2.5,75.4){$\up$}
\put(42.5,75.4){$\up$}
\put(-2.5,90.4){$\up$}
\put(27.5,90.4){$\up$}
\put(-2.5,105.4){$\up$}
\put(42.5,105.4){$\up$}

\put(-2.5,60.4){$\up$}
\put(27.5,60.4){$\up$}
\put(12.5,65.1){$\down$}
\put(42.5,60.4){$\up$}
\put(12.5,77.2){$\circ$}
\put(27,78.2){$\cross$}
\put(12.5,92.2){$\circ$}
\put(42,93.2){$\cross$}
\put(12.5,107.2){$\circ$}
\put(27,108.2){$\cross$}
\put(0.2,65){\line(0,1){68}}
\put(45.2,110){\line(0,1){23}}
\put(45.2,65){\line(0,1){15}}
\put(22.7,125){\oval(15,15)[b]}
\put(22.7,125){\oval(15,15)[t]}
\put(22.7,65){\oval(15,15)[t]}
\qbezier(45.2,110)\qbezier(45.2,102.5)\qbezier(37.7,102.5)
\qbezier(30.2,95)\qbezier(30.2,102.5)\qbezier(37.7,102.5)
\qbezier(45.2,80)\qbezier(45.2,87.5)\qbezier(37.7,87.5)
\qbezier(30.2,95)\qbezier(30.2,87.5)\qbezier(37.7,87.5)


\put(-2.5,45.4){$\up$}
\put(42.5,45.4){$\up$}
\put(12.5,30.4){$\up$}
\put(42.5,30.4){$\up$}
\put(-2.5,15.4){$\up$}
\put(42.5,15.4){$\up$}

\put(12.5,47.2){$\circ$}
\put(27,48.2){$\cross$}
\put(-2.5,32.2){$\circ$}
\put(27,33.2){$\cross$}
\put(12.5,17.2){$\circ$}
\put(27,18.2){$\cross$}
\put(45.2,65){\line(0,-1){68}}
\put(0.2,20){\line(0,-1){23}}
\put(0.2,65){\line(0,-1){15}}
\put(22.7,5){\oval(15,15)[t]}
\put(22.7,5){\oval(15,15)[b]}
\put(22.7,65){\oval(15,15)[b]}
\qbezier(0.2,20)\qbezier(0.2,27.5)\qbezier(7.7,27.5)
\qbezier(15.2,35)\qbezier(15.2,27.5)\qbezier(7.7,27.5)
\qbezier(0.2,50)\qbezier(0.2,42.5)\qbezier(7.7,42.5)
\qbezier(15.2,35)\qbezier(15.2,42.5)\qbezier(7.7,42.5)
\put(62,62){$\rightsquigarrow$}
\end{picture}
\qquad\qquad
\begin{picture}(45,138)
\put(0,65){\line(1,0){45}}
\put(-2.5,120.4){$\up$}
\put(12.5,120.4){$\up$}
\put(27.5,125.1){$\down$}
\put(42.5,125.1){$\down$}
\put(-2.5,0.4){$\up$}
\put(12.5,0.4){$\up$}
\put(27.5,5.1){$\down$}
\put(42.5,5.1){$\down$}

\put(-2.5,75.4){$\up$}
\put(42.5,80.1){$\down$}
\put(-2.5,90.4){$\up$}
\put(27.5,95.1){$\down$}
\put(-2.5,105.4){$\up$}
\put(42.5,110.1){$\down$}

\put(-2.5,60.4){$\up$}
\put(27.5,60.4){$\up$}
\put(12.5,65.1){$\down$}
\put(42.5,65.1){$\down$}
\put(12.5,77.2){$\circ$}
\put(27,78.2){$\cross$}
\put(12.5,92.2){$\circ$}
\put(42,93.2){$\cross$}
\put(12.5,107.2){$\circ$}
\put(27,108.2){$\cross$}
\put(0.2,65){\line(0,1){68}}
\put(15.2,125){\line(0,1){8}}
\put(30.2,125){\line(0,1){8}}
\put(45.2,110){\line(0,1){23}}
\put(45.2,65){\line(0,1){15}}
\put(22.7,125){\oval(15,15)[b]}
\put(22.7,65){\oval(15,15)[t]}
\qbezier(45.2,110)\qbezier(45.2,102.5)\qbezier(37.7,102.5)
\qbezier(30.2,95)\qbezier(30.2,102.5)\qbezier(37.7,102.5)
\qbezier(45.2,80)\qbezier(45.2,87.5)\qbezier(37.7,87.5)
\qbezier(30.2,95)\qbezier(30.2,87.5)\qbezier(37.7,87.5)


\put(-2.5,45.4){$\up$}
\put(42.5,50.1){$\down$}
\put(12.5,30.4){$\up$}
\put(42.5,35.1){$\down$}
\put(-2.5,15.4){$\up$}
\put(42.5,20.1){$\down$}

\put(12.5,47.2){$\circ$}
\put(27,48.2){$\cross$}
\put(-2.5,32.2){$\circ$}
\put(27,33.2){$\cross$}
\put(12.5,17.2){$\circ$}
\put(27,18.2){$\cross$}
\put(45.2,65){\line(0,-1){68}}
\put(15.2,5){\line(0,-1){8}}
\put(30.2,5){\line(0,-1){8}}
\put(0.2,20){\line(0,-1){23}}
\put(0.2,65){\line(0,-1){15}}
\put(22.7,5){\oval(15,15)[t]}
\put(22.7,65){\oval(15,15)[b]}
\qbezier(0.2,20)\qbezier(0.2,27.5)\qbezier(7.7,27.5)
\qbezier(15.2,35)\qbezier(15.2,27.5)\qbezier(7.7,27.5)
\qbezier(0.2,50)\qbezier(0.2,42.5)\qbezier(7.7,42.5)
\qbezier(15.2,35)\qbezier(15.2,42.5)\qbezier(7.7,42.5)
\end{picture}
\qquad\qquad\qquad
\begin{picture}(45,138)
\put(0,65){\line(1,0){45}}
\put(-2.5,125.1){$\down$}
\put(12.5,125.1){$\down$}
\put(27.5,120.4){$\up$}
\put(42.5,120.4){$\up$}
\put(-2.5,0.4){$\up$}
\put(12.5,0.4){$\up$}
\put(27.5,5.1){$\down$}
\put(42.5,5.1){$\down$}

\put(12.5,75.4){$\up$}
\put(-2.5,80.1){$\down$}
\put(27.5,90.4){$\up$}
\put(-2.5,95.1){$\down$}
\put(42.5,105.4){$\up$}
\put(-2.5,110.1){$\down$}

\put(12.5,60.4){$\up$}
\put(-2.5,65.1){$\down$}
\put(42.5,65.1){$\down$}
\put(27.5,60.4){$\up$}
\put(27.5,77.2){$\circ$}
\put(42,78.2){$\cross$}
\put(12.5,92.2){$\circ$}
\put(42,93.2){$\cross$}
\put(12.5,107.2){$\circ$}
\put(27,108.2){$\cross$}
\put(0.2,65){\line(0,1){60}}
\put(45.2,110){\line(0,1){15}}
\put(15.2,65){\line(0,1){15}}
\put(22.7,125){\oval(15,15)[b]}
\put(22.7,125){\oval(15,15)[t]}
\put(22.7,125){\oval(45,30)[t]}
\put(37.7,65){\oval(15,15)[t]}
\qbezier(45.2,110)\qbezier(45.2,102.5)\qbezier(37.7,102.5)
\qbezier(30.2,95)\qbezier(30.2,102.5)\qbezier(37.7,102.5)
\qbezier(15.2,80)\qbezier(15.2,87.5)\qbezier(22.7,87.5)
\qbezier(30.2,95)\qbezier(30.2,87.5)\qbezier(22.7,87.5)


\put(27.5,45.4){$\up$}
\put(42.5,50.1){$\down$}
\put(12.5,30.4){$\up$}
\put(42.5,35.1){$\down$}
\put(-2.5,15.4){$\up$}
\put(42.5,20.1){$\down$}

\put(-2.5,47.2){$\circ$}
\put(12,48.2){$\cross$}

\put(-2.5,32.2){$\circ$}
\put(27.5,33.2){$\cross$}

\put(12.5,17.2){$\circ$}
\put(27,18.2){$\cross$}
\put(45.2,65){\line(0,-1){60}}
\put(30.2,65){\line(0,-1){15}}
\put(0.2,20){\line(0,-1){15}}
\put(22.7,5){\oval(15,15)[t]}
\put(22.7,5){\oval(15,15)[b]}
\put(22.7,5){\oval(45,30)[b]}
\put(7.7,65){\oval(15,15)[b]}
\qbezier(0.2,20)\qbezier(0.2,27.5)\qbezier(7.7,27.5)
\qbezier(15.2,35)\qbezier(15.2,27.5)\qbezier(7.7,27.5)
\qbezier(30.2,50)\qbezier(30.2,42.5)\qbezier(22.7,42.5)
\qbezier(15.2,35)\qbezier(15.2,42.5)\qbezier(22.7,42.5)
\put(62,62){$\rightsquigarrow$}
\end{picture}
\qquad\qquad
\begin{picture}(45,138)
\put(0,65){\line(1,0){45}}
\put(-2.5,120.4){$\up$}
\put(12.5,120.4){$\up$}
\put(27.5,125.1){$\down$}
\put(42.5,125.1){$\down$}
\put(-2.5,0.4){$\up$}
\put(12.5,0.4){$\up$}
\put(27.5,5.1){$\down$}
\put(42.5,5.1){$\down$}

\put(-2.5,75.4){$\up$}
\put(12.5,80.1){$\down$}
\put(-2.5,90.4){$\up$}
\put(27.5,95.1){$\down$}
\put(-2.5,105.4){$\up$}
\put(42.5,110.1){$\down$}

\put(-2.5,60.4){$\up$}
\put(12.5,65.1){$\down$}
\put(42.5,65.1){$\down$}
\put(27.5,60.4){$\up$}
\put(27.5,77.2){$\circ$}
\put(42,78.2){$\cross$}
\put(12.5,92.2){$\circ$}
\put(42,93.2){$\cross$}
\put(12.5,107.2){$\circ$}
\put(27,108.2){$\cross$}
\put(0.2,65){\line(0,1){68}}
\put(15.2,125){\line(0,1){8}}
\put(30.2,125){\line(0,1){8}}
\put(45.2,110){\line(0,1){23}}
\put(15.2,65){\line(0,1){15}}
\put(22.7,125){\oval(15,15)[b]}
\put(37.7,65){\oval(15,15)[t]}
\qbezier(45.2,110)\qbezier(45.2,102.5)\qbezier(37.7,102.5)
\qbezier(30.2,95)\qbezier(30.2,102.5)\qbezier(37.7,102.5)
\qbezier(15.2,80)\qbezier(15.2,87.5)\qbezier(22.7,87.5)
\qbezier(30.2,95)\qbezier(30.2,87.5)\qbezier(22.7,87.5)


\put(27.5,45.4){$\up$}
\put(42.5,50.1){$\down$}
\put(12.5,30.4){$\up$}
\put(42.5,35.1){$\down$}
\put(-2.5,15.4){$\up$}
\put(42.5,20.1){$\down$}

\put(-2.5,47.2){$\circ$}
\put(12,48.2){$\cross$}

\put(-2.5,32.2){$\circ$}
\put(27.5,33.2){$\cross$}

\put(12.5,17.2){$\circ$}
\put(27,18.2){$\cross$}
\put(45.2,65){\line(0,-1){68}}
\put(30.2,65){\line(0,-1){15}}
\put(15.2,5){\line(0,-1){8}}
\put(30.2,5){\line(0,-1){8}}
\put(0.2,20){\line(0,-1){23}}
\put(22.7,5){\oval(15,15)[t]}
\put(7.7,65){\oval(15,15)[b]}
\qbezier(0.2,20)\qbezier(0.2,27.5)\qbezier(7.7,27.5)
\qbezier(15.2,35)\qbezier(15.2,27.5)\qbezier(7.7,27.5)
\qbezier(30.2,50)\qbezier(30.2,42.5)\qbezier(22.7,42.5)
\qbezier(15.2,35)\qbezier(15.2,42.5)\qbezier(22.7,42.5)
\end{picture}
$$

\vspace{2mm}

\begin{Theorem}\label{iso3}
There is an algebra isomorphism 
\begin{equation*}
\Upsilon_{R}^{(k)}:
\End_{K(m|n)}(\lonestar R\,L(\zeta))^{\op}
\stackrel{\sim}{\rightarrow}
B_R(\delta) / B_R(\delta)_{> k}
\end{equation*}
mapping $\bar f_{\tS,\tT}$ for a good pair $(\tS,\tT)$ of $R$-tableaux
of type $\zeta$
to the canonical image in
$B_R(\delta) / B_R(\delta)_{> k}$ of $f_{\tS',\tT'} \in B_R(\delta)$, 
where $(\tS',\tT')$ is the image
of $(\tS,\tT)$ under the bijection from Lemma~\ref{trick}.
Moreover the following diagrams commute.
\begin{align*}
\begin{CD}
\End_{K(m|n)}(\lonestar R\,L(\zeta))^{\op}&@>\iota_{R;i}^{RE} >>& \End_{K(m|n)}(\lonestar(RE)\,L(\zeta))^{\op}\\
@V\Upsilon_{R}^{(k)} VV&&@VV\Upsilon_{RE}^{(k)} V\\
 B_R(\delta) / B_R(\delta)_{> k}&@>\iota_{R;i}^{RE} >>& B_{RE}(\delta) /
 B_{RE}(\delta)_{> k}
\end{CD}\end{align*}\begin{align*}
\begin{CD}
\End_{K(m|n)}(\lonestar R\,L(\zeta))^{\op}&@>\iota_{R;i}^{RF} >>& \End_{K(m|n)}(\lonestar(RF)\,L(\zeta))^{\op}\\
@V\Upsilon_{R}^{(k)} VV&&@VV\Upsilon_{RF}^{(k)} V\\
B_R(\delta) / B_R(\delta)_{> k}
&@>\iota_{R;i}^{RF} >>& B_{RF}(\delta) / B_{RF}(\delta)_{> k}.
\end{CD}
\end{align*}
\end{Theorem}

\begin{proof}
The map $\Upsilon_R^{(k)}$ is a vector space isomorphism by
Lemma~\ref{trick}
and (\ref{ks}).
It is multiplicative 
by the algorithmic descriptions of the multiplications in
$\End_{K(m|n)}(\lonestar R\,L(\zeta))^{\op})$
and $B_R(\delta) / B_R(\delta)_{> k}$.
Finally,
the commutativity of the diagrams 
follows by considering the algorithmic description
of the 
maps $\iota_{R;i}^{RE}$ and $\iota_{R;i}^{RF}$
from the end of $\S$\ref{sd} and (\ref{leg1})--(\ref{leg2}).
\end{proof}

\phantomsubsection{Morita equivalence}
The following proves Theorem~\ref{fin} 
from the introduction.

\begin{Theorem}\label{mt}
Fix $R \in \Seq_{r,s}$ and $k \gg 0$.
Then there is an equivalence of categories
\begin{equation*}
\G_R 
:
\mod{K_{r,s}(\delta)} \rightarrow \mod{B_{r,s}(\delta)}
\end{equation*}
defined 
by applying (the ungraded analogue of) the functor
$\F_R$
from (\ref{life})
followed by pull-back through the isomorphism $\Theta_R^{(k)}$.
Moreover the following hold.
\begin{itemize}
\item[(1)] The equivalence commutes with $i$-restriction
  and $i$-induction, i.e.
there are isomorphisms of functors
\begin{align*}
\G_{R} \circ \ires^{r+1,s}_{r,s} 
\cong \ires^{r+1,s}_{r,s}  \circ \G_{RE}&:\mod{K_{r+1,s}(\delta)} \rightarrow \mod{B_{r,s}(\delta)},\\
\G_{R} \circ \ires^{r,s+1}_{r,s} 
\cong \ires^{r,s+1}_{r,s}  \circ \G_{RF}&:\mod{K_{r,s+1}(\delta)} \rightarrow \mod{B_{r,s}(\delta)},\\
\G_{RE} \circ \iind^{r+1,s}_{r,s}
\cong \iind^{r+1,s}_{r,s} \circ \G_R&:\mod{K_{r,s}(\delta)} \rightarrow \mod{B_{r+1,s}(\delta)},\\
\G_{RF} \circ \iind^{r,s+1}_{r,s} 
\cong \iind^{r,s+1}_{r,s}  \circ \G_R&:\mod{K_{r,s}(\delta)} \rightarrow \mod{B_{r,s+1}(\delta)}.
\end{align*}
\item[(2)]
For any
$\la \in \La_{r,s}$,
we have that 
$\G_R(V_{r,s}(\la)) \cong C_{r,s}(\la)$
as 
$B_{r,s}(\delta)$-modules.
\end{itemize}
\end{Theorem}

\begin{proof}
The opening statement follows from Theorem~\ref{morita}.
Also (1) follows on combining the isomorphisms of functors in
Theorem~\ref{morita}
with the commutativity of the diagrams in
Theorem~\ref{isotheorem}.

To prove (2) we proceed by induction on $r+s$, the base case $r+s=0$ being trivial.
For the induction step, suppose we know already
that $\G_R(V_{r,s}(\la)) \cong C_{r,s}(\la)$
for all $\la \in \La_{r,s}$.
We need to show that
$\G_{RE}(V_{r+1,s}(\mu)) \cong C_{r+1,s}(\mu)$
for all $\mu \in \La_{r+1,s}$
and that
$\G_{RF}(V_{r,s+1}(\mu)) \cong C_{r,s+1}(\mu)$
for all $\mu \in \La_{r,s+1}$.
We just verify the first of these statements, since the proof of the
second is similar.
For this we proceed by another induction, this time on $s-|\mu^{\mathrm R}|$.

First consider the base case when $s-|\mu^{\mathrm R}| = 0$.
Then $\mu$ is minimal in $\La_{r+1,s}$
with respect to the Bruhat order.
So $V_{r+1,s}(\mu)$ is irreducible by Theorem~\ref{iscell}(4).
Let $\la$ be obtained from $\mu$ by removing a box of content $i$ from
$\mu^{\mathrm L}$.
We claim that $\iind_{r,s}^{r+1,s} V_{r,s}(\la)$ has irreducible head
isomorphic to $V_{r+1,s}(\mu)$.
By Theorem~\ref{finalbranch}(3), we either have
that $\iind_{r,s}^{r+1,s} V_{r,s}(\la) \cong V_{r+1,s}(\mu)$, when the
claim is trivially true, or $\la^{\mathrm R}$ has a removable box of content
$-i-\delta$.
In the latter case, let $\nu \in \La_{r+1,s}$ be obtained from $\la$ by removing this
box.
Then $\iind_{r,s}^{r+1,s} V_{r,s}(\la)$ has a two-step filtration with
$V_{r+1,s}(\nu)\langle 1 \rangle$ at the bottom and $V_{r+1,s}(\mu)$
at the top.
It follows that the head of 
$\iind_{r,s}^{r+1,s} V_{r,s}(\la)$ contains
$V_{r+1,s}(\mu)$ and possibly some other constituents which are
also constituents of $L :=
\operatorname{head} V_{r+1,s}(\nu)\langle 1 \rangle$.
We need to show that there are actually no extra constituents.
Suppose first that $\nu \in \dot\La_{r+1,s}$, in which case
$L = L_{r+1,s}(\nu)\langle 1 \rangle$. 
The weight diagram of $\nu$ has $i$th and $(i+1)$th vertices labelled
$\up$ and $\down$, respectively, so by the last part of Theorem~\ref{ipf}
we have that $E_i L(\nu) = \{0\}$.
This implies using also Lemma~\ref{comm} that
$\ires^{r+1,s}_{r,s} L_{r+1,s}(\nu)
= \ires^{r+1,s}_{r,s} (e_{r,s} L(\nu))
\cong e_{r,s}(E_i L(\nu)) = \{0\}$, hence
\begin{multline*}
\dim \hom_{K_{r+1,s}(\delta)}(\iind_{r,s}^{r+1,s} V_{r,s}(\la),
L_{r+1,s}(\nu))
= \\\dim
\hom_{K_{r,s}(\delta)}(V_{r,s}(\la),
\ires^{r+1,s}_{r,s} L_{r+1,s}(\nu))
= 0.
\end{multline*}
This proves that $L$ does not appear in this case.
It remains to treat the exceptional situation in which $\nu =
(\varnothing,\varnothing)$, $\delta = 0$, $r=0$ and $s=1$.
In this case Theorem~\ref{iscell}(4) implies 
that $L \cong V_{r+1,s}(\mu)\langle 2 \rangle$.
As $\ires^{r+1,s}_{r,s} V_{r+1,s}(\mu) \cong V_{r,s}(\la)$
by Theorem~\ref{finalbranch}(1), we have that
\begin{multline*}
\dim\hom_{K_{r+1,s}(\delta)}(\iind_{r,s}^{r+1,s} V_{r,s}(\la),
V_{r+1,s}(\mu))
=\\
\dim\hom_{K_{r,s}(\delta)}(V_{r,s}(\la),
\ires^{r+1,s}_{r,s} V_{r+1,s}(\mu))
= 1,
\end{multline*}
which is good enough to complete the proof of the claim.
Now to show that $\G_{RE} (V_{r+1,s}(\mu)) \cong C_{r+1,s}(\mu)$, we note 
by the definition (\ref{cell}) that $C_{r+1,s}(\mu)$ is irreducible
too.
Hence using the claim just established, it suffices to show that
$C_{r+1,s}(\mu)$ appears in the head of $\G_{RE} (\iind_{r,s}^{r+1,s}
V_{r,s}(\la))$.
By (1) and induction,
$\G_{RE} (\iind_{r,s}^{r+1,s}
V_{r,s}(\la))
\cong \iind_{r,s}^{r+1,s} C_{r,s}(\la)$, which has $C_{r+1,s}(\mu)$
in its head by Theorem~\ref{ibranch}(3).
This verifies the base case of the second induction.

For the induction step,
we take $\mu \in \La_{r+1,s}$ with $s-|\mu^{\mathrm R}| > 0$.
Let $\la$ be obtained from $\mu$ by adding a box of content
$-i-\delta$
to $\mu^{\mathrm R}$.
By induction we know that $\G_R(V_{r,s}(\la)) \cong C_{r,s}(\la)$.
Now apply the functor $\iind_{r,s}^{r+1,s}$.
If $\mu^{\mathrm L}$ has no removable box of content $i$ then
we have that $\iind_{r,s}^{r+1,s} V_{r,s}(\la)
\cong V_{r+1,s}(\mu)$ and $\iind_{r,s}^{r+1,s} C_{r,s}(\la)
\cong C_{r+1,s}(\mu)$ by Theorems~\ref{finalbranch}(3) 
and \ref{ibranch}(3), respectively. Using also (1), 
this implies that $\G_{RE}(V_{r+1,s}(\mu)) \cong C_{r+1,s}(\mu)$ as
required.
If $\mu^{\mathrm L}$ has a removable box of content $i$, let $\nu$
be obtained from $\la$ by adding this box to $\la^{\mathrm L}$.
Then Theorem~\ref{ibranch}(3) asserts that 
$M := \iind_{r,s}^{r+1,s} C_{r,s}(\la)$
has a unique submodule $M'$ such that $M / M' \cong C_{r+1,s}(\nu)$,
and this submodule $M'$ is isomorphic to $C_{r+1,s}(\mu)$.
Also Theorem~\ref{finalbranch}(3) and (1) show that
$\G_{RE}(\iind_{r,s}^{r+1,s} V_{r,s}(\la))
\cong M$ has a submodule $M'' \cong \G_{RE}(V_{r+1,s}(\mu))$
such that $M / M'' \cong \G_{RE}(V_{r+1,s}(\nu))$.
By the induction on $s-|\mu^{\mathrm R}|$ we know already 
that $\G_{RE}(V_{r+1,s}(\nu)) \cong C_{r+1,s}(\nu)$.
Hence by the unicity of $M'$ we deduce that $M'' \cong M'$.
This shows that $C_{r+1,s}(\mu) \cong \G_{RE}(V_{r+1,s}(\mu))$.
\end{proof}

\begin{Corollary}\label{othercellid}
For $\la \in \La_{r,s}$,
the graded cell module $C_R(\lambda)$
is isomorphic to the ordinary cell module $C_{r,s}(\lambda)$
on forgetting the grading and viewing it as a $B_{r,s}(\delta)$-module
via the isomorphism $\Theta_{R}^{(k)}$ from (\ref{thiso}).
\end{Corollary}

\begin{proof}
Immediate from Theorem~\ref{mt}(2) and Lemma~\ref{sun}.
\end{proof}

\begin{Corollary}\label{irredu}
For $\la \in \dot\La_{r,s}$,
$D_R(\lambda)$
is isomorphic to $D_{r,s}(\lambda)$
on forgetting the grading and viewing it as a $B_{r,s}(\delta)$-module
via the isomorphism $\Theta_{R}^{(k)}$.
\end{Corollary}
 
\phantomsubsection{Weight idempotents}
Recall the Jucys-Murphy elements
from (\ref{jm1})--(\ref{jm2}) and the corresponding
weight idempotents
$e(\bi) \in B_{r,s}(\delta)$; these depend
implicitly on $R$.
Recall also the idempotents 
$e(\bi) \in B_R(\delta)$. We observed in Remark~\ref{isotypic}
that these are isotypic, hence so are their canonical images
$\bar e(\bi) \in B_R(\delta) / B_R(\delta)_{> k}$.

\begin{Lemma}\label{imid}
For $\bi \in \Z^{r+s}$ and any $k \geq 0$, the map
$\Theta^{(k)}_R$ from (\ref{thetaonto}) sends $e(\bi) \in B_{r,s}(\delta)$
to 
$\bar e(\bi) \in B_R(\delta) / B_R(\delta)_{> k}$.
\end{Lemma}

\begin{proof}
We proceed by induction on $r+s$, the lemma being vacuous in the case
$r+s=0$.
For the induction step assume we have already proved that
$\Theta^{(k)}_R(e(\bi)) = \bar e(\bi)$.
We need to show that
 $\Theta^{(k)}_{RE}:B_{r+1,s}(\delta)
\rightarrow B_{RE}(\delta) / B_{RE}(\delta)_{> k}$
and
 $\Theta^{(k)}_{RF}:B_{r,s+1}(\delta)
\rightarrow B_{RF}(\delta) / B_{RF}(\delta)_{> k}$
map $e(\bi i)$ to $\bar e(\bi i)$
for $i \in \Z$.
We just explain the former.
By the definition (\ref{jm1})
we have in $B_{r+1,s}(\delta)$
that $\iota_{r,s}^{r+1,s} (e(\bi)) = \sum_{j \in \Z} e(\bi j)$.
Recalling the idempotent $1_{r,s;i}^{r+1,s} \in B_{r+1,s}(\delta)$
defined just after (\ref{xr}), it follows in $B_{r+1,s}(\delta)$ that
$$
e(\bi i) = \iota_{r,s}^{r+1,s} (e(\bi)) 1_{r,s;i}^{r+1,s}
= \iota_{r,s;i}^{r+1,s} (e(\bi)).
$$
Also by the observations after (\ref{EP})--(\ref{FP}), we know that
$\iota_{R;i}^{RE}(\bar e(\bi)) = \bar e(\bi i)$
in $B_{RE}(\delta) / B_{RE}(\delta)_{> k}$.
Now we are done by the induction hypothesis and 
the commutativity of the first square in Theorem~\ref{isotheorem}.
\end{proof}

\begin{Lemma}\label{idcorr}
Fix $R \in \Seq_{r,s}$.
For any $\bi \in \Z^{r+s}$, the idempotent
$e(\bi) \in B_{r,s}(\delta)$ is non-zero if and only if
there exists a restricted $R$-tableau $\tT$
with $\bi^\tT = \bi$.
Assuming this is the case, let $\la := \sh(\tT) \in \dot \La_{r,s}$.
Then $e(\bi)$ is isotypic
in the sense that the projective module
$B_{r,s}(\delta) e(\bi)$ 
is a direct sum of copies of 
the projective cover of
$D_{r,s}(\la)$.
\end{Lemma}

\begin{proof}
Assuming $k \gg 0$ so 
that
$\Theta^{(k)}_{R}$ is an isomorphism,
this follows from
Lemmas~\ref{primitivespre}--\ref{primitives}
together with
Lemma~\ref{imid} 
and Corollary~\ref{irredu}.
\end{proof}

\begin{Theorem}\label{lehrerzhang}
If either $m=0$ or $n=0$, i.e. $k =0$, then
$\ker \Psi^{m,n}_{r,s}$
is generated by the idempotents $e(\bi)$ for all $\bi \in \Z^{r+s}$
with $k(\bi) > 0$. 
\end{Theorem}

\begin{proof}
We know by Lemma~\ref{lehrerzhangsetup} that the two-sided ideal
of $B_R(\delta)$ generated by the idempotents $e(\bi)$ with $k(\bi) >
0$
is of the same dimension as $B_R(\delta)_{> 0}$.
Applying Lemma~\ref{imid} with $k \gg 0$, we deduce
that the two-sided ideal of $B_{r,s}(\delta)$ generated by
these
idempotents has the same dimension as $B_R(\delta)_{> 0}$.

Now assume that $k=0$ as in the statement of the theorem.
Combining the observation made in the 
previous paragraph with the opening statement of 
Corollary~\ref{isocor}, we deduce that
the two-sided ideal of $B_{r,s}(\delta)$ generated by the idempotents 
$e(\bi)$ with $k(\bi) > 0$ has the same dimension as $\ker
\Psi^{m,n}_{r,s}$.
It remains to show that all these idempotents belong to $\ker
\Psi^{m,n}_{r,s}$.
Since $\ker \Psi^{m,n}_{r,s} 
= \ker \Theta^{(k)}_R$
this amounts to showing that $\Theta^{(k)}_R(e(\bi)) = 0$ when $k(\bi)
> 0$. This follows from Lemmas~\ref{imid} and \ref{lehrerzhangsetup}
\end{proof}

\phantomsubsection{Cross bipartitions and irreducible representations}
Following \cite{CW}, a bipartition $\la \in \La$ is
an {\em $(m,n)$-cross bipartition} if there exists some $1 \leq i \leq
m+1$ with $\la_i^{\mathrm L} + \la_{m+2-i}^{\mathrm R} < n+1$.
If we represent $\la$ by its picture
obtained by rotating $\la^{\mathrm R}$ through $180^\circ$ as in the following example,
then $\la$ is an $(m,n)$-cross bipartition if and only if its picture can
be fitted into a cross of height $m$ and width $n$.
$$
\begin{picture}(280,95)
\put(-6,38){$\la = \left((3,2),(2,1^2)\right)
\qquad
\leftrightarrow
\quad$}
\put(139,37.3){$m\Big\{$}
\put(200.5,77){$\overbrace{\phantom{helloo}}^{\textstyle n}$}
\put(240,30){\line(0,1){10}}
\put(230,20){\line(0,1){20}}
\put(220,20){\line(0,1){20}}
\put(210,20){\line(0,1){50}}
\put(200,40){\line(0,1){30}}
\put(190,40){\line(0,1){10}}
\put(210,20){\line(1,0){20}}
\put(210,30){\line(1,0){30}}
\put(190,40){\line(1,0){50}}
\put(190,50){\line(1,0){20}}
\put(200,60){\line(1,0){10}}
\put(200,70){\line(1,0){10}}
\dashline{3}(200,30)(200,0)
\dashline{3}(230,30)(230,0)
\dashline{3}(200,50)(200,80)
\dashline{3}(230,50)(230,80)
\dashline{3}(230,50)(270,50)
\dashline{3}(230,30)(270,30)
\dashline{3}(200,50)(160,50)
\dashline{3}(200,30)(160,30)
\end{picture}
$$
Recall the number $k(\la)$ defined from the weight diagram $\la$
by (\ref{kdefla}), and recall that $k = \min(m,n)$.

\begin{Lemma}\label{comeslem}
Given $\la \in \La$, we have that
$k(\la) \leq k$ if and only if $\la$
is an $(m,n)$-cross bipartition.
\end{Lemma}

\begin{proof}
We say $\la$ is $(m,n)$-minimal if $\la$ is not an $(m,n)$-cross
bipartition but every bipartition $\mu$ obtained from $\la$ by removing a box is
an $(m,n)$-cross bipartition.
We say that $\la$ is $k$-minimal if $k(\la) > k$ but every bipartition
$\mu$
obtained from $\la$ by removing a box satisfies $k(\mu) \leq k$.

Suppose that $\mu$ is obtained from $\la$ by removing a box.
In terms of weight diagrams,
this means that $\mu$ is obtained from $\la$ by relabelling the $i$th
and $(i+1)$th vertices according to one of the entries in the
following table.
$$
\begin{array}{|l|llllllll|}
\hline
\la&\circ\up&\down\circ&\cross \up&\down \cross&\down\up&\down\up&\circ\cross&\cross\circ
\\\hline
\mu&\up\circ&\circ\down&\up \cross&\cross
\down&\circ\cross&\cross\circ&
\up\down&\up\down\\
\hline
\end{array}
$$
In all cases in this table, it is clear from
(\ref{kdefla}) either that $k(\mu) = k(\la)$
or that $k(\mu) = k(\la)-1$; the latter can hold only for the rightmost two
entries in the table.
In particular, 
we always have that $k(\mu) \leq k(\la)$.
Hence if $k(\la) \leq k$ then $k(\mu) \leq k$.
Also it is also obvious that
if $\la$ is an $(m,n)$-cross bipartition then so is $\mu$.
The last two observations reduce the proof of
the lemma to showing that $\la$ is $(m,n)$-minimal if and only if
$\la$ is $k$-minimal.

By the definition, $\la$ is $(m,n)$-minimal if and
only if
$\la^{\mathrm L}_i + \la^{\mathrm R}_{m+2-i} = n+1$ for each $i=1,\dots,m+1$
and $\la^{\mathrm L}_j = \la^{\mathrm R}_j = 0$ for each $j \geq m+2$.

Finally suppose that $\la$ is $k$-minimal.
Considering the cases in the above table, $k$-minimality implies
that $k(\la) = k+1$. Moreover
the weight diagram $\la$ has infinitely many
$\up$'s to the left, infinitely many
$\down$'s to the right, and in between 
there are only $\circ$'s and $\times$'s.
More precisely, there are exactly
$(n+1)$ vertices labelled $\circ$ and $(m+1)$ vertices 
labelled $\times$.
Now it follows from the weight dictionary that
$\la^{\mathrm L}_i + \la^{\mathrm R}_{m+2-i} = n+1$ for each $i=1,\dots,m+1$
and $\la^{\mathrm L}_j = \la^{\mathrm R}_j = 0$ for each $j \geq m+2$
just like in the previous paragraph.
\end{proof}

\begin{Theorem}\label{irrclass2}
If $\la \in \dot \La_{r,s}$ is an $(m,n)$-cross bipartition
then $\ker \Psi^{m,n}_{r,s}$ acts as zero on $D_{r,s}(\la)$,
hence $D_{r,s}(\la)$ induces an irreducible
$B_{r,s}(\delta) / \ker \Psi^{m,n}_{r,s}$-module.
The modules
$$\{D_{r,s}(\la)\:|\:\text{for all $(m,n)$-cross bipartitions }\la \in \dot \La_{r,s}\}$$ give
a complete set of pairwise non-isomorphic irreducible
$B_{r,s}(\delta) / \ker \Psi^{m,n}_{r,s}$-modules.
\end{Theorem}

\begin{proof}
The surjection $\Theta^{(k)}_R$ induces an isomorphism
between $B_{r,s}(\delta) / \ker \Psi^{m,n}_{r,s}$
and $B_R(\delta) / B_R(\delta)_{>k}$.
The pull-backs through this isomorphism of the irreducible
$B_R(\delta) / B_R(\delta)_{>k}$-modules
classified in Theorem~\ref{irrclass} give us a complete set of 
pairwise non-isomorphic irreducible $B_{r,s}(\delta) / \ker
\Psi^{m,n}_{r,s}$-modules. 
We do not
know {\em a priori} that the labellings match up correctly, i.e. that
the pull-back of $D_R(\la)$ is the same as
the push-forward of $D_{r,s}(\la)$, but this follows from
Lemmas~\ref{imid}--\ref{idcorr}.
Finally the resulting irreducible modules are indexed by the $(m,n)$-cross
bipartitions in
$\dot\La_{r,s}$ thanks to Lemma~\ref{comeslem}. 
\end{proof}

\phantomsubsection{Indecomposable summands of mixed tensor space}
In the final subsection we explain how to deduce 
the classification of indecomposable
summands of the mixed
tensor spaces $V^{\otimes r} \otimes W^{\otimes s}$ 
obtained recently by Comes and Wilson \cite{CW}.
Recall by Theorem~\ref{BS4a} 
that the category $\mathscr F(m|n)$ of rational
$\mathfrak{g}$-supermodules is equivalent to the
category of finite dimensional $K(m|n)$-modules,
and that the irreducible objects in either of these categories are
indexed
by the set $\Xi$ from (\ref{xidef}).

We begin by defining a map
\begin{equation}\label{crazymap}
\{\text{$(m,n)$-cross bipartitions}\}
\rightarrow
\Xi,
\quad
\la \mapsto \la^\dagger.
\end{equation}
Let $\la \in \La$ be an $(m,n)$-cross bipartition and
define $k(\la)$ according to (\ref{kdefla}).
Let $\alpha$ be the weight diagram obtained from the diagram $\eta$
from (\ref{othergroundstate}) 
by switching the rightmost $k(\la)$ $\up$'s with the
leftmost $k(\la)$ $\down$'s.
Let $t$ be the crossingless matching obtained by drawing the
cap diagram $\overline{\la}$ underneath the cup diagram
$\underline{\alpha}$,
then joining rays according to the unique order-preserving bijection
that is the identity
outside of some finite interval.
Then take the oriented crossingless matching
$\la t \alpha$, replace $\alpha$ with the weight diagram 
$\zeta$ from (\ref{groundstate}), and finally adjust
the labels of $\la$ that are at the bottoms of 
line segments, to obtain $\la^\dagger \in \Xi$ 
such that $\la^\dagger t \zeta$ is consistently oriented.
Here is an example, taking $m=4, n=5$
and $\la = ((5,3,2),(5,4,1^2))$, so $k = 4$ and $k(\la) = 3$:
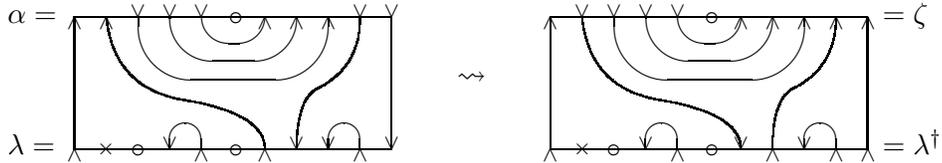
\begin{align*}
\begin{picture}(354,55)
\put(0,-1){$\la = $}
\put(0,49){$\alpha = $}
\put(170,26){$\rightsquigarrow$}
\put(25,1){\line(1,0){120}}
\put(25,51){\line(1,0){120}}
\put(331,49){$=\zeta$}
\put(331,-1){$=\la^\dagger$}
\put(205,1){\line(1,0){120}}
\put(205,51){\line(1,0){120}}
\put(22.5,-3.5){$\up$}
\put(34,-.8){$\cross$}
\put(46.6,-1.8){$\circ$}
\put(58.5,1.2){$\down$}
\put(70.5,-3.5){$\up$}
\put(83.6,-1.8){$\circ$}
\put(94.5,-3.5){$\up$}
\put(106.5,1.2){$\down$}
\put(118.5,1.2){$\down$}
\put(130.5,-3.5){$\up$}
\put(142.5,1.2){$\down$}
\put(22.5,46.5){$\up$}
\put(34.5,46.5){$\up$}
\put(46.5,51.2){$\down$}
\put(58.5,51.2){$\down$}
\put(70.5,51.2){$\down$}
\put(83.6,48.2){$\circ$}
\put(94.5,46.5){$\up$}
\put(106.5,46.5){$\up$}
\put(118.5,46.5){$\up$}
\put(130.5,51.2){$\down$}
\put(142.5,51.2){$\down$}
\put(202.5,-3.5){$\up$}
\put(214,-.8){$\cross$}
\put(226.6,-1.8){$\circ$}
\put(238.5,1.2){$\down$}
\put(250.5,-3.5){$\up$}
\put(263.6,-1.8){$\circ$}
\put(274.5,1.2){$\down$}
\put(286.5,-3.5){$\up$}
\put(298.5,1.2){$\down$}
\put(310.5,-3.5){$\up$}
\put(322.5,-3.5){$\up$}
\put(202.5,46.5){$\up$}
\put(214.5,51.2){$\down$}
\put(226.5,51.2){$\down$}
\put(238.5,51.2){$\down$}
\put(250.5,51.2){$\down$}
\put(263.6,48.2){$\circ$}
\put(274.5,46.5){$\up$}
\put(286.5,46.5){$\up$}
\put(298.5,46.5){$\up$}
\put(310.5,46.5){$\up$}
\put(322.5,46.5){$\up$}
\put(25.2,1){\line(0,1){50}}
\put(145.2,1){\line(0,1){50}}
\put(205.2,1){\line(0,1){50}}
\put(325.2,1){\line(0,1){50}}
\put(67.2,1){\oval(12,20)[t]}
\put(127.2,1){\oval(12,20)[t]}
\put(247.2,1){\oval(12,20)[t]}
\put(307.2,1){\oval(12,20)[t]}
\put(265.4,51){\oval(24,20)[b]}
\put(265.4,51){\oval(48,33)[b]}
\put(265.4,51){\oval(72,47)[b]}
\put(85.4,51){\oval(24,20)[b]}
\put(85.4,51){\oval(48,33)[b]}
\put(85.4,51){\oval(72,47)[b]}
\qbezier(109.3,1)(109.3,20)(117,24)
\qbezier(117,24)(133.3,32)(133.3,51)
\qbezier(289.3,1)(289.3,20)(297,24)
\qbezier(297,24)(313.3,32)(313.3,51)
\qbezier(97.3,1)(97.3,15)(65,20)
\qbezier(65,20)(40,25)(37.3,51)
\qbezier(277.3,1)(277.3,15)(245,20)
\qbezier(245,20)(220,25)(217.3,51)
\end{picture}
\end{align*}
In general the map (\ref{crazymap}) is neither injective nor surjective.

Continuing with $\la, \la^\dagger$ and $t$ as in the previous paragraph, 
let $\Ga$ be the block containing $\zeta$, let $\Delta$ be the
block
containing $\la^\dagger$, and define
\begin{equation}\label{peiple}
R(\la) := G^t_{\De\Ga} L(\zeta) \in \Mod{K(m|n)},
\end{equation}
where $G^t_{\De\Ga}$ is the indecomposable projective functor from
(\ref{legg1}).
This is a finite dimensional 
graded $K(m|n)$-module. Theorem~\ref{ipf} (and also \cite[Theorem 4.11]{BS2})
gives detailed information about the structure of 
this module. In particular the last
part of Theorem~\ref{ipf} implies the following lemma, recalling that the
defect $\defect(\la)$ is the number of caps in $\overline{\la}$.

\begin{Lemma}
For an $(m,n)$-cross bipartition $\la$, the graded $K(m|n)$-module
$R(\la)$ is indecomposable with irreducible head isomorphic
to $L(\la^\dagger) \langle - \defect(\la)\rangle$ and irreducible socle
isomorphic
to $L(\la^\dagger) \langle \defect(\la) \rangle$.
\end{Lemma}

For the final theorem, we transport the $K(m|n)$-module 
$R(\la)$ through the equivalence of categories from Theorem~\ref{BS4a} to obtain 
a finite dimensional $\mathfrak{g}$-supermodule, 
which we also denote by
$R(\la)$.

\begin{Theorem}\label{soupedup}
Up to isomorphism, the 
indecomposable summands of the mixed tensor
space
$V^{\otimes r} \otimes W^{\otimes s}$ are the $\mathfrak{g}$-supermodules
$$
\{R(\la)\:|\:\text{for all $(m,n)$-cross bipartitions }\la \in \dot
\La_{r,s}\}.
$$
More precisely, given any primitive idempotent
$e \in B_{r,s}(\delta)$,
let $\la \in \dot \La_{r,s}$
be the bipartition labelling the irreducible head of the projective
indecomposable module
$B_{r,s}(\delta) e$.
Then $e$ is non-zero on $V^{\otimes r} \otimes W^{\otimes s}$ if and
only if
$\la$ is an $(m,n)$-cross bipartition, in which case
$(V^{\otimes r} \otimes W^{\otimes s}) e \cong R(\lambda)$ as $\mathfrak{g}$-supermodules.
\end{Theorem}

\begin{proof}
The endomorphism algebra $\End_{\mathfrak{g}}(V^{\otimes r} \otimes
W^{\otimes s})^{\op}$ is identified with the quotient algebra
$B_{r,s}(\delta) / \ker \Psi^{m,n}_{r,s}$ thanks to
Theorem~\ref{main1}.
Hence the isomorphism classes of 
indecomposable summands of $V^{\otimes r} \otimes
W^{\otimes s}$ are parametrised by the isomorphism classes of
irreducible $B_{r,s}(\delta) / \ker \Psi^{m,n}_{r,s}$-modules, which
we determined in Theorem~\ref{irrclass2}. We deduce 
for a primitive idempotent $e \in B_{r,s}(\delta)$
of type $\la \in \dot \La_{r,s}$ as in
the statement of the theorem that $(V^{\otimes r} \otimes W^{\otimes
  s}) e = \{0\}$ unless $\la$ is an $(m,n)$-cross-bipartition, in
which
case $(V^{\otimes r} \otimes W^{\otimes s}) e$ is a representative for
the class of
indecomposable summands parametrised by $\la$.

It remains to show that $(V^{\otimes r} \otimes W^{\otimes s}) e
\cong R(\la)$.
Pick any $R \in \Seq_{r,s}$ and let $e(\bi) \in B_{r,s}(\delta)$
denote the corresponding weight idempotents.
Replacing $e$ by a conjugate if
necessary, we may assume that $e$ is contained in 
$e(\bi)$ for some $\bi \in \Z^{r+s}$, i.e. $e(\bi) =  e + f$
for some idempotent $f$ orthogonal to $e$.
By Lemma~\ref{idcorr}, we know that 
$e(\bi)$ is isotypic of type $\la$. Now we need to show that
$(V^{\otimes r} \otimes W^{\otimes s}) e(\bi)$ is isomorphic to a
direct sum of copies of $R(\la)$.
For this we observe from
Theorem~\ref{iso1}
that $(V^{\otimes r} \otimes W^{\otimes s}) e(\bi)
\cong \lonestar R_{\bi} \C$ as $\mathfrak{g}$-supermodules.
Transporting to $K(m|n)$ using Theorem~\ref{BS4a} we are reduced to
showing
that all indecomposable summands of
$\lonestar R_{\bi} L(\zeta)$ are isomorphic to $R(\lambda)$ (up to
degree shift).
Finally observe using Lemma~\ref{composition} that 
$\lonestar R_{\bi}$ is a direct sum of copies
of the indecomposable projective functor $G^t_{\De\Ga}$
appearing in
the definition (\ref{peiple}).
\end{proof}

In particular Theorem~\ref{soupedup} implies that the indecomposable summands of
the mixed tensor spaces 
$V^{\otimes r} \otimes W^{\otimes s}$ for all $r, s \geq 0$ are
parametrised by the $(m,n)$-cross bipartitions, exactly as in \cite{CW}.

\end{document}